\documentclass[12pt]{amsart}
\usepackage{amssymb,paralist,url,amscd,euscript,mathrsfs,stmaryrd}
\usepackage[centering,text={15.5cm,24cm}]{geometry}

\usepackage[all]{xy}
\SelectTips{cm}{}
\usepackage[dvips,final]{graphicx}
\usepackage[x11names]{xcolor} 
\usepackage{epsfig}
\usepackage{latexsym}
\usepackage{paralist}
\usepackage{xfrac}
\usepackage{faktor}
\usepackage[normalem]{ulem}
\usepackage{contour} 
\usepackage{footmisc}
\usepackage{enumitem}
\setlist[enumerate]{itemindent=0ex,labelwidth=2ex,labelsep=1ex,%
leftmargin=6ex,itemsep=0.5ex}
\usepackage{lineno}

\usepackage{hyperref}
\hypersetup{
  colorlinks   = true,          
  urlcolor     = violet,        
  linkcolor    = blue,          
  citecolor    = violet         
}

\makeatletter

\def\l@paragraph{\@tocline{4}{0pt}{1pc}{7pc}{}}
\def\l@subparagraph{\@tocline{5}{0pt}{1pc}{7pc}{}}
\makeatother

\contourlength{0.8pt}

\numberwithin{equation}{section}

\numberwithin{subsection}{section}

\allowdisplaybreaks[1]

\newtheorem*{namedtheorem}{\theoremname}
\newcommand{\theoremname}{testing}
\newenvironment{named}[1]{\renewcommand\theoremname{#1}
\begin{namedtheorem}}
{\end{namedtheorem}}

\theoremstyle{plain}

\newtheorem{theorem}{Theorem}[section]
\newtheorem{proposition}[theorem]{Proposition}
\newtheorem{proposition-definition}[theorem]{Proposition-Definition}
\newtheorem{lemma-definition}[theorem]{Lemma-Definition}
\newtheorem{corollary}[theorem]{Corollary}
\newtheorem{lemma}[theorem]{Lemma}

\theoremstyle{definition}
\newtheorem{definition}[theorem]{Definition}

\newtheorem{example}[theorem]{Example}
\newtheorem{examples}[theorem]{Examples}
\newtheorem{remark}[theorem]{Remark}

\newtheorem*{remark*}{Remark}

\theoremstyle{remark}

\newcommand\ul[1]{\underline{#1}}

\newcommand\fs{{\operatorname{fs}}}
\newcommand\fom{\mathfrak{m}} 

\newcommand\ocM{\overline{\mathcal{M}}}

\newcommand\ocK{\overline{\mathcal{K}}}
\newcommand\ocI{\overline{\mathcal{I}}}
\newcommand\ocN{\overline{\mathcal{N}}}

\newcommand\cl{\operatorname{cl}}
\newcommand\scrC{\mathscr{C}}
\newcommand\scrM{\mathscr{M}}

\newcommand\Int{\operatorname{Int}}

\newcommand{\ev}{\mathrm{ev}}
\newcommand{\spl}{\mathrm{spl}}
\newcommand{\pun}{\mathrm{p}}

\newcommand\basic{\mathrm{bas}}
\newcommand\gl{\mathrm{gl}}

\newcommand\cA{\mathcal{A}}

\newcommand\cE{\mathcal{E}}
\newcommand\cF{\mathcal{F}}

\newcommand\cI{\mathcal{I}}
\newcommand\cJ{\mathcal{J}}
\newcommand\cK{\mathcal{K}}
\newcommand\cL{\mathcal{L}}
\newcommand\cM{\mathcal{M}}
\newcommand\cN{\mathcal{N}}
\newcommand\cO{\mathcal{O}}
\newcommand\cP{\mathcal{P}}

\newcommand\cT{\mathcal{T}}

\newcommand\cV{\mathcal{V}}

\newcommand\cX{\mathcal{X}}

\newcommand\cZ{\mathcal{Z}}

\newcommand\uC{\underline{C}}

\newcommand\uY{\underline{Y}}

\renewcommand\AA{\mathbb{A}}

\newcommand\CC{\mathbb{C}}

\newcommand\EE{\mathbb{E}}
\newcommand\FF{\mathbb{F}}
\newcommand\GG{\mathbb{G}}

\newcommand\kk{\Bbbk}
\newcommand\LL{\mathbb{L}}

\newcommand\NN{\mathbb{N}}

\newcommand\PP{\mathbb{P}}
\newcommand\QQ{\mathbb{Q}}
\newcommand\RR{\mathbb{R}}

\newcommand\ZZ{\mathbb{Z}}

\newcommand\bA{\mathbf{A}}

\newcommand\bE{\mathbf{E}}

\newcommand\bL{\mathbf{L}}
\newcommand\bM{\mathbf{M}}

\newcommand\bS{\mathbf{S}}

\newcommand\bg{\mathbf{g}}
\newcommand\bp{\mathbf{p}}
\newcommand\bu{\mathbf{u}}
\newcommand\bsigma{{\boldsymbol{\sigma}}}

\newcommand\biota{{\boldsymbol{\iota}}}
\newcommand\btau{{\boldsymbol{\tau}}}

\newcommand\bfu{\mathbf{u}}

\newcommand\fC{\mathfrak{C}}

\newcommand\fM{\mathfrak{M}}
\newcommand\fN{\mathfrak{N}}

\newcommand\frf{\mathfrak{f}}

\newcommand\tfM{\widetilde{\mathfrak{M}}}
\newcommand\tfC{\widetilde{\mathfrak{C}}}

\newcommand{\rk}{\operatorname{rk}}
\newcommand{\Sets}{\mathbf{Sets}}
\newcommand{\Cat}{\mathbf{Cat}}

\newcommand{\Cones}{\mathbf{Cones}}
\newcommand{\Mon}{\mathbf{Mon}}
\newcommand{\Sch}{\mathbf{Sch}}
\newcommand\arr{\ifinner\to\else\longrightarrow\fi}
\newcommand\larr{\longrightarrow}

\newcommand\op{^{\mathrm{op}}}

\newcommand\im{\operatorname{im}}

\def\displaytimes_#1{\mathrel{\mathop{\times}\limits_{#1}}}
\def\displayotimes_#1{\mathrel{\mathop{\bigotimes}\limits_{#1}}}

\newcommand\ext{\operatorname{Ext}}

\newcommand\Aut{\operatorname{Aut}}
\newcommand\Map{\operatorname{Map}}
\newcommand\Func{\operatorname{Func}}

\newcommand\sing{\mathrm{sing}}

\newcommand\pic{\operatorname{Pic}}
\newcommand\spec{\operatorname{Spec}}
\newcommand\Spec{\operatorname{Spec}}
\newcommand\bSpec{\operatorname{\textbf{Spec}}}
\newcommand\Sh{\operatorname{\textbf{Sh}}}

\newcommand\rank{\operatorname{rank}}
\newcommand\virt{{\operatorname{virt}}}
\newcommand\id{\mathrm{id}}
\newcommand\pr{\operatorname{pr}}
\newcommand\indlim{\varinjlim}

\newcommand\colim{{\operatorname{colim}}}

\newcommand\tr{\operatorname{tr}}
\newcommand\Pt{\operatorname{\mathbf{Pt}}}
\newcommand\Strata{\operatorname{Strata}}
\newcommand\Stalks{\operatorname{Stalks}}

\newdir{ >}{{}*!/-5pt/@{>}}

\newcommand\doublelong[2]{\mathbin{\xymatrix{{}\ar@<3pt>[r]^{#1}
\ar@<-3pt>[r]_{#2}&}}}

\newlength{\ignora}

\newcommand{\red}{{\mathrm{red}}}
\newcommand{\bmu}{{\boldsymbol{\mu}}}
\newcommand{\et}{{_{\text{\'et}}}}

\newcommand{\ol}{\overline}
\newcommand{\Log}{{\operatorname{Log}}}

\newcommand{\gp}{{\mathrm{gp}}}

\newcommand{\sat}{{\mathrm{sat}}}

\renewcommand{\setminus}{\smallsetminus}

\newcommand{\Hom}{{\operatorname{Hom}}}
\newcommand{\cHom}{{\mathcal{H}om}}

\begin{document}

\title{Punctured logarithmic maps}


\author{Dan Abramovich}
\address{\tiny Department of Mathematics, Brown University, Box 1917,
Providence, RI~02912, USA}
\email{dan\_abramovich@brown.edu}

\author{Qile Chen}
\address{\tiny Department of Mathematics, Boston College, Chestnut Hill,
MA~02467-3806, USA}
\email{qile.chen@bc.edu}

\author{Mark Gross}
\address{\tiny DPMMS, Centre for Mathematical Sciences, Wilberforce Road,
Cambridge, CB3 0WB, UK}
\email{mgross@dpmms.cam.ac.uk}

\author{Bernd Siebert}
\address{\tiny Department of Mathematics, The Univ.\ of Texas at Austin,
2515 Speedway, Austin, TX 78712, USA}
\email{siebert@math.utexas.edu}

\date{\today}
\begin{abstract}
We introduce a variant of stable logarithmic maps, which we call \emph{punctured
logarithmic maps}. They allow an extension of logarithmic Gromov-Witten theory
in which marked points have a negative order of tangency with boundary divisors.

As a main application we develop a gluing formalism which reconstructs stable
logarithmic maps and their virtual cycles without expansions of the target, with
tropical geometry providing the underlying combinatorics.

Punctured Gromov-Witten invariants also play a pivotal role in the intrinsic
construction of mirror partners by the last two authors, conjecturally relating
to symplectic cohomology, and in the logarithmic gauged linear sigma model in
work of Qile Chen, Felix Janda and Yongbin Ruan.
\end{abstract} 

\maketitle

\setcounter{tocdepth}{2}
\tableofcontents


\section{Introduction}

Logarithmic Gromov-Witten theory, developed by the authors in
\cite{Chen}, \cite{AC}, \cite{LogGW}, has proved a successful generalization
of the notion of relative Gromov-Witten invariants developed in
\cite{LR}, \cite{JunLi1}, \cite{JunLi2}. Relative Gromov-Witten invariants
are invariants of pairs $(X,D)$ where $X$ is a non-singular variety and $D$ is
a smooth divisor on $X$; these invariants count curves with imposed
contact orders with $D$ at marked points. Logarithmic Gromov-Witten theory
allows $D$ instead to be normal crossings, or more generally,
allows $(X,D)$ to be a toroidal crossings variety.

\addtocontents{toc}{\protect\setcounter{tocdepth}{1}}
\subsection{Scope and motivation}
The purpose of the present work is to extend logarithmic Gromov-Witten theory to
admit \emph{negative contact orders}. Working over a field $\kk$, an example for
how negative contact orders arise naturally is by restricting a normal crossings
degeneration, such as
\[
\pi:\AA^2=\Spec\kk[z,w]\arr\AA^1=\Spec\kk[t],\quad \pi^\sharp(t)=zw,
\]
to the irreducible component $C=V(w)$ of the central fiber $\pi^{-1}(0)$.
Viewing $\pi$ as a morphism of log spaces for the toric log structures on
$\AA^2$ and $\AA^1$, denote by $s_z,s_w,s_t$ the global sections of the log
structure $\cM_{\AA^2}$ induced by $z,w$ and $\pi^\sharp(t)$, respectively. The
induced log structure $\cM_C= \cM_{\AA^2}|_C$ on $C=\Spec\kk[z]$ is generated by
the restrictions of $s_z,s_w$, denoted by the same symbols. Note that the
structure morphism $\cM_C\arr \cO_C$ maps $s_z$ to $z$, which has a first order
zero at the origin $0\in\AA^1$ as given by a marked point, while $s_w$ and $s_t$
map to $0$. The point is that viewed as a log space over $\AA^1$, the equation
$s_z s_w=s_t$ implies that away from $0\in\AA^1$, {we have} 
\[
s_w= z^{-1}\cdot s_t.
\]
Such sections do not exist on log smooth curves over the standard log point. The
power of $z$ occuring in this equation reflects the negative contact order.
Since $z^{-1}$ is defined on the punctured curve $\AA^1\setminus\{0\}$, we call
the resulting extension of log smooth curves, stable logarithmic maps and
logarithmic Gromov-Witten theory \emph{punctured curves}, \emph{punctured
(logarithmic) maps} and \emph{punctured Gromov-Witten theory}.

Our motivation for studying punctured Gromov-Witten theory comes from three
sources. First, as illustrated in the example, negative contact orders arise
naturally when gluing a logarithmic stable map from its restrictions to closed
subcurves, as desired in degeneration situations \cite{decomposition}.
{Note that in transverse situations, as achieved by the expanded
degeneration technique in \cite{JunLi2}, negative contact orders can be avoided
by turning a punctured map over a standard log point into a stable logarithmic
map to an irreducible component of the target over the trivial log point; see
\cite[\S\S6,7]{G22} for details. This simplification is not possible when an
irreducible component of the curve maps into a deeper stratum. See
\cite[\S5.2.4]{decomposition} for an example where no decomposition of the
target splits any of the nodes into a pair of marked points with non-negative
contact order.} A treatment of gluing situations based on punctured maps is
contained in \S\ref{sec:split-glue}.

The second motivation comes from mirror constructions and their link to
symplectic cohomology, relating to the program on mirror symmetry of Gross and Siebert via toric degenerations. It turns out that the algorithmic
construction of mirrors via wall structures in \cite{GSAnnals} admits a vast,
intrinsic generalization by using punctured invariants \cite{GSAssoc}. Punctured
invariants are used in this context to define the structure coefficients of the
coordinate ring of the mirror degeneration, with the space of non-negative
contact orders representing generators. The structure coefficients require
punctured invariants with two positive and one negative contact order. The
gluing techniques developed in \S\ref{sec:split-glue} are the crucial
ingredient in proving associativity of the resulting multiplicative structure.
In \cite{GSWallStructures}, the gluing techniques for
punctured invariants are also crucial in constructing a consistent wall
structure in the intrinsic mirror symmetry setup, thus linking the mirror
constructions in \cite{GSAnnals} and \cite{GSAssoc} via \cite{GStheta}. Further,
building on \cite{GSWallStructures}, \cite{AG} gives an algorithmic method of
calculating certain one-pointed punctured invariants on blow-ups of toric
varieties.

Another interesting related fact is the interpretation of punctured invariants
as structure coefficients in some versions of symplectic cohomology. Thus
punctured invariants provide an algebraic-geometric path to computing otherwise
hard to compute symplectic invariants. See \cite{Seidel, Pascaleff, Arguez,
GanatraPomerleano1, GanatraPomerleano2} for some steps in this direction.

The third motivation is from work of the second author on the logarithmic gauged
linear sigma model. In the papers \cite{CJR20P, CJR21P}, punctured maps
are shown to be a key for computing the invariants of the logarithmic
gauged linear sigma model of \cite{CJR19P}.\footnote{In \cite{CJR20P},
punctured maps to a smooth boundary divisor with extra structure called R-maps,
are studied. The moduli of punctured maps provide different virtually birational
models over which effective formulas for computing higher genus Gromov-Witten
invariants hold \cite{CJR21P}. These crucial virtually birational models do not
exist as moduli of rubber maps with expansions \cite{JunLi2, GV05}.} This
provides the geometric foundation for calculating higher genus invariants of
quintic $3$-folds \cite{GJR17P, GJR18P}, and for proving Conjecture A.1 of
\cite{PPZ16P} on the cycle of holomorphic differentials \cite{CJRS20P}.

\subsection{Main features of punctured Gromov-Witten theory}

Several aspects of the theory of punctured Gromov-Witten invariants appear to be
straightforward generalizations from ordinary logarithmic Gromov-Witten theory.
The formal similarity can, however, be quite misleading. In fact, finding
the right setup and point of view took a very long time, and was only made
possible by developing the theory simultaneously with the mentioned
applications.

One major difference is the more singular and more interesting nature of the
base space for moduli spaces of punctured maps. In ordinary Gromov-Witten
theory, the natural base space is the Artin stack $\bM$ of nodal curves. While
non-separated, $\bM$ is smooth, hence is locally pure dimensional. The relative
obstruction theory of the moduli space of stable maps over $\bM$ thus produces a
virtual fundamental cycle by virtual pull-back of the fundamental class $[\bM]$.
The picture in logarithmic Gromov-Witten theory is much the same, with $\bM$ now
replaced by the stack $\fM=\Log_\bM$ of log smooth curves of the given genus and
numbers of marked points over fine saturated (fs) log schemes. This stack is log
smooth over the base field, hence is also locally pure-dimensional.

For punctured invariants, the analogue of $\fM$ is the stack $\breve\fM$ of
logarithmic curves with punctures. One crucial feature of the deformation theory
of punctured curves is that $\breve\fM$ is typically not pure-dimensional. In
fact, the map $\breve\fM \arr \bM$ forgetting the log structure turns out to be
only idealized logarithmically \'etale (Proposition~\ref{prop:puncurve-moduli}).
This means that locally in the smooth topology $\breve\fM \arr \bM$ is
isomorphic to the composition of a closed embedding defined by a monomial ideal
followed by a toric morphism of affine toric varieties with associated lattice
homomorphism an isomorphism over $\QQ$.

The induced stratified structure of punctured maps turns out to be captured by
tropical geometry. The second main feature of punctured Gromov-Witten theory is
thus the central role of tropical geometry, exceeding by far its increasingly
recognised role in logarithmic Gromov-Witten theory. Working over a base space
$B$, we first factor the log smooth target $X\arr B$ over the \emph{relative
Artin fan} {$\cX\arr B$ from \cite[Cor.\,3.3.5]{ACMW}},
an algebraic stack glued from
quotients of toric charts for $X\arr B$ by the fiberwise acting torus, see
\cite[\S2.2]{decomposition}. Working with $\cX$ as a target amounts to working
with nodal curves and compatible families of tropical maps, thus making the
theory of such punctured maps a combinatorially enriched version of the theory
of stable curves. A better base space than $\breve\fM$ to work with is then the
algebraic stack $\fM(\cX/B)$ of punctured maps to $\cX/B$. Indeed, the forgetful
map
\[
\fM(\cX/B)\arr \bM\times B
\]
is also idealized logarithmically \'etale (Theorem~\ref{thm:idealized-etale}),
but now with idealized structure easy to extract from the tropical geometry of
the situation. In particular, Remark~\ref{Rem: stratified structure of
fM(cX,tau)} gives a complete characterization of the strata of $\fM(\cX/B)$ in
terms of types of tropical maps. Due to its fundamental nature for punctured
Gromov-Witten theory, we emphasize the role of tropical geometry throughout,
including adapted presentations of material from \cite{LogGW} in
\S\ref{sec:punctured-map}, \S\ref{sec:stack} {and Appendix~\ref{App: Functorial tropicalization}.}

A third feature of punctured Gromov-Witten theory developed here, but already
relevant to ordinary logarithmic Gromov-Witten theory, is the introduction of
\emph{evaluation stacks} for imposing point conditions compatible with the
virtual formalism. Since $X\arr\cX$ is smooth in the ordinary sense, we can
choose a lift to $X$ of the image of each marked point in $\cX$ to arrive at an
algebraic stack $\fM^\ev(\cX/B)$ smooth over $\fM(\cX/B)$ and such that the
relative obstruction theory over $\fM(\cX/B)$ arises from a relative obstruction
theory over $\fM^\ev(\cX/B)$. It is this \emph{evaluation stack}
$\fM^\ev(\cX/B)$ that one needs to work with to impose conditions on the
evaluations at the marked points rather than the product
$X\times_B\ldots\times_B X$ in ordinary Gromov-Witten theory. Evaluation stacks
also play a crucial role in our gluing formalism, see \S\ref{ss:gluing at
virtual level}.

\subsection{Statements of main results}

For simplicity of the presentation of the main results we now assume $X\arr B$
to be a projective log smooth morphism of log schemes
and $B$ {the standard log point $\Spec(\NN\to\kk)$ or $B$} log smooth
over the trivial log point $\Spec\kk$, where $\kk$ is a field of
characteristic~$0$. For more detailed statements we refer to the main body of
the text.

Similar to logarithmic Gromov-Witten theory, as presented in
\cite[\S2.5]{decomposition}, we introduce (decorated) types of punctured maps
and of (families of) tropical maps. Types restrict the combinatorics of
punctured maps over geometric points as seen by their tropicalizations, such as
the dual intersection graph, the contact orders of punctured and nodal points
and the genera and the curve classes of its irreducible components
(Definition~\ref{Def: type}). An appropriate global version of contact orders
developed in \S\ref{subsec: global contact} leads to the notion of global
decorated type $\btau$ that can be used to define moduli spaces of punctured
maps \emph{marked by $\btau$} (Definition~\ref{Def: stacks of decorated puncted
maps}). Theorem~\ref{Thm: scrM and fM are algebraic}, Corollary~\ref{Cor: scrM/B
is proper} and Theorem~\ref{thm:idealized-etale} establish the basic properties
of these moduli spaces, which can be summarized as follows.

\begin{named}{Theorem~A}
Let $\btau$ be a decorated global type. Then the stacks $\fM(\cX/B,\btau)$ and
$\scrM(X/B,\btau)$ of $\btau$-decorated basic stable punctured maps
(Definition~\ref{Def: stacks of decorated puncted maps}) to $\cX\to B$ and to
$X\to B$, respectively, are logarithmic algebraic stacks. Moreover,
{$\scrM(X/B,\btau)$ is Deligne-Mumford and proper over
$B$.\footnote{{We prove properness under the technical assumption that
the log structure on $X$ is Zariski and $\ocM_X^\gp$ is globally generated. The
latter assumption has meanwhile been removed in \cite{Johnston}.}} If in
addition $X$ is simple \cite[Def.\,2.1]{decomposition},} $\fM(\cX/B,\btau)$ is
idealized smooth over $B$.
\end{named}

For the definition of punctured Gromov-Witten invariants we work over the
evaluation stacks $\fM^\ev(\cX/B,\tau)$, which lift the evaluations at a set
$\bS$ of punctured and nodal points from $\cX$ to $X$ (Definition~\ref{Def:
evaluation stack}). We suppress $\bS$ in the notation. The following theorem
summarizes Proposition~\ref{Prop: Obstruction theory for morphisms} and
Proposition~\ref{Prop: Point condition virtual bundle}.

\begin{named}{Theorem~B}
For $\btau$ a decorated global type let $\pi: \scrC(X/B,\btau)\arr
\scrM(X/B,\btau)$ and $f:\scrC(X/B,\btau)\arr X$ be the universal curve and
universal punctured map over the moduli space $\scrM(X/B,\btau)$ of
$\btau$-marked basic punctured maps to $X\arr B$. Let $\fM^\ev(\cX/B,\btau)$
be the corresponding evaluation stack for a subset $\bS$ of
punctured sections, and $Z\subset \ul \scrC(X/B,\btau)$ the closed
substack defined by the union of the images of these sections. Then there is a
canonical perfect relative obstruction theory
\[
\GG\simeq \big(R\pi_* f^*\Theta_{X/B}(-Z)\big)^\vee\arr \LL_{\scrM(X/B,\btau)/ \fM^\ev(\cX/B,\btau)}
\]
for the natural morphism $\varepsilon:\scrM(X/B,\btau)\arr
\fM^\ev(\cX/B,\btau)$. {Here $\Theta_{X/B}$ denotes the
logarithmic tangent bundle of $X$ over $B$.}
\end{named}

A similar statement holds if $\bS$ also contains nodal sections, see
Proposition~\ref{Prop: Point condition virtual bundle}. Virtual pull-back
\cite{Mano} now provides punctured Gromov-Witten invariants, with the basic
correspondence the homomorphism
\[
\textstyle
(\ul\ev\times p)_*\varepsilon_\GG^!: A_*\big( \fM^\ev(\cX/B,\btau)\big)
\arr A_{*+d(g,k,A,n)}\big(\prod_L \ul Z_L\times \scrM_{g,k}\big)
\]
on rational Chow groups defined in Definition~\ref{Def: Punctured GW
correspondence}. Here each $\ul Z_L$ is an evaluation stratum, the closed
stratum of $\ul X$ that evaluation at a punctured section maps to by
the choice of decorated global type $\btau$, and $\scrM_{g,k}$ is the
Deligne-Mumford stack of $k$-marked stable curves of genus $g$. A formula for
the relative virtual dimension $d(g,k,A,n)$ is stated in \eqref{Eqn: relative
virtual dimension}.

The most challenging part of this paper was an efficient and practically useable
treatment of gluing. While the final results may look straightforward, they rely
on a number of careful choices and subtle points which became clear to us only
after a long series of futile attempts.\footnote{For some time our formalism
only worked for {gluing problems appearing in} certain mirror
constructions. We emphasize that the final results below are general and
practical, as demonstrated in \cite{Yixian}.} Here we only summarize the results
and refer to Remark~\ref{Rem: History of gluing} for some further discussion.

The formal setup for gluing takes a decorated global type $\btau$ and splits the
graph underlying $\btau$ at a subset of edges, leading to a set
$\{\btau_1,\ldots,\btau_r$\} of decorated global types, with each split edge now
producing a pair of legs in the graphs for the $\btau_i$. Our first result on
gluing reduces all gluing questions to the unobstructed evaluation stacks, as
proved in Proposition~\ref{Prop: Gluing via evaluation spaces} and
Theorem~\ref{Thm: Virtual gluing via evaluation spaces2}.

\begin{named}{Theorem~C}
There is a cartesian diagram 
\[
\vcenter{\xymatrix@C=30pt
{
\scrM(X/B,\btau) \ar[r]^(.45){\delta}\ar[d]_{\varepsilon} &
\prod_{i=1}^r \scrM(X/B,\btau_i)\ar[d]^{\hat\varepsilon=\prod_i\varepsilon_i}\\
\fM^\ev(\cX/B,\tau) \ar[r]^(.45){\delta^\ev} &
\prod_{i=1}^r \fM^\ev(\cX/B,\tau_i)
}}
\]
with horizontal arrows the splitting maps from
Proposition~\ref{prop:splitting-map}, finite and representable by
Corollary~\ref{Cor: Gluing theorem for evaluation stacks}, and the vertical
arrows the canonical strict morphisms.\footnote{There is an entirely equivalent
formalism allowing for disconnected punctured curves and disconnected types, in
which case the products on the right form a single moduli stack corresponding to
a disconnected decorated global type $\btau$.}

Moreover, if $\hat\varepsilon^!$ and $\varepsilon^!$ denote Manolache's virtual
pull-back defined using the two given obstruction theories for the vertical
arrows, then for $\alpha\in A_*\left(\fM^\ev(\cX/B,\tau)\right)$, we have
the identity
\[
\hat\varepsilon^!\delta^\ev_*(\alpha)=
\delta_*\varepsilon^!(\alpha)\,.
\]
\end{named}

A numerical gluing formula follows from Theorem~C for those Chow classes
$\alpha$ on $\fM^\ev(\cX/B,\btau)$ such that $\delta^\ev_*$ can be written as a
sum of products, see Corollary~\ref{Cor: decomposition into
product}.\footnote{Conversely, if there is no such decomposition, a numerical
gluing formula cannot be achieved within Chow theory --- a phenomenon already
present in the classical case of a smooth gluing locus. A generally applicable
gluing formula should therefore require working with a homology theory with a
K\"unneth decomposition. It is possible that the formalism of
virtual pull-back in Borel-Moore homology developed in \cite{KV} may
be useful.} A straightforward consequence of Theorem~C is a gluing
formula for the degeneration situation from \cite{decomposition}, see
Corollary~\ref{Cor: gluing formula for degenerations}.

Our second result on gluing provides a description of the splitting map
$\delta^\ev$ in terms of an fs-fiber diagram, which apart from proving the
properties stated in Theorem~C, provides a route to using Theorem~C in explicit
computations. For the following statement the log stacks $\fM^\ev(\cX/B,\btau)$,
$\fM^\ev(\cX/B,\btau_i)$ have to be replaced by closely related log stacks
$\tfM'^\ev(\cX/B,\btau)$ and $\tfM'^\ev(\cX/B,\btau_i)$, which however have the
same underlying reduced stacks, hence have identical Chow theories
(Proposition~\ref{Prop: tilde log structures don't change
reductions}).\footnote{These stacks are closely related to Parker's moduli of
cut curves introduced in \cite{Parker: log}.} These stacks come with logarithmic
evaluation morphisms such as $\ev_\bE: \tfM'^\ev(\cX/B,\btau)\arr
\prod_{E\in\bE} X$, where $\bE$ is the set of nodal sections to split. The
following is Corollary~\ref{Cor: Gluing theorem for evaluation stacks} to which
we refer for more details.

\begin{named}{Theorem~D}
The splitting morphism $\delta^\ev: \tfM'^\ev(\cX/B,\btau)\to
\prod_i\tfM'^\ev(\cX/B,\btau_i)$ fits into the cartesian diagram
\[
\xymatrix{
\tfM'^\ev(\cX/B,\btau) \ar[r]^(.45){\delta^\ev}\ar[d]_{\ev_\bE} &
\prod_{i=1}^r \tfM'^\ev(\cX/B,\btau_i)\ar[d]^{\ev_\bL}\\
\prod_{E\in \bE} X\ar[r]^(.45)\Delta& \prod_{E\in \bE} X\times_B X 
}
\]
of fs log stacks. Here $\Delta$ restricts to the diagonal morphism $X\arr
X\times_B X$ on each factor.
\end{named}

We emphasize that the diagram in Theorem~D is typically not cartesian on the
level of underlying stacks due to the more subtle nature of fs fiber products.
Theorem~D is nevertheless a powerful tool for explicit computations. For
example, under the assumption of toric gluing strata, Yixian Wu in \cite{Yixian}
derives from Theorem~D an explicit decomposition of the term
$\delta^\ev_*[\fM^\ev(\cX/B,\btau)]$ appearing in Thereom~C in terms of the
strata in $\prod_i \fM(\cX/B,\btau_i)$.

\subsection{Organization of the paper}

\S\ref{sec:punctured-map} contains the basic definitions and related concepts
concerning punctured curves and punctured maps, with \S\ref{Sect: Punctured maps
definitions} introducing pre-stable and stable punctured maps, and
\S\ref{sec:tropical interpretation} along with Appendix~\ref{App: Functorial
tropicalization} the tropical language, including the definition of types and
the modified balancing condition. The subject of \S\ref{subsec: global contact}
is the discussion of contact orders in a simplified version sufficient for most
applications, {and the associated notion of global types}. The more
involved general picture {concerning contact orders} is discussed in
Appendix~\ref{sec:contact}. \S\ref{ss: Basicness} presents the concept of
basicness for punctured maps, which while largely the same as for logarithmic
stable maps, emphasizes the tropical point of view, and hence might be of some
independent interest. \S\ref{ss:log-ideal} introduces the new phenomenon of
puncturing log ideals that each family of punctured curves or punctured maps
comes with. {\S\ref{ss: targets with monodromy} discusses the
generalization to targets with monodromy.}

\S\ref{sec:stack} deals with the moduli theory of punctured maps, proving
Theorem~A among other things. \S\ref{ss: stacks of punctured curves}
introduces stacks of punctured curves, with the main result the idealized
smoothness statement in Proposition~\ref{prop:puncurve-moduli}, followed in
\S\ref{ss: stacks marked by types} by definitions of various stacks of
punctured maps marked by types. \S\ref{ss:boundedness} and \S\ref{ss:stable
reduction} establish properness of the moduli spaces to projective targets by
adapting the boundedness and stable reduction theorems from \cite{LogGW}. The
topic of \S\ref{ss:idealized smoothness} is the idealized smoothness of the
spaces $\fM(\cX/B,\btau)$ and the induced stratified structure, all
characterized in terms of tropical geometry.

\S\ref{sec:obstruction} deals with obstruction theories, using the
approach of \cite{Behrend-Fantechi}. \S\ref{ss:Obstruction theories for pairs}
gives a careful treatment of functoriality as well as of compatibility of
obstruction theories for maps of pairs. \S\ref{ss: Obstruction theories with
point conditions} applies this discussion to construct the desired relative
obstruction theory for $\scrM(X/B,\btau)\arr \fM^\ev(\cX/B,\btau)$, in particular
proving Theorem~B. Another application of this discussion is to the virtual
treatment of gluing, presented in \S\ref{ss:gluing at virtual level}. The
section ends with the definition of punctured Gromov-Witten invariants in
\S\ref{ss: punctured GW invts}.

The last section \S\ref{sec:split-glue} contains our results on gluing.
\S\ref{ss:splitting-general} introduces the splitting morphism, while
\S\ref{ss:Gluing punctured maps to cX} treats the converse operation of
gluing, essentially proving Theorem~D, maybe the hardest single result in the
paper with a very long genesis. The virtual treatment of gluing, proving
Theorem~C, is the objective of \S\ref{ss:gluing at virtual level}. The last
section \S\ref{sec:degeneration gluing} applies our results to the
degeneration situation of \cite{decomposition}.

\subsection{Relation to other work}

We end this introduction by discussing related work. First, our approach owes a
great deal to Brett Parker's program of exploded manifolds, \cite{Parker,
Parker: reg, Parker: cmp, Parker: Kuranishi, Parker: vfc, Parker: gluing}. We
have often found ourselves trying to translate Parker's results in the category
of exploded manifolds into the category of log schemes. Indeed, some of the
original versions of the definition of punctured invariants, as well as the
approach to gluing, arose after discussions with Parker.

A gluing formula for logarithmic Gromov-Witten invariants without expansions in
the case of a degeneration $X_0$ with smooth singular locus is due to Kim, Lho
and Ruddat \cite{KimLhoRuddat}. This case does not require punctured invariants
or evaluation spaces, but is otherwise close in spirit to the present treatment.
A gluing formula in a special case for certain rigid analytic Gromov-Witten
invariants has been proved by Tony Yu \cite{Yu}.

After the earlier manuscript version of this paper was distributed, Mohammed
Tehrani \cite{Tehrani}, in developing a symplectic analogue of stable
logarithmic maps, found that punctures were naturally described in the theory.
Even more recently, \cite{FWY1,FWY2} defined negative contact order
Gromov-Witten invariants by a limiting version of orbifold Gromov-Witten
invariants. However, as observed by Dhruv Ranganathan, the invariants
as currently defined cannot coincide with logarithmic invariants as
they do not satisfy the correct invariance properties under log \'etale
modifications.
{Work of Battistella, Nabijou and Ranganathan
\cite{BNR1} takes this into account and shows how genus zero logarithmic
invariants can be recovered from the orbifold invariants after sufficient
blowing up. Further work in preparation \cite{BNR2} considers the case
of negative contact orders in the orbifold theory, and makes a somewhat
more subtle comparison which involves the puncturing ideal defined in
\S\ref{ss:log-ideal}. We send the
reader to those papers for more details.}

Besides the immediate applications of punctures already mentioned above,
punctures also have been used by H\"ulya Arg\"uz in \cite{Arguez} to build a
logarithmic analogue of certain tropical objects in the Tate elliptic curve
related to Floer theory.

Finally, we also mention recent work of Dhruv Ranganathan \cite{Ranganathan}
taking a different point of view on gluing in log Gromov-Witten theory using an
approach closer in spirit to the expanded degeneration picture of Jun Li.

\subsection{Acknowledgements}
Research by D.A. was supported in part by NSF grants DMS-1162367, DMS-1500525,
DMS-1759514, and DMS-2100548.

Research by Q.C. was supported in part by the Simons Foundation, NSF grant DMS-1403271, DMS-1560830,
DMS-1700682, and DMS-2001089.

M.G. was supported by NSF grant DMS-1262531, EPSRC grant EP/N03189X/1, a Royal
Society Wolfson Research Merit Award, and ERC Advanced Grant MSAG.

Research by B.S. was partially supported by NSF grant DMS-1903437.

We would like to thank Dhruv Ranganathan and Brett Parker for many useful
conversations, Barbara Fantechi for discussions on obstruction
theories with point conditions, {and Jonathan Wise for providing Example~\ref{Expl: Jonathan's example}.

We thank one of the anonymous referees for their detailed and insightful
comments on a previous version of this paper.}

\subsection{Conventions}
\label{Sec:convention}
All schemes and stacks are defined over an algebraically closed field $\kk$ of
characteristic~$0$. {By a logarithmic algebraic stack, we mean an
algebraic stack equipped with a log structure.} We follow the convention that if
$X$ is a log scheme or stack, then $\ul{X}$ is the underlying scheme or stack.
To unclutter notation, we nevertheless often write $\cO_X$ instead of $\cO_{\ul
X}$, and $f^*$ instead of $\ul f^*$ for the pull-back by the schematic morphism
underlying a log morphism $f:X\arr Y$. Unless stated otherwise, $\cM_X$ denotes
the sheaf of monoids on $X$, and $\alpha_X:\cM_X \rightarrow\cO_X$ the structure
map. The affine log scheme with a global chart defined by a homomorphism $Q\arr
R$ from a monoid $Q$ to a ring $R$ is denoted $\Spec(Q\arr R)$. We use the
notations $X\times_Z Y$, $X\times^{\mathrm{f}}_Z Y$, $X\times^\fs_Z Y$ to
distinguish the fiber products in the category of all log schemes, and in the
fine or the fine and saturated categories, respectively.

{Throughout $B$ denotes either a log point $\Spec (Q\arr \kk)$ with $Q$
a toric monoid with $Q^\times=0$, or an fs log scheme log smooth over
$\Spec\kk$.}\footnote{{We only use these assumptions in the proof of
Theorem~\ref{thm:idealized-etale} to assure that the reduced logarithmic strata
are defined by logarithmic ideals. This theorem is at the heart of everything we
do involving moduli spaces of punctured maps marked by a type.}}

A \emph{nodal curve} over a scheme $\ul W$ is a proper flat morphism $\ul C\arr
\ul W$ with all geometric fibers reduced of dimension one and with at worst
nodes as singularities. A \emph{pre-stable curve} is a nodal curve with all
geometric fibers connected.

The category of rational polyhedral cones from \cite[\S2.1]{decomposition} is
denoted $\Cones$. An object $\sigma$ of $\Cones$ comes with a lattice $N_\sigma$
with $\sigma\subseteq (N_\sigma)_\RR= N_\sigma\otimes_\ZZ \RR$, and we denote by
$\sigma_\ZZ = \sigma\cap N_\sigma$ the submonoid of integral points of $\sigma$.
If $P$ is a monoid, we write $P^{\vee}:=\Hom(P,\NN)$, $P^*=\Hom(P,\ZZ)$,
and $P_\RR^\vee=\Hom(P,\RR_{\ge0})$. {We write $P^{\times}$ for
the group of invertible elements of $P$. We write $\kk[P]$ for the
monoid ring of $P$ with coefficients in the field $\kk$, with $\kk$-basis
consisting of symbols $z^p$ for $p\in P$.}


\addtocontents{toc}{\protect\setcounter{tocdepth}{2}}

\section{Punctured maps}
\label{sec:punctured-map}

\subsection{Definitions}
\label{Sect: Punctured maps definitions}

\subsubsection{Puncturing}\label{sss:puncturing}

\begin{definition}
\label{def:puncturing}
Let $Y = (\ul{Y}, \cM_Y)$ be a fine and saturated logarithmic scheme with a
decomposition $\cM_Y= \cM\oplus_{\cO^{\times}}\cP$. A \emph{puncturing} of $Y$
along $\cP \subset \cM_Y$ is a fine sub-sheaf of monoids
\[
\cM_{Y^{\circ}} \subset \cM\oplus_{\cO^{\times}}\cP^{\gp}
\]
containing $\cM_Y$ 
with a structure map $\alpha_{{Y^\circ}}:\cM_{Y^\circ}
\rightarrow \cO_Y$ such that
\begin{enumerate}
\item
The inclusion $\pun^{\flat}: \cM_Y \to \cM_{Y^\circ}$ is a morphism of
logarithmic structures on $\ul{Y}$.
\item
For any geometric point $\ol x$ of $\ul{Y}$ let $s_{\ol x}\in
\cM_{Y^\circ,\ol x}$ be such that $s_{\ol x}\not \in
\cM_{\ol{x}}\oplus_{\cO^{\times}}\cP_{\ol{x}}$. Representing $s_{\ol
x}=(m_{\ol x},p_{\ol x})\in \cM_{\ol x}\oplus_{\cO^{\times}} \cP_{\ol
x}^{\gp}$, we have $\alpha_{{Y^\circ}}(s_{\ol x})=\alpha_{\cM}(m_{\ol x})=0$
in $\cO_{Y,\ol x}$.
\end{enumerate}
Denote by $Y^{\circ} = (\ul{Y}, \cM_{Y^\circ})$. We will also call the induced
morphism of logarithmic schemes $\pun: Y^{\circ} \to Y$ a \emph{puncturing} of
$Y$ along $\cP$, or call $Y^{\circ}$ a \emph{puncturing} of $Y$ along $\cP$.

We say the puncturing is \emph{trivial} if $\pun$ is an isomorphism.

\begin{figure}[htb]
\input{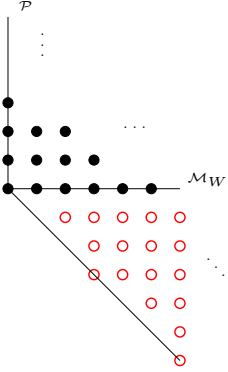}
\caption{A puncturing $Y^\circ$ of a monoid $\cM=\cM_W$. Note that the part with
negative projection in $\cP^{\gp}$ (open circles) is not necessarily
saturated.}
\label{fig:puncturing}
\end{figure}

\end{definition}

\begin{remark}
\label{rem:DF1}
In all examples in this paper, $\cP$ is a $DF(1)$ log structure, that is, there
is a surjective sheaf homomorphism $\ul{\NN}\rightarrow \overline\cP$. In this
case the condition $\alpha_{\cM}(m_{\ol x})=0$ is redundant. Indeed, for
$s_{\ol x}=(m_{\ol x}, p_{\ol x}) \not\in \cM_{\ol x}\oplus_{\cO^{\times}}
\cP$, suppose $\alpha_{{Y^\circ}}(s_{\ol x})=0$. Note that the $DF(1)$
assumption implies that $p_{\ol x}^{-1}\in
\cP_{\ol x}$, so $\alpha_{\cM}(m_{\ol x})=\alpha_Y(m_{\ol x},1)=
\alpha_{{Y^\circ}}\hspace{-1pt}\big(s_{\ol x}\cdot p_{\ol x}^{-1}\big)=0$.
More generally, the same argument works if $\cP$ is valuative.

{For more general puncturings, {the second vanishing
$\alpha_\cM(m_{\ol x})=0$ in Definition~\ref{def:puncturing},(2)} is not
automatic, but is needed to obtain good behavior under base-change
(Proposition~\ref{prop:pullback-puncturing}). Our log stacks
$\widetilde\fM'(\cX/B,\tau)$ in \S\ref{Sec:stacks of maps with sections}
naturally carry such a more general puncturing. {While} these more
general log structures have no further use in this paper, {they} may be
of use elsewhere.}

{Note also that if $\cP$ is a $DF(1)$ log structure and $\ol y$ is a geometric
point of $\ul Y$, then
\begin{equation}
\label{Eqn: puncturedmonoids}
\ocM_{Y,\ol y}\oplus\NN\subseteq \ocM_{Y^\circ,\ol y}\subset
\ocM_{Y,\ol y}\oplus\ZZ,\quad
\ocM_{Y^\circ,\ol y}\cap\big(\{0\}\times\ZZ_{<0}\big)=\emptyset.
\end{equation}}%
{We will see in Lemma~\ref{Lem: length function versus cone} how such
monoids can easily be encoded in the dual tropical picture.}
\end{remark}

\begin{remark}
\label{rem:final puncturing}
Puncturings $\cM^\circ$ of $\cM\oplus_{\cO^\times} \cP$ are not unique. In a
widely distributed early version of this manuscript as well as in \cite{Utah}, 
we found it instructive to work with a uniquely defined object $\cM^\cP$
we call here the \emph{final puncturing}. It may be defined as the direct limit
\[
\cM^{\cP} := \indlim_{\cM^{\circ} \in \Lambda} \cM^{\circ},
\]
over the collection $\Lambda$ of all puncturings of $\cM\oplus_{\cO^\times}
\cP$. This exists in the category of quasi-coherent, not necessarily coherent,
logarithmic structures. It has the advantage of being independent of any choice.
Its disadvantage, apart from not being finitely generated, is in that its
behavior under base change is subtle.

We emphasize that 
\begin{enumerate} 
\item
all puncturings used in this paper, with the exception of the remark above, are
fine, and in particular they are finitely generated. 
\item
On the other hand, the puncturings we use are rarely saturated, even though the
logarithmic structure they puncture are themselves saturated. The
reason is that base change of a saturated puncturing can lead to a non-saturated
puncturing. Imposing a saturation condition would therefore lead to a subtle
fiberwise saturation procedure. Instead we find that the notion of pre-stability
of {Definition~\ref{def:pre-stability}} below suffices to control these logarithmic structures and their moduli.
\end{enumerate} 
\end{remark}

{
\begin{remark}
In the introduction, we motivated punctures as arising from 
restrictions of log structures on log smooth curves to irreducible components.
Indeed, this is one way of producing punctures: see 
Proposition~\ref{prop:splitting-curve} for details. However, since we
allow fine rather than fine saturated log structures for the puncturing,
it is clear that not all the punctures we consider are of this form.
{See also Lemma~\ref{Lem: length function versus cone} for a description
of the submonoids of $\ocM_{Y,\ol y}\oplus\ZZ$ that can arise.}

It is worth making a historical remark here. When we began this project,
we first considered what we called ``pre-nodal" log structures in which
we allowed precisely those log structures coming via restriction from a
log smooth curve. However, we found the moduli space of pre-nodal
log maps was very poorly behaved, almost never Deligne-Mumford. The notion
of punctured points along with the notion of pre-stability of
Definition~\ref{def:pre-stability} resolved these technical difficulties,
and made gluing possible.
\end{remark}
}
\subsubsection{Pre-stable punctured log structures}
In case a puncturing is equipped with a morphism to another fine log structure
there is a canonical choice of puncturing. The following proposition
follows immediately from the definitions.

\begin{proposition} 
\label{prop:pre-stable}
Let $X$ be a fine log scheme, and $Y$ as in Definition~\ref{def:puncturing},
with a choice of puncturing $Y^{\circ}$ and a morphism $f:Y^{\circ}
\rightarrow X$. Let $\widetilde Y^{\circ}$ denote the puncturing of $Y$
given by the subsheaf of $\cM_{Y^{\circ}}$ generated by
$\cM_Y$ and $f^{\flat}(f^*\cM_X)$. Then 

\begin{enumerate} 
\item We have $\cM_{\widetilde Y^{\circ}}$ is a sub-logarithmic structure of
$\cM_{Y^{\circ}}$.
\item There is a factorization
\[
\xymatrix{
Y^{\circ} \ar[rr]^{f} \ar[rd] && X. \\
&\widetilde{Y}^{\circ} \ar[ur]_{\tilde{f}}&
}
\]
\item
Given $Y^\circ_1 \to Y^\circ_2 \to Y$ with both $Y^\circ_1, Y^\circ_2$
puncturings of $Y$, and compatible morphisms $f_i: Y_i^\circ \to X$, then
$\widetilde Y^\circ_1 = \widetilde Y^\circ_2$.
\end{enumerate}

\end{proposition}

\begin{definition}
\label{def:pre-stability}
A morphism $f:Y^{\circ}\rightarrow X$ from a puncturing of a log scheme
$Y$ is said to be \emph{pre-stable} if the induced morphism
$Y^{\circ}\rightarrow \widetilde{Y}^{\circ}$ in the above proposition
is the identity. In particular, one has $f=\tilde f$. 

\begin{figure}[htb]
\input{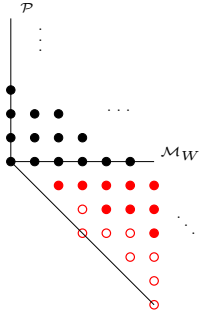}
\caption{A morphism of the previous puncturing $Y^\circ$ which is not
pre-stable, with $f^\flat \cM_X$ generated by $(2,-1)$. The submonoid generated
by $\cM_Y$ and $f^\flat \cM_X$, shown in solid dots, is a different
puncturing $\widetilde Y^{\circ}$ which is pre-stable.}
\label{fig:non-pre-stable}
\end{figure}

\end{definition}

Proposition~\ref{prop:pre-stable} yields the following criterion for
pre-stability of a morphism from a punctured log scheme. 

\begin{corollary}
\label{cor:pre-stable}
A morphism $f:Y^{\circ}\rightarrow X$ is pre-stable if and only if
the induced morphism of sheaves of monoids $f^*\ocM_X\oplus\ocM_Y
\rightarrow \ocM_{Y^{\circ}}$ is surjective.
\end{corollary}

\subsubsection{Pull-backs of puncturings}

\begin{proposition}
\label{prop:pullback-puncturing}
Let $Z$ and $Y$ be fs log schemes with log structures $\cM_Z$ and $\cM_Y$,
and suppose given a morphism $g:Z\rightarrow Y$. Suppose also given
an fs log structure $\cP_Y$ on $\ul{Y}$ and an induced log structure
$\cP_Z:=g^*\cP_Y$ on $\ul{Z}$. Set
\[
Z'=(\ul{Z}, \cM_Z\oplus_{\cO_Z^{\times}} \cP_Z),\quad
Y'=(\ul{Y}, \cM_Y\oplus_{\cO_Y^{\times}} \cP_Y)\,.
\]
Further,
let $Y^{\circ}$ be a puncturing of $Y'$ along $\cP_Y$. Then there
is a diagram
\[
\xymatrix@C=30pt
{
Z^{\circ}\ar[r]^{g^{\circ}}\ar[d]&Y^{\circ}\ar[d]\\
Z'\ar[r]^{g'}\ar[d]&Y'\ar[d]\\
Z\ar[r]_g&Y
}
\]
with all squares Cartesian in the category of underlying schemes, the lower
square Cartesian in the category of fs log schemes, and the top square Cartesian
in the category of fine log schemes. Furthermore, $Z^{\circ}$ is a puncturing of
$Z'$ along $\cP_Z$, and $g^{\circ}$ is pre-stable.
\end{proposition}

\begin{proof}
We define $Z^{\circ}$ to be the fiber product $Z'\times^{\mathrm{f}}_{Y'}
Y^{\circ}$ in the fine log category. The bottom square is Cartesian in all
categories as $\cP_Y$ is assumed saturated. Thus it is sufficient to show (1)
the upper square is Cartesian in the ordinary category, that is, the underlying
map of $Z^{\circ}\rightarrow Z'$ is the identity and (2) $Z^{\circ}$ is a
puncturing of $Z'$. 

Note that the fiber product $Z'\times_{Y'} Y^{\circ}$ in the category
of log schemes is defined as $\big(\ul{Z},\cM:=\cM_{Z'}\oplus_{g^*\cM_{Y'}}
g^*\cM_{Y^\circ}\big)$. This push-out need not, in general, be integral,
so we must integralize. Note there is a canonical isomorphism
\[
\cM^{\gp}
=
\cM_{Z'}^{\gp}\oplus_{g^*\cM_{Y'}^{\gp}}
g^*\cM_{Y^\circ}^{\gp}
\cong \cM_{Z'}^{\gp}
\]
given by $(s_1,s_2)\mapsto s_1\cdot (g')^{\flat}(s_2)$, where
$(g')^{\flat}:g^*\cM_{Y'}^{\gp}\rightarrow \cM_{Z'}^{\gp}$ is induced by $g'$.
The integralization $\cM^{\mathrm{int}}$ of $\cM$ is then the image of $\cM$ in
$\cM^{\gp}$, which thus can be described as the subsheaf of $\cM_{Z'}^{\gp}$
generated by $\cM_{Z'}$ and $(g')^{\flat}(g^*\cM_{Y^{\circ}})$. Note $\cM_{Z'}$
and $(g')^{\flat}(g^*\cM_{Y^{\circ}})$ both lie in $\cM_Z\oplus_{\cO^{\times}_Z}
\cP_Z^{\gp}$, and hence we can replace $\cM^{\gp}$ with this subsheaf of
$\cM^{\gp}$ in describing $\cM^{\mathrm{int}}$.  

It is now sufficient to show that we can define a structure map
$\alpha:\cM^{\mathrm{int}} \rightarrow \cO_Z$ compatible with the structure maps
$\alpha_{Z'}:\cM_{Z'}\rightarrow\cO_Z$ and
$\alpha_{Y^{\circ}}:g^*\cM_{Y^{\circ}}\rightarrow \cO_Z$. If
$s\in\cM^{\mathrm{int}}$ is of the form $s_1\cdot (g')^{\flat}(s_2)$ for $s_1\in
\cM_{Z'}$ and $s_2\in g^*\cM_{Y^{\circ}}$, then we define
$\alpha(s)=\alpha_{Z'}(s_1)\cdot \alpha_{Y^{\circ}}(s_2)$. We need to show this
is well-defined. If $s_2\in g^*\cM_{Y'}$, then $(g')^{\flat}(s_2)\in \cM_{Z'}$,
and thus as $g'$ is a log morphism,
\[
\alpha(s)=\alpha_{Z'}(s_1)\cdot\alpha_{Y^{\circ}}(s_2)=
\alpha_{Z'}(s_1)\cdot \alpha_{Z'}((g')^{\flat}(s_2))=\alpha_{Z'}(s)\,.
\]
In particular, $\alpha(s)$ only depends on $s$, and not on the particular 
representation of $s$ as a product, provided that $s_2\in g^*\cM_{Y'}$.

On the other hand, if $s_2\in (g^*\cM_{Y^{\circ}})\setminus (g^*\cM_{Y'})$, then
$\alpha_{Y^{\circ}}(s_2)=0$ by definition of a puncturing. So in this case
$\alpha(s)=0$. Hence to check that $\alpha$ is well-defined, it is enough to
show that if $s=s_1\cdot (g')^{\flat}(s_2)=s_1'\cdot (g')^{\flat}(s_2')$ with
$s_2\in g^*\cM_{Y'}$ but $s_2'\not\in g^*\cM_{Y'}$, then $\alpha_{Z'}(s_1)\cdot
\alpha_{Y^{\circ}}(s_2)= \alpha_{Z'}\big(s_1\cdot (g')^\flat(s_2)\big)=0$.
Writing $s_i=(m_i,p_i), s_i'=(m_i',p_i')$ using the descriptions
$\cM_{Z'}=\cM_Z\oplus_{\cO^{\times}_Z} \cP_Z$ and $g^*\cM_{Y^{\circ}}\subset
g^*\cM_Y\oplus_{\cO_Z^{\times}} \cP_Z^{\gp}$, we note that we must have
$m_1g^{\flat}(m_2)=m_1'g^{\flat}(m_2')$. As $s_2'\not\in g^*\cM_{Y'}$, by
Condition~(2) of Definition~\ref{def:puncturing} we necessarily have
$\alpha_Y(m_2')=0$. Hence $\alpha_Z(m_1'g^{\flat}(m_2'))=0$, so
$\alpha_Z(m_1g^{\flat}(m_2))=0$. We deduce that
$\alpha_{Z'}(s_1(g')^{\flat}(s_2))=0$, as desired. This shows $\alpha$ is
well-defined.

Finally, it is clear from the above description that $Z^{\circ}$ is a
puncturing. By Corollary~\ref{cor:pre-stable}, the pre-stability of $g^{\circ}$
follows from the surjectivity of 
\[
g^{-1}(\ocM_{Y^{\circ}})\oplus \ocM_Z \to
g^{-1}(\ocM_{Y^{\circ}}){\oplus_{g^{-1}(\ocM_Y)}^{\mathrm{f}}} \ocM_Z=
\ocM_{Z^{\circ}},
\]
where {$\oplus^{\mathrm{f}}$} denotes the fibered coproduct in the category of
fine monoids.
\end{proof}

\begin{definition}
\label{def:pullback-puncturing}
In the situation of Proposition~\ref{prop:pullback-puncturing}, we say
that $Z^{\circ}$ is the \emph{pull-back of the puncturing} $Y^{\circ}$.
\end{definition}

\begin{corollary} 
\label{Cor: pre-stability-pullback}
Consider the situation of Proposition~\ref{prop:pullback-puncturing},
and suppose in addition given a pre-stable morphism $f:Y^{\circ} 
\rightarrow X$. Then the composition $f\circ g^{\circ}:Z^{\circ}
\rightarrow X$ is also pre-stable.
\end{corollary}

\begin{proof}
This follows immediately from the definition of pre-stability and
the construction of $Z^{\circ}$ in the proof of Proposition
\ref{prop:pullback-puncturing}.
\end{proof}

\subsubsection{Punctured curves}
\label{sss: punctured curves}

Throughout the paper, we will essentially only be interested in puncturing along
logarithmic structures from designated marked points of logarithmic curves. Let
$\pi: C \to W$ be a logarithmic curve in the sense of \cite{FKato}:
\begin{enumerate}
\item
The underlying morphism $\ul{\pi}$ is a family of ordinary pre-stable
curves with {pairwise} disjoint sections $p_1,\ldots,p_n$ of $\ul{\pi}$
{disjoint from the critical locus of $\ul{\pi}$}.
\item
$\pi$ is a proper logarithmically smooth and integral morphism of fine and
saturated logarithmic schemes.
\item
If $\ul{U} \subset \ul{C}$ is the non-critical locus of $\ul{\pi}$ then
$\ocM_{C}|_{\ul{U}} \cong
\ul{\pi}^*\ocM_{W}\oplus\bigoplus_{i=1}^{n}p_{i*}\NN_W$.
\end{enumerate}
{Note that by (3), all marked points receive a nontrivial logarithmic
structure.} We write $\alpha_C:\cM_C\rightarrow\cO_C$ for the structure map of
the logarithmic structure on $C$. We call a geometric point of {$\ul C$}
\emph{special} if it is either a marked or a nodal point.

\begin{definition}\label{def:punctured-curves}
A \emph{punctured curve} over a fine and saturated logarithmic scheme $W$ is
given by the following data:
\begin{equation}\label{punctured-curve}
\big( C^{\circ} \stackrel{\pun}{\longrightarrow} C
\stackrel{\pi}{\longrightarrow} W, \bp = (p_1, \ldots, p_n) \big)
\end{equation}
where
\begin{enumerate}
\item
$C \to W$ is a logarithmic curve 
{in the sense of \cite{FKato} with its collection
of pairwise disjoint sections  $p_1, \ldots, p_n$ {of the underlying curve} as above}.
\item
$C^{\circ} \to C$ is a puncturing of $C$ along $\cP$, where $\cP$ is the
divisorial logarithmic structure on $\ul{C}$ induced by the divisor
$\bigcup_{i=1}^n p_i(\ul{W})$.
\end{enumerate}
When there is no danger of confusion, we may call $C^{\circ} \to W$ a punctured
curve. Sections in $\bp$ are called \emph{punctured sections}, or
simply \emph{punctures}. If $\ul W=\Spec\kappa$ with $\kappa$ a field,
we also speak of a \emph{punctured point}. We also say $C^{\circ}$ is a
puncturing of $C$ along the punctured sections $\bp$.

If locally around a punctured point $p_i$ the puncturing is trivial,
we say that the punctured point is a \emph{marked point}. In this case,
the theory will agree with the treatment of marked points in 
\cite{Chen},\cite{AC},\cite{LogGW}.
\end{definition}

\begin{examples}
\label{puncturedcurvesexamples}
(1) Let $W=\Spec\kk$ be the point with the trivial logarithmic structure, and 
$C$ be a non-singular curve over
$W$. Choose a closed point $p\in C$ and a puncturing
$\cM_{C^{\circ}}$ of $C$ at $p$. Then $\cM_{C^{\circ}}=\cP$, as 
$\cM_{C^{\circ}}\subset \cP^{\gp}$ can have no sections
$s$ with $\alpha_{C^{\circ}}(s)=0$. Thus, in this case the only puncturing 
$C^{\circ} \to C$ is the trivial one.
\smallskip

\noindent
(2) Let $W=\Spec(\NN\rightarrow\kk)$ be the standard logarithmic point, and $C$
be a non-singular curve over $W$, so that $\cM_C=\cO_C^{\times} \oplus
\ul{\NN}$, where $\ul{\NN}$ denotes the constant sheaf on $C$ with stalk $\NN$.
Again choose a {closed point} $p\in C$ with defining ideal $(x)$. Let $\cM_{C^{\circ}}
\subset \pi^*\cM_W\oplus_{\cO_C^{\times}}\cP^{\gp}$ be a puncturing. Let $s$ be
a local section of $\cM_{C^{\circ}}$ near $p$. Write $s= \big((\varphi, n),
x^m\big)$ with $\varphi\in\cO_{C,p}^{\times}$, $n\in \NN$. If $m<0$, then
Condition (2) of Definition~\ref{def:puncturing} implies that
$\alpha_{\pi^*(\cM_W)}(\varphi,n)=0$, so we must have $n>0$. Thus we see that 
\[
\overline{\cM}_{C^{\circ},p}\subset\{(n,m)\in \NN\oplus\ZZ\,|\, 
\hbox{$m\ge 0$ if $n=0$}\}\,.
\]
Conversely, any fine
submonoid of the right-hand-side of the above inclusion
which contains $\NN\oplus\NN$ can be realised as the stalk of the ghost sheaf at
$p$ for a puncturing. Note the monoid on the right-hand side is not finitely
generated, and is the stalk of the ghost sheaf of the final puncturing, see
Remark~\ref{rem:final puncturing}.
\smallskip

\noindent
(3) Let $\ul{W}=\Spec \kk[\epsilon]/(\epsilon^{k+1})$, and let $W$ be given by
the chart $\NN\rightarrow \kk[\epsilon]/(\epsilon^{k+1})$, $1\mapsto \epsilon$.
Let $C_0$ be a non-singular curve over $\Spec \kk$ with the trivial logarithmic
structure, and let $C=W\times C_0$. Choose a section $p:W\rightarrow C$, with
image locally defined by an equation $x=0$. Condition (2) of Definition
\ref{def:puncturing} now implies that a section $s$ of a puncturing
$\cM_{C^{\circ}}$ near $p$ takes the form $\big((\varphi, n),x^m\big)$ where
$\varphi\in \cO_{C,p}^{\times}$, and $0\le n\le k$ implies $m\ge 0$. In
particular, 
\[
\overline{\cM}_{C^{\circ},p}\subset\{(n,m)\in \NN\oplus\ZZ\,|\, 
\hbox{$m\ge 0$ if $n \leq k$}\},
\]
and any fine submonoid of the right-hand side containing $\NN\oplus\NN$ can be
realised as the stalk of the ghost sheaf at $p$ of a puncturing.

\begin{figure}[htb]
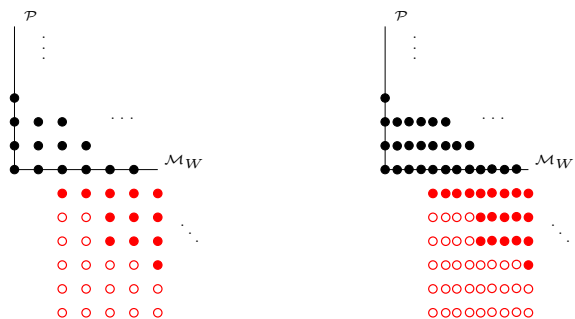

\input{deformation-pullback.pspdftex}\hspace*{1in}
\input{deformation-pullback2.pspdftex}
\caption{The solid puncturing on the left extends to
$\kk[\epsilon]/(\epsilon^2)$ but no further --- the circled elements are the ones
allowed for $k=1$. Its pull-back (see below) via $\cE^2 = \epsilon$ is pictured
on the right --- it is defined on $\kk[\cE]/(\cE^4)$ but does not extend further.}
\label{fig:deformation-pullback}
\end{figure}

\end{examples}

\subsubsection{Pull-backs of punctured curves}
Consider a punctured curve $(C^{\circ} \to C \to W, \bp)$ and a morphism of fine
and saturated logarithmic schemes $h: T \to W$. Denote by $(C_T \to T, \bp_T)$
the pull-back of the log curve $C \to W$ via $T \to W$. By
Proposition~\ref{prop:pullback-puncturing}, we obtain a commutative diagram
\begin{equation}
\label{Diagram: pull-back of punctured curves}
\vcenter{\xymatrix{
C^{\circ}_T \ar[r] \ar[d]_{\pun_T} & C^{\circ} \ar[d]^{\pun} \\
C_T \ar[r] \ar[d]_{\pi_T} & C \ar[d]^{\pi} \\
T \ar[r] ^{h}& W
}}
\end{equation}
where the bottom square is cartesian in the fine and saturated category, and the
square on the top is cartesian in the fine category, and such that $C_T^{\circ}$
is a puncturing of the curve $C_T$ along $\pun_T$. 

\begin{definition}
\label{Def: pull-back of punctured curves}
We call $C^{\circ}_T \to T$ the \emph{pull-back} of the punctured curve
$C^{\circ} \to W$ along $T \to W$.
\end{definition}

\subsubsection{Punctured maps}
\label{sss:punctured-maps}
We now fix a morphism of fine and saturated logarithmic schemes $X \to B$.

\begin{definition}
A \emph{punctured map to $X\to B$} over a fine and saturated logarithmic scheme
$W$ over $B$ consists of a punctured curve $(C^{\circ} \to C \to W,
\bp)$ and a morphism $f$ fitting into a commutative diagram
\[
\xymatrix{
C^{\circ} \ar[r]^{f} \ar[d]_\pi & X \ar[d] \\
W \ar[r] & B
}
\]
Such a punctured map is denoted by $(\pi: C^{\circ} \to W, \bp, f)$ or
$(C^\circ/W,\bp,f)$. 

The \emph{pull-back} of a punctured map $(C^{\circ}/ W, \bp, f)$ along a
morphism of fine and saturated logarithmic schemes $T \to W$ is the punctured
map $(C_T^{\circ} / T, \bp_T, f_T)$ consisting of the pull-back $C_T^{\circ} \to
T$ of the punctured curve $C \to W$ and the pull-back $f_T$ of $f$.
\end{definition}

\begin{definition}
\label{def:stability}
A punctured map $(C^{\circ} \to W, \bp, f)$ is called \emph{pre-stable} if 
$f:C^{\circ}\rightarrow X$ is pre-stable in the sense of Definition
\ref{def:pre-stability}.

A pre-stable punctured map is called \emph{stable} if its underlying map, marked
by the punctured sections, is stable in the usual sense. 
\end{definition}

\begin{proposition}\label{prop:pre-stable-open}
Let $(C^{\circ} / W, \bp, f)$ be a punctured map over $W$. 
\begin{enumerate}
\item
The locus of points of $W$ with pre-stable fibers forms an open sub-scheme of
$\ul{W}$. 
\item
If $f: C^{\circ} \to X$ is pre-stable, then the pull-back $f_T: C_T^{\circ} \to
X$ along any morphism of fine and saturated logarithmic schemes $T \to W$ is
also pre-stable.
\end{enumerate}
\end{proposition}
\begin{proof}
The map $f: C^{\circ} \to X$ induces a morphism of fine logarithmic structures 
\[
f^{\flat}\oplus \pun^{\flat}: f^*\cM_X\oplus_{\cO_{C}^{\times}}\cM_{C}
\to \cM_{C^{\circ}}\,.
\]
The pre-stability of $f$ is equivalent to the condition that $f^{\flat}\oplus
\pun^{\flat}$ is surjective by Corollary~\ref{cor:pre-stable}. Statement (1)
can be proved by applying Lemma~\ref{lem:logsurjctivityopen} below to the
neighborhood of each puncture. Statement (2) follows immediately from
Corollary~\ref{Cor: pre-stability-pullback}.
\end{proof}

\begin{lemma}\label{lem:logsurjctivityopen}
Let $\uY$ be a scheme, and $\psi^{\flat}: \cM \to \cN$ be a morphism of fine log
structures on $\uY$. Then the locus $\uY' \subset \uY$ over which $\psi^{\flat}$ is
surjective forms an open subscheme of $\uY$. 
\end{lemma}

{
\begin{proof}
We thank the anonymous referee for suggesting the following simplified proof.
Since both $\cM$ and $\cN$ are $\cO^{\times}_{\uY}$-torsors over $\ocM$ and
$\ocN$ respectively, the surjectivity of $\psi^{\flat}$ is equivalent to the
surjectivity of the induced morphism $\ocM \to \ocN$ of ghost sheaves. Since the
statement is local on $\uY$, we may assume that $\ocN$ is globally generated.  

Suppose $y \in \uY$ is a geometric point over which $\ocM_{y} \to \ocN_y$ is
surjective. Then each global section of $\ocN$ lifts to a section of $\ocM$ in
an \'etale neighborhood of $y$. Since $\Gamma(\uY,\ocN)$ is finitely generated,
there is a common \'etale neighborhood of $y$ over which all the global sections
of $\ocN$ lift to $\ocM$. This finishes the proof. 
\end{proof}
}

The most interesting aspect of punctured curves is the appearance of negative
contact orders, defined as follows.

\begin{definition}
\label{Def: contact order}
The \emph{contact order} of a punctured map $(C^\circ/W,\bp,f)$ to $X\to B$ over
a log point $W=\Spec(Q\arr \kappa)$ at $p\in\bp$ is the composition
\begin{equation}
\label{Eqn: contact order}
u_p: \ocM_{X,\ul f(p)}\stackrel{f^\flat}{\longrightarrow}
\ocM_{C,p}\arr Q\oplus\ZZ \stackrel{\pr_2}{\longrightarrow}\ZZ
\end{equation}
with the second map the canonical inclusion. We say that $u_p$ is
\emph{negative} if $u_p(\ocM_{X,\ul f(p)})\not\subseteq\NN$.
\end{definition}

The difference with the case of logarithmic stable maps \cite[Def.\,1.8]{LogGW}
is the appearance of $\ZZ$ instead of $\NN$. {The tropical
interpretation of this condition will be discussed in \S\ref{sec:tropical
interpretation} below.} Note that if $(C^\circ/W,\bp,f)$ is pre-stable, the
contact order at $p\in\bp$ is negative if and only if $p$ is not a marked point.

\begin{example}
\label{ex:intuition}
Here is a simple example featuring a negative contact order. Let $\ul{X}$ be a
smooth surface, $\ul{D}\subseteq \ul{X}$ a non-singular rational curve with
self-intersection $-1$ inducing the divisorial log structure $X$ on $\ul{X}$.
Let $C\rightarrow W$ be the punctured curve of
Examples~\ref{puncturedcurvesexamples},(2), with $C\cong \PP^1$. Let
$\ul{f}:\ul{C}\rightarrow \ul{X}$ be an isomorphism of $\ul{C}$ with $\ul{D}$.
This can be enhanced to a punctured map $C^{\circ}\rightarrow X$ as follows. 

We first define $\bar f^{\flat}:\ul{f}^*\overline{\cM}_X=\ul{\NN}\rightarrow
\overline{\cM}_{C^{\circ}}\subseteq \overline{\cE}=\ul{\NN}\oplus \ZZ_p$ by
$1\mapsto (1,-1)$, where $\ZZ_p$ denotes the sky-scraper sheaf at $p$ with stalk
$\ZZ$. Note that {the inverse image of $1\in\Gamma(X,\overline{\cM}_X)$
under the projection map $\cM_X\rightarrow\overline{\cM}_X$ is the}
$\cO_X^{\times}$-torsor contained in $\cM_X$ corresponding to the line bundle
$\cO_X(-D)$, and thus $1\in\Gamma(C,\ul{f}^*\overline{\cM}_X)$ 
{similarly} yields the
$\cO^{\times}_C$-torsor corresponding to $\cO_C(1)$, using $-D^2=1$. On the
other hand, the torsor contained in $\cM_{C^{\circ}}$ corresponding to
$(1,0)$ is the torsor of $\cO_C$, and the torsor corresponding to $(0,1)$ is the
torsor of the ideal $\cO_C(-p)$. Hence $(1,-1)\in
\Gamma(C,\overline{\cM}_{C^{\circ}})$ corresponds to $\cO_C(1)$. Choosing an
isomorphism of torsors then lifts the map $\bar f^{\flat}$ to a map
$f^{\flat}:\ul{f}^*\cM_X\rightarrow \cM_{C^{\circ}}$ inducing a morphism
$f:C^{\circ}\rightarrow X$ (Figure~\ref{fig:-1-curve}).

Note this morphism does not lift to $C'\rightarrow W'=\Spec(\kk[\epsilon]/
(\epsilon^2))$ as in Examples~\ref{puncturedcurvesexamples},(3), since we can't
even lift $\bar f^{\flat}$ at the level of ghost sheaves. Indeed, $(1,-1)$ is
not a section of the ghost sheaf of $(C')^{\circ}$.

\begin{figure}[htb]
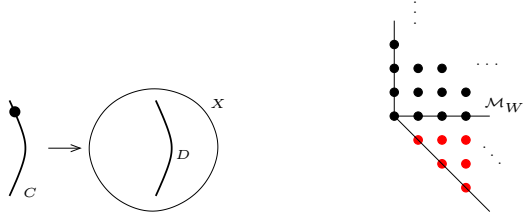

\input{minus-1-curve.pspdftex}
\hspace*{2cm}
\input{minus-1-curve-monoid.pspdftex}
\caption{The $(-1)$-curve and its monoid.}
\label{fig:-1-curve}
\end{figure}

\end{example}

\begin{remark}[Geometric implication of negative contact orders]
\label{rem:contained-in-strata}
Let $f:C^{\circ}/W\rightarrow X$ be a punctured map with $W=\Spec(Q\rightarrow
\kk)$. Suppose $p\in C$ is a punctured point which is not a marked point, and
let $C'$ be the irreducible component containing $p$, with generic point $\eta$.
Then, intuitively, $C'$ has negative order of tangency with certain strata in
$X$, and this forces $C'$ to be contained in those strata. 

Explicitly, let $P_p=\ocM_{X,\ul f(p)}$ and let $u_p:P_p\arr \ZZ$ be as in Definition~\ref{sec:tropical interpretation}. Then if $\delta\in P_p$ with $u_p(\delta)<0$, we must
have $\pr_1\circ\bar f_p(\delta)\not=0$, as there is no element of
$\overline{\cM}_{C^{\circ},p}\subset Q\oplus\ZZ$ of the form $(0,n)$ with $n<0$.
Thus if $\chi:P_p\rightarrow \ocM_{X,\ul f(\eta)}$ denotes the generization map,
we must have $u_p^{-1}(\ZZ_{<0})\cap \chi^{-1}(0) =\emptyset$. This restricts
the strata in which $f(C')$ can lie.

For example, if $X=(\ul{X},D)$ for a simple normal crossings divisor $D$ with
irreducible components $D_1,\ldots,D_n$, then $P_p=\bigoplus_{i: f(p)\in D_i}
\NN$. The value $u_p$ on the generator of $P_p$ corresponding to $D_i$ is the
contact order with $D_i$. Then $f(C')$ must lie in the intersection of 
those $D_i$ that have negative contact order at $p$.

A critical aspect of this phenomenon is discussed in \S\ref{ss:log-ideal}, see
especially Proposition~\ref{prop:puncturing-ideal-vanish} and Example~\ref{Expl:
Puncturing log ideal AA^2}.
\end{remark}


\subsection{The tropical interpretation}
\label{sec:tropical interpretation}
We now introduce the tropical picture, which gives the underlying
organizing language for punctured Gromov-Witten theory.
We assume familiarity with the discussion in ordinary logarithmic Gromov-Witten
theory as presented in \cite[\S2]{decomposition}. We review in \S\ref{sss:
tropical punctured maps} the notations and basic concepts briefly while
discussing the modifications needed for including non-trivial punctures.

\subsubsection{Tropical punctured maps}
\label{sss: tropical punctured maps}

{
In Appendix~\ref{App: Functorial tropicalization} we define tropicalization as a
functor associating to a fine log algebraic stack a generalized cone complex
$\Sigma(X)$. There is one stratum of $|\Sigma(X)|$ for each logarithmic stratum
of $X$, the latter defined as a maximal connected locally closed subset
$Z\subseteq|\ul X|$ with $\ocM_X|_Z$ locally constant.} {For each
logarithmic stratum $Z$ we choose, once and for all, a geometric point $\ol x_Z$
with image in $Z$. Then $\Sigma(X)$ is defined as the diagram with only one cone
\begin{equation}
\label{Eqn: sigma_Z}
\sigma_Z= \Hom(\ocM_{X,\ol x_Z},\RR_{\ge0})
\end{equation}
for each logarithmic stratum $Z$, along with all its faces, and arrows induced
by all sequences of generization morphisms and all face inclusions, including
inverses of those that are isomorphisms. Note that due to monodromy, $\Sigma(X)$
may contain nontrivial arrows $\sigma\to\sigma$. {The group
\[
\Aut_{\Sigma(X)}(\sigma)= \big\{\sigma\to\sigma\text{ arrow in }\Sigma(X)\big\}
\]
is a subgroup of the permutation group of the set of rays of
$\sigma$, hence is always finite. Note that the map
\[
\sigma_Z/\Aut_{\Sigma(X)}(\sigma_Z)\arr|\Sigma(X)|
\]
induced from $\sigma_Z\arr|\Sigma(X)|$ may still not be injective due to arrows
from strata of $X$ whose closure intersect the closure of
$Z$ and that are} not induced by monodromy on $Z$. Accordingly, the image of
$\sigma$ in $|\Sigma(X)|$ may be a finite quotient even on its interior.

By abuse of notation, $\Sigma(X)$ denotes both the distinguished presentation or
the equivalence class as a generalized cone complex. When writing
$\sigma\in\Sigma(X)$ we refer to the chosen presentation, so there is a unique
logarithmic stratum $Z\subseteq X$ with $\sigma=\sigma_Z$. For any geometric
point $\ol x$ with image in $Z$ we have the cone
\[
\sigma_{\ol x}= \Hom(\ocM_{X,\ol x},\RR_{\ge0})
\]
together with an isomorphism
\begin{eqnarray}
\label{Eqn: sigma_x -> sigma_Z}
\sigma_Z\arr \sigma_{\ol x},
\end{eqnarray}
but this isomorphism is only unique up to pre-composition with arrows
$\sigma_Z\to\sigma_Z$ in $\Sigma(X)$. In other words, the isomorphism
$\sigma_Z\arr\sigma_{\ol x}$ is unique up to the action of the monodromy group
$\Aut_{\Sigma(X)}(\sigma_Z)$ of the logarithmic stratum $Z$.

For $\sigma\in\Sigma(X)$ we denote by 
\begin{equation}
\label{Eqn: strata of X}
X_\sigma= \big\{ x\in \ul X\,\big|\, \text{there exists an arrow}\ \sigma\to
\sigma_{\ol x}\ \text{in $\Sigma(X)$}\big\}\subseteq \ul X
\end{equation}
the closed set of points $x\in \ul X$ with {$\sigma$ connected to
$\sigma_{\ol x}=\Hom(\ocM_{X,\ol x},\RR_{\ge0})$} by a sequence of generizations
and inverses of invertible generizations of the stalks of $\ocM_X$. We endow
$X_\sigma$ with the reduced induced scheme structure. {In practice, say
when $X$ is log smooth over a log point, $X_\sigma$} is the closure of the
logarithmic stratum given by $\sigma\in\Sigma(X)$.} {For brevity} we
refer to the $X_\sigma$ as \emph{strata} of $X$, but note that from the point of
view of stratified spaces, {and differing from the use in
Appendix~\ref{App: Functorial tropicalization}}, these are {at best}
closures of strata. Note also that for $\sigma=\{0\}$ we obtain $X_{\{0\}}=\ul
X$ {assuming $\Sigma(X)$ connected}, even if {there is no
geometric point $\ol x$ of $X$} with $\ocM_{X,\ol x}=0$.

A stable logarithmic map $(C/W, \bp, f)$ gives rise via functoriality
{of the tropicalization functor} $\Sigma$ to the diagram
\begin{equation}
\label{eqn: tropical curve}
\vcenter{\xymatrix@C=30pt
{
\Sigma(C)\ar[r]^{\Sigma(f)}\ar[d]_{\Sigma(\pi)}&\Sigma(X)\ar[d]\\
\Sigma(W)\ar[r]&\Sigma(B)
}}
\end{equation}
We will almost exclusively consider such diagrams {in which} $W$ is
covered by a single chart and $\Sigma(W)$ {has a single maximal cone
$\omega=(\cM_{W,\ol w}^\vee)_\RR$}
{for $\ol w$ some geometric point of $\ul W$}. Then it is shown in
\cite[Prop.\,2.25]{decomposition} that $\Sigma(\pi)$ along with the genera of
the irreducible components of the geometric fiber $C_{\ol w}$ is a (family of)
abstract tropical curves over $\omega$, also written $(G,\bg,\ell)$. Here $G$ is
the dual intersection graph of $C_{\ol w}$ with sets $V(G), E(G), L(G)$ of
vertices, edges and legs, and the maps
\[
\bg: V(G)\arr \NN,\qquad
\ell: E(G)\arr \Hom(\omega_\ZZ,\NN)\setminus\{0\},
\]
record the genera of the irreducible components of $C_{\ol w}$ and the lengths
of edges as functions on $\omega$, respectively, see
\cite[Def.\,2.19]{decomposition}. {If $G$ arises from the
tropicalization of a log curve over a geometric logarithmic point, we denote
elements of $V(G),E(G),L(G)$ both by their graph-theoretic notations as vertices
$v$, edges $E$, and legs $L$, or the corresponding algebraic geometric notations
as generic points $\eta$, nodes $q$, and marked points $p$.} By abuse of
notation, we view homomorphisms $\omega_\ZZ\arr\NN$ also as homomorphisms
$\omega\arr \RR_{\ge0}$ respecting the integral structure. Conversely, from
$(G,\bg,\ell)$, the cone complex
\[
\Gamma=\Gamma(G,\ell)
\]
recovering $\Sigma(C)$ has one copy of $\omega$ for each $v\in V(G)$, a cone
\begin{equation}
\label{Eqn: omega_E}
\omega_E=\big\{(s,\lambda)\in\omega\times\RR_{\ge0}\,\big|\,
\lambda\le \ell(E)(s)\big\}
\end{equation}
for each edge $E\in E(G)$, and a copy of $\omega\times\RR_{\ge0}$ for each leg.
Note that legs have infinite lengths for any parameter $s\in\omega_\RR$ when
viewing $\Gamma$ as a family of metric graphs.

The only change in the punctured setup is that a leg may now have finite length.
Indeed, if $L\in L(G)$ corresponds to a non-trivial puncture with
puncturing submonoid $Q^\circ\subset Q\oplus\ZZ$, then
$(Q^\circ)^\vee_\RR=\omega_L$ with
\begin{equation}
\label{Eqn: omega_L}
\omega_L=\big\{(s,\lambda)\in\omega\times\RR_{\ge0}\,\big|\,
\lambda\le \ell(L)(s)\big\}
\end{equation}
defined in analogy with \eqref{Eqn: omega_E} by a length function
$\ell(L):{\omega\arr \RR_{\ge 0}}$ with $\ell(L)\neq 0$. {Note,
however, that $\ell(L)$ is now only piecewise linear as illustrated in
Figure~\ref{fig: PL length}. Here a continuous function
$\ell:\omega\to\RR_{\ge0}$ on $\omega\in\Cones$ is \emph{piecewise linear} if there
exists a fan subdivision of $\omega$ such that $\ell$ is the restriction of a
linear form on each cone of the fan. For the following relation to monoids recall \eqref{Eqn: puncturedmonoids} from Remark~\ref{rem:DF1}.

{
\begin{figure}[htb]
\input{PL-length-fct.pspdftex}
\caption{The length of a bounded leg varies piecewise linearly under linear
variations of the adjacent vertex. The figure shows the intersection of the
situation with an affine hyperplane.}
\label{fig: PL length}
\end{figure}}

\begin{lemma}
\label{Lem: length function versus cone}
Let $Q$ be a sharp toric monoid and $\omega= Q_\RR^\vee$. Assume further that
$Q^\circ\subseteq Q\oplus \ZZ$ is a finitely generated submonoid with
$Q\oplus\NN\subsetneq Q^\circ$, $Q^\circ\cap\big(\{0\}\times
\ZZ_{<0}\big)=\emptyset$. Then there exists a nonzero, concave,
piecewise linear function
\[
\ell: \omega\arr \RR_{\ge0}
\]
with rational slopes such that
\begin{equation}
\label{Eqn: (Q^circ)^vee_RR from ell}
\big(Q^\circ)^\vee_\RR= \big\{ (s,\lambda)\in\omega\times\RR_{\ge 0}\,\big|\,
0\le\lambda\le\ell(s)\big\}.
\end{equation}

Each such $\ell:\omega\to \RR_{\ge0}$ arises in this fashion, and two submonoids
$Q_1^\circ, Q_2^\circ\subseteq Q\oplus\ZZ$ with $Q_i\oplus\NN\subsetneq
Q_i^\circ$, $Q_i^\circ\cap\big(\{0\}\times \ZZ_{<0}\big)=\emptyset$, $i=1,2$,
lead to the same $\ell$ {if and} only if $(Q_1^\circ)^\sat=
(Q_2^\circ)^\sat$.
\end{lemma}

\begin{proof}
Let $(s,\lambda)\in (Q^\circ)^\vee_\RR$. Then since $Q\oplus\NN\subseteq Q^\circ$, necessarily $s\in\omega=Q_\RR^\vee$ and $\lambda\ge 0$. Conversely, $(s,0)\in (Q^\circ)^\vee_\RR$ for all $s\in \omega$, and in fact,
\[
\omega\times\{0\} \subseteq (Q^\circ)^\vee_\RR
\]
is a facet. Since $Q^\circ\neq Q\oplus\NN$ no ray of $(Q^\circ)_\RR^\vee$ is
vertical, that is, agrees with $\RR_{\ge0} \cdot(0,1)$. Thus the union of the
maximal cells of $\partial (Q^\circ)^\vee_\RR$ neither contained in
$\omega\times\{0\}$ nor in $\partial\omega\times\RR$ form the graph of a
piecewise linear function $\ell:\omega\arr\RR_{\ge 0}$ as in the statement of
the lemma. Convexity of $(Q^\circ)^\vee_\RR$ implies that $\ell$ is concave.
Finally, $\ell\neq0$ for otherwise $(0,-1)\in Q^\circ_\RR$, contradicting
$Q^\circ\cap\big(\{0\}\times \ZZ_{<0}\big)=\emptyset$.

Conversely, given a nonzero, concave, piecewise linear
$\ell:\omega\to\RR_{\ge0}$ with rational slopes, the cone $\sigma$ on the
right-hand side of \eqref{Eqn: (Q^circ)^vee_RR from ell} contains
$\omega\times\{0\}$ as a facet. Hence $\sigma^\vee\subseteq \omega^\vee\times\RR
= Q_\RR\times\RR$ and $Q_\RR\times\RR_{\ge 0}\subseteq \sigma^\vee$. The case
$Q_\RR\times\RR_{\ge 0}=\sigma^\vee$ does not occur since $\sigma\neq
\omega\times\RR_{\ge0}$ by finiteness of the values of $\ell$. Moreover,
$\ell\neq 0$ implies $\sigma$ is a full-dimensional cone, and hence
{$(0,-1)\not\in \sigma^\vee$,} or
$\sigma^\vee\cap\big(\{0\}\times \ZZ_{<0}\big)=\emptyset$. This shows that
knowing $\ell$ retrieves the convex hull of $Q^\circ$ in $Q_\RR\times\RR$, hence
the set of integral points {of its saturation} $(Q^\circ)^\sat$.
\end{proof}
}

\begin{definition}
\label{Def: Punctured tropical curve}
(1)~A (family of) \emph{punctured tropical curves} over a cone $\omega\in
\Cones$ is a graph $G$ together with {two maps
\[
\bg: V(G)\arr\NN,\quad
\ell: E(G)\cup L^\circ(G) \arr \Map(\omega,\RR_{\ge0})
\]
for some subset $L^\circ(G)\subseteq L(G)$, with
$\ell(E)\in\Hom(\omega_\ZZ,\NN)\setminus\{0\}$ for $E\in E(G)$ and
$\ell(L):\omega\arr\RR_{\ge 0}$ for $L\in L^\circ(G)$ nonzero, concave,
piecewise linear, with rational slopes}. We refer to elements of $L^\circ(G)$ as
\emph{finite or punctured legs}, all other legs as \emph{infinite or
marked}.\\[1ex]
(2)~A (family of) \emph{punctured tropical maps} over $\omega\in\Cones$ is a map
of {generalized} cone complexes $h: \Gamma\arr\Sigma(X)$ for $\Gamma=\Gamma(G,\ell)$ the cone
complex defined by a punctured tropical curve $(G,\bg,\ell)$ over $\omega$.
\end{definition}

{For readability and as in \cite{decomposition} throughout, we assume
for the rest of this subsection that $\Sigma(X)$ is simple
\cite[Def.\,2.1]{decomposition}. This means that for each $\sigma\in \Sigma(X)$
the map $\sigma\arr|\Sigma(X)|$ is injective. We will treat the general case in
\S\ref{ss: targets with monodromy}.} As in \cite[Prop.\,2.26]{decomposition}, it
{then} follows readily from the definitions
that the tropicalization of a punctured map to $X$ over a logarithmic point
$\Spec(Q\arr \kappa)$ {with $\kappa$ algebraically closed} is a punctured
tropical map over $Q^\vee_\RR$.

Given a punctured tropical map, one extracts associated discrete data as in
\cite[Rem.\,2.22]{decomposition}. These are an \emph{image cone} map
\begin{equation}
\label{Eqn: image cone map}
\bsigma: V(G)\cup E(G)\cup L(G)\arr \Sigma(X)
\end{equation}
associating to each object of $G$ the {(distinguished representative of
the)} minimal cone of $\Sigma(X)$ it maps to, and, {referring again to the
notation in Section \ref{Sec:convention},}
\emph{contact orders}
\begin{equation}
\label{Eqn: contact orders}
u_q=u_E\in N_{\bsigma(E)},\qquad u_p=u_L\in N_{\bsigma(L)}
\end{equation}
for edges $E=E_q\in E(G)$ and for legs $L=L_p\in L(G)$, respectively.

Contact orders are defined by the image of the tangent vector $(0,1)$
{in the tangent space $N_{\omega}\times\ZZ$ of} $\omega_E$ or $\omega_L$
under $h$. {The contact order for an edge $E$ depends, up to sign, on a
choice of orientation on $E$}{, which we suppress in the notation.} For
legs, this definition is consistent with the definition of contact orders of
punctured maps in Definition~\ref{Def: contact order}.

Note that the contact order $u_p\in N_{\bsigma(L_p)}$ of a marked point $p \in
C_{\ol w}$ lies in $\bsigma(L_p)$. Conversely, a non-trivial puncture is forced
by a leg $L=L_p$ {if} for any parameter value $s\in\omega$, the line
segment $h\big(\{s\}\times[0,\ell(L)(s)]\big)$ inside the image cone
$\bsigma(L)\in \Sigma(X)$ does not extend to a half-line.

There is a simple tropical interpretation of pre-stability saying that images of
legs extend as far as possible inside their image cones. See
Figure~\ref{fig:tropical-curve} for an illustration. We call such tropical
punctured maps \emph{pre-stable}.

\begin{proposition}
\label{Prop: pre-stable tropical punctured maps}
{Let $(C^\circ/W,\bp,f)$ be a pre-stable punctured map over a log point
$W=\Spec(Q\arr \kappa)$ and $h=\Sigma(f):\Gamma(G,\ell)\arr \Sigma(X)$ its
tropicalization. For each {finite leg} $L\in L^\circ(G)$, we write 
$\omega_L\subseteq \omega\times \RR_{\ge 0}$ as in \eqref{Eqn: omega_L}.
Then for all
$s\in\omega$, we have}
\[
h(s,\ell(L)(s)){=h(s,0)+\ell(L)(s)\cdot u_L}\in \partial \bsigma(L),
\]
while $h(s,\ell(L)(s))+\varepsilon u_L\not\in \bsigma(L)$ for
all $\varepsilon>0$.
\end{proposition}

\begin{figure}[htb]
\input{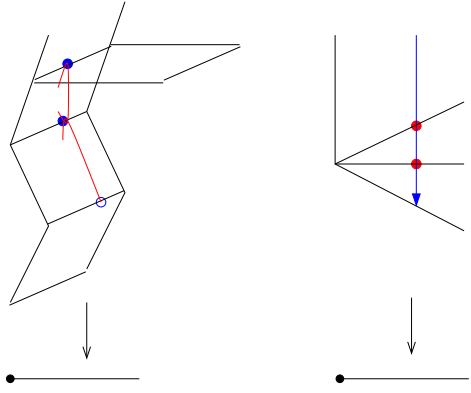}
\caption{A curve in {the fiber of a one-parameter family of surfaces (a threefold)} and its tropicalization. There are two
components, represented by vertices; one node represented by an edge; one
regular marked point represented by an infinite leg and one puncture represented
by a finite leg, which, by pre-stability, extends exactly as far as the cone
allows.}
\label{fig:tropical-curve}
\end{figure}

\begin{proof}
Let $\ol p\arr C$ be the punctured point defined by $L$, and write
$\omega=Q^\vee_\RR$, $\sigma=P^\vee_\RR$ for $P=\ocM_{X,\ul f(\ol p)}$. The map
$h_L:\omega_L\arr \sigma$ defined by $h$ is dual to
\[
\bar f^\flat_{\ul f(\ol p)}: P \arr
\ocM_{C,\ol p}= Q^\circ\subset Q\oplus\ZZ\,.
\]
By pre-stability, $Q^\circ$ is generated by $Q\oplus\NN$ and by the image of
$\bar f^\flat_{\ul f(\ol p)}$. Dually, we obtain
\[
\omega_L=(Q^\circ)^\vee_\RR= (\omega\times\RR_{\ge 0})\cap h_L^{-1}(\sigma)\,.
\]
Now $\omega_L$ is the convex hull of $\omega\times\{0\}$ and of
$\big\{(s,\ell(L)(s))\in\omega\times \RR_{\ge0}\big\}$, the graph of $\ell(L)$
as a map $\omega\arr\RR_{\ge0}$. This shows that no point
$\big(s,\ell(L)(s)\big)$ maps to an interior point of $\sigma$, {and the
line segment in $\sigma$ connecting $h(s,0)$ {with} $h(s,\ell(L)(s))$
can not be extended,} as claimed.
\end{proof}

Note that while $\omega_L^\vee\cap (N_\omega\times\ZZ)^*$ only computes the
saturation of $Q^\circ$, the tropical picture also contains the map $P\arr
Q\oplus\ZZ$. {{In the pre-stable case, $Q^\circ$ is then the
submonoid generated by the image of this map and by $Q\oplus \NN$, so can be
fully computed tropically.}}

\subsubsection{Types of punctured maps}
\label{sss: types of  punctured maps}
As in \cite[Def.\,2.23]{decomposition} for stable logarithmic maps, we now
capture the combinatorics underlying punctured maps and their tropicalization by
the notion of \emph{type}.

\begin{definition}
\label{Def: type}
(1)~The \emph{type of a family of tropical punctured maps}
$h:\Gamma=\Gamma(G,\ell)\arr \Sigma(X)$ over $\omega\in\Cones$ is the tuple
\[
\tau=(G,\bg,\bsigma,\bu)
\]
consisting of the associated genus-decorated connected graph $(G,\bg)$, the
image cone map $\bsigma$ from \eqref{Eqn: image cone map} and the collection
$\bu=\{u_q,u_p\}_{p,q} =\{u_E,u_L\}_{E,L}$ of contact orders from \eqref{Eqn:
contact orders}. In particular, for $x\in E(G)\cup L(G)$ we require $u_x\in
N_{\bsigma(x)}$. {We also sometimes write $\bu(x)$ instead of $u_x$ when
referring to a contact order given by a type rather than by a punctured
map.}\\[1ex]
(2)~The \emph{type of a punctured map $(C/W,\bp,f)$ to $X$ at a geometric point
{$\ol w$ of $\ul W$}} is the type of the associated tropical map
$\Gamma\arr\Sigma(X)$ over ${\omega=}(\ocM_{W,\ol w}^\vee)_\RR$.
\end{definition} 

Thus the type records the combinatorial data associated to
$h:\Gamma\arr\Sigma(X)$, but forgets the length function $\ell:E(G)\cup
L^\circ(G)\arr {\Map(\omega,\RR_{\ge0})}$.

For a punctured map over a logarithmic point, one sometimes also wants to keep
the curve classes $\bA(v)=\ul f_*\big([\ul C(v)]\big)$ for $\ul C(v)\subset \ul
C$ the irreducible component of $\ul C$ given by $v \in V(G)$. Here $\bA(v)$ is
a class of curves in singular homology of the corresponding stratum
$X_{\bsigma(v)}$, or in some other appropriate monoid of curve classes, written
$H_2^+(X_\sigma)$ for $\sigma\in\Sigma(X)$ in any case.\footnote{The notation
{allows to} \emph{define} $H_2^+(X_\sigma):=H_2^+(X)$ for all
$\sigma\in\Sigma(X)$, by interpreting classes of curves in a stratum $X_\sigma$
as classes of curves in $X$.} We refer to \cite[Basic Setup~1.6]{GSAssoc} for a
listing of the properties of $H_2^+$ assumed throughout. Adding this
information, one arrives at the notion of
\emph{decorated type}
\begin{equation}
\label{Eqn: decorated type}
\btau=(\tau,\bA)= (G,\bg,\bsigma,\bu,\bA).
\end{equation}

Finally, just as in logarithmic Gromov-Witten theory, generization of punctured
maps gives rise to contraction morphisms of graphs: Let $(C^\circ/W,\bp,f)$ be a
punctured map to $X$ and let {$\ol w'\arr\ol w$ be a specialization
arrow of geometric points of $W$. Denote by $h:\Gamma=\Gamma(G,\ell)\arr
\Sigma(X)$ and $h':\Gamma'=\Gamma(G',\ell')\arr \Sigma(X)$ the tropicalizations
of the strict restrictions of $(C^\circ/W,\bp,f)$ to $\ol w,\ol w'$.} Then as in
\cite[(2.15)]{decomposition}, generization defines a contraction morphism of the
associated decorated graphs
\[
\phi:(G,\bg)\arr (G',\bg'),
\]
{given} by contracting those edges $E=E_q\in E(G)$ with corresponding
node $\ol q\arr \ul C_{\ol w}$ not contained in the closure of the nodal locus
of $\ul C_{\ol w'}$. By abuse of notation we write $\phi$ also for the maps
\[
V(G)\arr V(G'),\quad L(G)\arr L(G'),\quad
E(G)\setminus E_\phi\stackrel{\mathrm{bij}}{\longrightarrow} E(G')  
\]
defining $\phi$. Here $E_\phi\subseteq E(G)$ is the subset of contracted edges.
Analogous to \cite[Def.\,2.24]{decomposition} there is a corresponding natural
notion of contraction morphism of (decorated) types of tropical
punctured maps
\begin{equation}
\label{Eqn: contraction morphism tropical type}
\begin{array}{lcl}
\tau=(G,\bg,\bsigma,\bu)&\longrightarrow& \tau'=(G',\bg',\bsigma',\bu')\\
\btau=(G,\bg,\bsigma,\bu,\bA)&\longrightarrow& \btau'=(G',\bg',\bsigma',\bu',\bA')
\end{array}
\end{equation}
Under such contraction morphisms, legs never get contracted. Moreover,
identifying $L(G)=L(G')$, the contact order $\bu(L)\in N_{\bsigma(L)}$ of a leg
of $G$ is the image of {$\bu'(L)\in N_{\bsigma'(L)}$} under the
inclusion of lattices $N_{\bsigma'(L)}\arr N_{\bsigma(L)}$ induced by the face
map $\bsigma'(L)\arr \bsigma(L)$. {An analogous statement applies to
contact orders of non-contracted edges.}

\begin{proposition}
\label{Prop: Generization leads to contraction morphisms}
Let {$(C^\circ/W,\bp,f)$} be a stable {punctured} map to $X$
over some logarithmic scheme $W$ and $(\tau_{\ol w},\bA_{\ol w})$ with
$\tau_{\ol w}= (G_{\ol w},\bg_{\ol w},\bsigma_{\ol w},\bu_{\ol w})$ its
decorated type at the geometric point $\ol w\to W$ according to
Definition~\ref{Def: type} and \eqref{Eqn: decorated type}.

Then if {$\ol w'\arr \ol w$ is a specialization arrow of geometric
points of $W$}, the map
\[
(\tau_{\ol w},\bA_{\ol w})\arr (\tau_{\ol w'},\bA_{\ol w'})
\]
induced by generization is a contraction morphism.
\end{proposition}

\begin{proof}
The proof is essentially identical to \cite[Lem.\,2.30]{decomposition}, noting
that the proof of \cite[Lem.\,1.11]{LogGW} also works for contact orders at
punctures.
\end{proof}

\subsubsection{The balancing condition} 
\label{sss: balancing condition}

The above discussion fits well with the tropical balancing condition at vertices
of the dual graph of $C^\circ$. In fact, the statement \cite[Prop.\,1.15]{LogGW}
holds unchanged as there is no balancing condition at the endpoint of a leg
$L\in L(G)$. As we will need the balancing condition to prove boundedness, we
review this statement here. We note that the balancing conditions 
discussed here are heavily used in applications such as \cite{GSAssoc},
\cite[\S4]{GHKS-cubic} or \cite{GSWallStructures}, {as balancing severely limits the possible combinatorial types.}

Suppose given a stable punctured map $(C^\circ/W, \bp,f)$ with
$W=\Spec(\NN\rightarrow\kappa)$ the standard log point over an
{algebraically} closed field, and denote by $(G,\bg,\bsigma,\bu)$ its
type. Let $g:\tilde D\rightarrow C$ be the normalization of an irreducible
component $D$ with generic point $\eta$ of $C$. One then obtains, with
$\overline{\cM}=\ul{f}^*\overline{\cM}_X$, composed maps
\begin{align*}
&\tau_{\eta}^X: \Gamma(\tilde D,g^*\overline{\cM})\longrightarrow \pic \tilde
D \stackrel{\deg}{\longrightarrow}\ZZ\\
&\tau_{\eta}^C: \Gamma(\tilde D,g^*\overline{\cM}_{C^{\circ}})
\longrightarrow \pic \tilde
D \stackrel{\deg}{\longrightarrow}\ZZ
\end{align*}
with the first map on each line given by taking a section of the ghost sheaf
to the corresponding $\cO_{\tilde D}^{\times}$-torsor, {the inverse
image of this section in $g^*\cM$ or $g^*\cM_{C^{\circ}}$.}
These are compatible: the pull-back of $f^{\flat}$
to $\tilde D$, $\varphi:g^*\cM\rightarrow g^*\cM_{C^{\circ}}$, induces 
$\bar\varphi:g^*\overline{\cM}\rightarrow g^*\overline{\cM}_{C^{\circ}}$
and a commutative diagram 
\[
\xymatrix@C=30pt
{\Gamma(\tilde D, g^*\overline{\cM})\ar[r]^{\bar\varphi}\ar[dr]_{\tau^X_{\eta}}&
\Gamma(\tilde D, g^*\overline{\cM}_{C^{\circ}})\ar[d]^{\tau_{\eta}^C}\\
&\ZZ
}
\]
The map $\tau_{\eta}^X$ is given by $\ul{f}$ and $\cM$, so depends on
the logarithmic geometry of $f: C^\circ \to X$; however if $\ul{f}$ contracts
$D$, then $\tau_{\eta}^X=0$. On the other hand, $\tau_{\eta}^C$ is determined
completely by the geometry of $D\subseteq C$ and $g^*\ocM_{C^\circ}$ as follows.
We use the notation in \cite[\S1.4]{LogGW}. For each point $q\in D$ over a
node of $\uC$ we have $\overline{\cM}_{C^{\circ},\ol q}=S_{e_q}$, the submonoid
of $\NN^2$ generated by $(0,e_q)$, $(e_q,0)$ and $(1,1)$. The generization map
$\chi_q:\overline{\cM}_{C^{\circ},\ol q}\rightarrow \overline{\cM}_{C^{\circ},
\ol \eta}=\NN$ is given by projection to the second coordinate: $\chi_q(a,b)=b$.
{In what follows, we use $q$ always to denote points over nodes and
$p$ to denote punctured points.} We then have
\[
\Gamma(\tilde D, g^*\overline{\cM}_{C^{\circ}})
\subseteq \ \ \left\{(n_q)_{q\in \tilde D}\,\bigg|\,
\begin{array}{c} \hbox{$n_q\in S_{e_q}$ and $\chi_q(n_q)=\chi_{q'}(n_{q'})$} \\
\hbox{ for $q,q'\in \tilde D$}\end{array}\right\} \ \oplus\ 
\bigoplus_{p\in\tilde D}\ZZ.
\]
This inclusion induces an equality at the level of groups. The
equation $\chi_q(n_q)=\chi_{q'}(n_{q'})$ allows us to write $b=b_q=\chi_q(n_q)$
independent of $q$. We then obtain, with proof identical to that of
\cite[Lem.\,1.14]{LogGW}:

\begin{lemma}
\label{lem:tauetaC}
$\tau_{\eta}^C\big(((a_q,b)_{q\in\tilde D},(n_p)_{p\in\tilde D})\big)
=-\sum_{p\in\tilde D}n_p  +\sum_{q\in\tilde D} \frac{b-a_q}{e_q}$,
\end{lemma}

The equation $\tau_{\eta}^X=\tau_{\eta}^C\circ \varphi$ is a formula in
$N_D:=\Gamma(\tilde D, g^*\overline{\cM}^{\gp})^*$, which is described in
\cite[Eqns.~(1.12), (1.13)]{LogGW} as follows. Let $\Sigma\subset\tilde D$ be the
set of points $x$ in $\tilde D$ mapping to a special point of $C$. Thus $\Sigma$
can be identified with the subset of $E(G)\cup L(G)$ of edges or legs adjacent
to the vertex $v$ corresponding to $\eta$. For any point $x\in\tilde D$,
we write $P_x:=\overline{\cM}_{X,g(x)}$. Then
\[
N_D=\lim_{\substack{\longrightarrow\\ x\in\tilde D}} P_x^*
=\Big(\bigoplus_{x\in \Sigma} P_x^*\Big)\Big/ \sim
\]
where for any $a\in P_{\eta}^*$ and any $x,x'\in\Sigma$,
\[
(0,\ldots,0,\iota_{x,\eta}(a),0,\ldots,0) \sim (0,\ldots,0,\iota_{x',\eta}(a),0,
\ldots,0)\,.
\]
Here $\iota_{x,\eta}:P_{\eta}^*\rightarrow P_x^*$ is the dual of generization,
and the non-zero entries lie in the position indexed by $x$ and $x'$
respectively. Thus an element of $N_D$ is represented by a choice of
tangent vector $n_x\in N_{\bsigma(x)}=P_x^*$, one for each preimage $x\in\tilde
D$ of a special point of $C$; and two such choices are identified if they can be
related by repeatedly subtracting a tangent vector in $N_{\bsigma(v)}= P_\eta^*$
from one of the $n_x$ and adding it to another.

We then have, exactly as in \cite[Prop.\,1.15]{LogGW}, the balancing condition:

\begin{proposition}
\label{prop:balancing}
Suppose $(C^\circ/W,\bp,f)$ is a stable punctured map to $X/B$
with $W=\Spec (\NN\rightarrow\kappa)$ a standard log point. Let $D\subseteq
\ul{C}$ be an irreducible component with generic point $\eta$ and $\Sigma
\subset\tilde D$ the preimage of the set of special points. If $\tau^X_{\eta}
\in \Gamma(\tilde D,g^*\overline{\cM}^{\gp})^*$ is represented by 
$(\tau_x)_{x\in\Sigma}$, then
\[
(u_x)_{x\in\Sigma}+(\tau_x)_{x\in\Sigma}=0
\]
in $N_D=\Gamma(\tilde D,g^*\overline{\cM}^{\gp})^*$.
\end{proposition}

\begin{remark}
\label{Rem: tropical interpretation of balancing}
With regards to the above interpretation of elements of $N_D$ in terms
of the type of $(C^\circ/W,\bp,f)$, Proposition~\ref{prop:balancing} says the
following. The degree data of the $\cO_C^\times$-torsors contained in
$g^*\cM$ defines a tuple of tangent vectors $\tau_x\in
N_{\bsigma(x)}$, one for each edge or leg $x\in E(G)\cup L(G)$ adjacent to the
vertex $v$ corresponding to $\eta$, well-defined up to trading elements of
$N_{\bsigma(v)}$ via the embedding $N_{\bsigma(v)}\hookrightarrow
N_{\bsigma(x)}$ defined by the face morphism $\bsigma(v)\arr \bsigma(x)$. Then
(1)~$\tau_x+ u_x$ lies in the image of $P_\eta^*\to P_x^*$, and (2)~the
traditional tropical balancing condition holds in $P_\eta^*$ for $\tau_x+u_x$,
$x$ running over the set of special points.

Traditional tropical geometry arises for the case that $X$ is a toric variety
with its toric log structure. Then $\cM_X^\gp$ is the sheaf of rational
functions that are invertible on the big torus. Monomial functions define
trivial $\cO_X^\times$-subtorsors of $\cM_X^\gp$. Denoting by $N$ the
cocharacter lattice of the torus, we thus have a canonical monomorphism
\[
N^*\arr \Gamma(X,\cM_X^\gp)\arr \Gamma(X,\ocM_X^\gp)
\]
with composition with $\tau_\eta^X$ identically zero. Composing the equation in
Proposition~\ref{prop:balancing} with the induced map $N_D\arr N$ then
yields the traditional balancing condition $\sum_x \ol u_x=0$ for $\ol u_x$ the
image of $u_x$ under the embedding $N_{\sigma(x)}\arr N$.
\end{remark}

The following is an encapsulation of balancing which gives easy to use
restrictions on curve classes realized by punctured maps with given contact
orders. For the statement we denote by $\cL_s^{\times}$ the torsor corresponding
to $s\in \Gamma(X,\overline{\cM}_X^{\gp})$, that is, the inverse image of $s$
under the homomorphism $\cM_X^{\gp}\rightarrow \overline{\cM}_X^{\gp}$, and
write $\cL_s$ for the corresponding line bundle. Furthermore, the {germ} 
of $s$
at $f(p_i)$ lies in $P_{p_i}^{\gp}{=\ocM_{X,f(p_i)}^\gp}$ and hence defines a homomorphism
$P_{p_i}^*\rightarrow\ZZ$, which we write as $\langle \cdot, s\rangle$.

{
\begin{proposition}
\label{prop:intersectionnumbers}
Suppose given a punctured map $(C^\circ/W,\bp,f)$ for $W$ a log point.
Let $(G,\mathbf{g},\bsigma,\mathbf{u})$ be the type of this map,
and let $\ul{D}\subseteq \ul{C}$ be an irreducible component of the domain,
corresponding to $v\in V(G)$.
Let $p_1,\ldots,p_n\in \bp$ be the punctured points of $C^{\circ}$ 
contained in $\ul{D}$, and let $q_1,\ldots,q_m$ be the nodes of $\ul{C}$
contained in $\ul{D}$ but which are not nodes of $\ul{D}$. 
This gives rise to contact orders
$u_{p_i}, u_{q_j}$, noting that for the contact orders of the nodes,
we orient the corresponding edge away from $v$. Then we have
\[
\deg \left(\ul{f}^*\cL_s\right)|_{\ul{D}}=
-\sum_{i=1}^n \langle u_{p_i},s\rangle
-\sum_{i=1}^m \langle u_{q_i},s\rangle\,.
\]
\end{proposition}

\begin{proof}
First, by making a base-change, we can assume $W$ is the standard log point.
Note $\ul{f}^*\cL_{s}$ must be isomorphic to the line bundle
$\cL_{\bar f^{\flat}(s)}$ associated to the torsor corresponding to $\bar
f^{\flat}(s)$.

Now the total degree of $\cL_{\bar f^{\flat}(s)}$ can be calculated using
Lemma~\ref{lem:tauetaC} and details of the proof of \cite[Prop.\,1.15]{LogGW}.
Let $g:\tilde{\ul{D}}\rightarrow \ul{C}$ be the normalization of 
$\ul{D}$, and let $\eta$ be the generic point of $\ul{D}$. Then
\begin{align*}
\deg (f\circ g)^* \cL_s = {} &  \deg g^*\cL_{\bar f^{\flat}(s)}
\ =\  \tau_{\eta}^C(\varphi(s))\\
= {} & \sum_{q\in \tilde D} \frac{1}{e_q} \left(\langle V_{\eta},s\rangle
-\langle V_{\eta_q},s\rangle\right) -\sum_{{p_i}\in \tilde D}
\langle u_{{p_i}},s\rangle,
\end{align*}
in the notation of \cite[{Lem.\,1.14, Prop.\,}1.15]{LogGW}, and the
last equality coming from the proof of \cite[Prop.\,1.15]{LogGW}.
Here $V_{\eta}:P_{\eta}\rightarrow\NN$ is the map 
$\bar f^{\flat}:\overline{\cM}_{X,f(\eta)}\rightarrow \overline{\cM}_{C,\eta}$,
and similarly $V_{\eta_q}$, where $\eta_q$ is the generic point
of the other branch of $C$ at the node $q$. By \cite[(1.9)]{LogGW},
$\frac{1}{e_q}(V_{\eta}-V_{\eta_q})=-u_q$, where $u_q$ is the contact
order of the node $q$ with corresponding edge oriented away from $v$. 
Note that self-nodes of $\ul{D}$ appear twice in this sum, with
opposite sign, and hence cancel. This then yields the desired formula.
\end{proof}

\begin{corollary}
\label{intersectionnumbers}
Suppose given a punctured curve $(C^\circ/W,\bp,f)$ with $W$ a log point,
$\bp=\{p_1,\ldots,p_n\}$. Then we have
\[
\deg \ul{f}^*\cL_s=-\sum_{i=1}^n \langle u_{p_i},s\rangle\,.
\]
\end{corollary}

\begin{proof}
This is obtained from the previous proposition by summing over all
irreducible components of $\ul{C}$.
\end{proof}
}

\subsection{Basicness}
\label{ss: Basicness}

A key concept in logarithmic moduli problems is the existence of \emph{basic} or
\emph{minimal} logarithmic structures. The existence of such distinguished
logarithmic structures on the base space of families is a necessary condition
for a logarithmic moduli problem to be represented by a logarithmic algebraic
stack. A good notion of basicness should be an open property, and hence is
typically defined by a condition at geometric points.

The definition of basic stable logarithmic maps from \cite[\S1.5]{LogGW} is
based on universality of the associated family of tropical maps. The original
definition in \cite[Def.\,1.20]{LogGW} phrases this property in terms of the dual
monoids and only indicates the tropical interpretation in
\cite[Rem.\,1.18]{LogGW}. A proof of the equivalence of the definitions in the
present notation is given in \cite[Prop.\,2.28]{decomposition}. This equivalence
of descriptions really only reflects the anti-equivalence between the categories
of fs monoids and of rational polyhedral cones. In the following, we freely use
this equivalence of categories when referring to material from \cite{LogGW}.

The definition of basicness in the punctured case is formally the same
as for stable logarithmic maps. Here we take the concrete, tropical view. {For readability we again assume that $X$ is simple, deferring the general discussion to \S\ref{ss: targets with monodromy}.}

\begin{definition}
\label{Def: Basicness}
A pre-stable punctured map $(C/W,\bp,f)$ is \emph{basic at a geometric point
{$\ol w$ of $\ul W$}} if the associated family of tropical maps
\[
h:\Gamma=\Gamma(G,\ell)\arr \Sigma(X)
\]
over $(\ocM_{W,\ol w})^\vee_\RR$ is universal among tropical maps of the same
type $(G,\bg,\bsigma,\bu)$. {This means that each family of stable
tropical maps of type $(G,\bg,\bsigma,\bu)$ {over some cone $\omega$}
arises by pull-back from $h:\Gamma\arr\Sigma(X)$ via a unique map $\omega\to
(\ocM_{W,\ol w})^\vee_\RR$ in $\Cones$.} Basicness without specifying $\ol w$
refers to basicness at all geometric points.
\end{definition}

The monoids $\ocM_{W,\ol w}$ obtained from basic punctured maps also
{formally have} the same description as for stable logarithmic maps
described in \cite[Prop.\,1.19]{LogGW} and \cite[Prop.\,2.28]{decomposition}. We
provide a full proof of this description emphasizing the tropical perspective.

\begin{proposition}
\label{Prop: Basic monoid}
Let $({C^\circ/ W},\bp,f)$ be a basic, pre-stable punctured map over a
logarithmic point $\Spec(Q\arr \kappa)$ with $\kappa$ an algebraically closed
field, and let $(G,\bg,\bsigma,\bu)$ be its type. For each generic point
{$\eta\in \ul C$} with $v=v_\eta\in V(G)$ the associated vertex write
\[
P_\eta=\ocM_{X,\ul f(\eta)} = \big(\bsigma(v)_\ZZ\big)^\vee.
\]
Then the map
\begin{equation}
\label{Eqn: basic monoid}
\textstyle
Q^\vee\arr \big\{\big((V_\eta)_\eta,(\ell_q)_q\big)\in\prod_\eta P_\eta^\vee\times
\prod_q\NN\,\big|\, V_\eta- V_{\eta'}= \ell_q\cdot \bu(q) \big\}
\end{equation}
given by the duals of ${\big(\ol{\pi^\flat_\eta}\big)^{-1}}\circ\bar
f^\flat_\eta: P_\eta\arr Q$ and of the classifiying map $\prod_q \NN\arr Q$ of
the log smooth curve $C/W$, is an isomorphism. Here $q$ runs over the set of
nodes of $\ul C$ and, in the equation, $\eta,\eta'$ are the generic points of
the adjacent branches, with the order chosen as in the definition of $\bu$.
\end{proposition}

\begin{proof}
Denote by $\omega\in\Cones$ the cone defined by the right-hand side of
\eqref{Eqn: basic monoid}. We first construct a tropical punctured map
\[
h_0:\Gamma=\Gamma(G,\ell_0)\arr \Sigma(X)
\]
over $\omega$ as follows. Define
\begin{equation}
\label{Eqn: ell_0, h_0}
\ell_0(E):\omega_\ZZ\arr \NN,\qquad
h_0\big(v):\omega_\ZZ\arr P_\eta^\vee
\end{equation}
for $E=E_q\in E(G)$ and $v=v_\eta\in V(G)$ as the projections to the $q$-th
factor in $\prod_q \NN$ and to $P_\eta^\vee= \bsigma(v)_\ZZ$, respectively. For
an edge $E=E_q$ with adjacent vertices $v$, $v'$ and associated cone $\omega_E$
from \eqref{Eqn: omega_E}, the map $h_0$ is defined by
\begin{align*}
{(h_0)_E: \omega_E}&\longrightarrow
(P_q^\vee)_\RR= \bsigma(E),\\
(s,\lambda)&\longmapsto h_0(v)(s)+\lambda \cdot \bu(E)
= h_0(v')(s)+(\ell_0(E)-\lambda)(-\bu(E)),
\end{align*}
with the sign of $\bu(E)$ chosen according to the orientation of $E$. In this
definition, we view $h_0(v(s))$, $h_0(v'(s))$ as elements of $(P_q^\vee)_\RR$ via
the face inclusions $P_\eta^\vee, P_{\eta'}^\vee\arr P_q^\vee$. The
equality holds by the relation in the definition of $\omega$ {by the
right-hand side of \eqref{Eqn: basic monoid}}. In particular,
$(h_0)_E$ restricts to $h_0(v)$, $h_0(v')$ on its two faces defined by $v,v'$.

Finally, for a leg $L=L_p\in L(G)$ with adjacent vertex $v\in V(G)$, the length
function $\ell_0(L)$ and the map $(h_0)_L$ defined on $(\omega_L)_\ZZ$ is
uniquely determined by $h_0(v)$ and by the contact order $\bu(L)$ via
pre-stability (Proposition~\ref{Prop: pre-stable tropical punctured maps}). This
finishes our construction of a pre-stable tropical punctured map $h_0$ over
$\omega$.

Conversely, if $h:\Gamma=\Gamma(G,\ell)\arr \Sigma(X)$ is a tropical punctured
map of type $(G,\bg,\bsigma,\bu)$ over some cone $\omega'\in\Cones$, the map
\[
\omega'\arr \omega,\quad s\longmapsto \big(h(v_\eta(s)),\ell(E_q)\big)_{\eta,q},
\]
with $v_\eta:\omega'\arr \Gamma$ the section of $\Gamma\arr\omega'$
defined by $v_\eta\in V(G')$, is readily seen to be the unique morphism in
$\Cones$ producing $h$ by pull-back from $h_0$.
\end{proof}

{}

\begin{definition}
\label{Def: basic monoid}
The fs monoid $Q$ defined by \eqref{Eqn: basic monoid} is called the \emph{basic
monoid associated to the type $\tau=(G,\bg,\bsigma,\bu)$}, while its dual
$Q^\vee\in \Cones$ (or $Q_\RR^\vee$ with the integral structure understood) is
called the \emph{associated basic cone}.
\end{definition}

{
Note that while the definition of the basic monoid makes sense for all types,
the length function $\ell_0(E)$ {constructed in \eqref{Eqn: ell_0, h_0}
in the proof of Proposition~\ref{Prop: Basic monoid}} may be zero for some edge
$E$. In this case, the universal tropical domain $\Gamma(G,\ell_0)$ in the proof
of Proposition~\ref{Prop: Basic monoid} is not the domain of a tropical
punctured map according to Definition~\ref{Def: Punctured tropical curve}. The
basic monoid is therefore only meaningful if there exists at least one tropical
punctured map of the given type.}\footnote{The analogue of this statement in
\cite{LogGW} is the condition $\ul{\text{GS}}(\ocM)\neq\emptyset$.} Observe also
that just as marked points do not enter in the definition of basicness, there is
no role for punctures in the statement of Proposition~\ref{Prop: Basic monoid}.

\begin{proposition}\label{basic-open}
Let $(C^\circ/W, \bp,f)$ be a pre-stable punctured map. Then 
\[
\Omega:=\{\ol w\in |\ul{W}|\,\big|\,
\hbox{$\{\ol w\}\times_{\ul W}(C^\circ/W,\bp,f)$ is basic}\}
\]
is an open subset of $|\ul{W}|$.
\end{proposition}

\begin{proof}
This is identical to \cite[Prop.\,1.22]{LogGW}.
\end{proof}

\begin{proposition}
\label{universalproperty}
Any pre-stable punctured map to $X\to B$ arises as the pull-back from a
basic pre-stable punctured map to $X\to B$ with the same underlying ordinary
pre-stable map. Both the basic pre-stable punctured map and the morphism are
unique up to a unique isomorphism.
\end{proposition}

\begin{proof}
The proof is almost identical to \cite[Prop.\,1.24]{LogGW}. Let $(\pi:C\arr W,
\bp,f)$ be a pre-stable punctured map over $B$. For each geometric point $\ol
w\arr \ul W$ one obtains a tropical punctured map
\[
h_{\ol w}: \Gamma_{\ol w}\arr \Sigma(X)
\]
over $\omega_{\ol w}=(\ocM_{W,\ol w})^\vee_\RR$, of some type $(G_{\ol w},
\bg_{\ol w},\bsigma_{\ol w}, \bu_{\ol w})$. By {Proposition~\ref{Prop:
Generization leads to contraction morphisms}}, generization $\ol w\in \cl(\ol
w')$ {(i.e.\ existence of a specialization arrow $\ol{w}'\to{\ol w}$ as
in Appendix~\ref{App: Functorial tropicalization})} leads to a contraction
morphism \eqref{Eqn: contraction morphism tropical type}
\[
(G_{\ol w}, \bg_{\ol w},\bsigma_{\ol w}, \bu_{\ol w}) \arr
(G_{\ol w'}, \bg_{\ol w'},\bsigma_{\ol w'}, \bu_{\ol w'}).
\]
This contraction morphism induces an embedding of $\Gamma_{\ol w'}$ as a
subcomplex of $\Gamma_{\ol w}$ such that $h_{\ol w'}$ becomes the restriction
of $h_{\ol w}$. These maps are compatible with the classifying maps to the dual
of the respective basic monoids in Proposition~\ref{Prop: Basic monoid},
producing a cartesian diagram of pre-stable tropical punctured maps.

As in the proof of \cite[Prop.\,1.24]{LogGW}, this situation produces monoid
sheaves $\ocM_{C^\circ}^\basic$, $\ocM_W^\basic$ on $\ul C$ and $\ul W$,
respectively, and a commutative diagram
\begin{equation}
\label{Eqn: universal diagram for basic sheaf}
\vcenter{\xymatrix{
\ul f^*\ocM_X\ar[r]& \ocM_{C^\circ}^\basic\ar[r]&\ocM_{C^\circ}\\
& \ul\pi^*\ocM_W^\basic\ar[r]\ar[u]&\ul\pi^*\ocM_W\ar[u]
}}
\end{equation}
In case $B$ has a non-trivial log structure, all morphisms are compatible with
morphisms from the pull-back of $\ocM_B$. Continuing as in
\cite[Prop.\,1.24]{LogGW}, we can now define the desired basic log structures by
fiber product:
\[
\cM_{W}^\basic=\cM_W\times_{\overline{\cM}_W} \overline{\cM}_{W}^\basic,\quad
\cM_{C^{\circ}}^\basic=\cM_{C^{\circ}}\times_{\overline{\cM}_{C^{\circ}}} 
\overline{\cM}_{C^{\circ}}^\basic.
\]
Each of these defines a log structure with the structure map being the
composition of the projection to the first factor followed by the structure map
for that log structure. The pair of induced morphisms
\[
\pi_\basic:C^\circ_\basic=(\ul C,\cM_{C^\circ}^\basic)\arr W_\basic
=(\ul W,\cM_W^\basic),\quad f_\basic: C^\circ_\basic\arr X
\]
have tropicalizations at any geometric point {$\ol w$ of $\ul W$} given
by the universal pre-stable tropical punctured map to $\Sigma(X)$ over
$\Sigma(B)$ of type $(G_{\ol w}, \bg_{\ol w},\bsigma_{\ol w}, \bu_{\ol w})$.
Thus $(C^\circ_\basic/W_\basic,\bp,f)$ is a basic punctured map to $X$. By the
construction by fiber products of monoid sheaves, it follows that $f_\basic$
commutes with the morphisms to $B$, and that $(C^\circ/W,\bp,f)$ is the
pull-back of $(C^\circ_\basic/W_\basic,\bp,f)$ by $W\arr W_\basic$. The
constructed basic punctured map is also pre-stable since $(C^\circ/W,\bp,f)$ is
and by the definition of $\cM_{C^\circ}^\basic$ as a fiber product. Finally, the
universal property of the basic monoid with regards to pre-stable tropical
punctured maps in Proposition~\ref{Prop: Basic monoid} implies uniqueness.
\end{proof}

{
\begin{remark}
\label{Rem: Gap in GS2}
Following \cite{LogGW}, our construction of the basic pre-stable punctured map
in the proof of Proposition~\ref{universalproperty} argues pointwise and uses
compatibility with generizations to obtain the universal diagram of ghost
sheaves. However, the existence of an \'etale sheaf with the stated stalks and
generization maps is never checked, notably in the proof of
\cite[Lem.\,1.23]{LogGW}. We use this occasion to close this gap.

The basic monoids and generization homomorphisms define a
contravariant functor
\begin{equation}
\label{Eqn: functor of stalks}
\Pt(W)\arr \Mon,\quad \ol w\longmapsto Q_{\ol w}
\end{equation}
from the category of geometric points $\Pt(W)$ with specialization arrows,
recalled at the beginning of Appendix~\ref{App: Functorial tropicalization}, to
the category of monoids. A specialization arrow $\ol w\arr\ol w'$ maps to an
{epimorphism} of monoids $Q_{\ol w'}\arr Q_{\ol w}$ given by localization at
a face and subsequently dividing out the subgroup of invertible elements. In any
case, from a functor as in \eqref{Eqn: functor of stalks} one can define an
\'etale sheaf $\ocM^\basic$ by associating to an \'etale map $h:U\arr X$ the
monoid
\[
\ocM^\basic(U)=\colim_{\ol w\arr h} Q_{\ol w},
\]
together with the obvious restriction maps. Here the colimit is taken over all
factorizations of $\ol w:\Spec \kappa(\ol w)\to X$ over $h$. The gap in
\cite{LogGW} concerns the implicit claim that for a geometric point $\ol w$ of
$X$ the natural map
\[
Q_{\ol w}\arr \ocM^\basic_{\ol w}
\]
is an isomorphism.

This claim is \'etale local in $\ul W$. Hence we can assume that the given
(non-basic) log structure $\cM_W$ on $W$ is a Zariski log structure with a
global chart that is neat at some geometric point $\ol w$. We may also assume
that the logarithmic stratum containing $\ol w$ lies in the closure of all
other strata, and that the restriction map
\[
\Gamma(W,\ocM_W)\arr \ocM_{W,\ol w}
\]
is an isomorphism. By \cite[Prop.\,II.2.1.2]{Ogus} we obtain a continuous map
\[
g:|\ul W|\arr S=\Spec \ocM_{W,\ol w}
\]
from the topological space underlying $\ul W$ to the monoidal scheme of prime
ideals of $\ocM_{W,\ol w}$, a finite topological space, together with an
isomorphism
\[
g^{-1}\ocM_S\arr \ocM_W.
\]
Here $\ocM_S$ is the structure sheaf of $\Spec \ocM_{W,\ol w}$, a sheaf of sharp
monoids.\footnote{We have reinterpreted the statement in \cite{Ogus} as a
statement for Kato fans to avoid dealing with invertible elements, which are
irrelevant for our discussion.}

Note that a finite topological space is an Alexandrov space. Thus a subset is
closed iff it is closed under specialization, and sheaves (of sets, say) are
indeed given by contravariant functors from the set of points to $\Sets$, see
e.g.\ \cite[\S2]{Ladkani}.

The universal property of basic monoids provides a monoid homomorphism
\[
Q_{\ol w}\arr \ocM_{W,\ol w},
\]
hence a morphism of monoid spectra
\[
k:S=\Spec \ocM_{W,\ol w}\arr S_\basic=\Spec Q_{\ol w}.
\]
Compatibility of the basic monoids and their universal property with
generization now shows first that $(k\circ g)^{-1}\ocM_{S_\basic}$ is a sheaf
of monoids with stalks equal to the basic monoids on $W$ and having the expected
generization homomorphisms, hence defines $\ocM_W^\basic$, and second that the
composition
\[
\ocM_W^\basic=(k\circ g)^{-1}\ocM_{S_\basic}\arr g^{-1}\ocM_S\arr \ocM_W
\]
stalkwise restricts to the classifying homomorphisms for $\ocM_W$.

A similar argument on $C$ provides the remaining parts of Diagram~\eqref{Eqn: universal
diagram for basic sheaf}.
\end{remark}
}

\begin{proposition}
\label{noautomorphisms}
An automorphism $\varphi:C^{\circ}/W\rightarrow C^{\circ}/W$ of a basic
pre-stable punctured map $(C^{\circ}/W, \bp, f)$ with
$\ul{\varphi}=\id_{\ul{C}^{\circ}}$ is trivial.
\end{proposition}

\begin{proof}
This is identical to \cite[Prop.\,1.25]{LogGW}.
\end{proof}

\subsection{Global contact orders {and global types}}\mbox{}
\label{subsec: global contact}
A fundamental ingredient in the definition of logarithmic
Gromov-Witten invariants is the global specification of contact orders at the
marked points. The local behaviour of contact orders in families of stable
logarithmic maps is captured by the notion of morphism of types \eqref{Eqn:
contraction morphism tropical type}, implying that generization leads to the
possible propagation of contact orders via face inclusions in $\Sigma(X)$. The
global definition can be subtle in the presence of monodromy, as the following
examples show.

\begin{example}
\label{ex:mobius}
This example is modelled on the well-known toric construction of the
Tate curve. Let $Y$ be the three-dimensional toric variety (not of finite type)
defined by the fan consisting of the collection of three-dimensional cones
\begin{align*}
\Sigma^{[3]}={} & \big\{\RR_{\ge 0}(n,0,1)+\RR_{\ge 0}(n+1,0,1)+\RR_{\ge 0}
(n,1,1)+\RR_{\ge 0}(n+1,1,1)\,\big|\,n\in\ZZ\big\}
\end{align*}
and their faces. Projection onto the third coordinate yields a
toric morphism $Y\rightarrow \AA^1$. After a base-change $\widehat{Y}
=Y\times_{\AA^1}\Spec \kk[[t]]\rightarrow \Spec\kk[[t]]$, one may divide
out $\widehat{Y}$ by the action of $\ZZ$ defined as follows.
This action is generated by an automorphism of $\widehat{Y}$ induced
by an automorphism of $Y$ defined over $\AA^1$. This automorphism is
given torically via the linear transformation $\ZZ^3\rightarrow\ZZ^3$
given by the matrix 
\[
\begin{pmatrix} 
1&0&\ell\\
0&-1&1\\
0&0&1
\end{pmatrix}
\]
where $\ell$ is a fixed positive integer.We then define
$X=\widehat{Y}/\ZZ$, with log structure induced by the toric log
structure on $Y$ (Figure~\ref{fig:monodromy}). 

\begin{figure}[htb]
\input{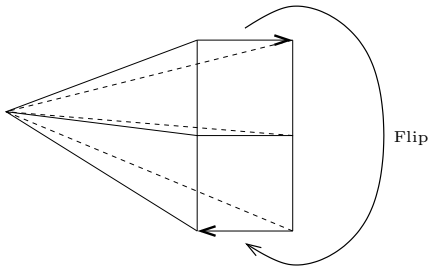}
\caption{Tropicalization of a Zariski logarithmic scheme with contact
order monodromy: $\ell=2$.}
\label{fig:monodromy}
\end{figure}

Then $X\rightarrow \Spec\kk[[t]]$ is a degeneration of the total space of a
$\GG_m$-torsor over an elliptic curve, the torsor corresponding to a $2$-torsion
element of the Picard group of the elliptic curve. As long as $\ell\ge 2$, $X$
has a Zariski log structure. Further, $\Sigma(X)$ is a cone over a M\"obius
strip composed of $\ell$ squares. If one takes $u=(0,1,0)\in\sigma^{\gp}$ for
any three-dimensional cone in $\Sigma(X)$, then propagating $u$ via chains of
face inclusions identifies $u$ with $-u$ due to the twist in the M\"obius strip.
\end{example}

{
\begin{example}
\label{Expl: Jonathan's example}
A variant of the previous example {that we learnt from Jonathan Wise}
also produces monodromy of infinite order. \label{Expl: Infinite monodromy}

Let $\sigma\subset  \RR^4$ be the cone generated by the following column vectors:
\[
\begin{aligned}
v_1 &= (0,0,0,1)^t,\
v_2 &= (0,1,0,1)^t,\
v_3 &= (0,0,1,1)^t,\
v_4 &= (0,1,1,1)^t,\\
v_5 &= (1,0,1,1)^t,\
v_6 &= (1,1,1,1)^t,\
v_7 &= (2,1,0,1)^t,\
v_8 &= (2,2,0,1)^t.
\end{aligned}
\]
The linear transformation of $\RR^4$ with matrix
\[
A=\left(\begin{matrix}
1& 0& -1& 2\\
0& 1& -1& 1\\
0& 0& 1& 0\\
0& 0& 0& 1\\
\end{matrix}\right)
\]
fulfills
\[
Av_1=v_7,\quad Av_2= v_8,\quad Av_3=v_5,\quad Av_4=v_6.
\]
Thus $A\tau_1=\tau_2$ for the two facets 
\[
\tau_1 = \langle v_1, v_2, v_3, v_4\rangle,\quad
\tau_2 = \langle v_5,v_6,v_7,v_8\rangle
\]
of $\sigma$.

Now $\tau_1^\gp\cap \tau_2^\gp$ is the lattice spanned by
\[
x =  2 v_3 - v_1 = 2 v_6 - v_8 =(0,0,2,1) ,\quad 
y = v_2 - v_1 = v_6 - v_5 = (0,1,0,0).
\]
The restriction of $A$ to this sublattice is a shear transformation, hence is of
infinite order:
\[
Ax= x-y, \quad  Ay=y.
\]

It is not hard to define a log structure on the nodal cubic curve $\ul X$ with
$\ocM_{X,q}^\vee\simeq \sigma\cap \ZZ^4$ at the node $q$, and the
generization maps to the two branches of $C$ at $q$ dual to the inclusions
$\tau_1,\tau_2\hookrightarrow \sigma$. Then $X=(\ul X,\cM_X)$ has infinite
monodromy.

By pulling back $\cM_X$ to a two-nodal curve of arithmetic genus $1$, with the
map to $X$ contracting a $\PP^1$, produces an example with Zariski log structure
and infinite monodromy.

Note that the feature of infinite monodromy can not be seen from the underlying
topological space of the tropicalization $\Sigma(X)$. In fact, as a topological
space, $\sigma$ is the cone over a polyhedron $\Xi\subset\RR^3$ that is the
convex hull of two disjoint facets with four vertices each, the intersections of
$\tau_1,\tau_2$ with the affine hyperplane $x_4=1$ for $x_1,\ldots,x_4$ the
coordinates on $\RR^4$. Thus $|\Sigma(X)|$ is the cone over the cell complex
obtained from $\Xi$ by identifying these two facets. But replacing $v_7,v_8$
with $(2,0,0,1)^t$, $(2,1,0,1)^t$ {and adapting $A$ accordingly}
produces an example with homeomorphic $|\Sigma(X)|$ and without monodromy.
\end{example}}

In the presence of monodromy as in Examples~\ref{ex:mobius} {and
\ref{Expl: Infinite monodromy}}, the na\"ive definition of global contact orders
by a reduced subscheme $\ul Z\subset \ul X$ and a section $s\in \Gamma(\ul Z,
({\ocM_X}|_{\ul Z})^*)$ not extending to any larger subscheme from
\cite[Def.\,3.1]{LogGW} does not work. We provide here an alternative
{treatment based on a notion of tangent vectors for the generalized cell
complex $\Sigma(X)$} that suffices for the definition of finite type moduli
spaces and of certain punctured Gromov-Witten invariants also in cases with
monodromy. Some applications such as gluing (Theorem~\ref{Thm: Gluing theorem})
in rare cases may require the more refined definition presented in
Appendix~\ref{sec:contact}. For the sake of simplicity of presentation, we
merely indicate what has to be modified to treat such rare cases.

\subsubsection{Global contact orders}
\label{sss: global contact orders}
For $\sigma\in\Sigma(X)$ denote by $\Sigma_\sigma(X)$ the star of $\sigma$,
considered as the category $\Sigma(X)$ under $\sigma$, i.e., the category with
objects face embeddings $\sigma\to\sigma'$ in $\Sigma(X)$ and {arrows}
given by morphisms $\sigma'\rightarrow\sigma''$ commuting with the given
morphisms from $\sigma$. Thus the star $\Sigma_\sigma(X)$ is formed by all cones
${\sigma_{\ol x}}=\Hom(\ocM_{X,\ol x},\RR_{\ge 0})$ with $\ol x$ running
over the {geometric points of} $X_\sigma$. Associating to
$(\sigma\to\sigma')\in \Sigma_\sigma(X)$ the free abelian group $N_{\sigma'}$,
viewed as a set, gives a diagram in the category of sets indexed by
$\Sigma_\sigma(X)$. {This diagram can be viewed as the diagram of
integral tangent vectors of $\Sigma_\sigma(X)$. Taking the colimit in the
category of sets provides a set of homomorphisms $\ocM_{X,\ol x}\arr\ZZ$ for
geometric points $\ol x$ of $X_\sigma$ compatible with all generization
homomorphisms. Elements of this colimit therefore provide a way to specify
{compatible sets of} contact orders along the stratum $X_\sigma$
independently of monodromy.}

\begin{definition}
\label{Def: global contact order}
Let $\sigma\in\Sigma(X)$ and $\fN_\sigma: \Sigma_\sigma(X)\arr \Sets$ be the
diagram in the category of sets mapping $\sigma\to\sigma'$ to $N_{\sigma'}$. A
\emph{{global} contact order for $\sigma\in\Sigma(X)$}, or for the
corresponding stratum $X_\sigma\subseteq \ul X$, is an element $\ol u$ of
\[
\fC_\sigma(X):= \colim^\Sets\, \fN_\sigma
= \colim^\Sets_{\sigma\to\sigma'} N_{\sigma'},
\]
the \emph{set of contact orders for $\sigma$}. For $\sigma'\in\Sigma_\sigma(X)$,
or for {a geometric point $\ol x$ of} $X_\sigma$, we denote by
\[
\iota_{\sigma\sigma'}: N_{\sigma'}\arr \fC_\sigma(X),\quad
{\iota_{\sigma \ol x}:N_{\sigma_{\ol x}}\arr \fC_\sigma(X)}
\]
the canonical maps.

A \emph{global contact order} is a contact order for {some}
$\sigma\in\Sigma(X)$. The set of global contact orders is denoted
$\fC(X):=\coprod_{\sigma\in\Sigma(X)} \fC_\sigma(X)$. 

We say a contact order $\ol u$ for $\sigma\in\Sigma(X)$ has \emph{finite
monodromy} if for {all} ${(\sigma\rightarrow}\sigma')
\in\Sigma_\sigma(X)$ the set $\iota_{\sigma\sigma'}^{-1}(\ol u)\subseteq
N_{\sigma'}$ is finite.

A global contact order $\ol u\in \fC_\sigma(X)$ is \emph{monodromy-free} if for
{all} $(\sigma\arr\sigma')\in \Sigma_\sigma(X)$ there exists at most one
$u\in N_{\sigma'}$ with $\ol u=\iota_{\sigma\sigma'}(u)$.
\end{definition}

{To be explicit, we spell out the definition of $\iota_{\sigma\ol
x}$ for $\ol x$ a geometric point of $X_\sigma$. Let $Z\subseteq X$ be the
smallest logarithmic stratum containing the image of $\ol x$. Then since $Z\cap
X_\sigma\neq\emptyset$, the definition of $\Sigma(X)$ provides an isomorphism
\[
\sigma_{\ol x}=\Hom(\ocM_{X,\ol x},\RR_{\ge0})
\stackrel{\simeq}{\longrightarrow}
\Hom(\ocM_{X,\ol x_Z},\RR_{\ge0})=\sigma_Z
\]
together with a face map $\sigma\arr\sigma_Z$, unique up to arrows
$\sigma\arr\sigma$ and $\sigma_Z\arr\sigma_Z$ in $\Sigma(X)$. Then
$\iota_{\sigma\ol x}$ is defined by composing the induced isomorphism of
lattices $N_{\sigma_{\ol x}}\simeq N_{\sigma_Z}$ with $\iota_{\sigma\sigma_Z}$.
The definition of $\fC_\sigma(X)$ is designed to make all maps $\iota_{\sigma\ol
x}$ independent of choices. In particular, a contact order as in \eqref{Eqn:
contact order} and \eqref{Eqn: contact orders} has an associated global contact
order.}

{Note that if $\ocM_X$ has monodromy along $X_\sigma$, there is a
non-trivial group $G$ of arrows $\sigma\to\sigma$ in $\Sigma(X)$. In this case,
the map $\iota_{\sigma\sigma}: N_\sigma\to \fC_\sigma(X)$ factors over the
quotient $N_\sigma\arr N_\sigma/G$ of the induced linear action of $G$ on
$N_\sigma$. In particular, two tangent vectors $u,u'\in N_\sigma$ define the
same global contact order $\ol u=\iota_{\sigma\sigma}(u)=
\iota_{\sigma\sigma}(u')$ if they are related by monodromy {along
$X_\sigma$}.}

Given a punctured map $(C^\circ/W,\bp,f)$ to $X$ and $s:\ul W\arr \ul C$ a
punctured or nodal section, each geometric point {$\ol w$ of $\ul W$}
has an associated contact order $u_{s(\ol w)}$ at $s(\ol w)$, {giving the
contact orders $u_p,u_q$ of} \eqref{Eqn: contact orders} of the associated
tropicalization:
\[
u_{s(\ol w)}: \ocM_{X,\ul f(s(\ol w))}\arr \ZZ.
\]
{Recall also that the contact order for a node, defined in
\cite[(1.8)]{LogGW}, depends on the choice of an ordering of the two branches of
$\ul{C}_{\ol w}$ through the node {$q$}, just as $u_E{=u_q}$
depends on the choice of orientation of the edge $E$.} {Now} for any
$\sigma\in\Sigma(X)$ with $\im(\ul f\circ s)\subseteq X_\sigma$ and any $\ol
w\arr \ul W$, {we obtain the induced global contact order}
\begin{equation}
\label{Eqn: global contact order at geometric point}
u_s^\sigma(\ol w)= \iota_{\sigma \ul f(s(\ol w))}(u_{s(\ol w)})
\end{equation}
{
The following lemma shows that fixing global contact orders in families of
punctured maps is both an open and closed condition. In particular, prescribing
global contact orders for strata, formalized in the notion of marking below
(Definition~\ref{Def: marking by type}), works well in moduli problems.}

\begin{lemma}
\label{Lem: contact orders are locally constant}
Let $(C^\circ/W,\bp,f)$ be a punctured map, $s:\ul W\arr \ul C$ a punctured or
nodal section, and $\sigma\in\Sigma(X)$ with $\im(\ul f\circ s)\subseteq
X_\sigma$. Then the function $\ol w\mapsto u_s^\sigma(\ol w)$ from \eqref{Eqn:
global contact order at geometric point}, associating to a geometric point
{$\ol w$ of $\ul W$} the {global} contact order of
$(C^\circ_{\ol w}/\ol w,\bp_{\ol w}, f_{\ol w})$ for $\sigma$, is locally
constant.
\end{lemma}

\begin{proof}
{
The existence of neat charts for the morphism $f:C^\circ\arr X$
\cite[Thm.\,III.1.2.7]{Ogus} shows that the composition
\[
s^{-1}f^{-1}\ocM_X\arr s^{-1}\ocM_{C^\circ}\arr \ul\ZZ,
\]
is a morphism of constructible sheaves of sets. See also
\cite[Thm.\,II.2.5.4]{Ogus}. This composition defines the contact order as a
function on $\ul W$. Hence the subset of $\ul W$ with $f$ of a given contact
order is a constructible set. It remains
}
to show that contact orders are compatible with generization. Consider a
specialization $\ol w'$ of $\ol w$, with $\ul f\circ s (\ol w') = \ol x'$ a
specialization of $\ul f\circ s (\ol w) = \ol x$. By Proposition~\ref{Prop:
Generization leads to contraction morphisms} the face embedding $N_{\sigma_{\ol
x}} \to N_{\sigma_{\ol x'}}$ dual to generization, which is an arrow in
$\fN_\sigma$, maps {the contact order} $u_{\ol x}\in N_{\sigma_{\ol x}}$
to $u_{\ol x'}\in N_{\sigma_{\ol x'}}$. Hence $\iota_{\sigma \ol x}(u_{\ol x}) =
\iota_{\sigma \ol x'}(u_{\ol x'})$, as needed.
\end{proof}

\begin{definition}
\label{Def: global contact order along section}
Let $(C^\circ/W,\bp,f)$ be a punctured map, and $s:\ul W\arr \ul C$ a punctured
or nodal section with $\im(\ul f\circ s)\subseteq X_\sigma$ for some $\sigma\in
\Sigma(X)$. Then $(C^\circ/W,\bp,f)$ is said to have \emph{global contact order
$\ol u\in \fC_\sigma(X)$ for $\sigma$ along $s$} if for each geometric point
{$\ol w$ of $\ul W$} the function in \eqref{Eqn: global contact order at
geometric point} fulfills $u_s^\sigma(\ol w)=\ol u$.
\end{definition}

{
\begin{remark}
A previous version of this paper contained a notion of evaluation stratum for a
global contact order. This was meant as the analogue of the pull-back via
$X\arr\cA_X$ of the image of $\cZ_\sigma\to\cA_X$ in the notion of contact
orders based on the Artin fan of $X$ developed in Appendix~\ref{sec:contact
zariski}. We decided to remove this part for several reasons.

First, the given treatment was ad hoc since unlike in the notion based on Artin
fans, there is no good functorial characterization of schematic evaluation
strata based on families of punctured curves. This lack of a universal property
is due to possible obstructions to deformations of punctured maps not coming
from obstructions to deformations of the evaluation point.

Second, contact orders are naturally selected after fixing a reference stratum,
see \S\ref{ss: stacks marked by types} below. In the most important case of
realizable types of punctured maps (Definition~\ref{Def: global type},(2)
below), the reference stratum already defines a reduced closed subscheme of the
evaluation stratum for the given contact order. Thus defining a non-reduced
evaluation stratum is pointless in this case. Indeed, so far there has not been
any use of non-reduced evaluation strata in practice, and notably not in the
applications mentioned in the introduction.

Third, should there ever be a need to define a non-reduced evaluation stratum,
it can easily be defined via the theory of contact orders developed in
Appendix~\ref{sec:contact}.
\end{remark}
}

\subsubsection{Global types}
\label{sss: global types}
As emphasized throughout the paper, a central aspect of the theory of
punctured maps involves the underlying combinatorics in terms of tropical
geometry. On the level of moduli spaces, this aspect is captured by the notion
of marking by tropical types.

{For this purpose,} we need a global version of the type of punctured maps (Definition~\ref{Def:
type}). {Crucially we replace contact orders by the
global contact orders from Definition~\ref{Def: global contact order}.}
{For readability we again restrict to the case of simple $X$ first. The discussion of the additional data needed for the general case is contained in
\S\ref{ss: targets with monodromy}.}

\begin{definition}
\label{Def: global type}
(1)~A \emph{global type (of a family of tropical punctured maps {to $\Sigma(X)$})} is a tuple
\[
\tau=(G,\bg,\bsigma,\bar\bu)
\]
consisting of a genus-decorated connected graph $(G,\bg)$ and two maps
\[
\textstyle
\bsigma: V(G)\cup E(G)\cup L(G)\arr \Sigma(X),\quad
\bar\bu: E(G)\cup L(G)\arr \fC(X)
\]
with $\bar\bu(x)\in \fC_{\bsigma(x)}$ for each $x\in E(G)\cup L(G)$. A (type of)
punctured maps has an \emph{associated global type} by replacing the contact
orders by the associated global contact orders. \emph{Morphisms of global types}
are defined analogously to morphisms of types of tropical punctured maps in
\eqref{Eqn: contraction morphism tropical type}.

{If the composition of $\bar\bu$ with the natural map $\fC(X)\arr\fC(B)$ equals $0$, we say $\tau$ is a global type for $X/B$ or relative $B$.}
\\[1ex]
(2)~A global type $\tau$ is \emph{realizable}\footnote{The term signifies that
the combinatorial data underlies a tropical object. It should not be confused
with realizability in tropical algebraic geometry, which signifies that a
tropical object is the tropicalization of an algebraic object.} if there exists
a tropical map to $\Sigma(X)$ with associated global type $\tau$.\\[1ex]
(3)~A \emph{decorated global type} $\btau=(\tau,\bA)$ of tropical punctured maps
is obtained by adding a curve class $\bA$ as in \eqref{Eqn: decorated
type}.\\[1ex]
(4)~A \emph{class of tropical punctured maps} {for a connected $X$} is a
decorated global type with a graph $G$ with only one vertex $v$, no edges, and
all strata $\bsigma(x)=\{0\}$. We write a class of tropical punctured maps as
$\beta=(g,\bar \bu,A)$ with $g\in\NN$, $A\in H_2^+(X)$ and $\bar\bu:
L(G)\arr\fC_{\{0\}}(X)$. The \emph{class of a decorated global type} is the
class of tropical punctured maps obtained by contracting all edges and keeping
the set of legs, but with associated strata $0\in\Sigma(X)$ and each global
contact order the image under the canonical map
\[
\fC_{\bsigma(L)}(X)\arr \fC_{\{0\}}(X).
\]
For a class $\beta$ of a global type we write $\ul\beta= (g,k,A)$ with
$k=|L(G)|$ for the class of the underlying ordinary stable map.

{If $X$ is disconnected, one takes one class of tropical punctured map
for each connected component of $X$.}
\end{definition}

We will often drop the adjective \emph{``tropical''} and refer to a global type,
decorated global type, or class of punctured maps.

{The following lemma will only be used in the proof of
Proposition~\ref{prop:I when tau is realizable}, which in turn is only used in
the dimension formulas of Proposition~\ref{Prop: pure-dimensional fM(cX,tau)}.}

\begin{lemma}
\label{Lem: unique type for global type}
{Let $(G,\bg,\bsigma,\bar\bu)$ be a realizable global type, and assume
all logarithmic strata $Z_\sigma\subseteq X$ for $\sigma\in\im(\bsigma)$ are
monodromy-free.} Then there is a unique type $\tau=(G,\bg,\bsigma,\bu)$ of
punctured maps with associated global type $(G,\bg,\bsigma,\bar\bu)$.
\end{lemma}
\begin{proof} 
Indeed, realizability implies in particular that for each $x\in E(G)\cup L(G)$,
the contact order $u_x\in \fC_{\bsigma(x)}(X)$ lies in the image of the natural
map $N_{\bsigma(x)} \rightarrow \fC_{\bsigma(x)}(X)$. 
However, it follows
immediately from the definition of $\fC_{\sigma}(X)$ that the map
$N_\sigma\rightarrow \fC_{\sigma}(X)$ is injective for each
$\sigma\in\Sigma(X)$.
\end{proof}

{A sufficient condition for the absence of monodromy in Lemma~\ref{Lem:
unique type for global type} is of course that $X$ is simple.}

\begin{remark}[Relation with types]
There are two differences of the notion of global type to the notion of type in
Definition~\ref{Def: type}. First, contact orders are replaced by
global contact orders. Second, the requirement
$\bar\bu(x)\in\fC_{\bsigma(x)}(X)$ for $x\in E(G)\cup L(G)$ does not imply
$u_x\in N_{\bsigma(x)}$. The lack of the latter condition for edges makes it
impossible to define a basic monoid just depending on a global type.

However, some useful discrete data remain. {For simplicity we assume
$X$ is simple again, deferring the discussion of the general case to \S\ref{sss:
basic monoid with monodromy}.} Consider a tropical punctured map $\Gamma \to
\Sigma(X)$, with associated type $\tau' = (G',\bg',\bsigma',\bu')$, basic monoid
$Q_{\tau'}$ as in \eqref{Eqn: basic monoid}, and dual monoid $Q_{\tau'}^\vee$
underlying the corresponding moduli of tropical maps. We have an associated
\emph{global} type $\bar\tau'=(G',\bg',\bsigma',\bar\bu')$ as in
Definition~\ref{Def: global type},(1) obtained by replacing the contact orders
${\bf u}'(x)$ with their images in $\fC_{\bsigma(x)}(X)$.

Now fix a contraction morphism $\phi:\bar\tau'\arr \tau$ to a global type
$\tau=(G,\bg,\bsigma,\bar\bu)${, with set of contracted edges $E_\phi$}.
We claim that there is a well-defined face $Q_{\tau\tau'}^\vee$ of
$Q_{\tau'}^\vee$, see \eqref{Def: Q_{tau tau'}^vee}, with dual localization
\eqref{Eqn: localization of basic monoids}, not requiring a morphism of types
lifting $\bar\tau' \to \tau$. Fix a point of $Q_{\tau'}^\vee$ given as a tuple
$(V_v,\ell_E)_{v\in V(G), E\in E(G)}$. Then $(V_v,\ell_E)_{v,E} \in
Q_{\tau\tau'}^\vee$ if and only if
\begin{enumerate} 
\item the
position $V_v$ of any vertex $v$ maps to the cell $\bsigma(\phi(v))$
associated to $\phi(v)\in V(G)$ by $\tau$, and
\item if $E\in V(G')$ is an edge
contracted by $\phi$ then $\ell_E=0$.
\end{enumerate}
Here we replaced generic points $\eta$ and nodal points $q$ in \eqref{Eqn: basic
monoid} by vertices $v\in V(G')$ and edges $E\in E(G')$. It is critical that
$\bsigma(\phi(v))$ is a well-defined face of $\bsigma(v)$. {This is
where we use the simplicity assumption.} Define $Q_{\tau\tau'}$ as the dual of
this face, given precisely as:
\begin{equation}
\label{Def: Q_{tau tau'}^vee}
Q_{\tau\tau'}^\vee= \bigg\{ (V_v,\ell_E)\in Q_{\tau'}^\vee\,\bigg|\,
\begin{array}{c}\forall v\in V(G'):\, V_v\in \bsigma(\phi(v))\\
\forall E\in  E_\phi: \ell_E=0 \end{array}\bigg\}.
\end{equation}
We then obtain a localization morphism
\begin{equation}
\label{Eqn: localization of basic monoids}
\chi_{\tau\tau'}: Q_{\tau'}\arr Q_{\tau\tau'},
\end{equation}
just as for basic monoids associated to types of tropical punctured maps
(\cite[Definition~2.31(3)]{decomposition}). The difference is that now both
$Q_{\tau\tau'}$ and $\chi_{\tau\tau'}$ depend not only on the morphism
$\phi:\bar\tau'\arr \tau$ of global types, but also on the lift of $\bar\tau'$
to a type $\tau'$ of tropical punctured maps.
\end{remark}

\subsection{Puncturing log-ideals}
\label{ss:log-ideal} 

The punctured points which are not marked points impose extra important
constraints on the possible deformations of a punctured curve, hence of
punctured stable maps, captured by an ideal in the base monoid. This is a key
new feature of the theory which we now describe.

{\subsubsection{Review of idealized log schemes}

We review here the notion of idealized log schemes from \cite{Ogus}, as
this notion is considerably less common in the literature.

Given a sheaf of monoids $\cM$ on a scheme $X$, we use the term \emph{log-ideal}
for a sheaf of monoid ideals $\cK\subseteq\cM$. 
The sheaf of monoid ideals $\cK$ is said to be \emph{coherent} (see 
\cite[Prop.\,II.2.6.1]{Ogus}) if locally on $X$, $\cK$ is generated
by a finite set of sections.

An \emph{idealized log
scheme} is data $(X,\cM_X,\alpha_X,\cK_X)$ where $(X,\cM_X,\alpha_X)$ is
an ordinary log scheme, with $\alpha_X:\cM_X\rightarrow\cO_X$ the structure
map, and $\cK_X\subseteq\cM_X$ a log-ideal such that $\cK_X
\subseteq \alpha^{-1}_X(0)$.
A \emph{morphism of idealized log schemes}
$f:(X,\cK_X)\rightarrow (Y,\cK_Y)$ is a morphism $f:X\rightarrow
Y$ of log schemes such that $f^{\flat}(f^{-1}(\cK_Y))\subseteq \cK_X$.
See \cite[Def.\,III.1.3.1]{Ogus}.

If $f:X\rightarrow Y$ is a morphism of log schemes and $\cK_Y\subseteq
\cM_Y$ is a log-ideal, we adopt the notation of \cite{Ogus}
by writing $f^{\bullet}(\cK_Y)\subseteq \cM_X$ as the log-ideal
generated by $f^{\flat}(f^{-1}(\cK_Y))$.
We say a morphism $f:X\rightarrow Y$ of idealized log
schemes is {\emph{idealized strict} (\cite[Def.\,III.1.3.2]{Ogus})} if $\cK_X=f^{\bullet}\cK_Y$.

If $W$ is a fine log scheme and $\cK\subseteq \cM_W$ is a log-ideal,
then $\cK$ is invariant under the multiplicative action of $\cO_W^\times$,
and the quotient $\ocK=\cK/\cO_W^\times$ is a log-ideal in $\ocM_W$. 
As the stalks of $\ocM_W$ are finitely generated monoids, the stalks of
$\ocK$ are then finitely generated ideals.

\begin{lemma}
\label{lem:coherent ideal}
Let $(W,\cM_W)$ be a fine log scheme and $\cK\subseteq \cM_W$ a log-ideal.
Then the following are equivalent:
\begin{enumerate}
\item
$\cK$ is a coherent sheaf of ideals;
\item
for any geometric points $\ol x, \ol y$ {of} $W$ with {$\ol y
\arr\ol x$ a specialization arrow}, the stalk $\cK_{\ol{y}}$ is
generated by the image of the generization map $\cK_{\ol{x}} \to \cK_{\ol{y}}$.
\end{enumerate} 
\end{lemma}

\begin{proof}
$(1)\Rightarrow (2)$:
Suppose $\cK$ is a coherent sheaf of ideals. Then given geometric points as in
the statement of the lemma, there is an open neighborhood $U$ of $\ol x$ and a
finite set of sections $S\subseteq \Gamma(U,\cM_W)$ generating $\cK|_U$. In
particular, $\ol y$ lifts to a geometric point of $U$ and hence $\cK_{\ol x}$
and $\cK_{\ol y}$ are both generated by $S$. In particular, the generization map
$\cK_{\ol x}\rightarrow \cK_{\ol y}$ is surjective.

$(2)\Rightarrow (1)$:
Suppose {the generatedness statement} always holds. Since $\cM_W$ is
fine, for any geometric point $\ol x$ of $W$, one may find an \'etale
neighborhood $U$ with a chart $\phi:Q\rightarrow \cM_W|_U$ inducing an
isomorphism $Q\rightarrow \ocM_{W,\ol x}$. Let $K\subseteq Q$ be the inverse
image of $\ocK_{\ol x}$ under this isomorphism, and let $S\subseteq K$ be a
finite generating set. Then $\phi(S)$ provides a subset of $\Gamma(U,\cM_W)$,
necessarily generating an ideal subsheaf $\cK'$ of $\cK$. However, because of
the assumed surjectivity, it follows immediately that $\cK'=\cK$.
\end{proof}}

{
Many notions in log geometry have idealized versions. In particular, there are
notions of idealized log \'etale and idealized log smooth morphisms, defined
using idealized versions of formal lifting. We send the reader to
\cite[IV.3.1]{Ogus} for details. Morally, an idealized log smooth morphism is
one modeled on a morphism between torus invariant subschemes of toric varieties;
alternatively {it is} a morphism $X\rightarrow Y$ such that there is a
closed substack $\cZ_{X/Y}$ of {a relative Artin fan} $\cA_{X/Y}$
\cite[Cor.\,3.3.5]{ACMW} defined by a monomial ideal such that the induced
morphism $X\rightarrow \cA_{X/Y}$ factors through a smooth morphism
$X\rightarrow\cZ_{X/Y}$. See Proposition~\ref{prop:charts} for precise
statements as needed in this paper.}

{
\begin{proposition}
\label{Prop: idealized log smooth}
If $X\rightarrow B$ is log smooth, and $B$ is log smooth over $\kk$ or is a log
point, then every stratum $X_{\sigma}$ of $X$ is idealized log smooth over $B$,
where $\sigma\in \Sigma(X)$. Here we endow $X_\sigma$ with its reduced induced
subscheme structure, and with the log structure induced by the closed embedding
$X_{\sigma}\hookrightarrow X$.
\end{proposition}}

{
\begin{proof}
Since the statement is \'etale local in $B$, we may assume there exists a global
chart $B\arr A_Q=\Spec\kk[Q]$. Note also that by Proposition~\ref{Prop: Log
strata in log smooth case}, $X_\sigma$ is irreducible, hence is
set-theoretically the closure of a geometric generic point $\ol\eta$ of
$X_\sigma$.

Define the log ideal $\cK\subseteq \cM_{X_{\sigma}}$ on $X_\sigma$ by 
\[
\cK(U):=\{s\in \cM_{X_{\sigma}}(U)\,|\,\alpha_{X_{\sigma}}(s)=0\}.
\]
To check that $(X_\sigma,\cK)\arr (B,\emptyset)$ is idealized log smooth near a
point $x\in X_\sigma$, we consider a chart for $X\arr B$ as in
Proposition~\ref{prop:etale charts}, an \'etale neighborhood $h:U\arr X$ of $x$
fitting into a commutative diagram
\[
\xymatrix{
U\ar[r]^(.3)g\ar[rd]& B\times_{\cA_Q} \cA_P\ar[d]\ar[r]&
\cA_P\ar[d]\\
&B\ar[r]&\cA_Q,
}
\]
with all horizontal arrows strict, $g:U\arr B\times_{\cA_Q} \cA_P$ smooth,
$P^\times=\{0\}$, and a lift $\tilde x$ of $x$ to $U$ mapping to the closed
(deepest) stratum of $\cA_P$. Then we obtain an isomorphism $\psi:P\arr
\ocM_{X,\ol x}= (\sigma^\vee_{\ol x})_\ZZ$. Each specialization arrow
$\ol\eta\arr \ol x$ defines a face inclusion $\sigma\arr \sigma_{\ol x}$, hence
a closed reduced substack $\cZ\subset \cA_P$ with $h(g^{-1}(\cZ))\subseteq
\cZ_\sigma$, where $\cZ_\sigma$ is the logarithmic stratum of $X$ with closure
$X_\sigma$. Thus if $F_i\subseteq P$ denotes the dual faces of $P$ defined by
such specializations, then by the definitions of $\cK$ and $X_\sigma$,
\begin{equation}
\label{Eqn: ocK stalkwise}
\psi\big(P\setminus\textstyle\bigcup_i F_i\big)
= \ocK_{\ol x}\subseteq\ocM_{X,\ol x}.
\end{equation}
Note this gives an alternative, stalkwise definition of the log ideal $\cK$,
using the reasoning in Remark~\ref{Rem: Gap in GS2}.

To show the claim on idealized smoothness, it thus remains to show that the
preimage in $U$ of the closed reduced substacks of $\cA_P$ are reduced for then
the subscheme of $U$ defined by $P\setminus\bigcup_i F_i$ agrees with
$h^{-1}(X_\sigma)$.

Now a closed reduced substack $\cZ\subseteq \cA_P$ maps onto a closed reduced
substack $\cT$ of $\cA_Q$, which by our assumptions on $B$ pulls back to a
reduced subscheme $S\subseteq B$. Therefore {$B\times_{\cA_Q} \cZ=
S\times_\cT \cZ$ is reduced since $S\arr \cT$ is smooth,} and so is its preimage
in $U$.
\end{proof}
}

\subsubsection{Log-ideals of punctured curves}
\label{sss: log ideals punctured curves}

Let $(\pi: C^\circ\arr W,\bp)$ be a punctured curve. For each of the punctures
$p:\ul W\to \ul C$ consider the composition
\begin{equation}
\label{Eqn: u_p}
{v_{p}}:p^*\cM_{C^\circ}\arr {\ocM_W}\oplus\ul\ZZ \arr \ul \ZZ
\end{equation}
of fine monoid sheaves, with the first map {induced by the canonical
inclusion $p^*\ocM_{C^\circ}\arr \ocM_W\oplus\ul\ZZ$} and the second map the
projection. Denote by $\cI_p\subseteq p^*\cM_{C^\circ}$ the sheaf of ideals
generated by $(v_{p})^{-1}(\ul \ZZ_{< 0})$.

\begin{definition}
\label{Def: Puncturing log ideal}
The \emph{puncturing log-ideal} $\cK_W\subseteq \cM_W$ of the punctured curve
$(\pi: C^\circ\arr W,\bp)$ is the ideal sheaf 
\[
\bigcup_p (\pi^\flat)^{-1}(\cI_p)\subseteq \cM_W,
\]
with $p$ running over all punctures.
\end{definition}

In the context of the definition we abuse notation when writing $\pi^\flat$ for
the composition
\[
\cM_W\stackrel{\pi^\flat}{\longrightarrow} \pi_*\cM_{C^\circ}
\arr \pi_* p_*p^*\cM_{C^\circ}= p^*\cM_{C^\circ},
\]
where as usual $p^*\cM_{C^\circ}$ denotes the pull-back log structure,
while the right arrow is induced by the adjunction unit morphism $1\to p_*
p^{-1}$ of the associated abelian sheaves.

We sometimes also refer to the quotient $\ocK_W$ of $\cK_W$ by $\cO_W^\times$ as
the puncturing log-ideal, but will then write $\ocK_W\subseteq \ocM_W$ for
clarity.

An illustration for the definition is contained in Figure~\ref{fig:idealized}.

\begin{figure}[htb]
\input{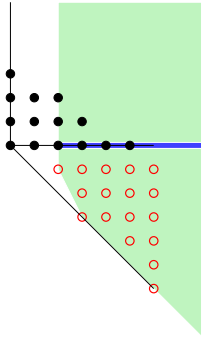}
\caption{An idealized punctured point (ideal lightly shaded) and the resulting
log ideal (the horizontal shaded ray). If there are several punctures, one takes
the ideal generated by these horizontal regions.}
\label{fig:idealized}
\end{figure}

{This picture indicates an equivalent way to identify $\ocK_W$.
{For the stalkwise characterization we may do a strict base change to a
geometric point of $\ul W$ and hence assume $W$ is a log point.} For a
marking $p$ on a component of $C^\circ$ with generic point $\eta$, consider the
generization map $\phi_{p,\eta}:\ocM_{C^\circ,p} \to \ocM_{C^\circ,\eta}\simeq
\ocM_W$. Identify $\ocM_W$ as a submonoid of $\ocM_{C^\circ,p}$ via $\pi^\flat$,
making $\phi_{p,\eta}$ an idempotent homomorphism on $\ocM_{C^\circ,p}$ with
image $\ocM_W$. An element $m \in \ocM_W$ is in $\ocK_W$ if and only if there is
a marking $p$ and an element $n\in (u_p^{\circ})^{-1} (\ZZ_{<0})$ such that
$\phi_{p,\eta}(n) = m$. Indeed if there is $n\in (v_{p})^{-1}(\ul \ZZ_{<
0}) $ and $n' \in \ocM_{C^\circ,p}$ with $\pi^\flat (m)=n+n'$ then, writing $n''
= n+ \phi_{p\eta}( n')$ we have $m = \phi_{p\eta}(n'')$; conversely, if $m =
\phi_{p\eta}(n'')$ with $v_{p}(n'')= -b<0$ then, using the notation of
\eqref{Eqn: u_p}, we have $\pi^\flat(m)= n'' + b\cdot (0,1)$.
}

\begin{lemma}
\label{curve-ideal-characteristic}
The puncturing log-ideal $\cK_W$ of a punctured curve $(\pi:C^\circ\arr W,\bp)$
is coherent.
\end{lemma}

\begin{proof}
{We verify the characterization of Lemma~\ref{lem:coherent ideal}. Let
$\ol x, \ol y \to W$ with $\ol x \in \cl(\ol y)$. Fix a generization map
$\chi_{\ol x\hspace{.7pt}\ol y}:\ocK_{\ol x} \to \ocK_{\ol y}$ and let $m_{\ol
y}\in \ocK_{\ol y}$. We wish to construct $m_{\ol x} \in \ocK_{\ol x}$ with
$\chi_{\ol x\hspace{.7pt}\ol y}(m_{\ol x}) = m_{\ol y}$. 

We refer to the following commutative diagram of generizations and contact orders:

\[
\xymatrix{
\ocM_{W,\ol x}=\ocM_{C^\circ, \eta_{\ol x}} \ar[rrr]^{\chi_{\ol
x\hspace{.7pt}\ol y} = \chi_{\eta_{\ol x}\eta_{\ol y}}} &&& \ocM_{C^\circ,
\eta_{\ol y}} = \ocM_{W,\ol y}\\
\ocM_{C^\circ, p_{\ol x}}\ar[u]^{\phi_{p_{\ol
x}\eta_{\ol x}}} \ar[rrr]_{\chi_{p_{\ol x}p_{\ol y}}}\ar[d]_{u^\circ_{p_{\ol
x}}}&&&\ocM_{C^\circ, p_{\ol y}} \ar[u]_{\phi_{p_{\ol y}\eta_{\ol
y}}}\ar[d]_{u^\circ_{p_{\ol y}}} \\
\ZZ \ar@{=}[rrr] &&& \ZZ
}
\]

Note that $m_{\ol y}\in \ocK_{\ol y}$ means that there is a puncture $p_{\ol y}$
lying on a component with generic point $\eta_{\ol y}$ of $C_{\ol y}$ and an
element $m_{p_{\ol y}}\in (v_{p_{\ol y}})^{-1}(\ul \ZZ_{< 0})$ whose generization
is $\phi_{p_{\ol y}\eta_{\ol y}}(m_{p_{\ol y}}) = m_{\ol y}$. 

Since $\cM_{C^\circ}$ is coherent, there is an element $m_{p_{\ol x}}\in
\ocM_{C^\circ,\ol x}$ such that $\chi_{p_{\ol x}p_{\ol y}}(m_{p_{\ol x}}) =
m_{p_{\ol y}}$. 

Note that $v_{p_{\ol y}}\circ\chi_{p_{\ol x}p_{\ol y}} = v_{p_{\ol x}}$, see
Lemma~\ref{Lem: contact orders are locally constant}. This implies that
$m_{p_{\ol x}}\in (v_{p_{\ol x}})^{-1}(\ul \ZZ_{< 0})$. Write $m_{\ol x} :=
\phi_{p_{\ol x}\eta_{\ol x}} ({m_{p_{\ol x}}})$. By definition $m_{\ol x}\in \ocK_{\ol
x}$.s

We obtain that $m_{\ol y}=\phi_{p_{\ol y}\eta_{\ol y}} \circ \chi_{p_{\ol
x}p_{\ol y}}(m_{p_{\ol x}}) = \chi_{\eta_{\ol x}\eta_{\ol y}} \phi_{p_{\ol
x}\eta_{\ol x}}(m_{p_{\ol x}}) = \chi_{\eta_{\ol x}\eta_{\ol y}}(m_{\ol x}) =
\chi_{\ol x\hspace{.7pt}\ol y}(m_{\ol x}),$ as needed.
}
\end{proof}

Puncturing log-ideals behave well under pull-backs.

\begin{proposition}
\label{prop:pull-back-ideal}
Let $(\pi:C^{\circ}\arr W, \bp)$ be a punctured curve, $(\pi_T: C^{\circ}_T\arr
T, \bp_T)$ its pull-back via $h: T \to W$ and $\cK_W$, $\cK_T$ the respective
puncturing log-ideals. Then {$\cK_T=h^\bullet\cK_{W}$}.
\end{proposition}

\begin{proof}
Denote by $g:C^\circ_T\arr C^\circ$ the pull-back of $h$ to the curves. By
coherence of $\cK_W$ and $\cK_T$ it suffices to check that for each geometric
point $\ol t\arr T$, the image of $\ocK_{W,\ul h(\ol t)}$ under $\bar
h_{\ol t}^\flat$ generates $\ocK_{T,\ol t}$. Denote by $\ol w=h(\ol
t)$. For a puncture $p$ of $C^\circ$ consider the commutative diagram
\[
\xymatrix{
\ocM_{W,\ol w}\ar[r]^{\bar\pi^\flat}\ar[d]^{\bar h^\flat}
& \ocM_{C^\circ,p( \ol w)}\ar[r]\ar[d]_{\bar g^\flat}
&\ocM_{W,\ol w}\oplus\ZZ\ar[r]\ar[d]&\ZZ\ar[d]^{=}\\
\ocM_{T,\ol t}\ar[r]^{\bar\pi_T^\flat}
& \ocM_{C_T^\circ,p_T(\ol t)}\ar[r]
&\ocM_{T,\ol t}\oplus\ZZ\ar[r]&\ZZ
}
\]
The two left squares are cocartesian in the category of fine monoids by the
definition of pull-back of punctured curves. This shows first that
$\bar g^\flat(\ocI_{p,\ol w})$ generates $\ocI_{p_T,\ol t}$, and in
turn that $\bar h^\flat\big((\bar\pi^\flat)^{-1}(\ocI_{p,\ol
w})\big)$ generates $(\bar\pi_T^\flat)^{-1}(\ocI_{p_T,\ol t})$. Taking the sum
over all punctures finishes the proof.
\end{proof}

Here comes the crucial vanishing property putting restrictions on deformations
of punctured curves.

\begin{proposition}
\label{prop:puncturing-ideal-vanish}
Let $(C^{\circ}/W, \bp)$ be a punctured curve and $\cK_W\subseteq \cM_W$ its
puncturing log-ideal. Then it holds
\[
\alpha_{W}(\cK_W) = 0.
\]
\end{proposition}

\begin{proof}
Let $\cI_p={v_p^{-1}}(\ZZ_{<0}) \subseteq p^*\cM_{C^{\circ}}$ be the
ideal sheaf defined after \eqref{Eqn: u_p}. Definition~\ref{def:puncturing},(2)
implies $(p^*\alpha_{C^\circ})(\cI_p) = 0$. Pulling back via $\pi^\flat:
\cM_W\arr p^*\cM_{C^\circ}$ thus yields
\[
\alpha_{W}\big( (\pi^\flat)^{-1}(\cI_p)\big)= (p^*\alpha_{C^\circ})(\cI_p)=0.
\]
The claimed vanishing follows by summing over the punctures $p$.
\end{proof}

Proposition~\ref{prop:puncturing-ideal-vanish} demonstrates {the
announced statement} that the base of a family of punctured curves is naturally
an idealized log scheme (or stack).

\begin{corollary}
\label{Cor: bases of punctured curves are idealized}
For a punctured curve $(C^{\circ}/ W, \bp)$ with $\cK_W$ its puncturing
log-ideal, the triple $(W,\cM_W,\cK_W)$ is a coherent idealized log scheme.
\end{corollary}

\begin{example}
\label{Expl: Puncturing log ideal AA^2}
Let $(C^\circ/W,\bp)$ be a punctured curve over the logarithmic point
$W=\Spec(Q\arr\kk)$, with $Q=\NN^2$, $\ul C$ a smooth and connected curve and
with only one punctured point $p$ with
\[
\ocM_{C^\circ,p}= (Q\oplus\NN)+\NN\cdot (a,0,-1)
+\NN\cdot(0,b,-1)\subset Q\oplus\ZZ,
\]
for some $a,b\in\NN\setminus\{0\}$. Then the puncturing log-ideal $\ocK_W$ is
generated by $(a,0),(0,b)$. This implies that if we view $W$ as the strict
closed subspace of $\AA^2=\Spec \kk[t_1,t_2]$ with its toric log structure, then
the maximal subscheme of $\AA^2$ to which $(C/W,\bp)$ extends is given by the
ideal $(t_1^a,t_2^b)\subset \kk[t_1,t_2]$.
\end{example}

\subsubsection{Log-ideals of punctured maps}

We define puncturing log-ideals only for pre-stable punctured maps.\footnote{{If
$(C^\circ/W,\bp,f)$ has associated pre-stable map $(\widetilde
C^\circ/W,\bp,\widetilde f)$ (Proposition~\ref{prop:pre-stable}), the ideal
$\cK_{W} $ of Definition~\ref{Def: Puncturing log ideal} associated to
$C^\circ/W$ may strictly include the corresponding ideal associated to
$\widetilde C^{\circ}/W$.}}

\begin{definition}
\label{Def: puncturing log ideal punctured map}\sloppy
The \emph{puncturing log ideal $\cK_W$} of a pre-stable punctured map
$(C^\circ/W,\bp,f)$ is the puncturing log-ideal of the punctured domain curve
$(C^\circ/W,\bp)$, as defined in Definition~\ref{Def: Puncturing log ideal}.
\end{definition}
\fussy

It is clear from the definition and Proposition~\ref{prop:pull-back-ideal} that
puncturing log ideals of punctured maps are stable under base change, and they
also enjoy the vanishing property $\alpha_W(\cK_W)=0$ from
Proposition~\ref{prop:puncturing-ideal-vanish}. 

We finish this subsection by giving a tropical interpretation {in the
spirit of Proposition~\ref{Prop: pre-stable tropical punctured maps}} of the
radical of the puncturing log-ideal $\cK_W$ of a pre-stable punctured map, see
Proposition~\ref{Prop: puncturing tropical interpretation}. This interpretation
is based on the following technical result concerning monoid ideals.

\begin{lemma}
\label{TQcomponents}
Suppose given a sharp toric monoid $Q$, and a collection of sharp toric 
monoids $P_{p_1},
\ldots,P_{p_r}$ along with monoid homomorphisms
$\varphi_{p_i}:P_{p_i}\rightarrow Q\oplus\ZZ$ with $u_{p_i}:=\pr_2\circ
\varphi_{p_i}$. Let $\ev_i:=(\pr_1\circ \varphi_{p_i})^t: Q^{\vee}_{\RR}
\rightarrow (P_{p_i})^{\vee}_{\RR}$. Let the ideal $I\subset Q$ 
be the monoid ideal 
\[
I=\bigcup_{i=1}^r \big\langle \pr_1\circ \varphi_{p_i}(m)\,\big|\,
\hbox{$m\in P_{p_i}$ and $u_{p_i}(m)<0$ }\big\rangle.
\]
For $\sigma$ a face of the cone $Q^{\vee}_{\RR}$,
let $A_\sigma=\Spec \kk[\sigma^{\perp}\cap Q]$ be the closed toric
stratum of $\Spec\kk[Q]$ corresponding to $\sigma$. 
Then there is a decomposition 
\[
\Spec\kk[Q]/\sqrt{I}=\bigcup_{\sigma} A_\sigma
\]
where the union is over all faces $\sigma$ of $Q^{\vee}_{\RR}$ such that if
$x\in \Int(\sigma)$, then $\ev_i(x)+\epsilon u_{p_i}\in 
(P_{p_i})^{\vee}_{\RR}$ for $\epsilon>0$ sufficiently small and $1\le i\le r$.\footnote{See Example \ref{ex:non equi dim} below.}
\end{lemma}

\begin{proof}
Let $I_{p_i}\subset Q$ be the monoid ideal
\[
I_{p_i}= \big\langle \pr_1\circ \varphi_{p_i}(m)\,\big|\,
\hbox{$m\in P_{p_i}$ satisfies $u_{p_i}(m)<0$}\big\rangle.
\]
Of course $V(I)=\bigcap_i V(I_{p_i})$. We first show that if $\sigma$ satisfies
the given condition, then $A_\sigma\subseteq V(I_{p_i})$ for each $i$. The
monomial ideal defining $A_\sigma$ is $Q\setminus (\sigma^{\perp}\cap Q)$, so
it is enough to show that $\sigma^{\perp}\cap I_{p_i}=\emptyset$. Choose an
$x\in\Int(\sigma)$. Let $q\in I_{p_i}$ be a generator of $I_{p_i}$, that is,
there exists an $m\in P_{p_i}$ such that $q=\pr_1(\varphi_{p_i}(m))$ and
$u_{p_i}(m)<0$. Since $m\in P_{p_i}$ and $\ev_i(x)+\epsilon u_{p_i}\in
(P_{p_i})^{\vee}_{\RR}$ for some $\epsilon>0$, we have
\[
0 \le \langle \ev_i(x)+\epsilon u_{p_i}, m\rangle.
\]
Thus $\langle u_{p_i}, m\rangle <0$ implies
$\langle \ev_i(x),m\rangle >0$, or $\langle x,\pr_1(\varphi_{p_i}(m))\rangle
=\langle x, q\rangle >0$, as desired.

Conversely, suppose that $A_\sigma\subseteq V(I)$ for some face
$\sigma$ of $Q^{\vee}_{\RR}$, but there exists an $i$ and some
$x\in\Int(\sigma)$
such that $\ev_i(x)+\epsilon u_{p_i}\not \in (P_{p_i})^{\vee}_{\RR}$
for any $\epsilon>0$. Then there exists an $m\in P_{p_i}$ such that
$\langle \ev_i(x)+\epsilon u_{p_i},m\rangle <0$ for all $\epsilon>0$. 
Since $\langle \ev_i(x),m \rangle \ge 0$, we must have 
$\langle \ev_i(x),m\rangle=0$
and $u_{p_i}(m)<0$. Thus $q=\pr_1(\varphi_{p_i}(m))$ lies in $I_{p_i}$.
We have $\langle x, q\rangle =\langle \ev_i(x), m\rangle=0$, so $q\in
\sigma^{\perp}$. In particular, $z^q$ does not vanish on $A_\sigma$,
contradicting $A_\sigma\subseteq V(I)$. 
\end{proof}

\begin{proposition}
\label{Prop: puncturing tropical interpretation}
Let $(C^\circ/W,\bp,f)$ be a punctured map to $X$ over the logarithmic point
$W=\Spec(Q\arr \kappa)$,
\[
h:\Gamma=\Gamma(G,\ell)\arr\Sigma(X)
\]
the associated tropical curve over $\omega= Q^\vee_\RR$, and
$(G,\bg,\bsigma,\bu)$ its type. Denote by $\sqrt{\ocK_W}\subset Q$ the
radical of the puncturing log-ideal of $(C^\circ/W,\bp,f)$.

Then a face $Q'\subseteq Q$ lies in $Q\setminus \sqrt{\ocK_W}$ if and
only if for any punctured leg $L\in L(G)$ it holds\footnote{Again, see
Example \ref{ex:non equi dim} below.}
\[
\ell(L)\big((Q')^\perp\cap\omega_\ZZ\big)\neq 0.
\]
In other words, $Q'$ determines a face $(Q')^{\perp}\cap \omega_{\ZZ}$
of $\omega_{\ZZ}$, and each point of this face corresponds to a tropical
map. Thus we require that the length function $\ell(L)$ of each punctured
leg be non-vanishing on this face of $\omega_{\ZZ}$.
\end{proposition}

\begin{proof}
By pre-stability, $\ocK_W$ is generated by those $q\in Q$ such that there exists
a puncture $p_i\arr \ul C$ of $C^\circ$ and $m\in \ocM_{X,\ul f(p_i)}$
with $\bar f^\flat(m)= (q, a)$ and $a=u_{p_i}(m)<0$. Thus $\ocK_W=I$
in Lemma~\ref{TQcomponents} applied with $P_{p_i}= \ocM_{X,\ul f(p_i)}$. 
Using the characterization of punctured legs in the
pre-stable case in Proposition~\ref{Prop: pre-stable tropical punctured maps},
the statement to be proved is then a reformulation of the conclusion of 
Lemma~\ref{TQcomponents} in terms of tropical maps.
\end{proof}

Phrased more geometrically, the conclusion of Proposition~\ref{Prop: puncturing
tropical interpretation} says that exactly those faces of the basic cone of a
tropical punctured map (Definition~\ref{Def: basic monoid}) can possibly arise
from a generization of punctured maps if the puncturing legs remain of positive
length.
\medskip

We end this section with an example highlighting the fact that the natural base
spaces in punctured Gromov-Witten theory are possibly reducible spaces due to
the puncturing ideals. See Theorem~\ref{thm:idealized-etale} and
Remark~\ref{Rem: local structure of fM(cX/B,tau)} for the general picture
underlying this phenomenon.

\begin{example}
\label{ex:non equi dim}
\textsc{Algebraic setup.}
Take $B=\Spec\kk$, and consider $X$ a smooth surface with log structure coming
from a smooth rational curve $D\subseteq X$ with $D^2=2$. Consider a type of
punctured maps of genus $0$, underlying curve class $[D]$, and four punctures,
$p_1,\ldots,p_4$, with contact orders $-1,-1,2$ and $2$ respectively. Consider a
punctured curve $f:C^{\circ}\rightarrow X$ where $C=C_1\cup C_2\cup C_3$ has
three irreducible components and two nodes $q_1=C_1\cap C_2$, $q_2=C_1\cap C_3$.
We assume $p_1,p_3\in C_2$, $p_2,p_4\in C_3$. Finally, $\ul{f}$ identifies $C_1$
with $D$ and contracts $C_2$ and $C_3$. Orienting the node $q_i$ from $C_1$ to
$C_i$, it is not difficult to check such a curve exists with $u_{q_1}=u_{q_2}=1$
(Figure \ref{fig:idealized-curve}).

\begin{figure}[htb]
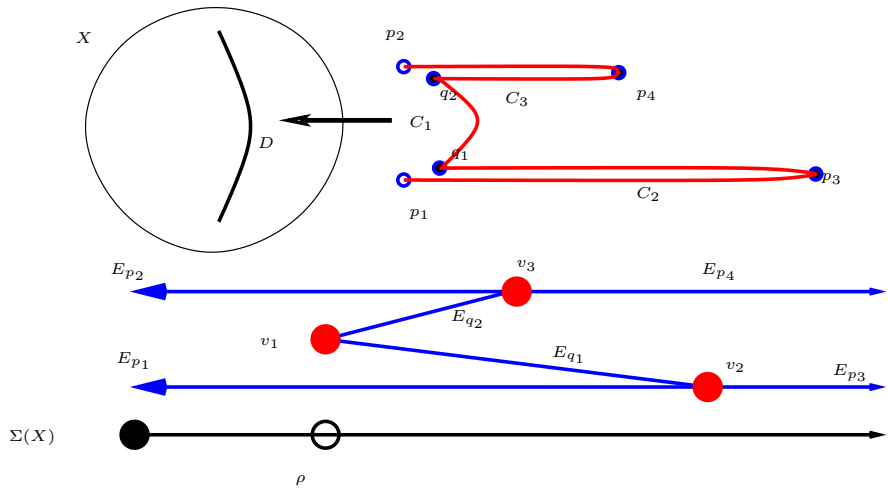

\input{idealized-punctured-curve-noexpand.pspdftex}
\input{idealized-punctured-tropical-only.pspdftex}
\caption{The algebraic map and its tropical counterpart. Here $\rho=1,\ell_1=2,$
and $\ell_2=1$.}
\label{fig:idealized-curve}
\end{figure}

\textsc{The tropical curve.}
The corresponding tropical curve $\Gamma$ has three vertices, $v_1,v_2,v_3$,
edges $E_{q_1}, E_{q_2}$, and legs $E_{p_1},\ldots,E_{p_4}$. The
moduli space of tropical curves of this type is
$\RR_{\ge 0}^3$, with coordinates $\rho,\ell_1,\ell_2$, where $\rho$ gives
the distance of the image of $v_1$ from the origin of $\Sigma(X)=\RR_{\ge 0}$,
and $\ell_1,\ell_2$ give the lengths of the edges $E_{q_1}, E_{q_2}$. In
particular, the basic monoid for this punctured log curve is 
$Q=\NN^3$, generated by $\rho,\ell_1,\ell_2$.

\textsc{The punctured ideal.}
In this case we may easily calculate the puncturing ideal (Definition~\ref{Def:
puncturing log ideal punctured map}). We have contributions from each of the two
punctures. Using the definition, we note that at the puncture $p_i$, $i=1$ or
$2$, the map $\varphi_{\ol\eta}\circ\chi_{\eta,p_i}: P_{p_i}=\NN\rightarrow Q$
is dual to $\ev_i:Q^{\vee}_{\RR}\rightarrow (P_{p_i})^\vee =\RR_{\ge 0}$
evaluating the tropical curve parameterized by a point at $Q^{\vee}_{\RR}$ at
$v_2$ or $v_3$, see Lemma~\ref{TQcomponents}. Thus for $m\in
Q^{\vee}_{\RR}$, $\ev_i(m) =\rho(m)+\ell_i(m)$. Dually
$\varphi_{\ol\eta}\circ\chi_{\eta,p_i}:P\rightarrow Q$ is given by $1\mapsto
\rho+\ell_i$. As $u_{p_i}(1)=-1$, $i=1,2$, we see the puncturing ideal $K$ is
generated by $\rho+\ell_1,\rho+\ell_2$. Writing $\kk[Q]=\kk[x,y,z]$, with the
three variables corresponding to $\rho,\ell_1,\ell_2$ respectively, we see
$\Spec \kk[Q]/K= \Spec \kk[x,y,z]/(xy,xz)$, which has two irreducible components
of differing dimension.

\textsc{The participating and excluded cones}.\ \sloppy
The decomposition
$\bigcup_{\sigma} A_\sigma$ of Lemma~\ref{TQcomponents} translates to the
statement that the cones \emph{excluded} in this decomposition are the origin,
the $\ell_1$-axis, and the $\ell_2$-axis. Indeed these are the cones where at
least one puncture is positioned with its tail at the origin, hence forced to
have length 0, {which is} excluded by Proposition~\ref{Prop: puncturing
tropical interpretation}.

\fussy
\textsc{The components of the algebraic moduli space.}
Note that deformation theory provides two deformation classes of the punctured
map. The first smooths one or both of the nodes, resulting in a punctured map
with at least one pair $p_1,p_3$ or $p_2,p_4$ now being distinct points on the
component of the domain mapping surjectively to $D$. Since this component
contains a negative contact order point, its image cannot be deformed away from
$D$ by Remark~\ref{rem:contained-in-strata}.

The second deformation class keeps the domain of $f$ fixed, but deforms the
image of $C_1$ away from $D$, so that it meets $D$ transversally in two points.
The remaining components $C_2$ and $C_3$ are then contracted to the points of
intersection of $f(C_1)$ with $D$. It is then no longer possible to smooth the
nodes.

\textsc{The data captured by the ideal.}
This local reducibility of moduli space happens despite the obstruction group
$H^1(C,f^*\Theta_X)$ for deformations with fixed domain (see
\S\ref{sec:obstruction}) being zero. The point of the puncturing ideal is that
it captures these intrinsic singularities of the moduli space. These
obstructions really come from obstructions to deforming the punctured
domain curve.

The general picture explaining this phenomenon is developed in
\S\ref{ss:idealized smoothness}. In particular, Example~\ref{Expl: non equi
dim revisited} revisits the present example from the general perspective.
\end{example}

\subsection{{Targets with monodromy}}
\label{ss: targets with monodromy}
We now drop the assumption that $X$ is simple and discuss what is needed to
treat the general case.

\subsubsection{Tropicalization of punctured maps with non-simple targets.}
\label{sss: refined tropicalization over geometric point}
Let $(C^\circ/W,\bp,f)$ be a punctured map over a logarithmic point
$\Spec(Q\arr\kappa)$ with $\kappa$ algebraically closed. Then the inclusion of a
nodal point $q$ or punctured point $p$ into $C^\circ$ is a geometric point of
$\ul C$ that we denote by $\ol q$ and $\ol p$, respectively. For a node $q$ of
$\ul C$, the generic points $\eta,\eta'\in \Spec\cO_{\ul C,\ol q}$ of the two
branches of $\ul C$ at $\ol q$ provide two specialization arrows of geometric
points (see Appendix~\ref{App: Functorial tropicalization})
\[
\ol\eta\arr \ol q,\quad \ol\eta'\arr\ol q,
\]
unique up to order and precomposition with an isomorphism in the category of
geometric points in $\Spec\cO_{\ul C,\ol q}$. The node $q$ is a
self-intersection point of $\ul C$ iff $\ol\eta,\ol\eta'$ have the same image in
$\ul C$, that is, iff they are isomorphic as geometric points of $\ul C$. In any
case, denoting by $G$ the dual intersection graph of $C^\circ$, each
specialization arrow $\ol\eta\arr \ol x$ with $x\in E(G)\cup L(G)$ gives rise to
a face inclusion
\begin{equation}
\label{Eqn: Face inclusions tropical domain}
Q^\vee= \ocM_{C,\ol\eta}^\vee\arr \ocM_{C,\ol x}^\vee.
\end{equation}
The equality on the left-hand side is the canonical isomorphism obtained since $C^\circ$ is a log smooth curve over $\Spec(Q\arr \kk)$.

Applying $f$ yields a specialization arrow $f(\ol \eta)\arr f(\ol x)$ and a corresponding face embedding
\begin{equation}
\label{Eqn: Face embedding from image generization}
\ocM^\vee_{X,f(\ol\eta)} \arr \ocM^\vee_{X,f(\ol x)}
\end{equation}
Our tropicalization procedure for $f:C^\circ\arr W$ requires us to choose, for
each $x\in V(G)\cup E(G)\cup L(G)$ with associated geometric point $\ol x$ of
$\ul C$, an isomorphism
\begin{equation}
\label{Eqn: Choice of isos with bsigma(x)}
\Hom\big(\ocM_{X,\ul f(\ol x)},\RR_{\ge 0}\big)\arr \bsigma(x)
\end{equation}
in $\Sigma(X)$. Composing these isomorphisms or their inverses with the arrow in
\eqref{Eqn: Face embedding from image generization} defines an arrow
\[
\iota_{x\eta}: \bsigma(\eta)\arr \bsigma(x)
\]
in $\Sigma(X)$. If $\Sigma(X)$ is simple there is only one arrow
$\bsigma(\eta)\arr \bsigma(x)$ in $\Sigma(X)$. In the general case, the
$\iota_{x\eta}$ are part of the data defining the tropicalization, up to the
simultaneous action of
\begin{equation}
\label{Eqn: GG}
\GG=\prod_{x\in V(G)\cup E(G)\cup L(G)} \Aut_{\Sigma(X)} (\bsigma(x))
\end{equation}
on the choices of isomorphisms \eqref{Eqn: Choice of isos with bsigma(x)}. Note
that $\GG$ may not act transitively on the set of arrows $\bsigma(\eta)\arr
\bsigma(x)$, and then the specialization morphism $\ol\eta\arr\ol x$ in $\ul C$
at a node or marked point distinguishes a $\GG$-orbit of such arrows.

We emphasize that if $x=q$ is a node there are two such arrows, regardless if
$q$ is self-intersecting or not, one for each branch of $\ul C$ at $q$. Thus the
proper labelling would not be by pairs $(\eta,q)$ but by half-edges of the dual
intersection graph $G$ of $C^\circ$. By abuse of notation we nevertheless denote
these two half-edges by $(q,\eta)$ and $(q,\eta')$.

Given a node $q$ with adjacent geometric generic point $\ol\eta$, we can compose
$f^\flat_{\ol\eta}: \ocM_{X,\ul f(\ol\eta)}\arr \ocM_{C,\ol\eta}$ with the
identification $\ocM_{C,\ol \eta}=Q$ and the isomorphisms \eqref{Eqn: Choice of isos
with bsigma(x)}, and dualize to obtain the map of cones
\[
V_\eta:Q^\vee\arr \bsigma(\eta).
\]
The defining equation~\cite[(2.22)]{decomposition} of the contact order $u_q\in
\bsigma(q)$ at $q$ now takes the form
\begin{equation}
\label{Eqn: edge equation with monodromy}
\iota_{q\eta'}\circ V_{\eta'}-\iota_{q\eta}\circ V_\eta= \ell(E_q)\cdot u_q,
\end{equation}
an equality in $\Hom(Q^\vee,\bsigma(q))$. Here $\eta'$ is the other geometric
generic point of $\Spec\cO_{C,\ol q}$ as above.

The pair $(V_\eta,V_{\eta'})$, or equivalently $(V_\eta,\ell(E_q),u_q)$,
determines the tropicalization of $(C^\circ/W,\bp,f)$ at $q$. At a marked point
$p$, the tropicalization is similarly defined by $V_\eta$ and the contact order
$u_p$.

\sloppy
Taken together, we obtain the following description of the tropicalization of
$(C^\circ/W,\bp,f)$.

\fussy
\begin{proposition}
\label{Prop: General tropicalization punctured map}
The tropicalization of a punctured map $(C^\circ/W,\bp,f)$ to $X$ with
$W=\Spec(Q\arr \kk)$ an algebraically closed logarithmic point is given by the
abstract tropical curve $(G,\bg,\ell)$, i.e.\ the tropicalization of
$C^\circ/W$, and the tuple
\[
(V_\eta,u_x,\iota_{x\eta})_{\eta,x},
\]
as discussed. Here $\eta\in V(G)$, $x\in E(G)\cup L(G)$, with $\eta$ adjacent to
$x$ for $\iota_{x\eta}$, and the data is subject to \eqref{Eqn: edge equation
with monodromy}. A self-intersecting node $q$ produces two arrows
$\iota_{x\eta}$, as commented on above. The tuple
$(V_\eta,u_x,\iota_{x\eta})_{\eta,x}$ is unique up to the obvious action of
$\GG$ from \eqref{Eqn: GG} on the set of tuples.
\end{proposition}

Conversely, a tropical punctured map over $\omega\in\Cones$ consists of two maps
$\Gamma\arr\omega$ and $\Gamma\arr \Sigma(X)$ of generalized cone complexes.
Lifting both maps locally near the strata of $|\Gamma|$ labelled by vertices,
edges and legs to maps of cone complexes provides a tuple
$(V_\eta,u_x,\iota_{x\eta})_{\eta,x}$ that is again unique up to the action of
$\GG$. Thus we have a one-to-one correspondence between tropical punctured maps
and $\GG$-orbits of tuples $(V_\eta,u_x,\iota_{\eta x})_{\eta,x}$. Note in
particular that each individual contact order $u_x\in\bsigma(x)$, $x\in E(G)\cup
L(G)$, is only defined up to the action of $\Aut_{\Sigma(X)} (\bsigma(x))$, but
more information is retained when considering contact orders simultaneously and
together with the set of face inclusions $\iota_{x\eta}$.
Here is a simple example illustrating the effect of monodromy on the procedure.

\begin{example}
This is a modification of the Whitney umbrella example in \cite[\S5.4.1]{ACMUW}.
Let $C$ be the nodal cubic with its log smooth structure over the standard log
point $\Spec(\NN\to\kk)$. Define $X$ as the quotient of
$(\AA^1\setminus\{0\})\times C$ by the $\ZZ/2$-action that swaps the two
branches of $C$ at the node and acts by multiplication by $-1$ on
$\AA^1\setminus\{0\}$. We can view $X$ as a non-trivial, log smooth fibration
over $(\AA^1\setminus\{0\})\times \Spec(\NN\to\kk)$ with all fibers $X_s$
isomorphic to the nodal cubic $C$. Thus $X$ is irreducible with two logarithmic
strata with closures $\ul X$ and $\ul X_\sing$, respectively. Denoting by
$\ol\eta_0$, $\ol\eta_1$ geometric generic points for these strata, we have
$\ocM_{X,\ol\eta_0}=\NN$, $\ocM_{X,\ol\eta_1}=\NN^2$. The tropicalization
$\Sigma(X)$ has a presentation with two non-zero cones $\sigma_0=\RR_{\ge0}$,
$\sigma_1=\RR_{\ge0}^2$, and non-trivial arrows the two face inclusions
$\sigma_0\arr \sigma_1$.

The inclusion $C\arr
X$ of a closed fiber defines a stable log map with unique generic point $\eta$,
one node $q$, and no marked points. We have $\bsigma(\eta)=\sigma_0$,
$\bsigma(q)=\sigma_1$, and a unique arrow \eqref{Eqn: Choice of isos with
bsigma(x)} in $\Sigma(X)$ for $x=\eta$, hence a unique map of cones $V_\eta: Q^\vee=\RR_{\ge 0}\arr \sigma_0$. There are, however, two choices of
isomorphisms
\[
\Hom(\ocM_{X,\ul f(\ol q)},\RR_{\ge0}) \arr\bsigma(q)=\sigma_1.
\]
Each such choice gives two arrows $\iota_{q\eta},\iota_{q\eta'}: \sigma_0\arr\sigma_1$ and a contact order $u_q$. If one choice gives
\[
\big(V_\eta, u_q,\iota_{q\eta},\iota_{q\eta'}\big)
\]
for the tuple in Proposition~\ref{Prop: General tropicalization punctured map},
the other choice swaps $\iota_{q\eta},\iota_{q\eta'}$ and replaces $u_q$ by
$-u_q$. This is indeed the action of $\GG=\ZZ/2$ on the set of tuples as stated
in the same proposition.

The relation to the Whitney umbrella $Y=V(x^2z-y^2)\subseteq \AA^3$ is as
follows. Endow $Y$ with the restriction of the divisorial log structure on
$\AA^3$ defined by $Y$. We view $Y\setminus V(z)$ as a fibration over
$\AA^1\setminus\{0\}$ by one-nodal rational curves via projection to the
$z$-coordinate. Then there is an \'etale map $Y\setminus V(z)\arr X$ of degree
two of fiber spaces over $\AA^1\setminus\{0\}$ that separates the branches of
the fibers of $X\to \AA^1\setminus\{0\}$.
\end{example}

\subsubsection{Types of punctured maps with non-simple targets}
\label{sss: types with non-simple targets}
One way to define the type of a punctured map in general is as an equivalence
class of tropicalizations which identifies two tropical punctured maps whenever
they fit into one family. The action of the automorphism group $\GG$ on
a face map $\iota_{x\eta}$ in Proposition~\ref{Prop: General tropicalization
punctured map} is induced by propagation along appropriate families. Thus in
the general case, the type of a punctured map at a geometric point, or of a
tropical punctured map, in addition to $(G,\bg,\bsigma,\bu)$ needs to specify
these face maps $\iota_{x\eta}$, at least up to the overall action by
$\GG$. This leads to the following modification of Definition~\ref{Def: type}.

\begin{definition}
\label{Def: Types with monodromy}
(1)~A \emph{framed type (of a family of tropical punctured maps)} is a tuple
$(G,\bg,\bsigma,\bu)$ with $\bu(x)\in N_{\bsigma(x)}$ for all $x\in E(G)\cup
L(G)$ as in Definition~\ref{Def: type}, together with arrows\footnote{In the
case of a self-intersecting node $x=q$ there are two such arrows, which as
before we do not distinguish by the notation.} in $\Sigma(X)$,
\[
\iota_{xv}: \bsigma(v)\arr \bsigma(x),
\]
for all $x\in E(G)\cup L(G)$ and $v\in V(G)$ an adjacent vertex.\\[1ex]
(2)~The \emph{type (of a family of tropical punctured maps)} is an equivalence
class of framed types under the obvious action of $\GG$ on the set of framed
types, as obtained from Proposition~\ref{Prop: General tropicalization punctured
map}. The notation for a framed type is $(G,\bg,\bsigma,\bu,\biota)$ with $\biota=(\iota_{xv})_{x,v}$.

The \emph{type of a punctured map $(C^\circ/W,\bp,f)$ to $X$ at a
geometric point $\ol w$ of $W$} is the type of the associated tropical map
$\Gamma\arr\Sigma(X)$ over $\omega=(\ocM_{W,\ol w}^\vee)_\RR$.
\end{definition}

Note that $\GG$ acts trivially on the domain data $(G,\bg)$, the strata map
$\bsigma$ and on global contact orders. So for framed types the action is on the
tuple $(\bu(x),\iota_{xv})$ with $x$ running through $E(G)\cup V(G)$ and $v$
through vertices adjacent to $x$. In particular, since the group $\GG$ acts
also trivially on the space $\fC_\sigma$ of global contact orders for
$\sigma\in\Sigma(X)$, the definition of global type in Definition~\ref{Def:
global type} remains unchanged.

We skip the obvious decorated versions of the notions of types in the general
case. These just add the data of curve classes to vertices.

\subsubsection{Contraction morphisms of types for non-simple targets}

The definition of contraction morphism of types
\[
\phi:\tau= (G,\bg,\bsigma,\bu)\arr \tau'=(G',\bg',\bsigma',\bu')
\]
from \cite[Def.\,2.24]{decomposition} imposes the condition that
$\bsigma'(\phi(x))$ is a face of $\bsigma(x)$ for all $x\in V(G)\cup
\big(E(G)\setminus E_\phi\big)\cup L(G)$. In the general case, this condition
has to be replaced by the choice of an arrow
\[
\bsigma'(\phi(x))\arr \bsigma(x)
\]
in $\Sigma(X)$ as part of the data defining $\phi$. We obtain the following definition.

\begin{definition}
\label{Def: Contraction morphism}
1)\ Let $\tau= (G,\bg,\bsigma,\bu,\biota)$, $\tau'=
(G',\bg',\bsigma',\bu',\biota')$ be two framed types. A \emph{contraction
morphism of framed types} $\tau\arr\tau'$ is a contraction morphism
$\phi:(G,\bg)\arr (G',\bg')$ of genus-decorated graphs together with arrows
\[
\iota_x: \bsigma'(\phi(x))\arr \bsigma(x)
\]
in $\Sigma(X)$ for all $x\in V(G)\cup \big(E(G)\setminus E_\phi\big)\cup L(G)$.
We require that the $\iota_x$ are compatible with $\biota,\biota'$, that is,
the diagrams
\begin{equation}
\label{Diag: contraction framed types}
\vcenter{\xymatrix{
\bsigma'(\phi(v))\ar[r]^{\iota_v}\ar[d]^{\iota'_{\phi(x)\phi(v)}}&
\bsigma(v)\ar[d]^{\iota_{xv}}\\
\bsigma'(\phi(x))\ar[r]^{\iota_x}&\bsigma(x)
}}
\end{equation}
commute, for all $x\in \big(E(G)\setminus E_\phi\big)\cup L(G)$ and all $v\in
V(G)$ an adjacent vertex.\footnote{Note that $\iota'_{\phi(x)\phi(v)}$ is uniquely determined by the diagram from $\iota_v,\iota_{xv}, \iota_x$.}\\[1ex]
2)\ An equivalence class for the obvious action of the group $\GG$ from
\eqref{Eqn: GG} acting on the set of contraction morphisms with domain framed
types with given $(G,\bg,\bsigma)$ defines the notion of \emph{contraction
morphism of types}.
\end{definition}

There is again no change in the definition of contraction morphism of global types compared to the case with simple $X$.

As in the discussion of types in the preceding \S\ref{sss: types with non-simple
targets}, we have again skipped spelling out the trivial generalization to the
decorated versions.

Contraction morphisms arise from specializations in families of punctured maps,
as proved in the case of simple $X$ in Proposition~\ref{Prop: Generization leads
to contraction morphisms}. Here is the version for the general case.

\begin{proposition}
\label{Prop: Contraction morphisms from specialization in general}
Let $(C^\circ/W,\bp,f)$ be a stable punctured map to $X$ over some logarithmic
scheme $W$, and let $\ol w'\arr \ol w$ be a specialization arrow of geometric
points of $W$. Let $(\tau,\bA)$ with $\tau= (G,\bg,\bsigma,\bu,\biota)$ be the
decorated framed type of $(C/W,\bp,f)$ at the geometric point $\ol w$
of $\ul W$ according to Definition~\ref{Def: Types with monodromy},(2) by a
choice of arrows \eqref{Eqn: Choice of isos with bsigma(x)}. Let similarly
$(\tau',\bA')$ with $\tau'= (G',\bg',\bsigma',\bu')$ be the decorated framed
type of $(C^\circ/W,\bp,f)$ at $\ol w'$, for the induced choice of arrows
\eqref{Eqn: Choice of isos with bsigma(x)}.

Then the map
\[
(\tau,\bA)\arr (\tau',\bA')
\]
induced by generization is a contraction morphism.
\end{proposition}

\begin{proof}
The proof is again identical to the proof of \cite[Lem.\,2.30]{decomposition}
save keeping track of the choices of arrows in $\Sigma(X)$.
\end{proof}

\subsubsection{The basic monoid and tropical moduli in general}
\label{sss: basic monoid with monodromy}
The definition of basicness (Definition~\ref{Def: Basicness}) makes sense in
complete generality by replacing ``type'' by ``a framed type representing the
type of $(C^\circ/W,\bp,f)$ at the geometric point $\ol w$''. Indeed, given a
framed type, the space of tropical curves of the given framed type is a subspace
of the set of tuples $(V_\eta,\ell_q)$ with entries taking values in strongly
convex rational polyhedral cones and subject to some integral equalities, hence
is parametrized by a strongly convex rational polyhedral cone itself. This
cone has been made explicit in Proposition~\ref{Prop: Basic monoid} in the case
of simple $X$. Here is the restatement of this proposition with reference to a
framed type.

\begin{proposition}
\label{Prop: Basic monoid for framed type}
Let $(\pi:C^\circ/W,\bp,f)$ be a basic, pre-stable punctured map over a
logarithmic point $\Spec(Q\arr \kappa)$ with $\kappa$ an algebraically
closed field. Denote by $G$ the dual intersection graph of $C^\circ$. For each
$x\in V(G)\cup E(G)\cup L(G)$ with associated geometric point $\ol x$ of $\ul
C^\circ$ and smallest stratum $\bsigma(x)\in\Sigma(X)$ containing $f(\ol x)$
choose an isomorphism
\[
\mu_x: \ocM_{X,f(\ol x)}\arr \big(\bsigma(x)_\ZZ\big)^\vee,
\]
dual to an arrow in $\Sigma(X)$ as in \eqref{Eqn: Choice of isos with
bsigma(x)}. Denote by $(G,\bg,\bsigma,\bu,\biota)$ the framed type of
$(\pi:C^\circ/ W,\bp,f)$ defined by this choice according to the discussion
leading to Proposition~\ref{Prop: General tropicalization punctured map}. Then
the map
\begin{equation}
\label{Eqn: basic monoid in general}
\textstyle
Q^\vee\arr \big\{\big((V_\eta)_\eta,(\ell_q)_q\big)\in\prod_\eta
\bsigma(\eta)_\ZZ\times \prod_q\NN\,\big|\, \iota_{q\eta}\circ V_\eta-
\iota_{q\eta'}\circ V_{\eta'}= \ell_q\cdot \bu(q) \big\}
\end{equation}
with $V_\eta$-entry the dual of ${\big(
\ol{\pi^\flat_\eta}\big)^{-1}} \circ\bar f^\flat_{\ol\eta}\circ \mu_\eta^{-1}:
\bsigma(\eta)_\ZZ^\vee\arr Q$ and $\ell_q$-entries given by the dual of
the classifiying map $\prod_q \NN\arr Q$ of the log smooth curve $C/W$, is an
isomorphism. Here $\eta$ and $q$ run over the set of generic points and nodes of
$\ul C$, respectively. The equation in the bracket holds in
$N_{\bsigma(q)}$ for all nodal points $q$ with adjacent generic points
$\eta,\eta'$ ordered according to the orientation of $E_q$ (with the usual
ambiguity of notation concerning self-intersecting nodes).
\end{proposition}

\begin{proof}
The proof is identical to the proof of Proposition~\ref{Prop: Basic monoid} once
the refined tropicalization procedure of \S\ref{sss: refined tropicalization
over geometric point} is taken into account.
\end{proof}

With this description of the basic monoid in the general case the proof of
Proposition~\ref{basic-open}, which proves that basicness is an open
condition, generalizes without problems.

The final point we want to discuss concerns the monoid quotient
\begin{equation}
\label{Eqn: chi_tautau'}
\chi_{\tau\tau'}: Q_{\tau'}\arr Q_{\tau\tau'},
\end{equation}
of basic monoids from \eqref{Eqn: localization of basic monoids} obtained from a
framed type $\tau'$ and contraction morphism $\ol \tau'\arr\tau$ of the
associated global type. The basic monoid $Q_{\tau'}$ depends only on the framed
type, as spelled out in \eqref{Eqn: basic monoid in general}. But note that the
group $\GG$ from \eqref{Eqn: GG} generally acts non-trivially on the right-hand
side of \eqref{Eqn: basic monoid in general}, so the basic monoid is \emph{not}
intrinsic to the type.

Similarly, the description of $Q_{\tau\tau'}$ in \eqref{Def: Q_{tau tau'}^vee}
requires the knowledge of the image of the arrows $\iota_v:
\bsigma(\phi(v))\arr \bsigma'(v)$, hence works only for a contraction morphism
of framed types as follows. Let $(C^\circ/W,\bp,f)$ be a basic punctured map and
$\ol w$ a geometric point of $\ul W$. Then a choice of isomorphisms in \eqref{Eqn:
Choice of isos with bsigma(x)}, or equivalently of $\bmu=(\mu_x)$ in
Proposition~\ref{Prop: Basic monoid for framed type}, provides a framed type
$\tau'=(G',\bg',\bsigma',\bu',\biota')$ and an isomorphism of $\ocM_{W,\ol
w}^\vee$ with the submonoid $Q_{\tau'}^\vee\subseteq \prod_\eta
\bsigma'(\eta)^\vee_\ZZ\times\prod_q \NN$ on the right-hand side of \eqref{Eqn:
basic monoid in general}. Let $\phi:\ol\tau'\arr\tau$ be a contraction
morphism of the global type $\ol\tau'$ associated to $\tau'$ to some other
global type $\tau =(G,\bg,\bsigma,\ol\bu)$. Then each choice $\biota_\bullet$ of
arrows
\[
\iota_v: \bsigma(\phi(v))\arr \bsigma'(v),\quad v\in V(G)
\]
in $\Sigma(X)$ provides a face $Q_{\tau\tau'}^\vee(\biota_\bullet)\subseteq
Q_{\tau'}$ as in \eqref{Def: Q_{tau tau'}^vee}, hence a dual localization
morphism
\[
\chi_{\tau\tau'}(\bmu,\biota_\bullet): \ocM_{W,\ol w}
\stackrel{\simeq}{\longrightarrow}
Q_{\tau'}\arr Q_{\tau\tau'}(\biota_\bullet)
\]
as in \eqref{Eqn: chi_tautau'}. Thus this quotient of $\ocM_{W,\ol w}$ depends
on both the choices of $\bmu$ and $\biota_\bullet$. Note that
$Q_{\tau\tau'}\neq 0$ only if there exists a degeneration of tropical punctured
maps of framed type $\tau$ compatible with the restriction on the images of
vertices given by $\biota_\bullet$.

The schematic restriction to punctured maps of global type $\tau$ is then
locally reflected in the monoid ideal
\begin{equation}
\label{Eqn: I_tautau'}
I_{\tau\tau'}=\bigcap_\biota \big(\chi_{\tau\tau'}(\bmu,\biota_\bullet)\big)^{-1}
(Q_{\tau\tau'}(\biota)\setminus \{0\}) \subseteq \ocM_{W,\ol w}.
\end{equation}
Note that unlike in the simple case, $\Spec\kk[Q_{\tau'}]/I_{\tau\tau'}$ may now
be a reducible scheme. See Definition~\ref{Def: marking by type},(3) for the use
of this ideal in a moduli context.


\section{The stack of punctured maps}
\label{sec:stack}

Throughout this section we fix as the target a morphism $X \to B$ locally of
finite type between separated, locally noetherian fs logarithmic schemes over
$\kk$. We assume further that $X$ is connected {and that $X\arr B$
fits into a commmutative diagram
\[
\xymatrix{
X\ar[r]\ar[d]&\cA_X\ar[d]\\
B\ar[r]&\cA_B
}
\]
with strict horizontal arrows, $\cA_B$ the Artin fan of $B$, and $\cA_X$ an
Artin fan representable over $\Log$ or over $\Log^1$. If $X$ has a Zariski log
structure and $X\arr B$ is log smooth then \cite[Prop.\,2.8]{decomposition}
shows that we can take the Artin fan of $X$ for $\cA_X$, which is representable
over $\Log$ by definition. In general, \cite[Cor.\,3.3.5]{ACMW} provides the
desired diagram with $\cA_X$ representable over $\Log^1$.\footnote{{The
representability assumption is used in the proof of Lemma~\ref{Lem: technical
properties of cX->B}.}} We define
\[
\cX= B\times_{\cA_B} \cA_X,
\]
which by abuse of notation we refer to as \emph{the relative Artin fan} of
$X\arr B$.}

\subsection{Stacks of punctured curves}
\label{ss: stacks of punctured curves}

The purpose of this section is the introduction of stacks of punctured curves as
domains for punctured maps.

\subsubsection{Stacks of marked pre-stable curves}
For a genus-decorated graph $(G,\bg)$
recall from \cite[\S2.4]{decomposition} the logarithmic stacks $\bM(G,\bg)$ of
$(G,\bg)$-marked pre-stable curves over the ground field $\kk$ with its basic
log structure as a nodal curve, and $\fM_B(G,\bg) = \Log_{\bM(G,\bg)\times B}$
of $(G,\bg)$-marked log smooth curves over $B$ with arbitrary fs log structures
on the base. For a leg $L\in L(G)$ denote by $p_L$ the associated marked
section. 

\subsubsection{The nodal log-ideal on $\bM(G,\bg)$} 
Since the basic monoid of an $r$-nodal curve is $\NN^r$, each $(G,\bg)$-marked
nodal curve $C\arr W$ comes with a homomorphism $\NN^r\arr \ocM_W$ with
$r=|E(G)|$. The image of $\NN^r\setminus\{0\}$ generates a coherent sheaf of
ideals $\ocI\subset \ocM_W$ with preimage $\cI\subset\cM_W$ mapping to $0$ under
the structure homomorphism $\cM_W\arr \cO_W$. Thus $\cI$ endows $\bM(G,\bg)$
with the structure of an idealized log stack.

\begin{definition}
\label{Def: nodal ideal sheaf}
We refer to $\cI$ and to any pull-back of $\cI$ to a stack over $\bM(G,\bg)$
such as $\fM(G,\bg)$ (and $\breve\fM(G,\bg)$ below) as the \emph{nodal
log-ideal}.
\end{definition}

The local structure of moduli spaces of nodal curves implies that $\bM(G,\bg)$
with the nodal log-ideal is idealized logarithmically smooth over the trivial
log point $\Spec\kk$. If $(C/W,\bp)$ is a $(G,\bg)$-marked curve, the log ideal
is generated at a geometric point $\ol w$ {of} $\ul W$ by those standard basis
vectors of $\ocM_{W,\ol w}\simeq\NN^r$ mapping to the smoothing parameters of
the nodes labelled by $E(G)$.

\subsubsection{Enter stacks of punctured curves}
We now define a stack $\breve\fM_B(G,\bg)$ of punctured curves by
admitting arbitrary puncturings at these marked sections.

\begin{definition}
\label{Def: stack of punctured curves}
Let $(G,\bg)$ be a genus-decorated graph. A \emph{$(G,\bg)$-marking} of a
punctured curve $(C^\circ/W,\bp)$ is a $(G,\bg)$-marking of the underlying
marked curve $(\ul C/\ul W,\bp)$.
The stack $\breve\fM_B(G,\bg)$ is the fibered
category over $(\Sch/\ul B)$ with objects $(G,\bg)$-marked punctured curves
$(C^\circ/W,\bp)$ over $B$. Morphisms are given by \emph{strict} fiber diagrams
of punctured curves respecting the markings by $(G,\bg)$.
\end{definition}

Note that the morphisms in $\breve\fM_B(G,\bg)$ are pull-backs of punctured
curves as defined in Definition~\ref{Def: pull-back of punctured curves}.

The maps associating to a $(G,\bg)$-marked punctured curve the underlying
$(G,\bg)$-marked nodal curve with its \emph{basic} log structure defines a
morphism of logarithmic stacks
\begin{equation}
\label{Eqn: fM(G)->bM(G)}
\breve\fM_B(G,\bg)\arr \bM(G,\bg).
\end{equation}

\subsubsection{The stacks of punctured curves are algebraic}

\begin{proposition}
\label{prop:puncurve-moduli}
(1)~The stack $\breve\fM_B(G,\bg)$ is a logarithmic algebraic stack.\\[1ex]
(2)~Endowing $\breve\fM_B(G,\bg)$ with the idealized log structure defined by the
union of its puncturing log-ideal (Definition~\ref{Def: Puncturing log ideal})
and its nodal log-ideal (Definition~\ref{Def: nodal ideal sheaf}) and
$\fM_B(G,\bg)$ with its nodal log-ideal, the {strict} morphism
\[
\breve\fM_B(G,\bg)\arr \fM_B(G,\bg)
\]
forgetting the puncturing, is {locally of finite type, quasi-separated, representable,
unramified, and idealized logarithmically \'etale.}
\end{proposition}
\begin{proof}
{We argue by showing that the morphism $\breve\fM_B(G,\bg)\arr
\fM_B(G,\bg)$ is representable by algebraic spaces, satisfying the adjectives
spelled out in (2).\footnote{A simple reduction to known stacks would be
welcome.} This is sufficient as $\fM_B(G,\bg)$ is a logarithmic algebraic stack.

The stack $\fM_B(G,\bg)$ is locally noetherian, so it has a covering $\sqcup
W_\alpha \to \fM_B(G,\bg)$ in the strict smooth topology, where $W_\alpha$ are
noetherian logarithmic schemes. Letting $W$ be one of these, 
define
\[
{\breve W} = W\times_{\fM_B(G,\bg)}\breve\fM_B(G,\bg),
\]
viewed as a category fibered in groupoids over $\ul W$, or, equivalently, over
the category of strict morphisms $T\arr W$. It suffices to prove that ${\breve
W}$ is an algebraic space satisfying the conditions of (2). 

We show this directly by exhibiting ${\breve W}$ as a \emph{sheaf of sets}, with
representable diagonal, having an \'etale covering by a scheme, and satisfying
the above conditions.

The morphism $W \to \fM_B(G,\bg)$ corresponds to a $(G,\bg)$-marked logarithmic
curve $\pi: C \to W$. Spelled out, the formation of ${\breve W}$} means that for
any strict morphism $T \to W$, the objects in ${\breve W}(T)$ are punctured
curves $(C^{\circ}_T\arr C_T\arr T,\bp_T)$ with punctures {at} the
markings of $C_T$. Here $C_T = C\times_W T \to T$ is the pull-back of the
logarithmic curve $C \to W$. Pull-backs in ${\breve W}$ are defined as
pull-backs of punctured curves along strict morphisms over $W$. The markings by
$(G,\bg)$ are inherited from $C/W$ and do not play any further role.

{First, we note that ${\breve W}$ is a sheaf of sets over $W$.
{We have to show that} any automorphism of the log curve parameterized
by $W$ induces {at most one} automorphism of any corresponding punctured
curve above it. Indeed, an
isomorphism of punctured curves over the identity of a given logarithmic curve
is a pullback diagram as in Diagram \eqref{Diagram: pull-back of punctured
curves}, with $h: T=W \to W$ and $C_T{=C} \to C$ the identity. Such an
isomorphism is an equality of the submonoids of
$\cM\oplus_{\cO^{\times}}\cP^{\gp}$ in the notation of
Definition~\ref{def:puncturing}. In particular, such an isomorphism is unique
when it exists.

{Second, $\operatorname{Isom}$ functors are representable, in fact by open
subschemes of the base $T$. Indeed, the locus on $C_T$ where two logarithmic
structures inside $\cM_{C_T}^\gp$ coincide is open in $C_T$ (as can be deduced
from Lemma~\ref{lem:logsurjctivityopen}), and its complement is a closed
subscheme of the markings of $C_T$, whose image in $T$ is closed. The complement
is the desired open subscheme of $T$. {In particular, $\breve{W} \to
\breve{W}\times_W\breve{W}$ is an open embedding; {once we prove
$\breve{W}$ is locally of finite type over $W$,} we will know the diagonal
$\breve{W} \to \breve{W}\times_W\breve{W}$ is quasi-compact. This will prove the
quasi-separatedness in (2).}

Third, it now remains to construct an \'etale atlas by a scheme, and verify the
various adjectives in (2).}

We note that the statements of the proposition are both local on $W$.
Further shrinking $W$, we may assume that the Artin fan $\cA_W$ {equals}
$\cA_Q$ for an fs and sharp monoid $Q$.}

To prove both statements {of the proposition}, it suffices to proceed as
follows: For any object $C_T^{\circ} \to C_T \to T$ in $\breve W(T)$, 
\begin{enumerate}
\item
we will construct {a locally of finite type, unramified, idealized logarithmically \'etale, and strict morphism  $V \to W$,} {for $V$ some log scheme,}
\item
show that $T \to W$ factors through $V$,
\item
construct a punctured curve $C^{\circ}_V \to C_V \to V$, and
\item
show that $C^{\circ}_T \to C_T \to T$ is the pull-back of $C^{\circ}_V \to C_V
\to V$. 
\item {Finally, we will show that the tautological morphism $V \to
{\breve W}$ defined by the family $C^{\circ}_V \to C_V \to V$ is \'etale.} 
\end{enumerate}
{In particular, we obtain an \'etale cover $\sqcup V \to \breve{W}$ of
the sheaf ${\breve W}$ {by ordinary schemes, or {equivalently}, by
strict \'etale morphisms of log schemes}.}

{Since the statements above are \'etale local on $T$, we may assume the
Artin fan $\cA_{T}$ {equals} $\cA_{Q'}$ for some fs sharp monoid $Q'$.
Since {the puncturing ideal $\cK_T$ of $(C_T^\circ\to T,\bp_T)$} is
coherent, further shrinking $T$ we may assume that there is a monoid ideal $K
\subset Q'$ such that the corresponding log ideal $\cK$ on $\cA_{Q'}$ pulls-back
to {$\cK_T$}.

The strict morphism $T \to W$ induces a strict open embedding $\cA_{Q'} \to \cA_{Q}$. Replacing $W$ by its strict open subscheme $W\times_{\cA_{Q}}\cA_{Q'}$, we may assume that $Q = Q'$.}
\smallskip

\noindent
\textsc{Step~1. Construction of $V\to W$.}
{Fix any point $t \in T$ over the unique closed point of $\cA_{Q}$. Consider the
monoid ideal $K = \ocK_T|_{t} \subset Q$. Let $\cV \to \cA_{Q}$ be the strict
closed embedding defined by the ideal $K$, and $\cK_{\cV}$ be the corresponding
log ideal over $\cV$. Then $\cV \to \cA_{Q}$ is finite type, strict, and
idealized logarithmically \'etale. Thus the projection $V := \cV\times_{\cA_Q}W
\to W$ with the log ideal $\cK_V := \cK_{\cV}|_{V}$ is a finite type, strict
closed embedding and idealized logarithmically \'etale.}

\smallskip

\noindent
\textsc{Step~2. $T \to W$ factors through $V$.}
{Recall that $\cK_T$ is the pull-back of $\cK$. By
Proposition~\ref{prop:puncturing-ideal-vanish} {applied to $C^\circ_T /T$}
the image {$\alpha_{T}(\cK_T)= \left(\alpha_{\cA_Q}(\cK)\right)_T$} is the
zero ideal. Hence the morphism $T \to \cA_{Q}$ factors through $\cV$.
Consequently, $T \to W$ factors through $V$, as claimed. 

For the point  $t$ as in Step~1, we denote its image in $V$ by $w$. }

\smallskip

\noindent
\textsc{Step~3. Construction of the punctured curves $C^{\circ}_V \to
C_{V} \to V$.}
To construct the sheaf of monoids $\ocM_{C^{\circ}_V}$, first notice that the
inclusion $\ocM_{C_V} \subseteq \ocM_{C^{\circ}_V}$ is an isomorphism away from
the points of $\bp$. For each puncture $p_w\in \bp_w$ over $w$, we define
$\ocM_{C^{\circ}_V, p_w} := \ocM_{C^{\circ}_T, p_t}$ using the fiber over $t$.
{Let $p_T,p_V$ be the punctured sections
corresponding to $p_w$ of $\ul{C}_T/\ul{T}$, $\ul{C}_V/\ul{V}$ respectively.}
{Note that we have 
\[
\ocM_{C^{\circ}_T, p_t} \stackrel{\cong}{\arr} \Gamma(\ul T,p_T^*\ocM_{C_T^{\circ}}),\qquad
\ocM_{C_T, p_t}= \ocM_{C_V,p_w}=Q\oplus\NN  \stackrel{\cong}{\arr} \Gamma(\ul V, p_V^* \ocM_{C_V}).
\]
Define $\ocM_{C_V^\circ}\subset \ocM_{C_V}^\gp$ as the subsheaf of fine monoids
generated by $\ocM_{C_V^\circ,p_w}\subset \ocM_{C_V,p_w}^\gp$.}

Consider $\cM_{C^{\circ}_V} := \cM_{C_V}^{\gp} \times_{\ocM_{C_V}^{\gp}}
\ocM_{C_V^{\circ}}$. Observe that $\cM_{C_V} \subseteq \cM_{C^{\circ}_V}$. We
define the structure morphism $\alpha_{{C^{\circ}_V}}: \cM_{C^{\circ}_V} \to
\cO_{C_V}$ as follows. First, we require $\alpha_{{C^{\circ}_V}}|_{{\cM_{C_V}}} =
\alpha_{{C_V}}$. Second, for a local section $\delta$ of $\cM_{C^{\circ}_V}$ not
contained in $\cM_{C_V}$, we define $\alpha_{{C^{\circ}_V}}(\delta) = 0$. 

This defines a monoid homomorphism. Indeed, using the decomposition
$\cM_{C^{\circ}_V}\subseteq \cM\oplus_{\cO^{\times}_{C_v}}\!\cP^{\gp}$
{as in Definition~\ref{def:puncturing},} write
$\delta={(\delta', \delta'')}$ with $\delta'$ the pull-back of a section of
$\cM_V$. It is sufficient to check that when $\delta \notin \cM_{C_V}$ we have
$\alpha_V(\delta')=0$.

In the notation of \S\ref{sss: log ideals punctured curves} the assumption
$\delta \notin \cM_{C_V}$ implies $\delta \in \cI_p$. Hence according to
Definition~\ref{Def: Puncturing log ideal} we have $\delta' \in \cK_V$. As $V$
is defined by $\alpha_V(\cK_V) = 0$, we have $\alpha_V(\delta')=0$ as needed.

This defines a logarithmic structure $\cM_{C^{\circ}_V}$ over $\ul{C}_V$. The
inclusion of logarithmic structures $\cM_{C_V} \subseteq \cM_{C^{\circ}_V}$ is a
puncturing, hence defines {a} punctured curve $C^{\circ}_V \to C_V \to V$.
\smallskip

\noindent
\textsc{Step~4. $C^{\circ}_T \to C_{T} \to T$ is the pull-back of $C^{\circ}_V
\to C_{V} \to V$ via $i: T \to V$.} Denote by $\ul{j}: \ul{C}_T\arr \ul{C}_V$
the pull-back of $\ul{i}$. Since $C_T \to T$ is given by base change from $C_V
\to V$, it suffices to show that $\ul{j}^{*}\cM_{C^{\circ}_V} =
\cM_{C^{\circ}_T}$ as sub-sheaves of monoids in $\cM^{\gp}_{C_T}$. Away from the
punctures, the equality clearly holds. Along each puncture $p \in \bp_T$, we
have the equality $j^{*}\ocM_{C^{\circ}_V, p_w} = \ocM_{C^{\circ}_T, p_t}$ at
$p_t$ {by the construction in Step~3}, which extends along the marking
$p$ by generization. This proves the desired equality.
\smallskip

\noindent
\textsc{Step~5. {\'Etale covering}.}
Consider a strict, square-zero extension $T \to T'$ over $W$ and a
family of punctured curves $C^{\circ}_{T'} \to C_{T'} \to T'$ such that $C_{T'}
= C\times_W T'$, and $C^{\circ}_T \to C_T \to T$ is the pull-back of
$C^{\circ}_{T'} \to C_{T'} \to T'$. Since the strict morphism $T' \to \cA_Q$
again factors through $\cA_{Q'}$, we may continue to assume $Q = Q'$. Applying
Step~2 again, we see that $T' \to W$ factors through $V$ uniquely.

Denote by $t' \in T'$ the image of $t$ via $T \to T'$. The family
$C^{\circ}_V \to C_{V} \to V$ is constructed using the same geometric fiber over
$t$. Applying Step~4 again, we see that $C^{\circ}_{T'} \to C_{T'} \to T'$ can
be obtained via pulling back $C^{\circ}_V \to C_{V} \to V$. 

This shows that $V \to \breve{W}$ is formally \'etale, and we claim it is
actually \'etale, in other words, for any scheme $T''$ and morphism $T'' \to
\breve{W}$, we need to show that $T''\times_{\breve{W}}V \to T''$ is locally of
finite presentation. The question being local, we may assume $T''\to \breve{W}$
factors through some $V'' \to \breve{W}$ in our covering, and may as well
replace $T''$ by $V''$. In this case $V \times_{\breve{W}} V'' \to V
\times_WV''$, the pull-back of the diagonal $\breve{W} \to
\breve{W}\times_W\breve{W}$ along $V \times_WV'' \to
\breve{W}\times_W\breve{W}$, is an open embedding. As $V, V''$ and $W$ are
noetherian, the map $V \times_{\breve{W}} V'' \to V''$ is of finite
presentation.

Moreover, since $\sqcup V \to W$ is locally of finite presentation and $\sqcup V
\to \breve{W}$ is \'etale and surjective, we have that $\breve{W} \to W$ is
locally of finite presentation, see {\cite[Sect.\,06Q1]{stacks-project}.} As
indicated earlier, this implies that the diagonal is quasiseparated, completing
the proof.
\end{proof}

\subsection{Stacks of punctured maps marked by tropical types}
\label{ss: stacks marked by types}

\subsubsection{Weak markings and markings}
In analogy with \cite[Definition~2.31]{decomposition} we define:
\begin{definition}
\label{Def: marking by type}
Let $\tau=(G,\bg,\bsigma,\bar\bu)$ be a global type of punctured maps
(Definitions~\ref{Def: global type}. A \emph{weak marking by $\tau$} of a
{basic} punctured map $(C^\circ/W,\bp,f)$ to $X$ is a $(G,\bg)$-marking
of the domain curve $(C^\circ/W,\bp)$ {(Definition~\ref{Def: stack of
punctured curves})} with the following properties:
\begin{enumerate}
\item
The restriction of $\ul f$ to the closed subscheme $Z\subseteq \ul C$ (a
subcurve or punctured or nodal section of $C$) defined by $x\in V(G)\cup
E(G)\cup L(G)$ factors through the closed stratum $X_{\bsigma(x)}\subseteq \ul
X$ (\S\ref{sss: tropical punctured maps}).
\item
For each geometric point {$\ol w$ of $\ul W$} with $\tau_{\ol w}=(G_{\ol
w},\bg_{\ol w},\bsigma_{\ol w},\bu_{\ol w})$ the associated type of
$(C^\circ/W,\bp,f)$ at $\ol w$ (Definition~\ref{Def: type}), the
{contraction} morphism $(G_{\ol w},\bg_{\ol w})\arr (G,\bg)$ of
decorated graphs given by the marking defines a contraction morphism of the
associated global types
\begin{equation}
\label{Eqn: marking contraction morphism}
(G_{\ol w},\bg_{\ol w},\bsigma_{\ol w},\bar \bu_{\ol w})\arr
\tau= (G,\bg,\bsigma,\bar\bu).
\end{equation}
\end{enumerate}

A weak marking of $(C^\circ/W,\bp,f)$ by $\tau$ is a \emph{marking} if in
addition the following condition holds.
\begin{enumerate}
\item[(3)]
{
For all geometric points {$\ol w$ of $\ul W$}, the ideal in $\cM_{W,\ol
w}$ defined by the monoid ideal $I_{\tau\tau_{\ol w}}\subseteq \ocM_{W,\ol w}$
in \eqref{Eqn: I_tautau'} maps to $0$ under the structure morphism $\cM_{W,\ol
w}\arr \cO_{W,\ol w}$.}
\end{enumerate}
A \emph{marking of $(C^\circ/W,\bp,f)$ by a decorated global type
$\btau=(\tau,\bA)$} is defined analogously, with the associated types replaced
by associated decorated types introduced in \eqref{Eqn: decorated type}.
\end{definition}

{In the definition, basicness is not necessary for (1) and (2), but is needed when referring to \eqref{Eqn: I_tautau'} in (3).}

{Note that a marking of a punctured map by a global type $\tau$ does not
mean that $\tau$ is realizable. It just means that there is a contraction
morphism $\ol\tau'\to\tau$ from {the global type $\ol\tau'$ associated
to} a realizable type $\tau'$, the type of the {given}
punctured map.}

{
\begin{remark}
\label{rem:weak vs strong}
The difference between weak markings and markings is fairly subtle
and is related to saturation in the definition of the basic monoid.
Recall first the construction of the basic monoid from
\cite[Const.\,1.16]{LogGW}. Let $f:C/W\rightarrow X$ be a punctured
map defined over a log point.
The basic monoid $Q$ associated to this log map was constructed
as the saturation of a quotient of the monoid $\prod_{\eta\in\ul{C}}
P_{\eta} \times \prod_{q\in \ul{C}}\NN$. Here $\eta$ runs over
generic points of $\ul{C}$ and $q$ runs over the nodes of $\ul{C}$.
Denote by $Q^{\mathrm{fine}}$ this quotient before saturating,
so that $Q$ is the
saturation of $Q^{\mathrm{fine}}$, as in \cite[(1.14)]{LogGW}.

Now suppose that $f:C/W \rightarrow X$ is a weakly $\tau$-marked log map with
$W$ an arbitrary fs log scheme, but suppose in addition that for every geometric
point $\ol w$ of $W$, $C_{\ol w}\rightarrow X$ is of type $\tau$. Thus
$\overline\cM_W$ is locally constant with stalk $Q$. The proof of Lemma
\ref{Prop: log ideal for type} {below} implies in particular that if $s$
is any section of $\cM_W$ whose image $\ol s$ in $\overline\cM_W$ has stalk
lying in $Q^{\mathrm{fine}}\setminus\{0\}$ at each geometric point, then
$\alpha_W(s)=0$. However, the condition for being marked requires this vanishing
even when $\ol s$ lies in $Q\setminus\{0\}$.

For an explicit example where $Q^{\mathrm{fine}}$ is not saturated,
see \cite[Ex.\,1.17,(3)]{LogGW}. There, $Q^{\mathrm{fine}}$ is the submonoid
of $\ZZ^2$ generated by $(1,-6), (0,2)$ and $(0,3)$.
In such a situation, it is not difficult
to construct an example of a weakly $\tau$-marked but not $\tau$-marked
curve, as follows.

Start with a basic $\tau$-marked log map $f:C/W\rightarrow X$
with $W$ a log point, and assume that $Q\not=Q^{\mathrm{fine}}$.
Let $W^{\mathrm{fine}} = 
\Spec(Q^{\mathrm{fine}}\rightarrow \kk)$. Since all nodal generators
$\rho_E\in Q$ already lie in $Q^{\mathrm{fine}}$ by construction, we may find a
sub-log structure
$\cM_{C^{\mathrm{fine}}}\subseteq \cM_C$ so that $C^{\mathrm{fine}}\rightarrow
 W^{\mathrm{fine}}$
is a log smooth curve (in the category of fine log schemes) and
$f$ induces a morphism $C^{\mathrm{fine}}\rightarrow X$.
Saturating $W^{\mathrm{fine}}$ may yield a non-reduced scheme
$W^{\sat}$ with reduction $W$.
The composition
\[
C^{\sat}:= C^{\mathrm{fine}}\times_{W^{\mathrm{fine}}} W^{\sat} \rightarrow 
C^{\mathrm{fine}} \rightarrow X
\]
yields a stable log map in the category of fs log schemes which is
weakly marked, but not marked, by $\tau$.

In the cited example \cite[Ex.\,1.17,(3)]{LogGW}, $Q$ is the submonoid
of $\ZZ^2$ generated by $(1,-6)$ and $(0,1)$, and one checks that
$\ul{W}^{\sat}\cong \Spec \kk[Q]/\langle z^{(1,-6)},
z^{(0,2)}\rangle$, which is a scheme of length two.
\end{remark}
}

{
Under the presence of monodromy, the following more refined version of marked
punctured maps using framed types rather than global types is sometimes more
appropriate, notably in gluing. Note however that framed types work with contact
orders living on a single stratum $X_\sigma$. Hence this refined notion is
inappropriate when studying punctured maps with a contact order propagating into
several $X_\sigma$ not contained in a single stratum.

\begin{definition}
\label{Def: marking by framed type}
Let $\tau=(G,\bg,\bsigma,\bu,\biota)$ be a framed type of a family of tropical
punctured maps (Definition~\ref{Def: type}). A \emph{weak marking by $\tau$} of
a basic punctured map $(C^\circ/W,\bp,f)$ to $X$ is a weak marking by the global
type $(G,\bg,\bsigma,\ol\bu)$ associated to $\tau$, along with, for each $x\in
E(G)\cup L(G)$ with associated nodal or punctured locus $Z_x\subseteq \ul C$, a
homomorphism of sheaves of monoids
\[
\mu(x): (\ul f|_{Z_x})^{-1}\ocM_X\arr \ul{\bsigma(x)_\ZZ^\vee},
\]
whose stalkwise duals at all geometric points $\ol w$ of $W$ are arrows in
$\Sigma(X)$, and which lift the contraction morphism of global types \eqref{Eqn:
marking contraction morphism} to a contraction morphism of framed types
(Definition~\ref{Def: Contraction morphism}). Here $\ul{\bsigma(x)_\ZZ^\vee}$ is
the constant sheaf with stalks the dual of the set of integral points of
$\bsigma(x)$.

A \emph{marking by a framed type} is then defined by replacing $I_{\tau\tau'}$
in Definition~\ref{Eqn: marking contraction morphism},(3) by
$\chi_{\tau\tau'}^{-1}(Q_{\tau\tau'}\setminus\{0\})$, noting that $\mu(x)$ makes
it possible to define $Q_{\tau\tau'}$ and $\chi_{\tau\tau'}$ unambiguously and
consistently.
\end{definition}

\begin{remark}
We expect that all results that we formulate for (weak) markings by global types
hold for (weak) markings by framed types. Since the framed notions have only
been included in a late revision of the paper, we nevertheless decided to leave
the full development of this modified theory to other occasions. We emphasize
that in most applications one is either interested in simple $X$ from the outset
or one can reduce to this situation, and in this case the framed perspective
does not provide any additional information.
\end{remark}
}

\subsubsection{Enter stacks of punctured maps}

{We continue to assume that $X\arr B$ is a morphism of fs log algebraic
schemes fulfilling the assumptions stated at the beginning of
\S\ref{sec:stack}.}

\begin{definition}
\label{Def: stacks of decorated puncted maps}
Let $\btau=(G,\bg,\bsigma,\bar \bu,\bA)=(\tau, \bA)$ be a decorated
global type (Definition~\ref{Def: global type}). Then
\[
\scrM(X/B,\btau)\quad\text{and}\quad \scrM(X/B,\tau) 
\]
are defined as the stacks over $(\Sch/\ul B)$ with objects basic stable
punctured maps to $X$ over $B$ (Definition~\ref{def:stability}) marked by
$\btau$ and by $\tau$, respectively (Definition~\ref{Def: marking by type}).

Weakening stability to pre-stability, the analogous stacks to the relative Artin
fan $\cX$ of $X$ over $B$, as defined at the beginning of \S\ref{sec:stack}, are
denoted\footnote{Stability being a concept for graphs decorated by genera and
curve classes, there does exist a stable version of $\fM(\cX/B,\btau)$. We omit
this variant.}
\[
\fM(\cX/B,\btau)\quad\text{and}\quad \fM(\cX/B,\tau).
\]
The corresponding stacks with markings replaced by weak markings are denoted by
the same symbols adorned with primes:
\[
\scrM'(X/B,\btau),\quad \scrM'(X/B,\tau) ,\quad \fM'(\cX/B,\btau),
\quad \fM'(\cX/B,\tau).
\]
\end{definition}

An important special case is that $\btau$ is the class $\beta=(g,\bar\bu,A)$ of
a punctured map (Definition~\ref{Def: global type}). Then $G$ is the graph with
only one vertex $v$ of some genus $g$, stratum $\bsigma(v)=
{0}\in\Sigma(X)$, and curve class $A$, no edges, and any number of legs.
Recalling from \S\ref{sss: tropical punctured maps} that the stratum of $X$
associated to the origin $0\in\Sigma(X)$ equals $\ul X$, the resulting stacks
\begin{equation}
\label{Eqn: fM(X,beta)}
\scrM'(X/B,\beta)=\scrM(X/B,\beta),\qquad \fM'(\cX/B,\beta)=\fM(\cX/B,\beta)
\end{equation}
restrict only the total genus and total curve class, as well
as the number of punctures and their global contact orders.

\begin{remark}
\label{Rem: remark on condition (4)}
We will see in Proposition~\ref{Prop: pure-dimensional fM(cX,tau)} that for a
realizable global type $\tau$ the moduli spaces $\fM(\cX/B,\tau)$ of
$\tau$-marked punctured maps to $\cX/B$ are reduced and pure-dimensional, 
{at least for simple $X$}. For a
general global type the reduction of $\fM(\cX/B,\tau)$ is stratified by the
images of the morphisms $\fM(\cX/B,\tau')\arr \fM(\cX/B,\tau)$ for realizable
types $\tau'$ dominating $\tau$, see Remark~\ref{Rem: stratified structure of
fM(cX,tau)} below. Thus from the stratified point of view, markings as in
{Definition~}\ref{Def: marking by type},3 are the correct notion. This
feature explains their appearance in \cite[Def.\,2.31]{decomposition}.

However, the notion of weak marking, as in {Definition~}\ref{Def:
marking by type},{(1)--(2)}, appears naturally in gluing situations.
Notably the commutative square in Theorem~\ref{Thm: Gluing theorem} is only
cartesian with weak markings. For applications in Gromov-Witten theory, one
works with cycles in the moduli spaces of punctured maps appearing in this
diagram and the difference between markings and weak markings disappears,
possibly up to computable multiplicities. See for example \cite{Yixian} where
this approach is taken.
\end{remark}

\subsubsection{The stacks are algebraic}

\begin{theorem}
\label{Thm: scrM and fM are algebraic}
{Let $X \to B$ be a morphism of fs logarithmic
schemes fulfilling the assumptions stated at the beginning of
\S\ref{sec:stack}, and let}
$\btau=(G,\bg,\bsigma,\bar \bu,\bA)=(\tau, \bA)$ be a decorated
global type of punctured maps {to $X$}. Then the stacks
\[
\scrM(X/B,\btau),\ \scrM(X/B,\tau),\ \fM(\cX/B,\btau),\ \fM(\cX/B,\tau)
\]
are logarithmic algebraic stacks locally of finite type over $B$.
Moreover, $\scrM(X/B,\btau)$ and $\scrM(X/B,\tau)$ are Deligne-Mumford,
and the forgetful morphisms to the stack $\scrM(\ul X/\ul B)$ of ordinary stable
maps are representable.

Analogous results hold for the weakly marked versions $\scrM'(X/B,\btau)$,
$\scrM'(\cX/B,\tau)$, $\fM'(\cX/B,\btau)$, $\fM'(\cX/B,\tau)$.
\end{theorem}

\begin{proof}
We first restrict to $\scrM'(X/B,\btau)$ and then comment on the minor changes
for the other cases. 
\smallskip

\noindent
\textsc{Step~1: An algebraic stack of prestable maps.}
Denote by
\[
\breve \fC=\breve\fC(G,\bg) \arr \breve\fM=\breve\fM(G,\bg)
\]
the universal curve over the logarithmic algebraic stack $\breve\fM(G,\bg)$ of
$(G,\bg)$-marked punctured curves from Definition~\ref{Def: stack of punctured
curves} and Proposition~\ref{prop:puncurve-moduli}. This morphism is proper,
flat, integral, of finite type and has geometrically reduced fibers. Hence
\cite[Cor.\,1.1.1]{Wise-minimality} applies to show that
\[
\Hom_{\breve\fM}(\breve\fC,\breve\fM\times^{\mathrm{f}}_B X)
\]
is representable by a logarithmic algebraic stack, locally of finite
type.\footnote{This last property is not explicitly stated in
\cite{Wise-minimality}, but follows by inspection of the proof.} 

The rest of the proof is analogous to \cite[Prop.\,2.34]{decomposition}.
\smallskip

\noindent
\textsc{Step~2: Carving out weakly marked basic stable maps.} Condition~(1) in
Definition~\ref{Def: marking by type} of marking by $\btau$ defines a closed
substack of $\Hom_{\breve\fM}(\breve\fC,\breve\fM\times^{\mathrm{f}}_B X)$,
while all the remaining conditions in Definition~\ref{Def: marking by type},(2)
are open, {see Proposition~\ref{Prop: Contraction morphisms from
specialization in general}}. Note here we are using that curve classes are
locally constant in flat families. The condition on a map being basic is open by
Proposition~\ref{basic-open}; stability is open since it is open on {the
underlying stable} maps. Thus the morphism
\[
\scrM'(X/B,\btau)\arr \Hom_{\breve\fM}(\breve\fC,\breve\fM\times^{\mathrm{f}}_B X)
\]
forgetting all parts of the marking except the $(G,\bg)$-marking of the domain
curve, identifies $\scrM'(X/B,\btau)$ with an open substack of a strict closed
substack of $\Hom_{\breve\fM}(\breve\fC,\breve\fM\times^{\mathrm{f}}_B X)$. 
\smallskip

\noindent
\sloppy
\textsc{Step~3: Verifying properties.}
By Proposition~\ref{noautomorphisms}, logarithmic automorphisms of basic stable
maps acting trivially on underlying maps are trivial. {Hence}
$\scrM'(X/B,\btau) \to \scrM(\ul X/\ul B)$ is representable. Since $\scrM(\ul
X/\ul B)$ is a Deligne--Mumford stack, so is $\scrM'(X/B,\btau)$. Ignoring curve
classes yields the statement for $\scrM'(X/B,\tau)$.
\smallskip

\fussy
\noindent
\textsc{Step~4: Weakly marked maps to $\cX$.}
{The morphism $\cX\arr B$ from the relative Artin
fan is well behaved:

\begin{lemma}
\label{Lem: technical properties of cX->B}
The morphism $\cX\arr B$ is quasiseparated,
locally of finite type, and has affine stabilizers.
\end{lemma}

\begin{proof}
It suffices to verify these properties for the morphism $\cA_X \to \cA_B$. This
is shown in \cite[Lem.\,~2.5.5]{AW} in case $X\to B$ is logarithmically smooth,
and we indicate here why the argument applies here. Since the properties claimed
are local in $B$ (or $\cA_B$), we may assume $\cA_B$ is an Artin cone
$\cA_\tau$. Since $\cA_X$ has a cover by \'etale maps from Artin cones
$\cA_{\sigma \to \tau}$, we have that $\cA_X$ is locally of finite type. 
 
Quasiseparation follows as in \cite[Lem.\,~2.3.8(ii)]{AW}, applied to $\cA_X \to
\cA_X \times_{\cA_B} \cA_X$ instead of $\cA_X \to \cA_X \times \cA_X$ and using
representability over $\Log^1$ instead of $\Log$: one needs to show, for two
charts $\cA_{\sigma_1 \to \tau}$ and $\cA_{\sigma_2 \to \tau}$ of $\cA_X$, that
$\cA_{\sigma_1 \to \tau} \times_{\cA_X}\cA_{\sigma_2 \to \tau}$ is quasicompact.
By \cite[Lem.\,~2.3.8(i)]{AW} and representability it suffices to show that the
stack $\cA_{\sigma_1 \to \tau} \times_{\Log^1}\cA_{\sigma_2 \to \tau}$ has
finitely many points. The argument of \cite[Lem.\,~2.3.8(ii)]{AW} then applies
as stated.
 
The claim about stabilizers follows as in \cite[Lem.\,~2.5.5]{AW}.
\end{proof}

It follows that \cite[Cor.\,1.1.2]{Wise-minimality} still applies.} The rest of
the proof for $\fM'(\cX/B,\btau)$ and $\fM'(\cX/B,\tau)$ is the same,
except we can not conclude the Deligne-Mumford property due to the absence of
stability.
\smallskip

\noindent
\textsc{Step~5: Marked maps.} 
Stacks of marked maps are closed substacks of stacks of weakly marked maps,
locally defined by the {log-ideal $I_{\tau\tau_{\ol w}}$} in
Definition~\ref{Def: marking by type},(3).\footnote{For a much more detailed
discussion of this point, in terms of the {idealized} structure defined by
markings, see \S\ref{ss:idealized smoothness} below, and notably
Theorem~\ref{thm:idealized-etale}.} Hence the result also holds for these cases.
\end{proof}

\subsection{Boundedness}
\label{ss:boundedness}
For ordinary stable logarithmic maps, boundedness of $\scrM(X/B,\beta)$ is
established in \cite{AC,LogGW} for projective $X\arr B$ under the technical
assumption that $\ocM_X$ is globally generated. \cite{ACMW} removed the
technical assumption by showing that there is a logarithmic blowing up $Y\arr X$
with $\ocM_Y$ globally generated and then using birational invariance of the
moduli spaces $\fM(\cX/B,\beta)$ under this process. Since this birational
invariance seems to be rather more subtle in the punctured case, we content
ourselves with a statement assuming global generatedness, which suffices for
most practical applications. {Throughout this and the next subsections
we assume that the log structure on $X$ is Zariski as in \cite{LogGW}, which we
follow. We believe this assumption could be removed by minor adaptations of the
proof.}

\begin{theorem}
\label{thm:boundedness}
Suppose the underlying family $X \to B$ is projective, and the
sheaf $\ocM_{X}^\gp\otimes_\ZZ\QQ$ is generated by its global
sections.\footnote{{Samuel Johnston in \cite{Johnston} has meanwhile
removed the global generatedness assumptions along the same
line as \cite{ACMW}.}} Then the projection $\scrM(X/B,\beta) \to B$ is of
finite type.
\end{theorem}

\begin{proof}
We split the proof into several steps. The theorem follows from
Propositions~\ref{prop:refined-stack-bounded}
and~\ref{prop:refinement-combinatorially-finite} below.
\end{proof}

Global generatedness of $\ocM_X^\gp\otimes_\ZZ\QQ$ can be easily read off from
the cone complex $\Sigma(X)$ as follows.

\begin{proposition}
\label{Eqn: Sigma -> RR^r}
The sheaf $\ocM_X^\gp\otimes_\ZZ\QQ$ is generated by global sections if and only
if there exists a continuous map
\[
\big|\Sigma(X)\big|\arr \RR^r
\]
with restriction to each $\sigma\in\Sigma(X)$ {an injective homomorphism
of additive monoids}.
\end{proposition}

\begin{proof}
A map $|\Sigma(X)|\arr \RR^r$ which is injective when restricted to any
$\sigma\in\Sigma(X)$ is dual to a system of surjective homomorphisms
\[
\varphi_\sigma: \RR^r \arr \Hom(\sigma,\RR),
\]
compatible with the dual of the face maps defining $\Sigma(X)$. But such
a compatible system $(\varphi_\sigma)_{\sigma\in\Sigma(X)}$ of surjections is
equivalent to a linear map
\[
\RR^r\arr \Gamma(\ul X,\ocM_X^\gp\otimes_\ZZ\RR)=
\Gamma(\ul X,\ocM_X^\gp\otimes_\ZZ\QQ)\otimes_\QQ\RR
\]
with composition to $\ocM_{X,x}^\gp\otimes_\ZZ\RR$ surjective for each
$x\in X$. The claim follows.
\end{proof}

\begin{remark}
We remark that if $\ocM_X^{\gp}\otimes_{\ZZ}\QQ$ is generated by global
sections, then all global contact orders of $X$ are monodromy free, which we see
as follows. The map of Proposition~\ref{Eqn: Sigma -> RR^r} gives a well-defined
map $\fC_{\sigma}(X)\rightarrow \ZZ^r\subseteq \RR^r$. Indeed, if
$\sigma'\in\Sigma_{\sigma}(X)$ and $u\in N_{\sigma'}$, we may view $u$ as
{an integral tangent vector (i.e., an element of $N_{\sigma'}$)} to
$\sigma'\in \Sigma(X)$ and take its image under the map $|\Sigma(X)|\rightarrow
\RR^r$. Since $u$ is compatible with inclusion of faces, this provides a point
$v$ of $\ZZ^r\subseteq \RR^r$ only depending on $\iota_{\sigma\sigma'}(u)$ (see
Definition~\ref{Def: global contact order} for notation). Since
$|\Sigma(X)|\rightarrow \RR^r$ is injective on cones, $v$ arises, for each
$\sigma'$, as the image of at most one $u\in N_{\sigma'}$. Hence all global
contact orders are monodromy free.
\end{remark}

\subsubsection{Boundedness of $\scrM(X/B, \beta)$}
\begin{definition}
A class $\beta$ of a punctured map (Definition~\ref{Def: global type}) is
called \emph{combinatorially finite} if the set of types (Definition~\ref{Def:
type}) of stable punctured maps with associated class $\beta$ is finite.
\end{definition}

\begin{proposition}\label{prop:refined-stack-bounded}
Suppose $\beta$ is combinatorially finite. Then the forgetful
map
\begin{equation}
\label{equ:remove-log}
\scrM(X/B, \beta) \to \scrM(\ul{X}/\ul{B},\ul{\beta})
\end{equation}
is of finite type.
\end{proposition}

\begin{proof}
The strategy of the proof is similar to those in \cite[\S3.2]{LogGW} and
\cite[\S5.4]{Chen} by showing that each stratum with constant combinatorial
structure is bounded. The proof is largely the same, with extra care needed only
in the proof of \cite[Prop.\,3.17]{LogGW}. 

{By Theorem~\ref{Thm: scrM and fM are algebraic}, $\scrM(X/B,\beta)
\rightarrow B$ is locally of finite type, and hence so is the morphism
$\scrM(X/B, \beta) \to \scrM(\ul{X}/\ul{B},\ul{\beta})$. Thus it is sufficient
to prove the latter morphism is quasi-compact. We thus need to show that
$\ul{W}\times_{\scrM(\ul{X}/\ul{B},\ul{\beta})} \scrM(X/B,\beta)$ is
quasi-compact for any quasi-compact scheme $\ul{W}$ and morphism
$\ul{W}\rightarrow \scrM(\ul{X}/\ul{B},\ul{\beta})$. Using
\cite[Lem.\,3.14]{LogGW}, it is enough to find a weak cover {in the
sense of \cite[Def.\,3.13]{LogGW}} of $\ul{W}\rightarrow
\scrM(\ul{X}/\ul{B},\ul{\beta})$ by finitely many quasi-compact subsets. We may
weakly cover $\ul{W}$ by a finite number of locally closed strata on which the
corresponding ordinary stable map is combinatorially constant (in the sense of
\cite[Def.\,3.15]{LogGW}), and replace $\ul{W}$ with one of these locally closed
strata. Thus we may assume given $\frf=(\ul{C}/\ul{W},\bp, \ul{f})$ a
combinatorially constant} ordinary stable map over an integral, quasi-compact
scheme $\ul{W}$. Then $\ul{W}\times_{\scrM({\ul X}/\ul{B},\ul{\beta})}
\scrM(X/B,\beta)$ classifies punctured enhancements of the ordinary stable maps
parameterized by $\ul{W}$, {and we need to show this fibre product is
quasi-compact}. 

As
the combinatorial type of a log curve with constant dual intersection graph is
locally constant, we have a decomposition
\[
\ul{W}\times_{\scrM({\ul X}/\ul{B},\ul{\beta})} \scrM(X/B,\beta)
=\coprod_{\bf u} \scrM(X,\frf,{\bf u})
\]
into disjoint open substacks according to the type ${\bf u}$. As $\beta$
is assumed combinatorially finite, this is a finite union. Hence it is
sufficient to show that each $\scrM(X,\frf,{\bf u})$ is quasi-compact.
 As in the proof of
\cite[Prop.\,3.17]{LogGW}, it is sufficient to construct a quasi-compact
stack $Z$ with a morphism $Z\rightarrow \scrM(X,\frf,{\bf u})$
which is surjective on geometric points.

To do so, set $Q_1 := \NN^{k}$, where $k$ is the number of nodes of any fiber of
$\ul{C}\rightarrow\ul{W}$. By Proposition~\ref{Prop: Basic monoid} and the fact
we have fixed the type ${\bf u}$, the basic monoid $Q$ is constant on
$\scrM(X,\frf,{\bf u})$, and there is a canonical morphism $Q_1 \to Q$. The
latter induces a morphism of Artin cones $\cA_{Q^{\vee}} \to \cA_{Q^{\vee}_1}$.
We equip $\ul{W}$ with the canonical log structure coming from the family of
pre-stable curves $\ul{C}\rightarrow\ul{W}$, and consider $Z_1 =
\cA_{Q^{\vee}} \times_{\cA_{Q^{\vee}_1}}W$.
Pulling back the universal family from $W$, we
obtain a family of log curves $C_1 \to Z_1$ and an ordinary stable
map $\ul{f}: \ul{C}_1 \to \ul{X}/\ul{B}$. Observe that there is a global
chart $Q \to \ocM_{Z_1}$. To check $Z_1$ is quasi-compact we can, and do,
replace $Z_1$ with its underlying reduced substack.

The type ${\bf u}$ prescribes, for each marked section $p \in \bp$, an
ideal sheaf
\[
\ol\cI_p \subseteq \ocM_W\oplus\ZZ\subseteq p^*\ocM_C^\gp
\]
generated by $u_p^{-1}(\ZZ_{<0})$, which, we note, is constant along $Z_1$.
These ideals produce an ideal $\ocK \subseteq Q$ as in Definition~\ref{Def:
Puncturing log ideal} by taking into account all punctures in $\bp$. Denote by
$\cK = \ocK\times_{\ocM_{Z_1}}\cM_{Z_1}$ the resulting log ideal, where the
arrow on the left is given by the composition $\ocK \to Q \to \ocM_{Z_1}$ with
the last arrow the global chart. 
 
To obtain a family of punctured stable maps of type ${\bf u}$ over $Z_1$ then
requires that $\alpha_{Z_1}(\cK)=0$ by
Proposition~\ref{prop:puncturing-ideal-vanish}. Thus in particular if $0\in
\ocK$, then there are no punctured maps of type ${\bf u}$ and we can ignore such
a ${\bf u}$; otherwise, as $Z_1$ is reduced and $\overline{\cM}_{Z_1}$ is
locally constant with stalk $Q$, necessarily $\alpha_{Z_1}(\cK)=0$. Indeed, any
local section $m$ of $\cK$ maps to a nowhere zero section of $\ocM_{Z_1}$, and
hence $\alpha_{Z_1}(m)$ is nowhere invertible, thus zero, since $Z_1$ is
reduced.

We now construct a punctured family of curves $C_1^{\circ}\rightarrow Z_1$.
First, the ghost sheaf $\ocM_{C^{\circ}_1}$ is identical to $\ocM_{C_1}$ away
from the punctures. Along each puncture $p\in \bp$, we take $\ocM_{C^{\circ}_1,
p} \subset \ocM^{\gp}_{C_1, p}$ to be the smallest fine submonoid generated by
$\ocM_{C_1, p}$ and the image of $f^{-1} \ocM_X \to \ocM_{C_1,p}^{\gp}$
determined by the type ${\bf u}$. As all the ghost sheaves and
morphisms between them are constant along $Z_1$, this yields a
well-defined sheaf of monoids $\ocM_{C^{\circ}_1}$, hence $\cM_{C^{\circ}_1} :=
\ocM_{C^{\circ}_1}\times_{\ocM^{\gp}_{C_1}}\cM^{\gp}_{C_1}$ over $\ul{C}_1$.

{
We define the structure homomorphism $\alpha_{{C^{\circ}_1}}: \cM_{C^{\circ}_1}
\to \cO_{C_1}$ by $\alpha_{{C^{\circ}_1}}|_{\cM_{C_1}} = \alpha_{{C_1}}$ and
$\alpha_{C^\circ_1}|_{\cM_{C^\circ_1}\setminus\cM_{C_1}}=0$. The same argument
as in the proof of Proposition~\ref{prop:puncurve-moduli}, Step~3, shows that
this defines a logarithmic structure $\cM_{C^{\circ}_1}$, hence the desired
punctured curve $C^{\circ}_1 \to Z_1$.}

The remainder of the proof is now identical to that of \cite[Prop.\,3.17]{LogGW}.
\end{proof}

\subsubsection{Finiteness of the combinatorial data}
\sloppy
In order to complete the proof that $\scrM(X/B,\beta)$ is finite type, it
remains to bound the combinatorial data. 

\fussy
\begin{proposition}
\label{prop:refinement-combinatorially-finite}
Suppose $\ocM_X^\gp\otimes_\ZZ\QQ$ is generated by its global sections. Then any
class {of punctured map} $\beta$ is combinatorially finite.
\end{proposition}

\begin{proof}
{Arguing stratawise as in \cite[\S3.2]{LogGW},} it is sufficient to show that for any combinatorially constant family of
ordinary stable maps $(\ul{C}/\ul{W},\bp,\ul{f})$ in the sense of
\cite[Def.\,3.15]{LogGW}, there are only finitely many combinatorial types of
liftings of such a family to a punctured log curve of type $\beta$. {Since types are constant along a combinatorially constant family, we may further assume that $\ul W$ is the spectrum of a field. Finiteness of the number of types of a logarithmic stable map with a given underlying stable map over a field with fixed contact orders $u_p$ is proved in \cite[Thm.\,3.9]{LogGW}. 

One small difference in our setup concerns the definition of contact orders. In
\cite{LogGW} these were given by a sheaf homomorphism $\ocM_Z\to\NN$, hence were
fixed at $\ul f(p)$ by the underlying ordinary stable map and the contact orders
$u_p$. In contrast, a global contact order may give an infinite set of maps
$\ocM_{X,\ul f(p)}\arr\ZZ$. The argument is saved under the assumption that
$\ocM_X^\gp\otimes_\ZZ\QQ$ is generated by its global sections:} The injectivity
statement in Proposition~\ref{Eqn: Sigma -> RR^r} implies that there is at most
one local representative of $u_p$.
\end{proof}

\subsection{Valuative criterion}
\label{ss:stable reduction}
We now show stable reduction for basic stable punctured maps, which allows us to
conclude properness of the moduli spaces of such maps. Recall that for a given
class $\beta = (g, \bar\bu, A)$ of stable punctured maps to $X \to B$,
we have the class $\ul{\beta} = (g, k, A)$ for ordinary stable maps to $\ul{X}
\to \ul{B}$ by removing contact orders. We will show that

\begin{theorem}
\label{stablereduction}
{Assume that the log structure on $X$ is defined in the Zariski
topology. Then} the tautological morphism removing all logarithmic
structures
\[
\scrM(X/B, \beta) \to \scrM(\ul{X}/\ul{B}, \ul{\beta})
\]
satisfies the valuative criterion for properness.
\end{theorem}

\begin{proof}
In what follows, we assume given $R$ a discrete valuation ring over $\ul{B}$
with maximal ideal $\fom$, residue field $\kappa=R/\fom$, and fraction field
$K$. Suppose we have a commutative square of solid arrows of the underlying
stacks:
\[
\xymatrix{
\spec K \ar[r] \ar[d] & \scrM(X/B, \beta) \ar[d] \\
\spec R \ar[r] \ar@{-->}[ru]& \scrM(\ul{X}/\ul{B}, \ul{\beta}).
}
\]
We want to show that
there is a dashed arrow making the above diagram commutative, which is unique up
to a unique isomorphism.

The top arrow of the above diagram yields a stable punctured map
\[
(\pi_K:C_K^\circ\rightarrow \Spec (Q_K\rightarrow K),  \bp_K, f_K)
\]
over the logarithmic point $\Spec (Q_K\rightarrow K)$. The bottom arrow of the
above diagram yields an ordinary stable map $(\ul{C}/\Spec R, \bp,\ul{f})$ with
its generic fiber given by the underlying stable map of $f_K$. To construct the
dashed arrow, it suffices to extend the stable punctured map $f_K$ across the
closed point $0 \in \spec R$ with the given underlying stable map $\ul{f}$. The
task is to then extend the logarithmic structures and morphisms thereof. The
proof is almost identical to that of \cite[Thm.\,4.1]{LogGW}. Since that proof
is quite long, we only note the salient differences.
\smallskip

Section 4.1 of \cite{LogGW} accomplishes this extension at the level of ghost
sheaves; in particular, \cite[Prop.\,4.3]{LogGW}, which states that the type of
the central fiber is uniquely determined by {the {stable log} map
on the generic fiber}, carries through with $u_p$ for a puncture $p$ determined
as for marked points. Indeed, if $p$ is a punctured point on $\ul{C}_0$ in the
closure of the punctured point $p_K$ on $\ul{C}_K$, then $u_p$ must be the
composition 
\begin{equation}
\label{upinduced}
P_p \longrightarrow P_{p_K} \stackrel{u_{p_K}}{\longrightarrow}\ZZ,
\end{equation}
where the first map is the generization map $(\ul{f}^*\ocM_X)_p \rightarrow
(\ul{f}^*\ocM_X)_{p_K}$. In particular, the contact orders $u_p$ and $u_{p_K}$
both have global contact order as specified in $\beta$.

By Proposition~\ref{Prop: Basic monoid}, the type of the central fiber then
determines the extension $\overline{\cM}_{C^{\circ}}$ of
$\overline{\cM}_{C_K^{\circ}}$ and a map $\bar f^{\flat}:
\ul{f}^*\overline{\cM}_X\rightarrow \overline{\cM}_{C^{\circ}}$ extending the
corresponding map on the generic fiber. Here $\overline{\cM}_{C^{\circ}}$ is
defined at punctures by pre-stability via Corollary~\ref{cor:pre-stable}.
\smallskip

Next, \cite[\S4.2]{LogGW}\footnote{{We take the opportunity to correct
an error, pointed out by the referee of the current paper, in the first
paragraph of \cite[\S4.2]{LogGW}. Two descriptions of a set
$U(\eta)$ are given. The first description, as the set of generizations
of points in $A$, the set of non-special points in $\ul{C}_0$, is not
correct (it is not necessarily an open set). Thus the reader should
rely only on the second description of the set $U(\eta)$.}} shows that the logarithmic structure on the base
$\Spec R$ is uniquely defined. In this argument, marked points play no role, and
the argument remains unchanged in the punctured case. In particular, this
produces a unique choice of logarithmic structure $\cM_R$ on $\Spec R$, which in
addition comes with a morphism of logarithmic structures $\cM_R^0\rightarrow
\cM_R$ where $\cM_R^0$ is the basic logarithmic structure (pulled back from the
moduli space of pre-stable curves $\bM$ with its basic logarithmic
structure, see \cite[App.\,A]{LogGW}) associated to the family $\ul{C}\rightarrow
\Spec R$. In particular, one obtains a logarithmic structure $(\ul{C},\cM_C')=
(\Spec R, \cM_R)\times_{(\Spec R,\cM_R^0)} (\ul{C}, \cM_C^0)$, where $\cM_C^0$
is the logarithmic structure pulled back from the basic logarithmic structure of
the universal curve over $\scrM(\ul{X}/\ul{B}, \ul{\beta})$. The logarithmic
structure $\cM_C'$ then has logarithmic marked points along the punctures $p$,
but there is a sub-logarithmic structure $\cM_C\subset \cM_C'$ which only
differs in that {we remove the marked points, that is, we make $(\ul
C,\cM_C)\arr (\Spec R,\cM_R)$ strict away from the nodes.}

By Corollary~\ref{cor:pre-stable}, there is a natural inclusion
$\ocM_{C^{\circ}} \subset (\ocM_C')^{\gp}$. We form $\cM_{C^{\circ}}:=
\ocM_{C^{\circ}} \times_{(\ocM_C')^{\gp}} (\cM_C')^{\gp}$ and define a structure
homomorphism $\alpha_{C^{\circ}}:\cM_{C^{\circ}} \rightarrow \cO_C$ by
$\alpha_{C^{\circ}}|_{{\cM_{C'}}}=\alpha_{C'}$ and $\alpha_{C^{\circ}}
(\cM_{C^{\circ}}\setminus\cM_C')=0$, as in Proposition
\ref{prop:puncurve-moduli}, Step~3. To show that this is a homomorphism, it is
enough to show that if $s\in\cM_{C^{\circ},p} \setminus\cM_{C,p}'$, writing
$s=(s_1,s_2)$ as a stalk of $\cM_C\oplus_{\cO_C^{\times}} \cP^{\gp}$, then
$\alpha_C(s_1)=0$. But necessarily 
{$(\ol s_1,\ol s_2)=\bar f^{\flat}(m)+
(\ol s_1',\ol s_2')$ for
some $m\in P_p$ with $u_p(m)<0$ and $(\ol s_1',\ol s_2')\in\ol\cM_{C,p}\oplus
\NN$.}
Write for points $x,x'\in \ul{C}$ with $x$ in
the closure of $x'$ the generization map $\chi_{x',x}:P_x\rightarrow P_{x'}$.
Then $u_{p_K}(\chi_{p_K,p}(m))=u_p(m)$ by \eqref{upinduced}. Thus
$u_{p_K}(\chi_{p_K,p}(m))<0$ and necessarily $\alpha_{C_K}(s_1|_{C_K})=0$. But
since $C$ is reduced and $C_K$ is dense in $C$, this implies $\alpha_C(s_1)=0$,
as desired. Thus we have a punctured log scheme $C^{\circ}$.

We can now extend $f^\flat_K:f_K^*\cM_X\rightarrow \cM_{C^{\circ}_K}$ to
$f^{\flat}:f^*\cM_X\rightarrow\cM_{C^{\circ}}$ as in \cite[\S4.3]{LogGW}.
\end{proof}

\begin{corollary}
\label{Cor: scrM/B is proper}
Let $\btau=(G,\bg,\bsigma,\bar\bu,\bA)$ be a decorated global type of punctured
maps (Definition~\ref{Def: global type}) and assume $X\arr B$ is projective,
{the log structure on $X$ is Zariski,} and $\ocM_X^\gp\otimes_\ZZ\QQ$ is
globally generated.\footnote{{{Again, the latter assumption
has been removed by \cite{Johnston}.}}} Then $\scrM(X/B,\btau)\arr B$ is
proper. In particular, $\scrM(X/B,\beta)$ is proper for any $\beta=(g,\bar\bu,
A)$.
\end{corollary}

\begin{proof}
Theorem~\ref{thm:boundedness} shows that $\scrM(X/B,\beta)\arr B$ is of finite
type. Properness for $\btau=\beta$ now follows from the valuative criterion
verified in Theorem~\ref{stablereduction}.

For general $\btau$, the proof of \cite[Prop.\,2.34]{decomposition} generalizes
to the present punctured setup to exhibit $\scrM(X/B,\btau)$ as a
closed substack of the base change of $\scrM(X/B,\beta)$ by the finite map
$\bM(G,\bg)\arr\bM$.
\end{proof}


\subsection{Idealized smoothness of \texorpdfstring{$\fM(\cX/B,\tau)\rightarrow
B$}{M(X/B,tau)->B}}
\label{ss:idealized smoothness}
{For simplicity of presentation, we restrict to $X$ simple throughout
this section. Thus for any $\sigma,\tau\in\Sigma(X)$ there is at most one arrow
$\sigma\arr\tau$ in $\Sigma(X)$.}

\subsubsection{Marking log-ideals}
Let $\tau=(G,\bg,\bsigma,\bar\bu)$ be a global type of
punctured maps. Recall from the discussion after Definition~\ref{Def: nodal
ideal sheaf} that the moduli stack $\bM(G,\bg)$ of $(G,\bg)$-marked pre-stable
curves with its nodal log ideal sheaf is idealized logarithmically smooth over
the trivial log point $\Spec\kk$. A similar result holds for our moduli spaces
$\fM(\cX/B,\tau)$. To introduce the idealized structure let $(\pi:C^\circ\to
W,\bp,f)$ be a $\tau$-marked basic punctured map and let {$\ol w$ of
$\ul W$} be a geometric point. Let $\tau_{\ol w}=(G_{\ol w},{\bf g}_{\ol w},
\bsigma_{\ol w}, \bu_{\ol w})$ be the type of the punctured map over $\ol w$,
equipped with its marking contraction morphism $\phi:\tau_{\ol
w}\rightarrow\tau$ (Definitions~\ref{Def: type} and \ref{Def: marking by
type},(2)), with set of contracted edges $E_\phi$. For the sake of
Definition~\ref{Def: Itau} below, we introduce the following notation. For $x\in
V(G_{\ol w})\cup L(G_{\ol w})\cup \big(E(G_{\ol w})\setminus E_\phi\big)$ the
face inclusion $\bsigma(\phi(x))\arr \bsigma_{\ol w}(x)$ is dual to a
localization map
\[
\chi_x: P_x\arr P_{\phi(x)}
\]
of stalks of $\ocM_X$. We also have homomorphisms
\[
\varphi_x: P_x\arr \ocM_{C^\circ,x},\quad u_x:P_x\arr \ZZ
\]
defined by $\bar f^\flat_{\ol w}$ and by the contact order $\bu_{\ol w}$.
{For uniformity of notation} we define $u_x=0$ for $x\in V(G_{\ol w})$.
Moreover, by Definition~\ref{Def: contact order} of contact order,
$\varphi_x(u_x^{-1}(0))\subseteq \ocM_{C^\circ,x}$ is contained in the image of
$\bar\pi^\flat_x: \ocM_{W,\ol w}\arr \ocM_{C^\circ,x}$. For the following
definition recall also the homomorphism $\chi_{\tau\tau_{\ol w}}: Q_{\tau_{\ol
w}}\arr Q_{\tau\tau_{\ol w}}$ from \eqref{Eqn: localization of basic monoids}.

\begin{definition}
\label{Def: Itau}
The \emph{$\tau$-marking ideal {$\ocI^\tau_W$}} of the $\tau$-marked
basic punctured map $(\pi:C\to W,\bp,f)$ is the sheaf of ideals in $\ocM_W$ with
stalk at the geometric point {$\ol w$ of $\ul W$} generated by the following
subsets:
\begin{enumerate}
\item[(i)] (Target stratum generators)
the preimage under $\bar\pi^\flat_x$ of $\varphi_x\big(P_x\setminus
\chi_x^{-1}(0)\big)$ for $x\in V(G_{\ol w})\cup L(G_{\ol w})\cup
\big(E(G_{\ol w})\setminus E_\phi\big)$;
\item[(ii)] (Nodal generators)
the nodal generators $\rho_E\in \ocM_{W,\ol w}=Q_{\tau_{\ol w}}$ for $E\in
E(G_{\ol w})\setminus E_\phi$;
\item[(iii)] (Basic monoid generators)
$\chi_{\tau \tau_{\ol w}}^{-1}(Q_{\tau\tau_{\ol w}}\setminus\{0\})$.
\end{enumerate}
\end{definition}

The collection of stalks $\ocI_{W,\ol w}^\tau\subset \ocM_{W,\ol w}$ in
Definition~\ref{Def: Itau} form a coherent ideal $\ocI_{W}^\tau\subset
\ocM_{W}$. {Indeed, we obtain a sheaf by the method of Remark~\ref{Rem:
Gap in GS2}, and, as $W$ is fine and saturated, we may apply Lemma
\ref{lem:coherent ideal}, noting that all generating sets are compatible with
generization.} As usual, we also refer to the preimage $\cI_W\subset \cM_W$ of
$\ocI_W$ under $\cM_W\arr \ocM_W$ as the \emph{$\tau$-marking ideal}. Without
the generators specified in (iii) we speak of the \emph{weak $\tau$-marking
ideal}.

\subsubsection{The base of a punctured map is idealized by the marking log-ideal}
The $\tau$-marking ideal defines an idealized log structure on base spaces of
$\tau$-marked punctured maps as follows.

\begin{lemma}
\label{Prop: log ideal for type}
Let $({C^\circ}/W,\bp,f)$ be a $\tau$-marked basic punctured map. Then the
$\tau$-marking ideal $\cI_W\subset\cM_W$ maps to $0$ under the structure
homomorphism $\cM_W\arr \cO_W$.
\end{lemma}

\begin{proof}
It is enough to show that any lift $s\in \cM_{W,\ol w}$ of an element of
one of the generating sets satisfies $\alpha_W(s)=0$. This holds for elements
described in (iii) of Definition~\ref{Def: Itau} by Definition~\ref{Def: marking
by type},(3). 

Similarly, Definition~\ref{Def: marking by type},(1) guarantees the required
vanishing for elements described in (i) of Definition~\ref{Def: Itau}. Indeed,
consider first the case of $x=v\in V(G_{\ol w})$, where we defined $u_v=0$.
Then $u_v^{-1}(0) \setminus\chi_v^{-1}(0) = P_v\setminus\chi_v^{-1}(0)$, and
$\alpha_X(P_v\setminus\chi_v^{-1}(0))\subset \cO_X$ locally generates the ideal
$\cI_{X_{\bsigma(\phi(v))}} \subset \cO_X$ of the stratum $X_{\bsigma(\phi(v))}$
in $X$. Thus, the condition that the restriction of $\ul{f}$ to the closed
subscheme of $\ul{C}$ corresponding to $\phi(v)$ factors through
$\ul{X}_{\bsigma(\phi(v))}$ implies the desired vanishing in this case. A
similar argument works for legs and edges.

Finally, the lift to $\cM_{W,\ol w}$ of a nodal generator $\rho_E\in
\ocM_{W,\ol w}$ lies in the nodal log-ideal (Definition~\ref{Def: nodal ideal
sheaf}) of the $(G,{\bf g})$-marked curve $C/W$, which maps to zero in $\cO_W$
by Proposition~\ref{prop:puncurve-moduli},(2).
\end{proof}

\begin{remark}
\label{Rem: on marking ideals}
Omitting the last set (iii) of generators in Definition~\ref{Def: Itau}
leads to the idealized structure for moduli spaces of \emph{weakly} marked
punctured maps (Definition~\ref{Def: marking by type}).
\end{remark}

As shown in Proposition~\ref{prop:puncturing-ideal-vanish}, the base $W$ is also
idealized by the puncturing log ideal $\cK$. It is therefore natural to combine
the two.

\begin{definition}
We call the union $\cI^\tau\cup\cK$ of the $\tau$-marking and the puncturing log
ideals the \emph{canonical idealized structure} on our $\tau$-marked moduli
spaces such as $\fM(\cX/B,\tau)$.
\end{definition}

\subsubsection{The realizable case}
While the definition of the $\tau$-marking ideal may seem complicated, in fact
in the case we most frequently need it, namely the realizable case, the
canonical idealized structure has a simpler description: By
Lemma~\ref{Lem: unique type for global type} there is a unique lift to a type,
and the associated basic monoid already knows about marked strata,
non-deforming nodes and punctures.

\begin{proposition} 
\label{prop:I when tau is realizable}
If $\tau$ is a realizable global type, then $\ocI^{\tau}_{W,\ol w}
+\ocK_{W,\ol w}$ with $\ocK_W$ the puncturing log ideal
(Definition~\ref{Def: puncturing log ideal punctured map}) is {given} 
by the
set $(iii)$ in Definition~\ref{Def: Itau}.
\end{proposition}

\begin{proof}
Denote by $\chi: Q_{\tau_{\ol w}}\arr Q_{\tau\tau_{\ol w}}$ the localization
homomorphism from \eqref{Eqn: localization of basic monoids} defined by the
$\tau$-marking of $(C^\circ/W,\bp,f)$. By Lemma~\ref{Lem: unique type for
global type} there is a unique type of punctured map with associated global type
$\tau$.
Hence in particular $Q_{\tau\tau_{\ol w}}$ agrees with the basic monoid
for a tropical punctured map of this type and does not depend on $\ol w$. We
write this basic monoid as $Q_{\tau}$. Denote by $R\subset Q_{\tau_{\ol w}}$
the ideal $\chi^{-1}(Q_\tau\setminus\{0\})$.

We need to show that $R$ contains the elements listed in (i) and (ii) of
Definition~\ref{Def: Itau} as well as generators of the puncturing log ideal
stated in Definition~\ref{Def: Puncturing log ideal}. Adopting the notation
given {in Definition~\ref{Def: Itau}}, for $v\in V(G_{\ol w})$ we have a
commutative diagram
\[
\xymatrix@C=30pt
{
P_v\ar[d]_{\chi_v} \ar[r]^{\varphi_v} & Q_{\tau_{\ol w}}\ar[d]^{\chi}\\
P_{\phi(v)}\ar[r]_{\varphi_{\phi(v)}} & Q_{\tau}
}
\]
The fact that $\tau$ is realizable implies that $\varphi_{\phi(v)}$ is a local
homomorphism, i.e., $\varphi_{\phi(v)}^{-1}(0)=\{0\}$. Indeed, dually, the map
$Q_{\tau}^{\vee} \rightarrow P_{\phi(v)}^{\vee}$ is given by evaluation of the
tropical map at the vertex $v$, and realizability implies the image of this map
intersects the interior of $P_{\phi(v)}^{\vee}$. This is equivalent to the local
homomorphism statement. But this implies that $\varphi_v(P_v\setminus
\chi_v^{-1}(0)) \subseteq \chi^{-1}(Q_{\tau}\setminus \{0\})=R$.

In the case of a leg $L$, we similarly have a diagram
\[
\xymatrix@C=30pt
{
P_L\ar[d]_{\chi_L} \ar[r]^>>>>>{\varphi_L} & Q_{\tau_{\ol w}}\oplus\ZZ
\ar[d]^{\chi\oplus\id}
\ar[r]^>>>>>{\pr_1}&Q_{\tau_{\ol w}}\ar[d]^{\chi}\\
P_{\phi(L)}\ar[r]_{\varphi_{\phi(L)}} & Q_{\tau}\oplus\ZZ\ar[r]_>>>>>{\pr_1}&
Q_{\tau}
}
\]
Again, $\varphi_{\phi(L)}$ is necessarily local by realizability. Note that,
with $\iota:{Q_{\tau_{\ol w}}}\rightarrow {Q_{\tau_{\ol
w}}}\oplus \ZZ$ given by $m\mapsto (m,0)$,
\[
\iota^{-1} 
\left(\varphi_L\big(P_L\setminus \chi_L^{-1}(0)\big)\right) = 
\iota^{-1} \left(\varphi_L\big(u_L^{-1}(0)\setminus 
\chi_L^{-1}(0)\big)\right)=
\pr_1\circ\varphi_L\big(u_L^{-1}(0)\setminus \chi_L^{-1}(0)\big).
\]
Thus $\pr_1\circ\varphi_L\big(u_L^{-1}(0)\setminus \chi_L^{-1}(0)\big) \subseteq
\chi^{-1}(Q_{\tau}\setminus \{0\})=R$, as desired. In fact we obtain more from
this. If instead $p\in P_L$ with $u_L(p)<0$, then $\pr_1(\varphi_L(p))$ is a
generator of $\ocK_{W,\ol w}$, and $\chi(\pr_1(\varphi_L(p)))$ is a generator
of the puncturing ideal for the type $\tau$. But as the type is realizable, this
ideal does not contain $0$. Thus $\pr_1(\varphi_L(p))\in R$, so $\ocK_{W,\ol w}
\subseteq R$.

For an edge $E\in V(G)$, the argument that $\phi_E\big(u_E^{-1}(0)\setminus
\chi_E^{-1}(0)\big) \subseteq R$ is similar and we leave the details to the
reader. Finally, for the corresponding nodal generator $\rho_E\in Q_{\tau_{\ol
w}}$ from Definition~\ref{Def: Itau},(ii), observe that $\chi(\rho_E)$ is the
edge length function of the edge $E$. Again, since $\tau$ is realizable,
$\chi(\rho_E)\not=0$ and $\rho_E\in R$.
\end{proof}

\subsubsection{The stacks are idealized log smooth}

\begin{theorem}
\label{thm:idealized-etale}
{Assume that $X$ is simple. Then} the forgetful morphisms
\[
\fM(\cX/B,\tau) \arr \bM(G,\bg)\times B
\]
remembering only the domain curve as a family of marked curves over $B$, is
idealized logarithmically \'etale for the canonical idealized structures. An
analogous result holds for $\tau$ replaced by a decorated global type
$\btau=(\tau,\bA)$ of a punctured map, and for weak markings.
\end{theorem}

\begin{proof}
\textsc{Step~1. Lifting to the stack of punctured curves.} 
We first note that the morphism in question is in fact idealized. Indeed, the
generators of the nodal log-ideal (Definition~\ref{Def: nodal ideal sheaf}) on
$\bM(G,\bg)\times B$ are pulled back to the nodal generator $\rho_E$
of Definition~\ref{Def: Itau},(ii) for $E\in E(G)$. The morphism then factors
over the idealized logarithmically \'etale morphism
\[
\breve\fM_B(G,\bg)\arr \fM_B(G,\bg)=\Log_{\bM(G,\bg)\times B}
\]
from Proposition~\ref{prop:puncurve-moduli},(2). Moreover, by
\cite[Thm.\,4.6,(iii)]{LogStack}, the morphism
\[
\Log_{\bM(G,\bg)\times B}\arr \bM(G,\bg)\times B
\]
is also logarithmically \'etale. It thus suffices to prove the statement with
$\bM(G,\bg)\times B$ replaced by the stack $\breve\fM_B(G,\bg)$ of
$(G,\bg)$-marked punctured curves. Note that the morphism $\fM(\cX/B,\tau)\arr
\breve\fM_B(G,\bg)$ is strict, but not in general idealized strict: {the
{nodal log-}ideal of {$\breve\fM_B(G,\bg)$} from
{Definition~\ref{Def: nodal ideal sheaf}} involves {only the
nodes of the domain curves}, whereas the {$\tau$-marking} ideal of
$\fM(\cX/B,\tau)$ from Definition~\ref{Def: Itau}, in particular part (i),
{also} records target data}.\\[1ex]
\textsc{Step~2. Lifting to the prestable map.}
According to the definition of idealized log \'etale, it is sufficient to
consider a diagram of solid arrows in the category of idealized log
spaces
\begin{equation}
\label{diag:universal-stack-idealize-etale}
\vcenter{\xymatrix{
T_0 \ar[rr]^{g_0} \ar[d] && \fM(\cX/B,\tau)\ar[d] \\
T \ar[rr]_g \ar@{-->}[urr] && \breve\fM_B(G,\bg)
}}
\end{equation}
where $T_0\hookrightarrow T$ is an idealized strict closed embedding defined by
a square-zero ideal. Denote by $\cK_{T_0}$ and $\cK_T$ the log-ideals of $T_0$
and $T$ respectively. We wish to show that there is a unique dashed arrow making
the above diagram commutative.

Denote by $f_{T_0}:C_{T_0}^{\circ}\rightarrow\cX$ the punctured map over $T_0$
corresponding to the morphism $g_0$, and by $C_{T_0}^{\circ}\hookrightarrow
C_T^\circ$ the extension given by $g$. Write also $\pi_{T_0}:C_{T_0}^{\circ}
\rightarrow T_0$, $\pi_T:C_T^{\circ}\rightarrow T$. Thus the lifting problem
\eqref{diag:universal-stack-idealize-etale} reduces to the following: 
\[
\xymatrix{
C^{\circ}_{T_0} \ar[rr]^{f_{T_0}} \ar[d] && \cX \ar[d] \\
C^{\circ}_{T} \ar[rr] \ar@{-->}[rru]^{f_T} && B
 }
\]
Since $\cX \to B$ is log \'etale, by the infinitesimal lifting property of log
\'etale morphisms in the category of idealized log schemes 
\cite[{p.399}]{Ogus},
such $f_T$ exists and is unique.
\smallskip

{It remains} to check that $f_T$ is also a $\tau$-marked curve. Item~(2)
of Definition~\ref{Def: marking by type} is automatic as $T_0$ and $T$ have the
same geometric points. {As a preparation for establishing~(1) and (3),
we first check the vanishing of the $\tau$-marking ideal.}\\[1ex]
\textsc{Step~3. The marking ideal vanishes.}
Fix a geometric point $\ol t$ of $T_0$. Let $I^{\tau}_0\subseteq \cM_{T_0,\ol
t}$ be the stalk of the log-ideal $g_0^{\bullet}\cI^{\tau}_{\fM(\cX/B,\tau)}$ at
$\ol t$, and write $\overline{I}_0^{\tau}\subseteq \overline{\cM}_{T_0,\ol t}$
for its image. As $\overline{\cM}_{T_0,\ol t}=\overline{\cM}_{T,\ol t}$, we also
obtain an ideal $I^{\tau}\subseteq \cM_{T,\ol t}$ as the inverse image of
$\overline{I}_0^{\tau}$ under the map $\cM_{T,\ol t}\rightarrow
\overline\cM_{T,\ol t}$. As $g_0$ is idealized, necessarily $I^{\tau}_0
\subseteq \cK_{T_0,\ol t}$. Since $T_0\rightarrow T$ is {idealized
strict}, we thus have $I^{\tau}\subseteq \cK_{T,\ol t}$ and hence
$\alpha_T(I^{\tau})=0$. {This finishes Step~3.}
\smallskip

Now let $x \in V(G)\cup E(G)\cup L(G)$, and let $Z\subseteq C_T$ be the
corresponding closed subscheme. To verify condition (1) of Definition~\ref{Def:
marking by type}, we need to show that $f_T|_{Z}$ factors through
$\cX_{\bsigma(x)}$. Let $\ol w= g_0(\ol t)$, with corresponding type of
tropical curve $\tau_{\ol w}$, equipped with a contraction morphism
$\phi:\tau_{\ol w}\rightarrow \tau$. {We now check the
needed factorization for each kind of $x$ in the following steps.}
\\[1ex]
\textsc{Step~4. The marking lifts at a vertex.}
First consider the case that $x$ is a vertex. In this case $Z$ is a sub-curve of
$\ul{C}_T$, flat over $\ul T$. Let $U\subseteq Z$ be the open subset of
non-special points; it is then sufficient to show that $\ul{f}_T|_U$ factors
through the closed substack $\cX_{\bsigma(x)}$. So let $\ol u$ be a geometric
point of $U$ lying over $\ol t$, contained in an irreducible component of
$Z_{\ol t}$ indexed by a vertex $v\in V(G_{\tau_{\ol w}})$. Note then that
$\phi(v)=x$. It is enough to show that $f_T^\sharp:\cO_{\cX,\ul f_T(\ol u)}
\rightarrow \cO_{C_T,\ol u}$ takes the stalk $\cJ_{\ul f_T(\ol u)}$ of the ideal
$\cJ$ of $\cX_{\bsigma(x)}$ in $\cX$ to $0$. Using the notation of
Definition~\ref{Def: Itau}, we have $\overline\cM_{\cX,\ul f_T(\ol u)}=P_v$ and
a generization map $\chi_v:P_v\rightarrow P_{\phi(v)}$. If $p \in P_v$, write
$s_p\in \cM_{\cX,\ul f_T(\ol u)}$ for a lift of $p$. {We next observe
that since $B$ is a log point or is log smooth over $\Spec\kk$ and $X$ is
simple, the ideal} $\cJ_{\ul f_T(\ol u)}$ is generated by the set
$\{\alpha_{\cX}(s_p)\,|\, p\in P_v\setminus \chi_v^{-1}(0)\}$. {Indeed,
this is the idealized smoothness statement of the strata in
Proposition~\ref{Prop: idealized log smooth}, applied on a smooth chart of
$\cX$, together with the stalkwise characterization \eqref{Eqn: ocK stalkwise}
of the log ideal $\cK$ in the proof of that proposition. Note that due to
simplicity, the only face map is $\chi_v^t:\bsigma(x)\arr P_\RR^\vee$ in the
present case, and hence $\ocK_{\ul f_T(\ol
u)}=P_v\setminus\chi_v^{-1}(0)$.}

{Now} by
Definition~\ref{Def: Itau},(i) and strictness of $\pi_T$ at $\ol u$,
{for each $p\in P_v$} there exists $s'_p\in I^{\tau}\subseteq \cM_{T,\ol
t}$ with $f_T^{\flat}(s_p)= h \cdot \pi_T^{\flat}(s'_p)$ for some $h\in
\cO_{C_T,\ol u}^{\times}$. Thus
\[
f_T^\sharp(\alpha_{\cX}(s_p))=\alpha_{C_T}(f_T^{\flat}(s_p))
=\alpha_{C_T}(h \cdot \pi_T^{\flat}(s'_p))=
h\cdot \pi_T^\sharp(\alpha_T(s'_p))= 0.
\]
This shows that $\ul{f}_T|_{U}$ factors through $\cX_{\bsigma(x)}$.
\\[1ex]
\textsc{Step~5. The marking lifts at a leg.}
Second consider the case that $x=L\in L(G)$. In this case $Z$ is the image of a
section of $\pi_T$, {with $\ul{Z}\cong\ul{T}$}. Let $\ol u$ be the unique
geometric point of $\ul Z$ over $\ol t$. We now have a generization map
$\chi_L:P_L= \overline{\cM}_{\cX,\ul f_T(\ol u)}\rightarrow P_{\phi(L)}$.
Following the same notation as in the previous paragraph, it is then sufficient
to show that for each $p\in P_L\setminus\chi_L^{-1}(0)$, we have
$0=\alpha_{C_T}(f_T^\flat(s_p))|_{{Z}}\in \cO_{{Z},\ol u}$. As in
the previous paragraph, this is forced by the generators of the puncturing ideal
in Definition~\ref{Def: Itau},(i) in case $u_L(p)=0$. If $u_L(p)>0$, then
$\alpha_{C_T}(f_T^{\flat}(s_p))$ contains a positive power of the defining
equation of {$Z$ as a subscheme of $C_T$}, and hence vanishes along
{$Z$}. If $u_L(p)<0$, then we achieve vanishing by
Definition~\ref{def:puncturing},(2). Thus we obtain the desired vanishing.
\\[1ex]
\textsc{Step~6. The marking lifts at an edge.}
The third case is $x=E\in E(G)$. The argument is similar to the second
case, and we leave the details to the reader. This verifies that
$f_T$ satisfies condition (1) of $\tau$-marked curve.
\\[1ex]
\textsc{Step~7. Base marking, decoration and weak marking.}
Finally, condition (3) holds. {Indeed, the generators in
Definition~\ref{Def: Itau},(iii) guarantee the desired
vanishing.}

This completes the proof for markings by $\tau$. The proof for $\tau$ replaced
by $\btau$ is identical. The weakly marked case is obtained by the same proof
omitting (iii) in Definition~\ref{Def: Itau}.
\end{proof}

{
\begin{remark}
\label{Rem: idealized smoothness in non-simple case}
The proof in the weakly marked case uses simplicity only when arguing that the
ideal defining $\cX_{\bsigma(x)}$ locally is generated by expressions
$\alpha_\cX(s_p)$, $p\in P_v\setminus\chi_v^{-1}(0)$ for the unique generization
map $\chi_v: P_v\to P_{\phi(v)}$, $P_v=\ocM_{\ul f_T(\ol u)}$. In general
there is still always a log ideal $K\subseteq \ocM_{\ul f_T(\ol u)}$ with this
property, as we saw in the proof of Proposition~\ref{Prop: idealized log
smooth}. This larger log ideal can be accounted for by modifying
Definition~\ref{Def: Itau},(i) accordingly. In the marked case, we also need to
refine $Q_{\tau\tau_{\ol w}}$ in Definition~\ref{Def: Itau},(iii) to the version
stated in \eqref{Eqn: I_tautau'} in \S\ref{sss: basic monoid with monodromy}.

Thus we expect the statement of Theorem~\ref{thm:idealized-etale} to hold true
in the non-simple case with these adjustments. Details are left to the
interested reader.
\end{remark}}

\begin{remark}[Local structure of stacks of prestable maps]
\label{Rem: local structure of fM(cX/B,tau)}
Theorem~\ref{thm:idealized-etale} gives the following local description of
$\fM(\cX/B,\tau)$. Let $(C^\circ/W,\bp,f)$ be a basic stable punctured map over
a log point $W=\Spec(Q\arr \kappa)$ over $B$ marked by $\tau$. Denote by $s$ the
number of edges of the graph $G$ given by $\tau=(G,\bg,\bsigma,\bar\bu)$ and
assume that $\ul C$ has $s+r$ nodes. Thus $r$ nodes of $\ul C$ can be smoothed
while keeping a marking by $(G,\bg)$. 

The underlying object $(\ul C/\ul W,\bp)$, viewed as a pre-stable curve with its
basic log structure, is a point $\Spec \kappa \to \bM(G,\bg)\times B$.

By the deformation theory of nodal curves,
there exists a strict smooth neighborhood of this point
\'etale locally isomorphic to
\begin{equation}
\label{Eqn: smooth model for bM(G,bg)}
\AA^r\times U\times B.
\end{equation}
Here $\AA^r$ is endowed with the idealized log structure obtained by restricting
the toric log structure of $\AA^{s+r}$ to an intersection of $s$ coordinate
hyperplanes, and corresponds to deforming the $r$ smoothable nodes; $U$ is
smooth with trivial log structure corresponding to equisingular deformations of
$\ul C$; and the \'etale local isomorphism is a product of an \'etale local
isomorphism of $\AA^r\times U$ with an open substack of $\bM(G,\bg)$ and
$\id_B$. 

Note that the image of $(C^\circ/W,\bp)$ in $\bM(G,\bg)$ is defined by the
underlying marked nodal curve $(\ul C/\ul W,\bp)$ endowed with its basic log
structure of marked nodal curves.

{
Consider the point $W \to \fM(\cX/B,\tau)$ corresponding to the object
$(C^\circ/W,\bp,f)$. Pulling back the neighborhood \eqref{Eqn: smooth model for
bM(G,bg)} along $\fM(\cX/B,\tau)\arr \bM(G,\bg)\times B$ gives a smooth
neighborhood $V$ of $W \to \fM(X/B,\tau)$ equipped with a
morphism $\phi:V\rightarrow \AA^r\times U\times B$. We may now apply
Proposition~\ref{prop:etale charts} to describe this neighborhood explicitly
\'etale locally, as follows.
We use the notation $\cA_P$ and $\cA_{P,I}$
defined in \eqref{eq:artin charts}, for $P$ a monoid and 
$I\subseteq P$ a monoid ideal.

The log-ideal $\cI^{\tau}_{\fM(\cX/B,\tau)}\cup\cK_{\fM(\cX/B,\tau)}$
induces a monoid ideal $I\subseteq Q$, as constructed in Definition~\ref{Def:
Itau},} {with associated idealized Artin fan {$\cA_{Q,I}$}.}
Let $Q_B$ be the stalk of $\overline{\cM}_B$ at the image point of the
composition $W\rightarrow \fM(\cX/B,\tau)\rightarrow B$. We may first replace
$B$ with an \'etale neighborhood of this image point and so assume given a map
$Q_B\rightarrow \overline\cM_B$, or equivalently a strict morphism $B\rightarrow
\cA_{Q_B}$. Then by Proposition~\ref{prop:etale charts}, possibly after passing
to an \'etale neighborhood of $V$, there is a diagram
\begin{equation}
\label{eq:local structure ogus}
\vcenter{\xymatrix@C=30pt
{
V\ar[rd]^{\theta}\ar@/_/[rdd]_{\phi}\ar@/^/[rrd]^{\psi}&&\\
&V'\ar[d]\ar[r]& \cA_{Q,I}\ar[d]^{\iota}\\
&\AA^r\times U \times B\ar[r]&\cA_{\NN^{s+r},J}\times \cA_{Q_B}
}}
\end{equation}
with the square Cartesian in {the log, fine and fs} categories, 
$\psi$ and both horizontal arrows
strict and {idealized strict}, and $\theta$ \'etale and strict. Further, $\iota$
is induced by the map on stalks of ghost sheaves $\NN^{s+r}\oplus Q_B\rightarrow
Q$ given by the morphism $\fM(\cX/B,\tau)\rightarrow {\bf M}(G,{\bf g})\times
B$. Finally, $J\subseteq \NN^{s+r}$ is the ideal generated by the first $s$
generators of $\NN^{s+r}$, so that the morphism $\AA^r\rightarrow
\cA_{\NN^{s+r},J}$ is strict and {idealized strict}.

In conclusion, we see that $V$ is
\'etale locally isomorphic to
\begin{equation}
\label{Eqn: local models for fM(cX,tau)}
V'\cong U\times ((\AA^r\times B)\times_{\cA_{\NN^{s+r},J}\times\cA_{Q_B}} 
\cA_{Q,I}).
\end{equation}
Thus the local models of $\fM(\cX/B,\tau)$ and their idealized structures 
are explicitly described from the types of tropical punctured maps admitting
a contraction morphism to $\tau$.
\end{remark}

\subsubsection{Dimension formulas}
Example~\ref{ex:non equi dim} exhibits a case where $\fM(\cX/B,\tau)$ is not
pure-dimensional. Before revisiting this example, we give a useful condition
which implies $\fM(\cX/B,\tau)$ is pure-dimensional, of the expected dimension.
{The statement involves a refinement of the notion of realizability of
global types from Definition~\ref{Def: global type},(2) relative to $B$.} 

\begin{definition}
\label{def:realizable over B}
Let $\tau$ be a global type of punctured map to $X$. We say that \emph{$\tau$
is realizable over $B$} if there exists a geometric point
{$\ol w$ of $\fM(\cX/B,\tau)$} such that the corresponding
punctured {map} has global type $\tau$.
\end{definition}

\begin{proposition}
\label{prop:realizable over B}
Suppose the Artin fan $\cA_X$ of $X$ is Zariski
{(Definition~\ref{Def: Zariski Artin fan})}. Then a global type
$\tau=(G,{\bf g},\bsigma,\bar{\bf u})$ is realisable over $B$ if and only if the
following conditions hold:
\begin{enumerate}
\item
$\tau$ is realizable, hence there is a universal family $h:\Gamma=\Gamma(G,\ell)
\rightarrow \Sigma(X)$ of type $\tau$, parameterized by
$\omega_{\tau}:=Q_{\tau,\RR}^{\vee}$, where $Q_{\tau}$ is the basic monoid for
tropical maps of type $\tau$.
\item
The universal family of tropical maps of type $\tau$ is defined
over $\Sigma(B)$, i.e., there is a map $\omega_{\tau}\rightarrow \Sigma(B)$
making the diagram
\[
\xymatrix@C=30pt
{
\Gamma\ar[r]^h\ar[d]&\Sigma(X)\ar[d]\\
\omega_{\tau}\ar[r]&\Sigma(B)
}
\]
commute.
\item
Let $\sigma\in \Sigma(B)$ be the minimal cone containing the
image of $\omega_{\tau}$. Then there exists a point $b\in B_{\sigma}$
such that $\sigma=\Hom(\overline\cM_{B,b}, \RR_{\ge 0})$.
\end{enumerate}
\end{proposition}

\begin{proof}
That conditions (1)--(3) are necessary is clear. Conversely, suppose (1)--(3)
hold. Let $\ul C/\Spec\kk$ be a pre-stable curve with dual intersection graph
$G$. Pull-back the basic log structure on $\ul C/\Spec\kk$ by the canonical
morphism $\NN^{|E(G)|}\arr Q_\tau$ from the nodal parameters to the basic monoid
for $\tau$ to define a log smooth curve $C/W$ over the log point
$W=\Spec(Q_\tau\arr\kk)$. We may then construct a morphism $W\rightarrow B$ with
image a point $b\in B_{\sigma}$ given by item (3) in the statement of the
proposition. Note we may take $b$ to be a closed point, so that $b=\Spec\kk$. At
the logarithmic level, this morphism can be taken so its induced tropicalization
is the given map $\omega_{\tau} \rightarrow \sigma$.

Next apply the correspondence \cite[Prop.\,2.10]{decomposition} (it is here we
need the hypothesis that $\cA_X$ is Zariski) between morphisms from a
logarithmic space to an Artin fan and their tropicalizations to first construct
a {saturated puncturing} $\widetilde C^\circ\arr C$ and then a
logarithmic map $\widetilde C^\circ\arr \cA_X$ with tropicalization of type
$\tau$. Prestabilizing then leads to a basic pre-stable punctured map
$(C^\circ/W,\bp,f)$ to $\cA_X$ of type $\tau$. {Note that $C^\circ$ is not
necessarily saturated.} On the other hand, we have a composed morphism
$C^{\circ}\rightarrow W\rightarrow B$, with $W\rightarrow B$ constructed in the
previous paragraph. The compositions $C^{\circ}\rightarrow \cA_X \rightarrow
\cA_B$ and $C^{\circ} \rightarrow B\rightarrow \cA_B$ agree by item (2) of the
proposition, and hence we obtain a punctured map $C^{\circ}\rightarrow
\cX=\cA_X\times_{\cA_B} B$ defined over $B$ with the necessary properties.
\end{proof}

\begin{proposition}
\label{Prop: pure-dimensional fM(cX,tau)}
Let $\tau=(G,\bg,\bsigma,\bar\bu)$ be a global type (Definition~\ref{Def: global
type}) and assume {$X$ is simple and} $B$ is either log smooth over
$\Spec\kk$ or $B=\Spec\kk^{\dagger}$, the standard log point. Assume further
$\tau$ is realizable over $B$. Then $\fM(\cX/B,\tau)$ is non-empty, reduced and
pure-dimensional. If $B$ is log smooth over $\Spec\kk$, then 
\[
\dim \fM(\cX/B,\tau)=3|\bg|-3+|L(G)|-\rk Q_\tau^\gp + \dim B,
\]
while if $B=\Spec\kk^{\dagger}$, then
\[
\dim \fM(\cX/B,\tau)=3|\bg|-3+|L(G)|-\rk Q_\tau^\gp + 1.
\]
\end{proposition}

\begin{proof}
By Proposition~\ref{prop:I when tau is realizable}, as $\tau$ is a realizable
type, the $\tau$-marked ideal at a point $\ol w'$ of $\fM(\cX/B,\tau)$ takes
the form $\chi_{\tau\tau_{\ol w'}}^{-1}(Q_{\tau}\setminus\{0\})$. Thus, in the
description of a smooth neighborhood $V$ of $\ol w'$ as given in
\eqref{eq:local structure ogus}, $\cA_{Q,I}$ is reduced, and if $B$ is log
smooth over $\Spec\kk$, the bottom horizontal arrow is smooth, and hence $V'$ is
also reduced as the square is Cartesian. This shows that $\fM(\cX/B,\tau)$ is
reduced in this case.

If on the other hand $B=\Spec\kk^{\dagger}$, we may take $Q_B=\NN$ in
\eqref{eq:local structure ogus}. Since $\fM(\cX/B,\tau)$ is defined over $B$,
the induced morphism of stalks of ghost sheaves $\NN\rightarrow Q_{\tau}$ is
local and hence $\NN\setminus \{0\}$ maps into $Q_{\tau}\setminus \{0\}$, and
thus more generally $\NN\rightarrow Q_{\tau_{\ol w'}}$ maps $\NN\setminus
\{0\}$ into $\chi_{\tau\tau_{\ol w'}}^{-1}(Q_{\tau} \setminus\{0\})$ by
compatibility of these maps with generization. Hence we may replace $\cA_{Q_B}$
with the closed substack $\cA_{\NN,\NN\setminus \{0\}}$ in \eqref{eq:local
structure ogus} without affecting this diagram in any other way. In particular,
the bottom horizontal arrow is now still smooth. So $V'$ is again reduced.

{ Let $\ol w$ be a point as in Definition~\ref{def:realizable
over B}.} We may now calculate dimensions by looking at the description of
\eqref{eq:local structure ogus} for a neighborhood of $\ol w$ {in
$\fM(\cX/B,\tau)$}. Since the corresponding curve $C^{\circ}/\ol w$ now has no
smoothable nodes, we may take $r=0$ and $s=|E(G)|$ in \eqref{eq:local structure
ogus}. Further, since $I=Q_{\tau}\setminus \{0\}$, necessarily $\dim
\cA_{Q_{\tau},I}=-\rank Q^{\gp}_{\tau}$. Thus we may calculate, with the cases
being for $B$ log smooth and $B=\Spec\kk^{\dagger}$ respectively,
\begin{align*}
\dim \fM(\cX/B)-\dim \bM(G,{\bf g}) \times B = {} & \dim V' - 
\dim U \times B\\
= {} & \begin{cases} \dim \cA_{Q_{\tau},I} - \dim \cA_{\NN^s,J}\times\cA_{Q_B} \\
\dim \cA_{Q_{\tau},I}-\dim \cA_{\NN^s,J}\times \cA_{\NN,\NN\setminus \{0\} }
\end{cases}\\
= {} & 
\begin{cases}-\rank Q^{\gp}_{\tau}-(-s)\\
-\rank Q^{\gp}_{\tau} - (-s-1)
\end{cases}
\end{align*}
As $\dim \bM(G,{\bf g})= 3|{\bf g}|-3+|L(G)|-|E(G)|$, and $s=|E(G)|$,
we then obtain the desired dimension formulas in the two cases.
\end{proof}

\begin{remark}
\label{Rem: stratified structure of fM(cX,tau)}
(\emph{Stratified structure of $\fM(\cX/B,\tau)$.}) If $\tau'\arr\tau$ is a
morphism of global types (Definition~\ref{Def: global type}), a marking by
$\tau'$ induces a marking by $\tau$ by composition of the marking morphism with
$\tau'\arr\tau$. The same arguments as for ordinary logarithmic maps
\cite[Prop.\,2.34]{decomposition} shows that the corresponding morphism of stacks
\[
j_{\tau\tau'}:\fM(\cX/B,\tau')\arr \fM(\cX/B,\tau)
\]
is finite and unramified. If $\tau'$ is realizable over $B$, then
Proposition~\ref{Prop: pure-dimensional fM(cX,tau)}, { under
the assumptions on $B$ stated there,} further shows that $\im(j_{\tau\tau'})$
defines a pure-dimensional substack of $\fM(\cX/B,\tau)$. Conversely, if there
is no $\tau''$ which is realizable over $B$ mapping to $\tau'$ then
$\fM(\cX,\tau')=\emptyset$. Thus the images of $j_{\tau\tau'}$ for morphisms of
global types $\tau'\arr \tau$ with $\tau'$ realizable over $B$ define a
stratification of $\fM(\cX/B,\tau)$ into pure-dimensional strata.

In particular, the closure of a maximal stratum is the image of
$\fM(\cX/B,\tau')$ for $\tau'$ a \emph{minimal} global type realizable over $B$
dominating $\tau$. Minimality here means that the morphism $\tau'\arr\tau$ does
not factor over any other global type realizable over $B$.

Note, however, that $\fM(\cX/B,\tau')$ is not in general irreducible even for
realizable $\tau'$, due to saturation phenomena already present in ordinary
stable logarithmic maps. In the logarithmic enhancement question for transverse
stable logarithmic maps of \cite[Thm.\,4.13]{decomposition}, this reducibility is
reflected in various choices of roots of unity.
\end{remark}

\begin{example}(Example~\ref{ex:non equi dim} revisited.)
\label{Expl: non equi dim revisited}
Let $\tau$ be the global type with $G$ having just one vertex of genus $0$, no
edges, and four legs, all image cones equal to $0\in\Sigma(X)=\{0,\RR_{\ge 0}\}$
and global contact orders $-1,-1,2,2$. This global type is not realizable
because there can be no positive length legs for the two punctures, but there
are several minimal realizable global types marked by $\tau$. Here are two of
them. The first, $\tau_1$, has the same $(G,\bg)$ as $\tau$, but all image cones
are $\RR_{\ge 0}$. In the notation of Example~\ref{ex:non equi dim}, the
tropical punctured map realizing this type has $\rho>0$ and $\ell_1=\ell_2=0$.
The other minimal realizable type, $\tau_2$, has $G$ with three vertices
$v_1,v_2,v_3$ with $\bsigma(v_1)=\{0\}$, $\bsigma(v_2)=\bsigma(v_3)=\RR_{\ge0}$
and two edges, connecting $v_1$ to $v_2$ and $v_3$, respectively, and one
positive and one negative leg attached to each of $v_2$ and $v_3$. This global
type is realizable by tropical punctured maps with $\rho=0$ and
$\ell_1,\ell_2>0$. {Note that by Proposition~\ref{Prop: pure-dimensional
fM(cX,tau)}, $\dim\fM(\cX/B,\tau_1)=0$ but $\dim\fM(\cX/B,\tau_2)=-1$, showing
non-pure-dimensionality of $\fM(\cX/B,\tau)$.}

\begin{figure}[htb]
\input{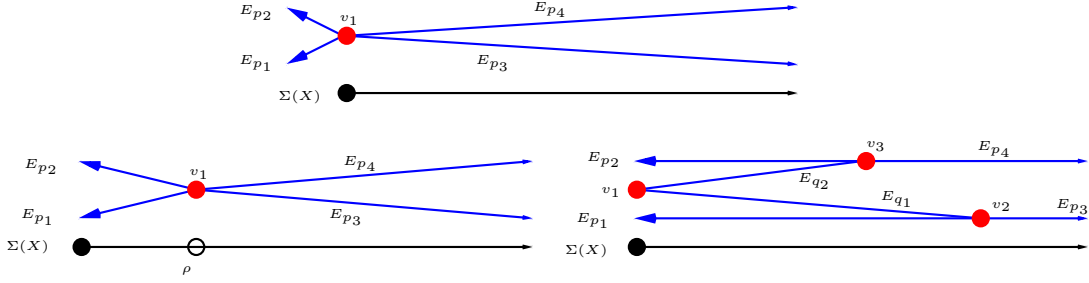}
\caption{The top combinatorial map is not tropically realizable since
$E_{p_1},E_{p_2}$ have nowhere to stretch. The first realizable type has no
nodes, with $\ell_1=\ell_2 = 0$, but with $v_1$ positioned at $\rho>0$. The
second has $v_1$ positioned at $\rho=0$ but then $\ell_1,\ell_2>0$.}
\label{fig:idealized-realizable}
\end{figure}
\end{example}

\subsubsection{Comparing marked and weakly marked stacks}
We end this section by showing that the marked and weakly marked moduli spaces
have the same reduction.

\begin{proposition}
\label{Prop: weakly marked versus marked}
Let $\tau=(G,\bg,\bsigma,\bu)$ be a global type of punctured maps {and
assume $X$ is simple}. Then the canonical morphism
\[
\fM(\cX/B,\tau)\arr \fM'(\cX/B,\tau)
\]
is a closed embedding defined by a nilpotent ideal. Analogous statements hold
for moduli spaces of punctured maps to $X/B$ and for decorated global types.
\end{proposition}

\begin{proof}
By the idealized description in Theorem~\ref{thm:idealized-etale} of the moduli
spaces in question, the statement amounts to showing that the $\tau$-marked
ideal from Definition~\ref{Def: Itau} is contained in the radical of the weakly
$\tau$-marked ideal defined in {Remark}~\ref{Rem: on marking ideals}. 

Let $(C/W,\bp,f)$ be a punctured map weakly marked by $\tau$ and {$\ol
w$ of $\ul W$} a geometric point. We adopt the notation from
Definition~\ref{Def: Itau} and in particular write $\phi:\tau_{\ol w}\arr \tau$
for the contraction morphism given by the marking and
\[
\chi_{\tau\tau_{\ol w}}: Q_{\tau_{\ol w}}\arr {Q_{\tau\tau_{\ol w}}}
\]
for the localization morphism of basic monoids. We have to show that for each
$q\in Q_{\tau_{\ol w}}$ with $\chi_{\tau\tau_{\ol w}}(q)\neq 0$ a
multiple $kq$ lies in the monoid ideal generated by the elements listed
in Definition~\ref{Def: Itau},(i) and (ii). The description of the dual basic
monoids in Proposition~\ref{Prop: Basic monoid} provides the following
commutative diagram with horizontal arrows surjective up to saturation,
with as usual $P_v$ denoting the monoid dual to the cone $\bsigma(v)$:
\[
\setlength{\fboxrule}{0pt}
\xymatrix{
\displaystyle\prod_{v\in V(G_{\ol w})} P_v\times
\prod_{E\in E(G_{\ol w})}\NN
\ar[r]\ar[d]& {Q_{\tau_{\ol w}}}
\ar@<-1ex>[d]^{\chi_{\tau\tau_{\ol w}}}\\
\displaystyle\prod_{v\in V(G)} P_v\times\prod_{E\in E(G)}\NN
\ar[r]& Q_{\tau\tau_{\ol w}}
}
\]
{  The left vertical homomorphism is as follows:
\[
\big( (p_v)_{v\in V(G_{\ol w})},(\ell_E)_{E\in E(G_{\ol w})}\big)
\longmapsto \Big(\big(\textstyle\sum_{\phi(v')=v}\chi_{v'}(p_{v'})\big)_v,
(\ell_{\phi^{-1}(E)})_E\Big).
\]
}
As the top arrow is surjective up to saturation, there exists $k\ge0$ such that
$kq\in Q_{\tau_{\ol w}}$ lifts to an element $(p_v,\ell_E)$ in the left upper
corner. Since $\chi_{\tau\tau_{\ol w}}(q)\neq 0$, the image of this lift in the
lower left corner is non-zero. We conclude that there exists (1)~$v\in
V(G_{\ol w})$ with $\chi_v(p_v)\neq0$ or (2)~$E\in E(G_{\ol w})\setminus
E_\phi$ with $\ell_E\neq0$. In the first case $kq$ lies in the ideal generated
by $\varphi_v(P_v\setminus\chi_v^{-1}(0))$, part of Definition~\ref{Def:
Itau},(i), while in the second case $kq$ lies in the ideal generated by the
nodal generator $q_E$ from Definition~\ref{Def: Itau},(ii).
\end{proof}


\section{The perfect obstruction theory}
\label{sec:obstruction}

Throughout this section, we fix a log smooth morphism
$X \to B$ {of fs logarithmic schemes
fulfilling the assumptions stated at the beginning of \S\ref{sec:stack}} and
$n\in\NN$. Crucial for the following discussion is the factorization of $X\to B$
over the relative Artin fan $\cX\arr B$.

Denote by $\scrM_n(X/B)$ {(resp.\ $\fM_n(\cX/B)$) the stack of marked or weakly
marked punctured maps to $X\to B$ (resp.\ $\cX\to B$),} with $n$ the
number of punctured or nodal sections, {fixing and} suppressing all other decorations
{in the notation}. In
\S\S\ref{ss:Obstruction theories for pairs} and~\ref{ss: Obstruction theories
with point conditions}, we construct two perfect relative obstruction theories,
in the sense of \cite[Def.\,4.4]{Behrend-Fantechi}, one for
$\scrM_n(X/B)\rightarrow \fM_n(\cX/B)$ and one for a related morphism
$\scrM_n(X/B)\rightarrow \fM^{\ev}_n(\cX/B)$; the latter space incorporates data
of maps to $X$ at a set of special points on the domain curve, see
\eqref{eq:fMev def}. Working over $\fM_n^\ev(\cX/B)$ is crucial for
understanding gluing at a virtual level in \S\ref{ss:gluing at virtual level}.

We will avail ourselves {of} the dualizing complex of various Gorenstein
morphisms $\pi$. To avoid adjusting for shifts of dimension in the formulas, we
denote by $\omega_\pi$ the relative dualizing complex, usually denoted
$\omega_\pi^\bullet$, of a relatively Gorenstein morphism $\pi$, that is, the
complex with the invertible relative dualizing sheaf {defined in
\cite[Ex.\,III.9.7]{Hartshorne} (see also \cite[p.157]{Conrad})} shifted to the
left by the relative dimension.

\subsection{Obstruction theories for logarithmic maps from pairs}
\label{ss:Obstruction theories for pairs}

All cases of interest fit into the following general setup. {For this
subsection we {do not enforce} the assumptions on $B$ from the Conventions,
\S\ref{Sec:convention}.}

\subsubsection{Source family} Let $S$ be a log
stack over $B$ and assume we are given a proper and representable morphism of
fine log stacks
\[
Y\larr S,
\]
with underlying map of ordinary stacks $\ul Y\to \ul S$ flat and relatively
Gorenstein. The fibers of this morphism serve as domains for a space of
logarithmic maps. 

In the application, $Y$ is either the universal curve over $S=\fM_n(\cX/B)$
{or $\fM^\ev_n(\cX/B)$}, or a union of sections of the universal curve
with induced log structure. 

\subsubsection{Target family} As a target, we take a composition of morphisms
of fine log stacks
\[
V\larr W\larr B,
\]
with $V\arr W$ log smooth. In applications this will be the sequence\footnote{In
this case, $V\to W$ is strict and we could indeed work with ordinary cotangent
complexes throughout, but for possible other applications we do not make this
assumption.} $X \to \cX \to B$. We assume further given a $B$-morphism $Y\to W$
defining a commutative square
\[
\xymatrixcolsep{3pc}
\xymatrix{
Y\ar[r]\ar[d]&W\ar[d]\\
S\ar[r] &B;}
\]
In our applications this is the universal family of maps to the Artin fan,
either prestable maps of curves, or the corresponding maps of the union of
sections, as the case may be.

\subsubsection{Moduli of lifted maps}
Let $M$ be an open algebraic substack of the following algebraic stack
over $S$. An object over an affine $S$-scheme $T$, considered as a log
scheme by pulling back the log structure from S, consists of a commutative
diagram
\begin{equation}
\label{Eqn: objects of M}
\vcenter{\xymatrixcolsep{3pc}
\xymatrix{
Y_T\ar[rd]\ar[rr]\ar[d]&&V\ar[d]\\
T\ar[rd]&Y\ar[r]\ar[d]&W\ar[d]\\
&S\ar[r]&B}}
\end{equation} 
where the square formed by $Y_T,T,S$ and $Y$ is cartesian. Thus we are
interested in lifting the map $Y\to W$ to $V$ fiberwise relative to
$S$. We endow $M$ with the log structure making the morphism $M\to S$ strict.
The pull-back of $Y$ to $M$ defines the universal domain $\pi:Y_M\to M$. We have
the following $2$-commutative diagram of stacks
\begin{equation}
\label{Eqn: universal object of M}
\vcenter{\xymatrixcolsep{3pc}
\xymatrix{
Y_M\ar[rd]\ar[rr]^f\ar[d]_\pi&&V\ar[d]\\
M\ar[rd]&Y\ar[r]\ar[d]&W\ar[d]\\
&S\ar[r]&B}}
\end{equation}

In the main application, with $Y \to S$ the family of prestable curves, $M$ is
an open substack {of} the stack of punctured maps of interest;
{thus} our deformation theory fixes both the domain of the punctured map
to $X$ and the map to the relative Artin fan $\cX$. In the secondary application,
with $Y \to S$ the family of {sections} with logarithmic structures, the
stack $M$ parametrizes {liftings} of the sections from $\cX$ to $X$.

\subsubsection{An obstruction theory}
Functoriality of log cotangent complexes \cite[1.1(iv)]{LogCot} yields
the morphism
\begin{equation}
\label{Eqn: functorial morphism of LL}
f^*\Omega_{V/W}=Lf^*\LL_{V/W}\larr \LL_{Y_M/Y} =\pi^* \LL_{M/S}. 
\end{equation}
The equality on the left holds by \cite[1.1\,(iii)]{LogCot} since $V\to W$ is
log smooth, while the equality on the right follows since $\LL_{M/S}=\LL_{\ul
M/\ul S}$ and $\LL_{Y_M/Y}=\LL_{\ul Y_M/{\ul Y}}$ by strictness of $M\to S$
\cite[1.1(ii)]{LogCot} and then using compatibility of the ordinary cotangent
complexes with flat pull-back by $\ul \pi$.

Since $\ul Y\larr \ul S$ is relatively Gorenstein by assumption, so is $\ul
Y_M\larr\ul M$ and we have a natural isomorphism of exact functors $\pi^!=
\pi^*\otimes\omega_\pi$. Thus \eqref{Eqn: functorial morphism of LL} is
equivalent to a morphism $f^*\Omega_{V/W}\otimes \omega_\pi\to \pi^!\LL_{M/S}$,
which by adjunction is equivalent to a morphism
\begin{equation}
\label{Eqn: obstruction theory Phi}
\Phi:\EE \larr \LL_{M/S}
\end{equation}
with $\EE= R\pi_*(f^*\Omega_{V/W}\otimes\omega_\pi)$. 

\subsubsection{Functoriality} 
We will show in Proposition~\ref{Prop: Obstruction theory for morphisms} that
$\Phi$ is a perfect obstruction theory for $M$ over $S$. A most transparent
proof that $\Phi$ is a perfect obstruction theory for $M$ over $S$ relies on the
fact that the construction of $\Phi$ is functorial. For lack of reference we
provide a proof for this well-known property in the following lemma. If $T\to M$
is any map, denote by
\[
\Phi_T: \EE_T\to \LL_{T/S}
\]
the morphism in \eqref{Eqn: obstruction theory Phi} constructed from \eqref{Eqn:
objects of M} instead of \eqref{Eqn: universal object of M}.

\begin{lemma}
\label{Lem: Functoriality of Phi}
The construction of $\Phi$ in \eqref{Eqn: obstruction theory Phi} is functorial
in the following sense: Let $\ul T\to \ul M$ be a morphism of stacks. Denoting
$T\to M$ the associated strict morphism of log stacks, we obtain the commutative
diagram
\[
\xymatrixcolsep{3pc}
\xymatrix{
Y_T\ar@/^2pc/[rrr]^{f_T}\ar[r]_{\tilde h}\ar[d]_{\pi_T}&
Y_M\ar[rd]\ar[rr]^f\ar[d]_\pi&&V\ar[d]\\
T\ar[r]^h&M\ar[rd]&Y\ar[r]\ar[d]&W\ar[d]\\
&&S\ar[r]&B}
\]
with the two squares of domains {(i.e., the left-most square
and the parallelogram)} cartesian. Then we have a commutative square
\[
\xymatrixcolsep{3pc}
\xymatrix{
Lh^*\EE\ar[r]^{Lh^*\Phi}\ar[d]_\beta&Lh^*\LL_{M/S}\ar[d]\\
\EE_T\ar[r]^{\Phi_T}&\LL_{T/S},
}
\]
with left-hand vertical arrow a natural isomorphism and the right-hand vertical
arrow defined by functoriality of cotangent complexes.
\end{lemma}

\begin{proof}
Naturality of the base change map \cite[Rem.\,07A7]{stacks-project} applied to
$f^*\Omega_{V/W}\otimes \omega_\pi\to \LL_{Y_M/Y}\otimes\omega_\pi$ together
with $f\circ\tilde h=f_T$ and $\tilde h^*\omega_\pi=\omega_{\pi_T}$
{\cite[Thm.\,3.6.1]{Conrad}}, leads to the commutative square
\begin{equation}
\label{Eqn: CartDiagFunctCotang}
\begin{CD}
{\makebox[\width][r]{$Lh^*\EE=$}}Lh^* R\pi_*(f^*\Omega_{V/W}\otimes\omega_\pi)@>>>Lh^* R\pi_*(\LL_{Y_M/Y}\otimes\omega_\pi)\\
@V{\beta}VV@VV{b}V\\
{\makebox[\width][r]{$\EE_T=$}}
R{\pi_T}_*(f_T^*\Omega_{V/W}\otimes\omega_{\pi_T})@>>>
R{\pi_T}_*(L\tilde h^*\LL_{Y_M/Y}\otimes\omega_{\pi_T}).
\end{CD}
\end{equation}
Now $\LL_{Y_M/Y}\simeq \pi^*\LL_{M/S}$, as remarked after \eqref{Eqn: functorial
morphism of LL}, and hence the adjunction counit $R\pi_*\pi^!\to {\id}$
applied in the construction of $\Phi$ in \eqref{Eqn: obstruction theory Phi} is
given by the projection formula followed by the trace morphism,
\[
R\pi_*(\pi^*\LL_{M/S}\otimes\omega_\pi)\stackrel{\simeq}{\larr}
\LL_{M/S}\otimes R\pi_*(\omega_\pi)
\stackrel{\operatorname{Tr}_{\omega_\pi}}{\larr} \LL_{M/S}.
\]
Thus the upper horizontal {map} of \eqref{Eqn: CartDiagFunctCotang} composed
with $Lh^*$ of this adjunction counit isomorphism yields $Lh^*\Phi$.

Similarly, extending the lower horizontal arrow by the map induced by
functoriality of cotangent complexes,
\[
L\tilde h^*\LL_{Y_M/Y}\larr \LL_{Y_T/Y}=\pi_T^*\LL_{T/S},
\]
composed with the adjunction counit morphism
$R{\pi_T}_*(\pi_T^*\LL_{T/S}\otimes\omega_{\pi_T}) \rightarrow \LL_{T/S}$ for
$\pi_T$ retrieves the definition of $\Phi_T$.

Moreover, by compatibility of both the projection formula
\cite[Lem.\,0B6B]{stacks-project} and the trace morphism
\cite[Lem.\,0E6C]{stacks-project} with base change, the following diagram
continuing \eqref{Eqn: CartDiagFunctCotang} on the right is commutative:
\[
\xymatrix{
Lh^* R\pi_*(\pi^*\LL_{M/S}\otimes\omega_\pi)\ar[r]^{\simeq}\ar[d]_{b} &
Lh^*\LL_{M/S}\otimes Lh^* R\pi_*\omega_\pi\ar[d]\ar[rd]^{\tr}\\
R{\pi_T}_*(\pi_T^* Lh^*\LL_{M/S}\otimes\omega_{\pi_T})\ar[r]^{\simeq}\ar[d] &
Lh^*\LL_{M/S}\otimes R{\pi_T}_*\omega_{\pi_T}\ar[r]\ar[d] &
Lh^*\LL_{M/S}\ar[d]\\
R{\pi_T}_*(\pi_T^*\LL_{T/S}\otimes\omega_{\pi_T}) \ar[r]^{\simeq} &
\LL_{T/S}\otimes R{\pi_T}_*\omega_{\pi_T}\ar[r] &
\LL_{T/S}.
}
\]
The three left horizontal isomorphisms are defined by projection formulas, the
diagonal and the two horizontal morphisms on the right induced by trace
homomorphisms, the two upper vertical arrows defined by base change, and the
three lower vertical arrows defined by functoriality of cotangent complexes.
For the identification of {the upper left vertical arrow} with the right
vertical arrow {labelled $b$} in \eqref{Eqn: CartDiagFunctCotang} note
that
\[
L\tilde h^*\LL_{Y_M/Y}\simeq L\tilde h^* \pi^* \LL_{M/S}
\simeq \pi_T^* Lh^* \LL_{M/S}.
\]
This establishes the claimed commutative diagram.

It remains to show that $\beta$ is a natural isomorphism. This follows from the
general base change statement \cite[Lem.\,0A1K]{stacks-project} applied to
$\pi:Y_M\to M$, with $f^*\Omega_{V/W}$ for the object in $D_{\rm
QCoh}(\cO_{Y_M})$ and with $\omega_\pi$ as complex of $\pi$-flat quasi-coherent
sheaves. 
\end{proof}

\begin{proposition}[$\Phi$ is a perfect obstruction theory]
\label{Prop: Obstruction theory for morphisms}
The morphism $\Phi:\EE\to
\LL_{M/S}$ constructed in \eqref{Eqn: obstruction theory Phi} is an obstruction
theory for $M\to S$ in the sense of \cite[Def.\,4.4]{Behrend-Fantechi}.
\end{proposition}

\begin{proof}
We check the obstruction-theoretic criterion
\cite[Thm.\,4.5.3]{Behrend-Fantechi}, applied in the setting relative to $S$,
similarly to ordinary logarithmic maps carried out in
\cite[Prop.\,5.1]{LogGW}.

Assume given a morphism $h:T\to M$, a square-zero extension $T\to \ol T$ with
ideal sheaf $\cJ$ and a morphism $\ol T\to S$, with log structures turning all
three morphisms strict. This situation leads to the following commutative
diagram:
\[
\xymatrix@-1.2pc{
&Y_T\ar@/^2pc/[rrrr]^{f_T}\ar[rr]^(.6){\tilde h}\ar[ld]\ar[dd]^(.3){\pi_T}&&
Y_M\ar[rr]_f\ar[ld]\ar[dd]^(.3)\pi&&V\ar[ld]\\
Y_{\ol T}\ar[rr]\ar[dd]&&Y\ar[dd]\ar[rr]&&W\ar[dd]\\
&T\ar[rr]^(.7)h\ar[ld]&& M\ar[ld]\\
\ol T\ar[rr]&&S\ar[rr]&&B.}
\]
All sides of the cube on the left are cartesian, but not in general the bottom
and top faces.

The obstruction class $\omega(h)\in \ext^1(Lh^*\LL_{\ul
M/\ul S},\cJ)$ for extending $h$ to {an $S$-morphism} $\ol T\to M$ is the composition
\[
Lh^*\LL_{\ul M/\ul S}\larr \LL_{\ul T/\ol{\ul T}}\larr
\tau_{\ge -1} \LL_{\ul T/\ul{\ol T}}=\cJ[1]
\]
with the first arrow defined by functoriality of cotangent complexes, see
\cite[Prop.\,2.2.4]{Illusie} with $X_0=T$, $X=\ol T$, $Y_0=Y=M$ and $Z_0=Z=S$.
Because $T \to \ol T$ and $M \to S$ are strict we can replace the ordinary
cotangent complex with the log cotangent complex in this construction
\cite[1.1(ii)]{LogCot}. 

Now $\Phi^*\omega(h)$ is the composition of this morphism with $Lh^*\Phi:
Lh^*\EE\to Lh^* \LL_{M/S}$. By functoriality of our obstruction theory
(Lemma~\ref{Lem: Functoriality of Phi}), this composition also has the
factorization
\[
\EE_T= R{\pi_T}_* (f_T^*\Omega_{V/W}\otimes\omega_{\pi_T})
\stackrel{\Phi_T}{\larr} \LL_{T/S}\larr \tau_{\ge -1}\LL_{T/\ol T}=\cJ[1],
\]
which by adjunction is equivalent to the composition
\[
f_T^*\Omega_{V/W}\otimes\omega_{\pi_T}\larr
\LL_{Y_T/Y}\otimes\omega_{\pi_T}\larr \tau_{\ge -1}\pi_T^!\LL_{T/\ol T}
=\pi_T^*\cJ[1]\otimes{\omega_{\pi_T}.}
\]
Up to tensoring with $\omega_{\pi_T}$ this is the obstruction class for
extending $f_T: Y_T\to V$ to $Y_{\ol T}$, as a morphism over $W$. By our
assumption on the objects of $M$, this extension exists if and only if $T\to M$
extends to $\ol T$. This shows the part of the criterion concerning the
obstruction.

A similar argument shows that once $\omega(h)=0$, the space of extensions form a
torsor under $\ext^0(Lh^*\LL_{\ul M/\ul S},\cJ)$, showing the second part of the
criterion.
\end{proof}
\medskip

\subsubsection{The dualizing complex of the embedding of markings}
After this recapitulation of obstruction theories for logarithmic maps with
proper and relatively Gorenstein domains, we are now in position to bring in
point conditions. Abstractly we consider a composition of proper, representable
morphisms of fine log stacks
\begin{equation}\label{Eqn: Morphism of domains}
Z\stackrel{\iota}{\larr} Y\larr S,
\end{equation}
with maps of {algebraic stacks} underlying $Z\to S$ and $Y\to S$ flat
and relatively Gorenstein as before. Note that while $\iota$ may not be flat and
hence cannot be considered relatively Gorenstein following the usual convention,
one can still define a relative dualizing sheaf
\begin{equation}
\label{Eqn: dualizing sheaf}
\omega_\iota= \omega_{Z/S}\otimes\iota^*\omega_{Y/S}^\vee
\end{equation}
fulfilling relative duality, hence defining a right-adjoint functor $\iota^!$ to
$R\iota_*$. This works as in the case of smooth morphisms discussed e.g.~in
\cite[\S3.4]{HuybrechtsFM}.

\subsubsection{Obstruction for markings}
\label{sss: obstructions for markings}
We now have another algebraic stack $N$, an open substack of the stack over $S$
with objects given by diagrams as in \eqref{Eqn: objects of M} with $Y$ replaced
by $Z$. We assume the open substack $N$ is chosen large enough so that
composition with $\iota: Z\to Y$ defines a morphism of stacks
\begin{equation}\label{Fig: Restriction square}
\varepsilon: M\larr N.
\end{equation}
As in \eqref{Eqn: obstruction theory Phi} we now obtain two obstruction
theories, one for $M\to S$, the other for $N\to S$,
\begin{equation}\label{Eqn: Two obstruction theories}
\Phi: \EE \larr \LL_{M/S},\qquad
\Psi: \FF\larr \LL_{N/S}.
\end{equation}
In our application, $Y\to S$ is some universal curve and $Z\to Y$ a strict
closed embedding with morphism to $S$ scheme-theoretically \'etale. In this
case, $\Psi$ is simply the obstruction theory for a number of points in $V/W$,
i.e., a trivial obstruction theory in the sense that there are no
obstructions. In particular, \'etale locally $\FF$ can be taken as the direct
sum of the pull-back of $\Omega_{V/W}$ by scheme-theoretic maps from $\ul N$ to
$\ul V$.

\begin{proposition}[Compatibility of obstruction theories]
\label{Prop: Compatibility of obstruction theories}
The two obstruction theories $\Phi$ and $\Psi$ in \eqref{Eqn: Two obstruction
theories} fit into a commutative square
\[
\begin{CD}
L\varepsilon^*\FF@>L\varepsilon^*\Psi>> L\varepsilon^*\LL_{N/S}\\
@VVV@VVV\\
\EE@>\Phi>> \LL_{M/S},
\end{CD}
\]
with the right-hand vertical morphism given by functoriality of the cotangent
complex.
\end{proposition}

\begin{proof}
Consider the following commutative diagram with the left four squares
cartesian.\\[2ex]
\[
\xymatrix{
Z\ar[d]_\iota&
Z_N\ar[d]\ar[l]\ar@/^2pc/[rrr]^g\ar@/_1.5pc/[dd]_(.3)p&
Z_M\ar[l]^{\tilde\varepsilon}\ar[d]^{\iota_M}
\ar[rr]^{h}\ar@/_1.5pc/[dd]_(.3){p_M}
\ar[rrd]|\hole
&&V\ar[d]\\
Y\ar[d]&
Y_N\ar[d]\ar[l]&
Y_M\ar[l]\ar[d]^\pi\ar[rr]\ar[rru]_(.75)f&&
W\ar[d]\\
S&
N\ar[l]&
M\ar[l]^\varepsilon\ar[rr]&&
B
}
\]
\\[1ex]
The left column is the given morphism~\eqref{Eqn: Morphism of domains} of
domains, the lower horizontal row contains the restriction morphism
$\varepsilon$ from~\eqref{Fig: Restriction square} and the morphisms to $B$ and
$S$, while $f: Y_M\to V$ and $g:Z_N\to V$ are the respective universal morphisms
defined on the universal domains $Y_M\to M$ and $Z_N\to N$.

{
The obstruction theory $\Phi$ in~\eqref{Eqn: Two obstruction theories} was
defined by applying $R\pi_*(\,\cdot\,\otimes\omega_\pi)$ to $f^*\Omega_{V/W}\to
\LL_{Y_M/Y}= \pi^*\LL_{M/S}$ followed by the adjunction counit
$R\pi_*\pi^!\arr\id$, using $\pi^!=\pi^*\otimes\omega_\pi$. For $\Psi$ one
analogously takes $Rp_*(\,\cdot\,\otimes\omega_p)$ of $g^*\Omega_{V/W}\to
\LL_{Z_N/Z}= p^*\LL_{N/S}$ followed by $Rp_*p^!\arr\id$. By functoriality of
obstruction theories (Lemma~\ref{Lem: Functoriality of Phi}), the pull-back
$L\varepsilon^*\Psi$ is similarly obtained by $R {p_M}_*(\,\cdot\,\otimes
\omega_{p_M})$ of
\begin{equation}
\label{Eqn: morphism to LL_{N/S}}
h^*\Omega_{V/W} \larr L\tilde\varepsilon^*\LL_{Z_N/Z}
= L\tilde\varepsilon^* p^*\LL_{N/S}
= p_M^*L\varepsilon^* \LL_{N/S},
\end{equation}
followed by $R{p_M}_* p_M^!\to {\id}$.

From $h=f\circ\iota_M= g\circ\tilde\varepsilon$ we can extend \eqref{Eqn:
morphism to LL_{N/S}} to the commutative diagram
\[
\xymatrix{
\tilde\varepsilon^*g^*\Omega_{V/W}\ar[r]\ar@{=}[d]&
L\tilde\varepsilon^*\LL_{Z_N/W}
\ar[d]\ar[r]&L\tilde\varepsilon^*\LL_{Z_N/Z}=
p_M^*L\varepsilon^*\LL_{N/S}\ar[d]\\
h^*\Omega_{V/W}\ar@{=}[d]\ar[r]&\LL_{Z_M/W\ar[r]\ar[d]}&\LL_{Z_M/Z}=
p_M^*\LL_{M/S}\ar[d]^{\simeq}\\
\iota_M^*f^*\Omega_{V/W}\ar[r]&L\iota_M^*\LL_{Y_M/W}\ar[r]&
L\iota_M^*\LL_{Y_M/Y}=p_M^*\LL_{M/S}
}
\]
The last row in this diagram is $L\iota_M^*$ of the morphism
$f^*\Omega_{V/W}\arr \pi^*\LL_{M/S}$ that gives rise to the obstruction theory
$\Phi$ for $M$. The essential part of this diagram is the square
\begin{equation}
\label{Eqn: Square connecting Phi and Psi}
\vcenter{\xymatrix{
h^*\Omega_{V/W}\ar@{=}[d]\ar[r]&p_M^*L\varepsilon^*\LL_{N/S}\ar[d]\\
\iota_M^*f^*\Omega_{V/W}\ar[r]&p_M^*\LL_{M/S}
}}
\end{equation}
Next observe that $\omega_{p_M}=\iota_M^*\omega_\pi\otimes\omega_{\iota_M}$, $h=f\circ\iota_M$, and $\iota_M^!=\iota_M^*\otimes\omega_{\iota_M}$ show that
\[
R{p_M}_*(h^*\Omega_{V/W}\otimes\omega_{p_M})=
R\pi_*R{\iota_M}_*\iota_M^!(f^*\Omega_{V/W}\otimes\omega_\pi).
\]
Thus $R{p_M}_*(\,\cdot\,\otimes\omega_{p_M})$ applied to \eqref{Eqn: Square
connecting Phi and Psi} yields the upper left square of the following
commutative diagram:
\begin{equation}
\label{Eqn: Big square for Phi versus Psi}
\vcenter{\xymatrix{
L\varepsilon^*\FF= R{p_M}_*(h^*\Omega_{V/W}\otimes\omega_{p_M})
\ar@{=}[d]\ar[r]&
R{p_M}_*p_M^!L\varepsilon^*\LL_{N/S}\ar[d]^a\ar[r]&
L\varepsilon^*\LL_{N/S}\ar[d]\\
R\pi_*R{\iota_M}_*\iota_M^!(f^*\Omega_{V/W}\otimes\omega_\pi)
\ar[d]\ar[r]_(.6)b&
R{p_M}_*p_M^!\LL_{M/S}\ar[d]\ar[r]& \LL_{M/S}\ar[d]\\
\EE=R\pi_*(f^*\Omega_{V/W}\otimes\omega_\pi)\ar[r]&
R\pi_*\pi^!\LL_{M/S}\ar[r]& \LL_{M/S}.
}}
\end{equation}
The upper right square is from functoriality of adjunction
$R{p_M}_*p_M^!\arr\id$ applied to the arrow marked $a$, the lower left one
similarly from $R{\iota_M}_*\iota_M^!\arr\id$ applied to the arrow marked $b$. The lower right square is from the natural isomorphism 
of the adjunction counit $R{p_M}_*p_M^!\arr\id$ with the composition
\[
R\pi_* R{\iota_M}_*\iota_M^!\pi^!\arr R\pi_* \pi^!\arr \id,
\]
see \cite[Prop.\,VII.3.4,(b)]{Hartshorne},\cite[Lem.\,3.4.3]{Conrad}.

The outer square of \eqref{Eqn: Big square for Phi versus Psi} provides the
claimed commutative diagram.
}
\end{proof}

\subsection{Obstruction theories for punctured maps with point conditions}
\label{ss: Obstruction theories with point conditions}

We are now in position to define obstruction theories for moduli spaces of
punctured maps with prescribed point conditions. Recall the log smooth morphism
$X\to B$ and its factorization over the relative Artin fan $\cX\to B$ from the
beginning of this section. We want to work relative to a stack $S$ of stable
punctured maps to $\cX/B$. Adopting the notation used elsewhere in the paper, we
now write $\fM$ instead of $S$ for the algebraic stack of domains together with
the tuple of points at which to impose point conditions. For example, $\fM$
could be $\fM(\cX/B,\tau)$ from Definition~\ref{Def: stacks of decorated puncted
maps}. Then $Y\to S=\fM$ is the universal curve, $Z\to Y$ the strict closed
embedding of a union of sections, one for each point condition to be imposed,
assumed ordered, and we have a universal diagram
\[
\begin{CD}
Y@>>> \cX\\
@VVV@VVV\\
\fM@>>> B.
\end{CD}
\]
As our target we now take the composition
\[
X\larr \cX\larr B.
\]
Note that $\cX\to B$ is log \'etale and $X\larr\cX$ is strict and log smooth.
Hence $\ul X\to\ul \cX$ is smooth as a morphism of stacks and we have a
sequence of canonical isomorphisms
\[
\LL_{X/B}=\Omega_{X/B}=\Omega_{X/\cX}=\Omega_{\ul X/\ul\cX}= \LL_{\ul X/\ul\cX}.
\]
For easier reference later on we also write $\scrM$ instead of $M$ for the
algebraic stack of punctured maps to $X$ to be considered.

For the moduli space $N$ of point conditions we take the space of factorizations
of the composition $Z\to Y\to\cX$ via $X\to\cX$. Note that since
$X\to\cX$ is strict, it is enough to provide the lift for $\ul X\to \ul\cX$,
that is, ignoring the log structure. Thinking of these factorizations as
providing evaluation maps $\fM\to \ul X$ at the marked points given by
the sections $Z$ of $Y\to S$, we denote the stack of such factorizations by
$\fM^\ev$. This stack is algebraic by the fiber product description
\begin{equation}
\label{eq:fMev def}
\fM^\ev= \fM\times_{\ul\cX\times_{\ul B}\ldots\times_{\ul B}\ul\cX}
(\ul X\times_{\ul B}\ldots\times_{\ul B} \ul X).
\end{equation}
Here the map $\fM\to \ul\cX\times_{\ul B}\ldots\times_{\ul B}\ul\cX$ is defined
by composing the sections $\fM\to\ul\fM\to \ul Z$ with the composition $\ul Z\to
\ul Y\to \ul\cX$ in the given order of the sections.

With this notation, the composition $M\to N\to S$ considered in the
proof of Proposition~\ref{Prop: Compatibility of obstruction theories} reads
\begin{equation}
\label{Eqn: scrM->fM^ev->fM}
\scrM\stackrel{\varepsilon}{\larr} \fM^\ev\larr\fM. 
\end{equation}
In \S\ref{ss:Obstruction theories for pairs} we recalled
the construction of obstruction theories for $\scrM/\fM$ and for $\fM^\ev/\fM$,
which in the situation at hand are perfect of amplitude contained in $[-1,0]$,
and showed their compatibility (Proposition~\ref{Prop: Compatibility of
obstruction theories}). As in \cite[Constr.\,3.13]{Mano}, this situation
provides perfect obstruction theories for $\scrM/\fM^\ev$ by completing the
compatibility diagram in Proposition~\ref{Prop: Compatibility of obstruction
theories} to a morphism of distinguished triangles:\\[1ex]
\begin{equation}
\label{Eqn: Cone construction obstruction theory}
\vcenter{\xymatrix{
L\varepsilon^*\FF \ar[r]\ar[d]& \EE \ar[r]\ar[d]& \GG \ar[r]\ar@{-->}[d]&
L\varepsilon^*\FF[1]\ar[d] \\
L\varepsilon^*\LL_{\fM^\ev/\fM} \ar[r] &\LL_{\scrM/\fM} \ar[r]&
\LL_{\scrM/\fM^\ev} \ar[r]& L\varepsilon^*\LL_{\fM^\ev/\fM}[1]}}
\end{equation}

{
\begin{remark}
\label{Rem: non-uniqueness of obstruction theories}
Note that while the isomorphism class of $\GG$ is unique, the dashed arrow is
not, so this recipe potentially provides several different obstruction theories
for $\scrM/\fM^\ev$. On the other hand, any two dashed arrows differ by an
element of the image of
\[
\Hom(\GG,\LL_{\scrM/\fM})\arr \Hom(\GG,\LL_{\scrM/\fM^\ev}).
\]
Thus the space of obstruction theories $\GG\arr \LL_{\scrM/\fM^\ev}$ constructed
as dashed arrow in \eqref{Eqn: Cone construction obstruction theory} is
parametrized by an affine space. This shows that the virtual classes constructed
from any two such obstruction theories agree.\footnote{We learnt this argument
from Tom Graber.}
\end{remark}}

For the sake of being explicit and for later use we now work out $\GG$. For
simplicity of notation write ${\pi}:C\to \scrM$ for the pull-back
$Y_\scrM$ of the universal curve $Y\to\fM$ to $\scrM$, and, in disagreement with
our usual conventions, write $\iota:Z\to C$ for the strict closed substack of
special points rather than $Z_\scrM$. We assume that $Z=Z'\amalg Z''$ with $Z'$
disjoint from the critical locus of $\ul C\to \ul\scrM$ and $Z''$ the images of
a set of nodal sections, as reviewed in Definition~\ref{Def: nodal sections}
below. Denote by $\kappa:\tilde C\to C$ the partial normalization of $\tilde C$
along the nodal sections exhibiting $\ul C$ as the fibered sum
\[
\ul C= \ul Z''\amalg_{\tilde{\ul Z}''}\tilde{\ul C}
\]
with $\tilde Z''=\kappa^{-1}(Z'')\arr Z''$ the two-fold unbranched cover induced
by $\kappa$. Write $\tilde\pi= \pi\circ\kappa: \tilde C\to \scrM$, $\tilde
f=f\circ\kappa: \tilde C\to X$ and $\tilde Z= \kappa^{-1}(Z)$, with the log
structures making $\tilde C\arr C$ and $\tilde Z\arr \tilde C$
strict.\footnote{The log structures on $\tilde C$ and $\tilde Z$ are irrelevant
for the following discussion and are merely chosen for the sake of uncluttering
the notation.}

For simplicity of the following statement \emph{we now assume the two-fold
covering $\tilde Z''\arr Z''$ is trivial}, that is, that there is an isomorphism
\[
\tilde Z''\simeq Z''\amalg Z''
\]
over $Z''$. This is sufficient for all aplications we
can currently think of. The general case can be treated by going over to an
orientation covering or by twisting with an orientation sheaf.

\begin{proposition}
\label{Prop: Point condition virtual bundle}
For the tangent-obstruction bundle in \eqref{Eqn: Cone construction
obstruction theory} it holds
\[
\GG\simeq R\pi_*\big(f^*\Omega_{X/B}\otimes
\kappa_*(\omega_{\tilde \pi}(\tilde Z))\big)
\simeq R\tilde\pi_*\big(\tilde f^*\Omega_{X/B}\otimes
\omega_{\tilde \pi}(\tilde Z)\big)
\simeq (R\tilde\pi_* \tilde f^*\Theta_{X/B}(-\tilde Z))^{\vee}.
\]
Moreover, $\GG$ is perfect of amplitude $[-1,0]$.
\end{proposition}

\begin{proof}
The second isomorphism follows by the projection formula, the third isomorphism
by relative duality.

For the first isomorphism we first claim there exists the following exact
sequence of complexes, all concentrated in degree $-1$:
\begin{equation}
\label{Eqn: Exact sequence dualizing complexes}
0\larr \omega_\pi \larr \kappa_*\big(\omega_{\tilde\pi}(\tilde Z)\big)
\larr \iota_*\cO_Z[1]\larr 0.
\end{equation}
On the complement of the nodal locus $Z''$, this sequence is defined by
\[
0\larr \omega_\pi \larr \omega_\pi(Z')\larr \omega_\pi\otimes_{\cO_C}
\iota_*\cO_{Z'}(Z')\larr 0
\]
by means of the canonical isomorphism
\[
\omega_\pi\otimes_{\cO_C} \iota_*\cO_{Z'}(Z')=
\iota_*(\iota^*\omega_\pi\otimes_{\cO_{Z'}}
\omega_\iota) \simeq \iota_*\cO_{Z'}[1]
\]
coming from the definition of $\omega_\iota{=
\iota^*\omega_\pi^\vee\simeq \cO_{Z'}(Z')}$ in \eqref{Eqn: dualizing sheaf}.
Explicitly, the homomorphism $\omega_\pi(Z')\to \iota^*\cO_{Z'}[1]$ takes the
residue along $Z'$.

Near the nodal locus, \eqref{Eqn: Exact sequence dualizing complexes} is defined
by
\[
0\larr \omega_\pi \stackrel{\kappa^*}{\larr}
\kappa_*\big(\omega_{\tilde\pi}(\tilde Z)\big)\larr \iota_*\cO_{Z''}[1]\larr 0.
\]
To obtain this sequence, recall that \'etale locally
$\omega_\pi=\Omega_{C/\scrM}[1]$ with $\Omega_{C/\scrM}$ the sheaf of relative
logarithmic differentials for $C/\scrM$, while
$\omega_{\tilde\pi}=\Omega_{\ul{\tilde C}/\ul \scrM}[1]$ with
$\Omega_{\ul{\tilde C}/\ul \scrM}$ the sheaf of relative ordinary differentials
for $\ul{\tilde C}/\ul \scrM$. In fiberwise coordinates $z,w$ for the two
branches of $C$ along $Z''$ on an \'etale neighborhood, $\Omega_{C/\scrM}$ is
locally generated by $z^{-1}dz= -w^{-1}dw$, hence pulls back to ordinary
differentials with simple poles along $\kappa^{-1}(Z'')\subseteq \tilde Z$. The
map to $\cO_Z$ takes the difference of the residues of such rational
differential forms on $\tilde C$ along the two preimages of the nodal locus.
Note that this map depends on an order of the two branches along each connected
component of $Z''$, hence relies on the assumption $\tilde Z''= Z''\amalg Z''$.
This establishes sequence~\eqref{Eqn: Exact sequence dualizing complexes}.

{Next note that $\omega_{p_{\!\scrM}}\simeq\cO_Z$ since $p_\scrM: Z\to
\scrM$ is \'etale. Using the projection formula we can thus rewrite
\[
L\varepsilon^*\FF= R{p_{\!\scrM}}_*(h^*\Omega_{X/B}\otimes\omega_{p_{\!\scrM}})
= R\pi_*\iota_*\iota^*f^*\Omega_{X/B}
=R\pi_*(f^*\Omega_{X/B}\otimes\iota_*\cO_Z). 
\]
Finally, apply $R\pi_*$ to \eqref{Eqn: Exact sequence dualizing complexes}
tensored with $f^*\Omega_{X/B}$ to produce the upper triangle of \eqref{Eqn:
Cone construction obstruction theory} with the claimed middle term
$\GG=R\pi_*\big(f^*\Omega_{X/B}\otimes\kappa_* (\omega_{\tilde \pi}(\tilde
Z))\big)$: }
\begin{equation}
\label{Eqn: EE-GG-FF sequence}
\begin{CD}
\EE@. \GG@. L\varepsilon^*\FF[1]\\
@| @| @|\\
R\pi_*(f^*\Omega_{X/B}\otimes\omega_\pi)
@>>> R\pi_*\big(f^*\Omega_{X/B}\otimes\kappa_*
(\omega_{\tilde \pi}(\tilde Z))\big)
@>>> {p_\scrM}_* (h^*\Omega_{X/B})[1]
\end{CD}
\end{equation}
Taking cohomologies, this diagram also shows the statement about the amplitude
of $\GG$.
\end{proof}
\bigskip

\subsection{Punctured Gromov-Witten invariants}
\label{ss: punctured GW invts}

Using properness of $\scrM(X/B,\beta)$ over $B$ (Corollary~\ref{Cor: scrM/B is
proper}) and the obstruction theory, we can now define punctured Gromov-Witten
invariants. {To be} explicit, we assume the ground field $\kk$ to be a
subfield of $\CC$ and take $H_2(X)$ to be singular homology of the base change
to $\CC$. Since $\fM(\cX/B,\beta)$ is typically non-equidimensional due to the
puncturing ideal, the general definition demands a stratum-by-stratum treatment.
Sometimes one can show independence of certain choices, e.g.\ in the setting of
\cite{GSAssoc}, but presently our understanding of the intersection theory of
$\fM(\cX/B)$ and in logarithmic geometry is too limited to make general
statements. Some steps in this direction have been taken in
\cite{Barrott,Yixian}.

Let $X\arr B$ be projective and log smooth, with Zariski logarithmic structure
on $X$. Let $\btau=(G,\bg,\bsigma,\bar\bu,\bA)$ be a decorated global type
(Definition~\ref{Def: global type}). Denote by $g$ the total genus and
$k=|L(G)|$. We assume $\ocM_X^\gp\otimes_\ZZ\QQ$ to be generated by global
sections to apply Corollary~\ref{Cor: scrM/B is proper}, or otherwise
$\scrM(X/B,\btau)\arr B$ to be proper. Denote by
{$Z_L=X_{\bsigma(L)}\subseteq X$ the evaluation
stratum for $L\in L(G)$}.

{Considering} for simplicity evaluations at all punctures {rather
than at a subset of punctures}, we then have an evaluation map
\[
\textstyle
\ul\ev: \ul{\scrM}({X/B},\btau)\arr \prod_{L\in L(G)} \ul Z_L,
\]
and, by \S\ref{ss: Obstruction theories with point conditions} and notably
\eqref{Eqn: Cone construction obstruction theory}, a perfect relative
obstruction theory $\GG$ for
\[
\varepsilon: \scrM({X/B},\btau)\arr \fM^\ev(\cX/B,\btau).
\]
The relative virtual dimension is given by the Riemann-Roch formula applied to
the {virtual} bundle in Proposition~\ref{Prop: Point condition virtual
bundle} as
\begin{equation}
\label{Eqn: relative virtual dimension}
d(g,k,A,n)= c_1(\Theta_{X/B})\cdot A+ n\cdot(1-g-k).
\end{equation}
Here $A=|\bA|$ and $g=|\bg|$ are the total {curve} class and total genus
of $\btau$, $k=|L(G)|$ the number of point conditions imposed and $n=\dim X-\dim
B$ the relative dimension of $X$ over $B$. Denote by $\varepsilon_\GG^!$ the
associated virtual pull-back from \cite{Mano}, an operational Chow-class for
$\varepsilon$. 

\begin{definition}
\label{Def: Punctured GW correspondence}
The \emph{punctured Gromov-Witten correspondence} defined by the global
decorated type $\btau$ is the homomorphism
\[
\textstyle
(\ul\ev\times p)_*\varepsilon_\GG^!: A_*\big( \fM^\ev(\cX/B,\btau)\big) \arr
A_{*+d(g,k,A,n)}\big(\prod_L \ul Z_L\times \scrM_{g,k}\big)
\]
of rational Chow groups.
\end{definition}

{Here $\prod_L$ denotes the cartesian product of spaces over $B$.}
As usual, pairing with cohomology classes in $\prod_L \ul Z_L\times \scrM_{g,k}$
and taking degrees then produces Gromov-Witten invariants. Note also that
Proposition~\ref{Prop: pure-dimensional fM(cX,tau)} defines pure-dimensional
cycles {in $\fM^\ev(\cX/B,\btau)$} as the images of the fundamental
classes of $\fM(\cX/B,\btau')$ for $\btau'\to\btau$ a contraction morphism from
a realizable global type.


\section{Splitting and gluing}
\label{sec:split-glue}

As discussed in the introduction, one crucial motivation for the introduction of
the notion of punctured maps is the desire to treat logarithmic Gromov-Witten
invariants by splitting the domain curves along nodal sections, in situations
where such sections occur uniformly in the moduli space.

After briefly formalizing this splitting operation, we present the second series
of main results of this paper, the reverse procedure of gluing a pair of
punctured sections, followed by its treatment in punctured Gromov-Witten theory.
We end this section with an application to the degeneration situation of
\cite{decomposition}.

{Throughout this section, $X \to B$ denotes a morphism of fs logarithmic
schemes fulfilling the assumptions stated at the beginning of
\S\ref{sec:stack}.}

\subsection{Splitting punctured maps}
\label{ss:splitting-general}

We first discuss the operation of splitting of punctured curves along nodal
sections.

\begin{definition}
\label{Def: nodal sections}
A \emph{nodal section} of a family of nodal curves $\ul\pi:\ul C\arr\ul W$ is a
section $s:\ul W\arr \ul C$ of $\ul\pi$ that \'etale locally in $\ul W$ factors
over the closed embedding defined by the ideal $(x,y)$ in the domain of an
\'etale map
\[
\bSpec \cO_W[x,y]/(xy)\arr \ul C.
\]
The \emph{partial normalization of $\ul C/\ul W$ along $s$} is the map
\begin{equation}
\label{Eqn: partial normalization at nodal section}
\ul \kappa:\tilde {\ul C}\arr \ul C
\end{equation}
that \'etale locally is given by base change from the normalization of the plane
nodal curve $\Spec\kk[x,y]/(xy)$. We say $s$ is \emph{of splitting type} if the
two-fold unbranched cover $\ul\kappa^{-1}\big(\im(s)\big)\arr \im(s)$ is
trivial.
 
A nodal section of a punctured curve $(C^\circ/W,\bp)$ or punctured map
$(C^\circ/W,\bp,f)$ is a nodal section of the underlying curve
$\ul C/\ul W$.
\end{definition}

Note that a nodal section $s$ of a nodal curve $\ul C/\ul W$ with partial
normalization $\ul\kappa:\tilde{\ul C}\arr \ul C$ and nodal locus $\ul Z=\im(s)$
exhibits $\ul C$ as the fibered sum
\begin{equation}
\label{Eqn: fibred sum node splitting}
\ul Z\amalg_{\ul\kappa^{-1}(\ul Z)}\tilde{\ul C}
\stackrel{\simeq}{\longrightarrow} \ul C.
\end{equation}

A punctured curve can be split along a nodal section of splitting type:

\begin{proposition}
\label{prop:splitting-curve}
Let $\ul\kappa:\tilde{\ul C}\arr \ul C$ be the partial normalization of a
punctured curve $(\pi:C^\circ\arr W,\bp)$ defined by the splitting at a nodal
section $s$ of splitting type. Let $p_1,p_2: \ul W\arr \tilde{\ul C}$ be two
sections of $\ul\kappa^{-1}(\im(s)\big)\arr \im(s)$ with disjoint images.

Then
\[
(\tilde C^\circ,\tilde\bp)=
\big(\tilde\pi: (\tilde{\ul C},\ul\kappa^*\cM_{C^\circ})
\stackrel{\kappa}{\to} C^\circ\to W,\{\hat\bp,p_1,p_2\}\big)
\]
with $\hat\bp:\ul W\to\ul{\tilde C}$ the unique set of sections with
$\bp=\ul\kappa\circ\hat\bp$, is a (possibly disconnected) punctured curve.
\end{proposition}

\begin{proof}
Since $\kappa:\tilde C^\circ\arr C^\circ$ is an isomorphism away from
$\im(p_1)\cup\im(p_2)$, it suffices to consider a neighborhood of a geometric
point $\ol p\arr \tilde {\ul C}$ of one of $\im(p_i)$, say $i=1$. Denote by
$\ol q=\ul\kappa\circ\ol p$ the corresponding geometric point of $\ul C$, thus
a geometric point of the image of the nodal section. By the structure of log
smooth curves, $\cM_{C^{\circ},\ol q}$ is generated by
$(\pi^*\cM_{W})_{\ol q}$, $s_x$ and $s_y$, where $s_x,s_y\in
\cM_{C^{\circ},q}$ are induced by the coordinates $x,y$ in Definition~\ref{Def:
nodal sections}. These are subject to the relation $s_x s_y=s_{\rho}$ for some
$s_{\rho}\in (\pi^*\cM_W)_{\ol q}$. Hence $(\pi^*\cM_W)_{\ol q}$ and $s_y$
locally generate $\cM_{C^{\circ}}^{\gp}$ as a group, with
$s_x=s_{\rho}s_y^{-1}$. Pulling back to $\tilde{\ul{C}}$, along the branch
$x=0$, hence with $y=0$ giving $\im(p_1)$, we see that
$(\kappa^*\cM_{C^{\circ}})^{\gp}$ is locally generated by
$(\tilde\pi^*\cM_W)_{\ol p}$ and $\kappa^\flat s_y$. Further,
$\kappa^\flat s_y$ is also a section of $\cP$, the divisorial log
structure given by $p_1$, and the image of $\kappa^\flat s_y$ in
$\overline{\cP}$ generates $\overline{\cP}$ as a monoid. Thus locally near $\ol
p$, 
\[
\tilde\pi^*\cM_W\oplus_{\cO^{\times}_{\widetilde C}}
\cP \subseteq \kappa^*\cM_{C^{\circ}} \subset
\tilde\pi^*\cM_W\oplus_{\cO^{\times}_{\widetilde C}}
\cP^{\gp}.
\]
Further, any local section of $\kappa^*\cM_{C^{\circ}}$ not contained in
$\tilde\pi^*\cM_W\oplus_{\cO^{\times}_{\widetilde C}} \cP$ can be written
in the form $s_x^as_y^bs_W$ with $a>0$, $b\ge 0$ and $s_W$ a local section of
$\tilde\pi^*\cM_W$. Since $\alpha(s_x)=0$ when $x=0$, we see that $\alpha$
applied to any such element is zero. Thus $(\tilde{C^{\circ}}/W,\tilde\bp)$ is a
punctured curve near $\ol p$.
\end{proof}

For the application to moduli spaces of punctured maps we formalize the
splitting procedure as an operation on graphs, hence on (global) types of
punctured maps.

\begin{definition}
\label{Def: graph splitting}
Let $G$ be a connected graph and $\bE\subseteq E(G)$ a subset of edges.
Replacing each $E\in \bE$ by a pair of legs $L_E, L'_E$ leads to a graph
$\widehat G$ with
\[
V(\widehat G)= V(G),\ E(\widehat G)=E(G)\setminus \bE,\
L(\widehat G)= L(G)\cup \{L_E,L'_E\}_{E\in\bE}.
\]
We call the collection of connected subgraphs $G_1,\ldots,G_r$ of $\widehat G$
\emph{the graphs obtained from $G$ by splitting along $\bE$}.

There is an obvious induced notion of splitting of a genus-decorated graph
$(G,\bg)$, of a (global) type $\tau$, or of a (global) decorated type $\btau$ of
a punctured map along a subset of edges of the corresponding graphs.
\end{definition}

\begin{proposition}
\label{prop:splitting-map}
{Let $X\arr B$ be a morphism of fs logarithmic schemes over $\kk$
fulfilling the assumptions stated at the beginning of \S\ref{sec:stack}.}
Let $\tau_1,\ldots,\tau_r$ be obtained from splitting a global type
$\tau=(G,\bg,\bsigma,\bar\bu)$ of a punctured map to $X/B$ along a subset of
edges $\bE\subseteq E(G)$. Then the splitting morphism from
Proposition~\ref{prop:splitting-curve} followed by pre-stabilization
(Proposition~\ref{prop:pre-stable}) defines morphisms of stacks
\[
\textstyle
\fM(\cX/B,\tau)\arr \prod_i \fM(\cX/B,\tau_i),\quad
\scrM({X/B},\tau)\arr \prod_i \scrM({X/B},\tau_i),
\]
with the products understood as fiber products over $B$.

Analogous results hold for decorated types and for moduli spaces of weakly
marked punctured maps.
\end{proposition}

\begin{proof}
The statement is immediate from Proposition~\ref{prop:splitting-curve} and
Proposition~\ref{prop:pre-stable}.
\end{proof}

{
\begin{example}
\begin{figure}
\input{splitting.pspdftex}
\caption{Tropical splitting}
\label{Fig: Tropical splitting}
\end{figure}
As an illustration of the splitting procedure consider the degeneration of
$\PP^1\times\PP^1$ to two copies of $\PP^2$ constructed as follows. Take the
polyhedral decomposition $\mathscr P$ of $\RR^2$ with two vertices at $(0,0)$,
$(1,1)$ and four maximal cells given by the dashed part of Figure~\ref{Fig:
Tropical splitting}. Embed $\RR^2$ as affine hyperplane $\RR^2\times\{1\}$ in
$\RR^3$ and take the closures of the cones {over cells} of $\mathscr P$
to define a fan $\Sigma$ in $\RR^3$ with support $|\Sigma|=
\RR^2\times\RR_{\ge0}$. The corresponding toric threefold $X$ comes with a flat
morphism
\[
\pi:X\arr\AA^1
\]
induced by the projection $|\Sigma|\arr \RR_{\ge0}$ to the last coordinate. It
is not hard to show that $\pi^{-1}(\AA^1\setminus\{0\})=
(\PP^1\times\PP^1)\times(\AA^1\setminus\{0\})$, a trivial family, and
$\pi^{-1}(0)=\PP^2\amalg_{\PP^1}\PP^2$, a gluing of two copies of $\PP^2$ along
a pair of toric divisors.

Figure~\ref{Fig: Tropical splitting} on the left shows the tropicalization of a
family of curves of bidegree~$(1,1)$ giving a type $\tau$. The figure shows the
intersection with the affine hyperplane $\RR^2\times\{1\}$. Splitting along the
edge $E$ yields the two types $\tau_1,\tau_2$ whose general members are depicted
on the right. Note also that the leg in $\tau_2$ obtained from splitting $\tau$
at $E$ extends to the boundary of the cell, while this is not true for $\tau_1$.
This illustrates the necessity of pre-stabilization in the splitting procedure.

The opposite process of shrinking legs to an edge of a tropical domain curve
appears in gluing, see Remark~\ref{Rem: ghost sheaf and
tropical description of gluing}.
\end{example}}


\subsection{Gluing punctured maps to \texorpdfstring{$\cX/B$}{cX/B}}
\label{ss:Gluing punctured maps to cX}
\subsubsection{Notation for splitting edges}
In this section we work in categories of spaces over $\ul B$ or $B$.
In particular, products are to be understood as fiber products over
$\ul B$ or $B$, as appropriate.

\sloppy
Let $\tau=(G,\bg,\bsigma,\bar\bu)$ be a global type of punctured tropical maps
and $\tau_i=(G_i,\bg_i,\bsigma_i,\bar\bu_i)$, $i=1,\ldots,r$, the global types
obtained by splitting $\tau$ at a subset $\bE\subseteq E(G)$ of edges
(Definition~\ref{Def: graph splitting}). We choose an orientation on each edge
$E\in \bE$ and refer to the two legs obtained by splitting {the edge $E$
with vertices $v,v'$} by the corresponding half-edges $(E,v)$, $(E,v')$, with
$E$ oriented from $v$ to $v'$.\footnote{We use this notation as it is easy to
parse, but note that $(E,v)$ is ambiguous if $E$ is a loop. It will always be
clear from the context how to fix this ambiguity with a heavier notation.}
Denote by $\bL\subseteq\bigcup_i L(G_i)$ the subset of all legs obtained from
splitting edges, and by $i(v)\in\{1,\ldots,r\}$ for $v\in V(G)$ the index $i$
with $v\in V(G_i)$.

\fussy
\subsubsection{The stack $\tfM'(\cX/B,\tau)$ and its evaluation morphism}
\label{Sec:stacks of maps with sections}
Evaluation at the nodal sections for $\bE$ defines the morphism
\[
\textstyle
\ul\ev_\bE: \ul{\fM'}(\cX/B,\tau)\arr \prod_{E\in\bE} \ul\cX.
\]
\sloppy
For each $E\in\bE$ denote by $\fM'_E(\cX/B,\tau)$ the image of the nodal section
$s_E:\ul{\fM'}(\cX/B,\tau)\arr \ul{\fC'}^\circ(\cX/B,\tau)$ with the restriction
of the log structure on the universal domain $\fC'^\circ(\cX/B,\tau)$.
Denote further by $\tfM'(\cX/B,\tau)$ the fs fiber product
\begin{equation}
\label{Eqn: tilde log structure}
\tfM'(\cX/B,\tau) = \fM'_{E_1}(\cX/B,\tau)\times^\fs_{\fM'(\cX/B,\tau)} \cdots
\times^\fs_{\fM'(\cX/B,\tau)} \fM'_{E_r}(\cX/B,\tau),
\end{equation}
\fussy
where $E_1,\ldots,E_r\in E(G)$ are the edges in $\bE$.\footnote{{Note that we
have suppressed the dependence of the stack on $\bE$ from the notation.}} With
this enlarged log structure, the pull-back $\tfC'^\circ(\cX/B,\tau)\arr
\tfM'(\cX/B,\tau)$ of the universal domain has sections $\tilde s_E$, $E\in
\bE$, in the category of log stacks. Moreover, $\ul\ev_\bE$ lifts to a
logarithmic evaluation morphism
\begin{equation}
\label{Eqn: ev_bE}
\textstyle
\ev_\bE: \tfM'(\cX/B,\tau) \arr \prod_{E\in \bE} \cX,
\end{equation}
with $E$-component equal to $\tilde f\circ \tilde s_E$ for $\tilde f:
\tfC'^\circ(\cX/B,\tau)\arr \cX$ the universal punctured morphism.

\subsubsection{The stacks $\tfM'(\cX/B,\tau_i)$, evaluation and splitting
morphisms}
Similarly, for each of the global types $\tau_i=(G_i,\bg_i,\bsigma_i,\bar\bu_i)$
obtained by splitting and $L\in L(G_i)$, denote by $\fM'_L(\cX/B,\tau_i)$ the
image of the punctured section $s_L:\ul\fM'(\cX/B,\tau_i)\arr
\ul{\fC'}^\circ(\cX/B,{\tau_i})$ defined by $L$, again endowed with the
pull-back of the log structure on ${\fC'}^\circ(\cX/B,{\tau_i})$. With
$L_1,\ldots,L_s$ the legs of $G_i$ {obtained from splitting,} define
the stack
\[
\widetilde\fM'(\cX/B,\tau_i)=
\big(\fM'_{L_1}(\cX/B,\tau_i)\times^\mathrm{f}_{\fM'(\cX/B,\tau_i)}
\cdots \times^\mathrm{f}_{\fM'(\cX/B,\tau_i)} \fM'_{L_s}(\cX/B,\tau_i)\big)^\sat,
\]
where $\sat$ denotes saturation{, bearing in mind that the log
structures on the stacks $\fM'_{L_j}(\cX/B,\tau_i)$ are not saturated.}

This stack differs from $\fM'(\cX/B,\tau_i)$ by adding the pull-back of the log
structure of each puncture {obtained from splitting}, so that the
pull-back $\widetilde {\fC'}^\circ(\cX/B,{\tau_i}) \arr
\widetilde\fM'(\cX/B,{\tau_i})$ of the universal curve now has
punctured sections in the category of log stacks. We define the evaluation
morphism
\begin{equation}
\label{Eqn: ev_bL}
\textstyle
\ev_\bL: \prod_{i=1}^r\tfM'(\cX/B,\tau_i) \arr \prod_{E\in\bE} \cX\times \cX,
\end{equation}
by taking as $E$-component the evaluation at the corresponding two sections
$s_{E,v},s_{E,v'}$, observing the chosen orientation of $E$.

{
\begin{lemma}
\label{Lem: enhanced splitting-map}
The splitting morphism $\fM(\cX/B,\tau)\arr \prod_i \fM(\cX/B,\tau_i)$ in
Proposition~\ref{prop:splitting-map} lifts to a morphism
\begin{equation}
\label{Eqn: splitting morphism}
\textstyle
\tfM(\cX/B,\tau)\arr \prod_{i=1}^r \tfM(\cX/B,\tau_i).
\end{equation}
Analogous statements hold for weak markings and for the moduli spaces of stable
maps to $X$ rather than $\cX$.
\end{lemma}

\begin{proof}
\sloppy
We only treat the case of marked moduli spaces of punctured maps to $\cX$, the
other cases being completely analogous.

It suffices to produce a morphism
\[
\fM_{E}(\cX/B,\tau) \arr \fM_{L}(\cX/B,\tau_i)
\]
lifting $\fM(\cX/B,\tau)\arr \fM(\cX/B,\tau_i)$
whenever $L=(E,v)\in L(G_i)$ is one of the two legs obtained from splitting $E$.
Indeed, this then provides a morphism of fibered products, which lifts to the
saturation by functoriality of saturation.

To construct this lifting let $\fC^\circ \to \fM:= \fM(\cX/B,\tau)$ be the
universal curve, and $\tilde\fC^\circ\arr\fC^\circ$ the splitting of all nodes
labelled by an element of $\bE$, strict as a morphism of log stacks. The graph
$G_i$ given by $\tau_i$ selects a connected component $\tilde\fC_i^\circ\subset
\tilde\fC^\circ$, and the nodal section $s_E$ lifts to a punctured section
$\tilde s_i: \ul\fM\arr \ul{\tilde\fC_i^\circ}$. Let similarly $\fC_L^\circ \to
\fM_i:= \fM(\cX/B,\tau_i)$ and $s_L: \ul \fM_i\arr \ul\fC_L$ the corresponding
universal curve and punctured section over $\fM_i$. Then $ \fM_{E}(\cX/B,\tau) =
\ul\fM \times_{\ul{\tilde\fC_i^\circ}} \tilde\fC_i^\circ$ since
$\tilde\fC_i^\circ\arr \fC^\circ$ is strict, and similarly
$\fM_{L}(\cX/B,\tau_i) = \ul\fM_i \times_{\ul \fC_L^\circ}\fC_L^\circ$. Now there
is a canonical morphism $\tilde\fC_i^\circ \to \fC_L^\circ$ lifting
$\fM(\cX/B,\tau)\arr \fM(\cX/B,\tau_i)$ --- the prestabilization morphism as a
punctured map. Pulling back we obtain the desired morphism $\fM_{E}(\cX/B,\tau)
\to \fM_{L}(\cX/B,\tau_i)$.
\end{proof}
}

{We next show that} enlarging the log structures for the punctures may
change the structure of the underlying stacks, but only by nilpotents in the
structure sheaf.

For $\bS\subseteq E(G)\cup L(G)$ we unify the notation, denoting by
$\tfM'(\cX/B,\tau)\arr \fM'(\cX/B,\tau)$ the corresponding fiber product over
both nodal and punctured sections. In this generality we have:

\begin{proposition}
\label{Prop: tilde log structures don't change reductions}
Let $\tau=(G,\bg,\bsigma,\bar\bu)$ be a global type of punctured maps,
$\bS\subseteq E(G)\cup L(G)$ and $\tfM'(\cX/B,\tau)$ the corresponding
stack of weakly $\tau$-marked punctured maps to $\cX/B$ with sections. Then the
canonical map
\[
\tfM'(\cX/B,\tau)\arr \fM'(\cX/B,\tau)
\]
induces an isomorphism on the reductions {of their underlying
stacks}. If moreover $\bS\subseteq E(G)$, the canonical map is an isomorphism
{on underlying stacks}.

Analogous results hold for the marked and decorated versions.
\end{proposition}

\begin{proof}
Going inductively, it suffices to treat the case that $\bS$ has only one
element. The case $\bS=\{E\}$ is an edge leads to the problem of going over from
a monoid $Q$ to the saturation of a monoid of the form $Q\oplus_\NN \NN^2$ with
$1\in\NN $ mapping to $(1,1)\in\NN^2$. Since the morphism $\NN\arr\NN^2$ is
saturated and integral, $Q\oplus_\NN \NN^2$ is saturated and integral as well by
\cite[Prop.\,I.4.8.5, Prop.\,I.4.6.3]{Ogus}. In particular, the fs fiber
product in \eqref{Eqn: tilde log structure} agrees with the ordinary fiber
product, and only changes the log structure.

For $\bS=\{L\}$ a leg, we need to take the saturation of the strict subspace
given by a punctured section. {Let $(C^\circ/W,\bp,f)$ be
a punctured map to $\cX/B$ with $W=\Spec(Q\arr A)$. Let $\ul W\arr \Spec\kk[Q^\circ]$ be a chart for the log structure induced by the punctured section corresponding to the leg $L$, with $Q^\circ\subset Q\oplus \ZZ$.} Then necessarily
the induced map $Q^{\circ} \arr A$ takes $Q^{\circ} \setminus (Q \oplus {0})$ to
zero.

The saturation of $\Spec(Q^\circ\arr A)$ equals $W' = \Spec (Q'\arr A')$ with
$Q'$ the saturation of $Q^{\circ}$ and $A' = A\otimes_{k[Q^{\circ}]} k[Q']$.
Necessarily, if $m\in Q' \setminus Q^{\circ}$ then $m\in Q \oplus \ZZ_{<0}$, and
so {its image} $z^m\in A'$ is nilpotent {(following the notation of \S\ref{Sec:convention})}. It is then immediate that $A \arr A'_{\red}$ is
surjective. This map factors through $A_{\red} \arr A'_\red$, so the latter is
surjective. Thus $W'_\red$ is a closed subscheme of $W_\red$. On the other hand,
by \cite[Prop.\,III.2.1.5]{Ogus} saturation is always a surjective
morphism, and hence $W'_\red \arr W_\red$ is an isomorphism.
\end{proof}

By Proposition~\ref{Prop: tilde log structures don't change reductions} the Chow
theories of the moduli stacks of punctured maps do not change by enlarging the
log structures. We can thus freely use the enlarged log structures in discussing
gluing.

We are now in position to state the central technical gluing result. It explains
how a $\tau$-marked punctured map is equivalent to giving a collection of
$\tau_i$-marked punctured maps obeying a logarithmic matching condition.

\begin{theorem}
\label{Thm: Gluing theorem}
{Let $X\arr B$ be a morphism of fs logarithmic schemes over $\kk$
fulfilling the assumptions stated at the beginning of
\S\ref{sec:stack}{, and assume $X$ is simple}.} Let
$\tau_1,\ldots,\tau_r$ be the global types of punctured maps
(Definition~\ref{Def: global type}) obtained by splitting a global type
$\tau=(G,\bg,\bsigma,\bar\bu)$ along a subset of edges $\bE$. Then the
commutative diagram
\[
\xymatrix{
\tfM'(\cX/B,\tau) \ar[r]^(.43){\delta_\fM}\ar[d]_{\ev_\bE} &
\prod_{i=1}^r \tfM'(\cX/B,\tau_i)\ar[d]^{\ev_\bL}\\
\prod_{E\in \bE} \cX\ar[r]^(.43)\Delta& \prod_{E\in \bE} \cX\times \cX 
}
\]
with $\Delta$ the product of diagonal embeddings and the other arrows defined in
\eqref{Eqn: ev_bE}, \eqref{Eqn: ev_bL}, and \eqref{Eqn: splitting morphism}, is
cartesian in the category of fs log stacks. We remind the reader that all
products in this square are taken over $B$.

An analogous statement holds for $\tau$ replaced by a decorated global type
$\btau=(\tau,\bA)$.
\end{theorem}

\begin{remark}
We note that it is important that we use the weakly marked moduli spaces
here. Indeed, there exist simple examples of (strongly) marked
punctured maps which may be glued to {obtain} a punctured map which is
only weakly marked. This arises as saturation issues in 
the above fiber product description may introduce nilpotents.
{For
an explicit example, see \cite[Ex.\,4.5]{G22}. We also note that this
is essentially the same saturation issue as in Remark~\ref{rem:weak vs strong},
and the examples are closely related.}
\end{remark}

The proof of the theorem, given further below, is based on the following gluing
result for punctures with a section.

\begin{lemma}
\label{Lem: Local construction of nodes}
Let $W$ be an fs log scheme and $U_i^\circ$ a puncturing along $\{0\}\times W$
of strict open neighborhoods $U_1,U_2\subseteq \AA^1\times W$ of $\{0\}\times
W$, $i=1,2$. {Here $\AA^1$ is endowed with its toric log structure.}
Furthermore let $s_i: W\arr U^\circ_i$ be sections with schematic image
$\{0\}\times\ul W$ of the composition $U_i^\circ\arr U_i\arr W$ of the
puncturing map and the projection.

Then there exists an enlarged puncturing $\hat U^\circ_i \to U^\circ_i \to U_i$
through which the sections $s_i$ factor, and a unique log smooth curve $\pi:
U\to W$ with maps
\[
\iota_1: \hat U_1^\circ\arr U,\qquad \iota_2: \hat U_2^\circ\arr U
\]
over $W$ inducing an isomorphism of underlying schemes $\ul
U_1\amalg_{\{0\}\times \ul W} \ul U_2\simeq \ul U$, strict away from
$\{0\}\times \ul W$, and such that $\iota_1\circ \hat s_1=\iota_2\circ \hat
s_2$, with $\hat s_i$ the lifts of $s_i$.
\end{lemma}

\begin{remark}
\label{Rem: Uniqueness of puncture enlargement for gluing}
The lifting of $s_i$ to $\hat U_i^\circ$ is unique. The enlarged puncturing
$\hat U_i^\circ$ is not unique, but may be chosen uniquely if we require that
$\hat U_i^\circ \to U_i^\circ\times U$ is prestable. We obtain a pushout diagram
up to unique punctured enlargement:
\[
\xymatrix{ &U_1^\circ & \hat U_1^\circ \ar@{.>}[l] \ar@{.>}[dr]^{\iota_1}\\ 
W \ar[ur]^{s_1}\ar@{.>}[urr]_{\hat s_1}
\ar[dr]_{s_2}\ar@{.>}[drr]^{\hat s_2} &&&U\\  
&U_2^\circ & \hat U_2^\circ \ar@{.>}[l] \ar@{.>}[ur]_{\iota_2}}
\]
\end{remark}
\begin{proof}[Proof of Lemma~\ref{Lem: Local construction of nodes}.]
The statement is about the unique definition of the log structure on $\ul U$
near the nodal locus $\{0\}\times \ul W\subset \ul U$. Since this is a local
question we can restrict attention to a neighborhood of a geometric point
{$\ol q=(0,\ol w)$ of $\{0\}\times \ul W$}. By the definition of
puncturing, the linear coordinate of $\AA^1$ defines elements
$\sigma_x\in\cM_{U^\circ_1,\ol q}$, $\sigma_y\in \cM_{U^\circ_2,\ol q}$.

Now assume that $U=(\ul U,\cM_U)\arr W$ is a log smooth curve with the required
properties {for some $\hat U_i^\circ$. Since $\hat U_i^\circ$, $U_i^\circ$ are both puncturings of $U_i$ we may identify $\ocM_{\hat U_i^\circ}^\gp= \ocM_{U_i^\circ}^\gp=\ocM_{ U_i}^\gp$.} Then
\[
\ol \iota_i^\flat:\ocM_{U,\ol q}^\gp\arr{ \ocM_{\hat U_i^\circ,\ol q}^\gp=}
\ocM_{U_i^\circ,\ol q}^\gp=
\ocM_{W,\ol w}^\gp \oplus\ZZ
\]
is an isomorphism with $\ocM_{W,\ol w}\oplus\NN \subseteq
\ol \iota_i^\flat(\ocM_{U,\ol q})$. Thus there exist $\tilde \sigma_x,\tilde
\sigma_y\in \cM_{U,\ol q}$ with
\[
\sigma_x= \iota_1^\flat(\tilde \sigma_x),\qquad
\sigma_y= \iota_2^\flat(\tilde \sigma_y).
\]
An important property of log smooth structures at nodes is that logarithmic
lifts of given local coordinates at the two branches of the node become unique
if one requires their product to lie in $\pi^\flat(\cM_{W,\ol w})$
\cite[\S3.8]{Mochizuki}. With this condition imposed on $\tilde \sigma_x$,
$\tilde \sigma_y$, we now obtain a unique element $\sigma_q\in \cM_{W,\ol w}$
with
\begin{equation}
\label{Eqn: nodal relation}
\tilde \sigma_x\cdot \tilde \sigma_y = \pi^\flat(\sigma_q).
\end{equation}
Under the assumption of the existence of { factorizations $\hat s_1,\hat s_2$ of the} sections $s_1,s_2$, we can compute
$\sigma_q$ from $\sigma_x$ and $\sigma_y$ as follows: With $\iota_1\circ
{\hat s_1}=\iota_2\circ {\hat s_2}$ we obtain
\[
\sigma_q= (\iota_1\circ {\hat s_1})^\flat\big(\pi^\flat(\sigma_q)\big)= (\iota_1\circ
{\hat s_1})^\flat(\tilde \sigma_x)\cdot (\iota_2\circ {\hat s_2})^\flat(\tilde \sigma_y)=
s_1^\flat(\sigma_x)\cdot s_2^\flat(\sigma_y).
\]
Note also that $\cM_{U,\ol q}$ is generated by $(\pi^*\cM_W)_{\ol
q}$ and $\tilde\sigma_x, \tilde\sigma_y$, with single relation \eqref{Eqn:
nodal relation}.

Conversely, we can define the structure of a log smooth curve at $\ol
q\arr \ul U$ with the requested properties simply by defining
\begin{equation}
\label{Eqn: Defining nodal relation}
\sigma_q= s_1^\flat(\sigma_x)\cdot s_2^\flat(\sigma_y),
\end{equation}
and
\[
\cM_{U,\ol q}:= (\pi^*\cM_W)_{\ol q}\oplus_\NN \NN^2,
\]
with the generator $1\in\NN$ {in the fibered sum} mapping to
$\pi^\flat(\sigma_q)\in (\pi^*\cM_W)_q$ and to $(1,1)\in\NN^2$, respectively.
The structure morphism
\[
\cM_{U,\ol q}\arr \cO_{U,\ol q}
\]
is defined by the structure morphism of $W$ on the first summand, {and}
by mapping $(a,b)\in\NN^2$ to $x^a y^b$ when writing ${\ul U\subseteq
W\times_\ZZ \Spec \ZZ[x,y]/(xy)}$. Since the projection
$U_i^\circ\setminus(\{0\}\times \ul W) \to W$ is strict, this log structure near
$\ol q$ patches uniquely to the given log structure on
$U_i^\circ\setminus(\{0\}\times \ul W)$ to define the desired log smooth curve
$U\arr W$.

The morphisms $\iota_i:\widehat U^\circ_i\arr U$ are then given by
\begin{equation}
\label{Eqn: U^circ_i->U}
\begin{aligned}
(\iota_1^\flat)_{\ol q}:\cM_{U,\ol q}
&\arr \cM_{U^\circ_1,\ol q}^\gp,\quad&(1,0)&
\longmapsto \sigma_x,&\quad (0,1)&
\longmapsto \sigma_x^{-1} \pi^\flat(\sigma_q),\\
(\iota_2^\flat)_{\ol q}:\cM_{U,\ol q}&
\arr \cM_{U^\circ_2,\ol q}^\gp,&(1,0)&
\longmapsto \sigma_y^{-1} \pi^\flat(\sigma_q),& (0,1)&
\longmapsto \sigma_y.
\end{aligned}
\end{equation}
These definitions are forced upon us by the structure homomorphisms on
$U^\circ_i$ and by the defining relation \eqref{Eqn: Defining nodal relation}
for $\cM_{U,\ol q}$. If $\sigma_x^{-1} \pi^\flat(\sigma_q)\not\in
\cM_{U^\circ_1,q}$, we may have to enlarge the puncturing of $U^\circ_1$ for
this map to define $\hat U_1^\circ\arr U$, and similarly for $\hat U^\circ_2$;
if we choose the enlargement to be generated by $\sigma_x^{-1}
\pi^\flat(\sigma_q)$ it is uniquely defined. Note that by \eqref{Eqn: Defining
nodal relation}, the image of $\sigma_q$ under the structure morphism is $xy=0$,
and hence this enlargement of puncturing is possible. Note also that $s_1$
factors uniquely over this extension of puncturing since by \eqref{Eqn:
Defining nodal relation},
\[
(s_1^\flat)^\gp \big(\sigma_x^{-1}\pi^\flat(\sigma_q)\big)=
s_2^\flat(\sigma_y),
\]
and similarly for $s_2$. Finally, to check the equality $\iota_1\circ s_1=
\iota_2\circ s_2$ we compute
\[
(s_1^\flat\circ \iota_1^\flat)(1,0)= s_1^\flat(\sigma_x)
= s_2^\flat(\sigma_y)^{-1}\sigma_q
= s_2^\flat(\sigma_y^{-1}\pi^\flat(\sigma_q)) =
(s_2^\flat\circ\iota_2^\flat)(1,0),
\]
and similarly for $(1,0)$ replaced by $(0,1)$. This shows the claimed properties
for $U\to W$ and $\iota_1,\iota_2$. Uniqueness follows from the discussion at
the beginning of the proof.
\end{proof}

\begin{remark}
\label{Rem: ghost sheaf and tropical description of gluing}
It is worthwhile to understand the gluing construction of a pair of punctured
points to a node on the level of ghost sheaves and in terms of the dual tropical
picture. The relevant monoids are
\[
Q=\ocM_{W,\ol w},\quad Q_i=\ocM_{U^\circ_i,\ol q}\subset Q\oplus\ZZ,
\]
and their duals
\[
\omega=\Hom(Q,\RR_{\ge0}),\quad
\tau_i=\Hom(Q_i,\RR_{\ge0})\subset\omega\times\RR_{\ge0}.
\]
We choose the embedding $Q_i\subset Q\oplus\ZZ$ such that $\bar\pi^\flat$
identifies $Q$ with $Q\oplus\{0\}$, while the puncturing log structure is
generated by $(0,1)\in Q\oplus\ZZ$. The sections $s_i$ define left-inverses
\[
\bar s_i^\flat: Q_i\arr Q
\]
to $\bar\pi^\flat$. Now the point of the
gluing construction is that there are exactly two automorphisms $\phi$ of
$Q^\gp\oplus\ZZ$ making the following
diagram of monoids commutative:
\[
\xymatrix@C=0pt@R=15pt{
Q\ar[rr]^\id\ar[d]_{\bar\pi^\flat}&&Q\ar[d]^{\bar\pi^\flat}\\
Q_1\ar[d]&& Q_2\ar[d]\\
Q^\gp\oplus\ZZ\ar[rr]^\phi\ar[rd]_{\bar s_1^\flat} &&
Q^\gp\oplus\ZZ\ar[ld]^{\bar s_2^\flat}\\
& Q^\gp
}
\]
Indeed, by commutativity of the square, $\phi(m,0)=(m,0)$ for all $m\in Q^\gp$.
Define $\rho_i=\bar s_i^\flat(0,1)$, $i=1,2$, and $\rho_q$ by
$\phi(0,1)=(\rho_q,d)$. Then $d=\pm1$ since $\phi(0,1)$ together with
$Q^\gp\oplus\{0\}$ generates $Q^\gp\oplus\ZZ$. This sign determines the two
possibilities. Commutativity of the triangle now shows
\[
\rho_1=\bar s_1^\flat(0,1)= \bar s_2^\flat(\rho_q,\pm1)= \rho_q \pm \rho_2.
\]
The situation obtained by splitting a node into two punctures produces the
negative sign. With this choice we obtain an isomorphism of the submonoid
\begin{figure}
\begin{center}
\input{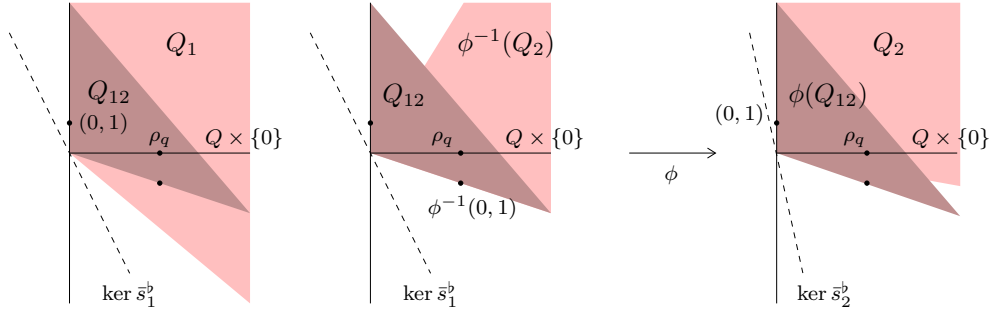}
\caption{\small The monoids $Q_1,Q_2, Q_{12}\subset Q\oplus\ZZ$ and their
comparison under $\phi: Q\oplus \ZZ\arr Q\oplus\ZZ$. {The hatched area
depicts $Q_{12}$, the solid shading $Q_1$, $Q_2$ or $\phi^{-1}(Q_2)$.} Note that
$\phi(Q_{12})=Q_{12}$ because both cones are spanned by $Q\times\{0\}$, $(0,1)$
and $\rho_q-(0,1)$. In the sketched situation, the puncturing for $U^\circ_2$
has to be enlarged, the one for $U^\circ_1$ does not.}
\label{Fig: Q12}
\end{center}
\end{figure}
$Q_{12}\subset Q^\gp\oplus\ZZ$ generated by $Q\oplus\NN$ and
$\phi^{-1}(Q\oplus\NN)$ with $Q\oplus_\NN \NN^2$, with $1\in\NN$ mapping to
$\rho_q$ and $(1,1)$, respectively. The defining equation $\rho_q=\rho_1+\rho_2$
retrieves \eqref{Eqn: Defining nodal relation} in the proof of Lemma~\ref{Lem:
Local construction of nodes} on the level of ghost sheaves. The change of
puncturing of $U^\circ_1$ becomes necessary if $Q_{12} \not\subset Q_1$, and
similarly if $\phi(Q_{12})\not\subset Q_2$ for $U^\circ_2$. Figure~\ref{Fig:
Q12} provides an illustration.

For the tropical interpretation, illustrated in Figure~\ref{Fig: tropical
gluing}, we have two factorizations
\[
\omega\stackrel{\Sigma(s_i)}{\longrightarrow} \tau_i
\stackrel{\Sigma(\pi)}{\longrightarrow} \omega,
\]
of $\id_\omega$. Here the second map is the projection to the first component
when writing $\tau_i\subseteq \omega\times\RR_{\ge0}$.
\begin{figure}
\begin{center}
\input{tau12.rev.pspdftex}
\caption{\small The dual tropical picture of Figure~\ref{Fig: Q12}. {The
hatched area covers $\omega_E$, the solid shading $\tau_1$, $\tau_2$ and
$\phi^t(\tau_2)$. The dashed line indicates the image of $\omega$ under
$\Sigma(s_1)$ or $\Sigma(s_2)$.}}
\label{Fig: tropical gluing}
\end{center}
\end{figure}
Thus $\Sigma(s_i)(h)=(h,\ell_i(h))$ for some {piecewise linear map}
\[
\ell_i:\omega\arr \RR_{\ge 0}.
\]
Thinking of $h$ as parametrizing a punctured tropical curve, $\ell_i(h)$
specifies a point on the puncturing interval or ray emanating from the unique
vertex $v_i$. The tropical glued curve then produces the metric graph with two
vertices $v_1,v_2$ by joining the two intervals at the specified points, hence
producing an edge $E$ of length $\ell_1(h)+\ell_2(h)$. The tropical glued curve
over $\omega$ thus has edge function $\ell:\omega\arr \RR_{\ge0}$ simply defined
by
\begin{equation}
\label{Eqn: length functions for gluing of nodes}
\ell=\ell_1+\ell_2.
\end{equation}
The process of producing the glued cone $\omega_E\subset\omega\times\RR_{\ge0}$
over $\omega$ is dual to the statement
$Q_{12}=(Q\oplus\NN)+\phi^{-1}(Q\oplus\NN)$:
\[
\omega_E=\Hom(Q_{12},\RR_{\ge0})= (\omega\times\RR_{\ge0})
\cap \phi^t(\omega\times\RR_{\ge0}).
\]
The change of puncturing is necessary if $\ell(h)$ is smaller than
either of the length functions obtained by tropicalizing the puncturing, or if
either one of $Q_{12}\cap Q_1$, $\phi(Q_{12})\cap Q_2$ is not saturated.
\end{remark}

We now turn to the proof of the gluing theorem for punctured maps to $\cX/B$.

\begin{proof}[Proof of Theorem~\ref{Thm: Gluing theorem}]
Write $\tau_i=(G_i,\bg_i,\sigma_i,\bar \bu_i)$. We check the universal
property of cartesian diagrams. 
\smallskip

\noindent
\textsc{Step~1: An object of the fibered product.}
Consider an fs log scheme $W$ with two morphisms
\begin{equation}
\label{Eqn: cartesianity check setup}
\textstyle
W\arr \prod_{E\in \bE} \cX,\qquad
W\arr \prod_{i=1}^r\tfM'(\cX/B,\tau_i)
\end{equation}
together with an isomorphism of the compositions to $\prod_E \cX\times \cX$.
Spelled out this means that (1)~for each $i=1,\ldots,r$ we have given a
{weakly} $\tau_i$-marked, pre-stable punctured map
\[
(\pi_i:C_i^\circ\to W,\bp_i,f_i:C_i^\circ\to \cX) 
\]
over $W$ and for each leg $(E,v)\in L(G_i)$ a section $s_{E,v}: W\to C_i^\circ$
with image the puncture labelled by the leg in $G_i$ generated by $E$; and (2)
the sections fullfill the logarithmic matching property
\begin{equation}
\label{Eqn: matching maps W->cX}
f_{i(v)}\circ s_{E,v} = f_{i(v')}\circ s_{E,v'},
\end{equation}
for each edge $E\in\mathbf{E}$ with adjacent vertices $v,v'$. Write $p_{E,v}$
for the strict closed subspace of $C^\circ_i$ defined by $(E,v)\in L(G_i)$.
\smallskip

\noindent
\textsc{Step~2. The glued curve.} 
{
Denote by $\ul C$ the family of nodal curves over $\ul W$ obtained by gluing
$\coprod_i \ul C_i$ schematically along pairs of punctures. Let $E\in\bE$ be an
edge with vertices $v,v'$, and $q_E$ the nodal section of $\ul C\arr\ul W$ given
by the image of the pair of punctures $p_{E,v}$, $p_{E,v'}$. Applying
Lemma~\ref{Lem: Local construction of nodes} \'etale locally near the image of
$q_E$ provides a local extension of the log structure defined by the $C_i^\circ$
away from $q_E$ to a log smooth curve over $W$.
} Thus there is a punctured curve
\[
(\pi: C^\circ\arr W,\bp)
\]
with underlying scheme $\ul C$ that replaces each pair of punctures $p_{E,v}$,
$p_{E,v'}$ in $\coprod_i C^\circ_i$, for an edge $E\in \bE$ with vertices
$v,v'$, by a node $q_E$. The lemma also provides a morphism of punctured curves
$\hat C^\circ_i\arr C^\circ_i$ with unique liftings $\hat s_{E,v}$ of each
section $s_{E,v}$ to $\hat C^\circ_{i(v)}$, and morphisms
\[
\iota_i: \hat C_i^\circ\arr C^\circ
\]
with $\iota_{i(v)}\circ \hat s_{E,v}= \iota_{i(v')}\circ \hat s_{E,v'}$, {and $\hat C_i^\circ$ equal to $C_i^\circ$ possibly up to  enlargement of the puncturing}. For
each edge $E\in \bE$ we can thus define the nodal section
\[
s_E:=\iota_{i(v)}\circ \hat s_{E,v}= \iota_{i(v')}\circ \hat s_{E,v'}:
W\arr C^\circ.
\]
\smallskip

\noindent
\textsc{Step~3. Gluing the tropical map.}  
Denote by $(\hat C^\circ_i,\hat\bp_i,\hat f_i)$ with $\hat f_i= f_i\circ (\hat
C_i^\circ\to C_i^\circ)$ the punctured stable map with the enlarged punctured
structure. It follows from the tropical description of the gluing construction
in Remark~\ref{Rem: ghost sheaf and tropical description of gluing} that the
tropicalizations
\[
\Sigma(\hat f_i):\Sigma(\hat C^\circ_i)\arr \Sigma(\cX)
\]
of $\hat f_i$ glue to a map of {generalized} cone complexes
\[
\Sigma(C^\circ)\arr \Sigma(\cX)
\]
which commutes with the map to $\Sigma(B)$. In fact, restricting to a geometric
point {$\ol w$ of $\ul W$} and adopting the notation from
Remark~\ref{Rem: ghost sheaf and tropical description of gluing}, at an edge
$E\in \bE$ with vertices $v_1,v_2$, the cone $\omega_E\subseteq
\omega\times\RR_{\ge0}$ of $\Sigma(C^\circ_{\ol w})$ is defined by the length
function $\ell_E=\ell_1+\ell_2$. Denote further $\sigma=(\ocM_{\cX,\ol
y}^\vee)_\RR$ for
\[
\ol y=  f_{i(v_1)}( s_{E,v_1}(\ol w)) =  f_{i(v_2)}(s_{E,v_2}(\ol w)). 
\]
Assuming $E$ oriented from $v_1$ to $v_2$, the contact orders obtained from
splitting $\tau$ at $E$ are related by
\[
u_{E,v_1}=u_E= -u_{E,v_2}\in\sigma^\gp_\ZZ.
\]
Now the map $\omega_E\arr \sigma\in \Sigma(\cX)$ can be defined by
\begin{equation}
\label{Eqn: map from omega_E}
\omega_E\ni (h,\lambda)\longmapsto V_1(h)+\lambda\cdot u_{E,v_1}
=V_2(h)+(\ell_E(h)-\lambda)\cdot u_{E,v_2},
\end{equation}
where $V_\mu:\omega\arr \sigma$ is the map for the vertex $v_\mu$ given by
$\Sigma(f_{i(v_\mu)})$, $\mu=1,2$. The image of this map lies in $\sigma$ since
the line segment $\{h\}\times[0,\ell_1(h)]$ is contained in $\tau_1\subseteq
\omega \times\RR_{\ge0}$ and $V_1(h)+\lambda\cdot u_{E,v_1}
=\Sigma(f_{i(v_1)})(h,\lambda)$, and similarly for the line segment $\{h\}\times
[\ell_1(h),\ell(h)]$ and $\Sigma(f_{i(v_2)})$. The equality in \eqref{Eqn: map
from omega_E} holds because
\[
V_1(h)+\ell_1(h)\cdot u_{E,v_1}= \Sigma(f_{i(v_1)}\circ s_{E,v_1})(h)
= \Sigma(f_{i(v_2)}\circ s_{E,v_2})(h) =V_2(h)+\ell_2(h)\cdot u_{E,v_2}.
\]
Note this last argument uses the assumption that $\bar u_E$ is
monodromy-free to assure that $u_{E,v_1}=-u_{E,v_2}$. This finishes the
construction of the map $\Sigma(C^\circ)\arr \Sigma(\cX)$.
\smallskip

\noindent
\textsc{Step~4. Gluing the punctured map.} 
In view of \cite[Prop.\,2.10]{decomposition}\footnote{{While
\cite[Prop.\,2.10]{decomposition} assumes a more restricted context, the proof
only uses that the Artin fan of the codomain is Zariski {(Definition~\ref{Def:
Zariski Artin fan})}. This is true here by simplicity of $X$ and our standing
assumptions on $B$.}}, we thus obtain a morphism $C^\circ\arr \cA_X$ over
$\cA_B$. By the same token, the composition $C^\circ\arr\cA_X\arr \cA_B$ agrees
with $C^\circ\arr B\arr\cA_B$. We thus obtain an induced morphism
\[
f: C^\circ\arr \cX= B\times_{\cA_B}\cA_X,
\]
commuting with the maps to $B$. By functoriality of this construction and the
tropical description of the gluing process, it holds $f\circ\iota_i= \hat f_i$
for all $i$.

The data $(C^\circ\arr W, f,\bp)$ and the collection of nodal sections $s_E$ now
define the desired morphism
\[
W\arr \tfM'(\cX/B,\tau).
\]
Indeed, splitting the domain $C^\circ$ at the nodes for edges $E\in\bE$ and
pre-stabilizing obviously retrieves the collection of pre-stable maps
$(C^\circ_i\to W,\bp_i,f_i)$ with compatible evaluation maps to $\cX$ and
sections $s_{E,v}$ that we started with. This finishes the existence part in
checking cartesianity.

Uniqueness follows from the uniqueness statement in Lemma~\ref{Lem: Local
construction of nodes}.
\end{proof}

\subsubsection{{Relative and absolute maps}}
We end this section by remarking that in many situations, working with all fiber
products over $B$ may be burdensome, as each product in the diagram of Theorem
\ref{Thm: Gluing theorem} is over $B$. In the standard degeneration situation
considered in \S\ref{sec:degeneration gluing} {below}, we might be
working over a standard log point $b_0$, and saturation issues even over $b_0$
can complicate the fiber product. Thus the following is generally useful.

\begin{proposition}
\label{Prop: absolute versus relative frakM}
Let $B$ be an affine log scheme equipped with a global chart $P\rightarrow
\cM_B$ inducing an isomorphism $P\cong \Gamma(B, \overline{\cM}_B)$. Let $\tau$
be a {global type of punctured} tropical map for $X/B$ {(Definition~\ref{Def: global type},(1))}, with underlying
graph connected.\footnote{Connectedness is generally assumed in this paper,
although usually not necessary, but here the result is not true without it.}
Then there are isomorphisms $\fM(\cX/B,\tau)\cong \fM(\cX/\Spec\kk, \tau)$ and
$\scrM(X/B,\tau)\cong \scrM(X/\Spec\kk,\tau)$.
\end{proposition}

\begin{proof}
We show the first isomorphism, the second being similar.
There is a canonical forgetful morphism $\fM(\cX/B,\tau)\rightarrow
\fM(\cX/\Spec\kk,\tau)$, and we need to show it is an isomorphism.
For this purpose, it is enough to demonstrate that given a
punctured map $f:C^{\circ}/W\rightarrow \cX$, there is a 
unique morphism $h: W \to B$ which fits into a
commutative diagram
\[
\xymatrix@C=20pt
{
C^{\circ}\ar[d]_{\pi}\ar[r]^f&\cX\ar[d]^g\\
W\ar[r]_h&B
}
\]
First, to define the underlying $\ul h: \ul W \to \ul B$ it is sufficient to
define $\ul h^{\#}:\Gamma(B,\cO_B)\rightarrow \Gamma(W,\cO_W)\cong \Gamma(W,
\pi_*\cO_C)$, the latter isomorphism from the fact that $\pi$ is flat, proper
with connected {and reduced} fibers and \cite[Lem.\,0E0S]{stacks-project}.
We take this map to coincide with $(g\circ f)^{\#}:\Gamma(B,\cO_B)\rightarrow
\Gamma(C,\cO_C)$.

We next enhance $\ul h$ to a log morphism, first by describing the map at the
level of ghost sheaves, or equivalently, at the tropical level. Fix $\ol w$
{a geometric point of $\ul W$}, and let $\tau'=(G',\bg',\bsigma', \mathbf{u}')$
be the type of $C_{\ol w}\rightarrow \cX$, so that in particular there is a
contraction morphism $\tau'\rightarrow \tau$. Since $\tau'$ and $\tau$ have the
same set of legs with the same contact orders, the fact that $\tau$ is defined
over $B$ implies that the composed map $\Sigma(g\circ f):\Sigma(C_{\ol
w})\rightarrow \Sigma(B)=P^{\vee}_{\RR}$ contracts all legs. However, $g\circ f$
is a punctured map with underlying schematic map constant, and thus by
Proposition~\ref{prop:balancing}, the restriction of $\Sigma(g\circ f)$ to any
fiber of $\Sigma(\pi)$ is a balanced tropical map. {Since all legs
are contracted, the image of this tropical map is compact. Hence, there
must be a hyperplane $H$ in the vector space $P^*_{\RR}$ containing the image
of a vertex of this map and with the entire image contained in a half-space
bounded by $H$. By  balancing, this is impossible unless
the tropical map is constant.} Hence the desired diagram exists at the
tropical level. {This shows that the map $P=\Gamma(B,\ocM_B)\arr
\Gamma(C^\circ_{\ol w}, \ocM_{C^\circ_{\ol w}})$ factors uniquely over
$\ocM_{W,\ol w}$.} In particular, we obtain a map $\bar
h^{\flat}:\Gamma(B,\overline\cM_B)\rightarrow \Gamma(W,\overline \cM_W)$. 

Finally, there is a unique lifting of $\bar h^{\flat}$ to
$h^{\flat}:\Gamma(B,\cM_B)\rightarrow \Gamma(W,\cM_W)$. Indeed, let $s\in
\Gamma(B,\cM_B)$ be a section which maps to $\bar s \in
\Gamma(B,\overline{\cM}_B)$. Then because the desired diagram exists at the
level of ghost sheaves, $(\bar f^{\flat}\circ\bar g^{\flat})(\ol
s)=\bar\pi^{\flat}(\ol t)$ for some $\ol t\in \Gamma(W,\overline{\cM}_W)$.
Thus \'etale locally on $W$, we may choose a lift $t\in \cM_W$ of $\ol t$, and
write $(f^{\flat}\circ g^{\flat})(s)=\psi\cdot \pi^{\flat}(t)$ for some $\psi\in
{\Gamma(\cO_C^{\times})}$. However, again by properness of $\pi$ and 
{connectivity and reducedness of the fibres of $\pi$,}
$\psi = \pi^{\#}(\psi')$ for some invertible function $\psi'$ on
$W$, and we may define $h^{\flat}(s)=\psi'\cdot t$. Because this choice of
$h^{\flat}(s)$ is determined uniquely by $(g\circ f)^{\flat}$, this local
description patches to give a section $h^{\flat}(s) \in\Gamma(W,\cM_W)$, making
the diagram commute.

We have thus defined a functor $\fM(\cX/\Spec\kk,\tau) \rightarrow
\fM(\cX/B,\tau)$ at the level of objects. By the uniqueness of the construction
of the morphism $W\rightarrow B$ given $f:C^{\circ}/W\rightarrow \cX$ above, a
morphism in the category $\fM(\cX/\Spec\kk,\tau)$ defines a morphism in the
category $\fM(\cX/B,\tau)$, hence completely defining the functor. This defines
the desired morphism $\fM(\cX/\Spec\kk,\tau)\rightarrow \fM(\cX/B,\tau)$ which
is inverse to the forgetful morphism $\fM(\cX/B,\tau) \rightarrow
\fM(\cX/\Spec\kk,\tau)$.
\end{proof}


\subsection{Evaluation stacks and gluing at the virtual level}
\label{ss:gluing at virtual level}

While Theorem~\ref{Thm: Gluing theorem} transparently describes the process of
gluing a collection of punctured maps at pairs of punctures with matching
contact orders, it lacks two crucial properties needed for applications in
punctured Gromov-Witten theory. First, since the diagonal map
$\Delta:\cX\arr\cX\times\cX$ is not proper except in trivial cases and neither
is the splitting map $\delta_\fM$, it is impossible to push forward cycles via
$\delta_\fM$ for the purpose of splitting computations according to the
splitting of $\tau$ along the chosen set of edges $\bE\subseteq E(G)$. And
second, the obvious commutative square lifting the splitting map $\delta_\fM$ to
a map $\scrM'(X/B,\tau)\arr \prod_i\scrM'(X/B,\tau_i)$ is far from being
cartesian even on the underlying stacks of (pre-) stable maps since it imposes
matching at the nodes only on $\cX$ rather than on $X$. (We remind the
reader that the products such as $\prod_i$ are all over the base log scheme $B$
in this discussion.) Hence this approach has no hope to be compatible with the
virtual formalism.

Both problems are solved by enriching the stacks $\fM(\cX/B,\tau)$ and
$\fM(\cX/B,\tau_i)$ of punctured maps to the relative Artin stack $\cX/B$, and
their various cousins $\fM'(\cX/B,\tau)$, $\tfM'(\cX/B,\tau)$ etc., by
providing a lift of the underlying evaluations to $X$. Note that such enriched
stacks of maps to $\cX$ have already been considered at the beginning of
\S\ref{ss: Obstruction theories with point conditions} in the context of
obstruction theories with imposed point conditions. 

For this discussion we mostly work with the stacks $\fM(\cX/B,\tau)$ of marked
maps (Definition~\ref{Def: stacks of decorated puncted maps}), except in the
analogue Corollary~\ref{Cor: Gluing theorem for evaluation stacks} of
Theorem~\ref{Thm: Gluing theorem}, which requires stacks $\tfM'(\cX/B,\tau)$
with weak markings and sections (\S\ref{Sec:stacks of maps with sections}).
All other results also hold in the weakly marked and decorated contexts.

{We continue to assume that $X\arr B$ is a morphism of fs logarithmic
schemes over $\kk$ fulfilling the assumptions stated at the beginning of
\S\ref{sec:stack}.}

\begin{definition}
\label{Def: evaluation stack}
Let $\tau=(G,\bg,\bsigma,\bar\bu)$ be a global type of punctured maps to $X$ and
$\bS\subseteq E(G)\cup L(G)$ a subset of edges and legs. The \emph{evaluation
stack} of $\fM(\cX/B,\tau)$ with respect to $\bS$ is the fiber product
\[
\textstyle
\fM^\ev(\cX/B,\tau)= \fM(\cX/B,\tau)\times_{\prod_{S\in\bS} \ul\cX}
\prod_{S\in\bS} \ul X
\]
of $\prod_{S\in\bS} \ul X\arr \prod_{S\in \bS} \ul\cX$ with the evaluation map
\[
\textstyle
\ul\ev_{\bS}: \fM(\cX/B,\tau)  \arr \prod_{S\in\bS} \ul \cX,\quad
(C^\circ/W,\bp,f)\longmapsto (\ul f\circ s_S)_{S\in\bS},
\]
evaluating at the punctured and nodal sections $s_S:\ul W\arr \ul C^\circ$ for
$S\in \bS$.

Analogous definitions apply in the weakly marked and decorated contexts
as in Definition~\ref{Def: stacks of decorated puncted maps}, or for the
stacks $\widetilde\fM'(\cX/B,\tau)$ of \S\ref{Sec:stacks of maps with
sections}.
\end{definition}

Note that $\fM^{\ev}(\cX/B,\tau)$ of course depends on the logarithmic
scheme $X$, but we suppress this in the notation as $\cX$ always denotes its
relative Artin fan. We also suppress $\bS$ in the notation of the evaluation
stacks and rather specify this subset whenever not clear from the context.

As indicated in the definition, we endow $\fM^\ev(\cX/B,\tau)$ with the log
structure making the projection to $\fM(\cX/B,\tau)$ strict, to obtain the
sequence of morphisms of log stacks
\[
\scrM(X/B,\tau)\stackrel{\varepsilon}{\longrightarrow}
\fM^\ev(\cX/B,\tau)\arr\fM(\cX/B,\tau)
\]
as in \eqref{Eqn: scrM->fM^ev->fM}. Recall that the obstruction theory for this
sequence of morphisms has been worked out in \S\ref{ss: Obstruction theories
with point conditions}. It was noted that, as the morphisms are strict, this
coincides with the obstruction theory for the underlying stacks. We further saw
that the obstruction theory of $\ul{\scrM}(X/B,\tau)$ over
$\ul{\fM}(\cX/B,\tau)$ is the composition of an obstruction theory for
$\varepsilon$ with the trivial obstruction theory in pure degree 0 of the
smooth morphism $\ul{\fM^\ev}(\cX/B,\tau) \arr\ul{\fM}(\cX/B,\tau)$ of relative
dimension $(\dim X-\dim B)\cdot|\bS|$.

We now adopt the setup of \S\ref{ss:Gluing punctured maps to cX} and split
$\tau$ at a subset $\bE\subseteq E(G)$ of edges with $\bE\subseteq\bS$ to obtain
global types $\tau_i=(G_i,\bg_i,\bsigma_i,\bar\bu_i)$. For the following
corollary of Theorem~\ref{Thm: Gluing theorem} for evaluation stacks, we write
$\widetilde\fM^\ev(\cX/B,\tau)$ for the evaluation stack of
$\widetilde\fM(\cX/B,\tau)$ {with evaluations at all nodes specified by
$\bE$}, thus {by Proposition~\ref{Prop: tilde log structures don't change reductions}} having the same underlying stack as $\fM^\ev(\cX/B,\tau)$, but with
the enlarged log structure admitting a logarithmic evaluation map analogous to
\eqref{Eqn: ev_bE}. Similarly, we obtain evaluation stack analogues of the
evaluation morphism for the $\tau_i$ \eqref{Eqn: ev_bL}, still denoted
$\ev_\bL$, and the splitting morphism \eqref{Eqn: splitting morphism}, now
denoted $\delta^\ev$.

\begin{corollary}
\label{Cor: Gluing theorem for evaluation stacks}
In the situation of Theorem~\ref{Thm: Gluing theorem}, the commutative diagram
\[
\xymatrix{
\tfM'^\ev(\cX/B,\tau) \ar[r]^(.45){\delta^\ev}\ar[d]_{\ev_\bE} &
\prod_{i=1}^r \tfM'^\ev(\cX/B,\tau_i)\ar[d]^{\ev_\bL}\\
\prod_{E\in \bE} X\ar[r]^(.45)\Delta& \prod_{E\in \bE} X\times X 
}
\]
with arrows defined by the above adaptations to the evaluation stacks for
$\bS\subseteq E(G)\cup L(G)$ with $\bE\subseteq \bS$, is cartesian in the
category of fs log stacks.

In particular, the splitting morphism $\delta^\ev$ is finite and representable.
\end{corollary}

\begin{proof}
The stated commutative square is the front face of the commutative box
\[
\xymatrix@R=.5em{
\tfM'^\ev(\cX/B,\tau) \ar[rrr]^(.45){\delta^\ev}\ar[ddd]_{\ev_\bE}
\ar@{.>}[rd]&&&
\prod_{i=1}^r \tfM'^\ev(\cX/B,\tau_i)\ar[ddd]^{\ev_\bL}\ar@{.>}[dl]\\
&\tfM'(\cX/B,\tau) \ar@{.>}[r]\ar@{.>}[d]&
\prod_{i=1}^r \tfM'(\cX/B,\tau_i)\ar@{.>}[d]\\
&\prod_{E\in \bE} \cX\ar@{.>}[r]& \prod_{E\in \bE} \cX\times \cX \\
\prod_{E\in \bE} X\ar[rrr]^(.45)\Delta\ar@{.>}[ur]&&&
\prod_{E\in \bE} X\times X \ar@{.>}[ul]
}
\] 
with back face the cartesian square from Theorem~\ref{Thm: Gluing theorem} and
the sides cartesian squares defining the evaluation stacks. Hence the stated
diagram is cartesian.

The claimed properties of the splitting morphism $\delta^\ev$ follow since an fs
fiber product is the saturation and integralization of the ordinary fiber
product.
\end{proof}

\begin{remark}
For systematic reasons we work in the category of log schemes over $B$ in this
section, and thus all products in the statement of Corollary~\ref{Cor: Gluing
theorem for evaluation stacks} are fiber products over $B$. For explicit
computations this leads to fibered sums of lattices, which sometimes require an
extra treatment of multiplicities due to saturation issues. This additional step
can be avoided by observing that the statement of Corollary~\ref{Cor: Gluing
theorem for evaluation stacks} holds unchanged when interpreting the products as
absolute products rather than as products over $B$, but still with $\cX$ the
relative Artin fan of $X/B$.

This statement is not a formal consequence of general properties of fiber
products, but is due to the connectedness of the graph $G$ given by
$\tau${, as in the argument in the proof of Proposition~\ref{Prop:
absolute versus relative frakM}}. To explain this let $\prod_B$ denote the
relative fiber product and $\prod$ the absolute one. To check the universal
property of the commutative square in Corollary~\ref{Cor: Gluing theorem for
evaluation stacks} with absolute products, let be given a morphism $W\to
\prod_i\tfM'^\ev(\cX/B,\tau_i)$, $i=1,\ldots,r$, such that the composition with
$\ev_{\bL}$ factors over $\Delta$. For each leg $L=(E,v)\in L(G_i)$ we obtain an
evaluation map $f_i\circ p_L: W\to \cX$, and by composing with $\cX\to B$ a map
$b_L: W\to B$. This map is independent of the choice of $L\in L(G_i)$ since the
$i$-th component of $W\to \prod_i\tfM'^\ev(\cX/B,\tau_i)$ defines a punctured
map over $B$, but a priori may vary with $i$. Now the factorization of $\ev_L$
over $\Delta$ implies that if the $i$-th and $j$-th vertex of $G$ are connnected
by an edge then the maps $W\to B$ obtained for $i$ and $j$ coincide. Since $G$
is connected we conclude that all these maps agree. Hence the map $W\to
\prod_i\tfM'^\ev(\cX/B,\tau_i)$ factors over $(\prod_B)_i
\tfM'^\ev(\cX/B,\tau_i)$, and in turn the composition with $\ev_\bL$ factors
over $(\prod_B)_E X$. We are then in position to apply Corollary~\ref{Cor:
Gluing theorem for evaluation stacks} in the stated form to obtain the unique
lift to $\tfM'^\ev(\cX/B,\tau)$.
\end{remark}

By the corollary, we obtain a proper push-forward homomorphism in Chow theory
for algebraic stacks, as defined by Kresch \cite{Kresch}, for the evaluation
stacks:
\begin{equation}
\label{Eqn: Proper push-forward of Chow classes in evaluation stacks}
\textstyle
\delta^\ev_*: A_*\big(\ul{\fM^\ev}(\cX/B,\tau)\big)\arr
A_*\big(\prod_i\ul{\fM^\ev}(\cX/B,\tau_i)\big),\quad
\alpha\longmapsto \delta_*(\alpha).
\end{equation}
Note that we can work with markings or weak markings here because the
corresponding stacks have the same reductions (Proposition~\ref{Prop: weakly
marked versus marked}).

It remains to relate $\delta^\ev$ with the splitting morphism for moduli spaces
of punctured maps to $X$ rather than $\cX$ and to show compatibility with the
obstruction theory. Note that these results use the unenhanced, basic log
structures on the moduli stacks.

\begin{proposition}
\label{Prop: Gluing via evaluation spaces}
{Let $X\arr B$ be a morphism of fs logarithmic schemes over $\kk$
fulfilling the assumptions stated at the beginning of \S\ref{sec:stack} Let
$\tau_1,\ldots,\tau_r$ be the global types of punctured maps
(Definition~\ref{Def: global type}) obtained by splitting a global type
$\tau=(G,\bg,\bsigma,\bar\bu)$ along a subset of edges $\bE$.
Then} there is a cartesian diagram 
\begin{equation}
\label{diag:evaluation-gluing}
\vcenter{\xymatrix@C=30pt
{
\scrM(X/B,\tau) \ar[r]^(.45){\delta}\ar[d]_{\varepsilon} &
\prod_{i=1}^r \scrM(X/B,\tau_i)\ar[d]^{\hat\varepsilon=\prod_i\varepsilon_i}\\
\fM^\ev(\cX/B,\tau) \ar[r]^(.45){\delta^\ev} &
\prod_{i=1}^r \fM^\ev(\cX/B,\tau_i)
}}
\end{equation}
with horizontal arrows the splitting maps from
Proposition~\ref{prop:splitting-map}, finite and representable by
Corollary~\ref{Cor: Gluing theorem for evaluation stacks}, and the vertical
arrows the canonical strict morphisms. Here we assume the set of edges and legs
$\bS\subseteq E(G)\cup L(G)$ used in the definition of the evaluation stacks
(Definition~\ref{Def: evaluation stack}) contains the set $\bE\subseteq E(G)$ of
splitting edges. 

Analogous statements hold for decorated and for weakly marked versions of the
moduli stacks (Definition~\ref{Def: stacks of decorated puncted maps}).
\end{proposition}

\begin{proof}
We argue by spelling out the definitions of the various stacks. Indeed, a
pair of morphisms from an fs log scheme $W$ to $\prod_{i=1}^r \scrM(X/B,\tau_i)$
and to $\fM^\ev(\cX/B,\tau)$ together with an isomorphism of their images in
$\prod_{i=1}^r \fM^\ev(\cX/B,\tau_i)$ is equivalent to (1)~an ordinary stable
map $(\ul C/\ul W,\ul p,\ul f)$ to $\ul X$ marked by the genus-decorated graph
$(G,\bg)$ given by $\tau$, and (2)~a punctured map
$(C^\circ/W,\bp,f_{\cX})$ to $\cX$ producing the morphism $W\arr \prod_{i=1}^r
\fM^\ev(\cX/B,\tau_i)$ by splitting at the nodes labelled by $\bE\subseteq
E(G)$. Note that (1) is obtained by the schematic matching condition
at the paired marked points provided by the evaluation stacks. Since $X\arr\cX$
is strict, $\ul f$ and $f_\cX$ together are the same as a log morphism
$f:C^\circ\arr X$. Moreover, a marking by $\tau$ is equivalent to markings by
$\tau_i$ of the punctured maps $(C^\circ_i/W,\bp_i,f_i)$ obtained by splitting.
The correspondence is also easily seen to be functorial. Thus the fiber
categories over $W$ of the cartesian product and of $\scrM(X/B,\tau)$ are
equivalent.
\end{proof}

\subsubsection{Notation for obstruction theories}
{To bring in the perfect obstruction theories discussed in
\S\ref{sec:obstruction}, we now in addition {to $\tau,\tau_i,\bE,\bS$ as
in Proposition~\ref{Prop: Gluing via evaluation spaces}} assume $X\arr B$ to be
log smooth.} To analyze the obstruction theories in
\eqref{diag:evaluation-gluing}, we introduce the following short-hand
notations:\footnote{For the sake of being specific we work with the marked
versions here. Analogous results also hold for the weakly marked cases.}
\begin{equation}
\label{Eqn: short-hand notations gluing stacks}
\begin{aligned}
\hspace*{5ex}\scrM_\gl&:=\scrM(X/B,\tau),&
\scrM_i&:=\scrM(X/B,\tau_i),&
\scrM_\spl&:=\textstyle\prod_{i=1}^r \scrM_i \hspace*{5ex}\\
\fM_\gl&:=\fM(\cX/B,\tau),&
\fM_i&:=\fM(\cX/B,\tau_i),&
\fM_\spl&:=\textstyle\prod_{i=1}^r \fM_i\\
\fM_\gl^\ev&:=\fM^\ev(\cX/B,\tau),&
\fM_i^\ev&:=\fM^\ev(\cX/B,\tau_i),&
\fM_\spl^\ev&:=\textstyle\prod_{i=1}^r \fM_i^\ev
\end{aligned}
\end{equation}
Denote further by $\ol C^\circ_i\to \scrM_i$ and by $C^\circ\to\scrM_\gl$ the
universal curves over $\scrM_i$ and $\scrM_\gl$, respectively, by $C^\circ_i\to
\scrM_\spl$ the pull-back of $\ol C^\circ_i$ under the projection from the
product $\scrM_\spl\arr \scrM_i$, and write $\pi_\spl: C_\spl^\circ=\coprod_i
C^\circ_i\to \scrM_\spl$. We also have universal morphisms $f:C^\circ\to X$,
$f_\spl: C_\spl^\circ\to X$, and the subspaces of special points to be
considered $\iota:Z\to C^\circ$, $\iota_\spl:Z_\spl\to C_\spl^\circ$ with
projections $p=\pi\circ\iota$ and $p_\spl=\pi_\spl\circ\iota_\spl$ to
$\scrM_\gl$ and $\scrM_\spl$, respectively. Here $Z$ is the union of the images
of the punctured and nodal sections labelled by $\bS\subseteq E(G)\cup L(G)$,
while $Z_\spl$ is the union of punctured and nodal sections given by
$\bS_i\subseteq E(G_i)\cup L(G_i)$, $i=1,\ldots,r$, obtained from $\bS$ by
splitting, both endowed with the induced log structures making $\iota$,
$\iota_\spl$ strict. 

\subsubsection{{The fundamental diagram}}
We consider the following commutative diagram:
\begin{equation}
\label{Eqn: Obstruction theory gluing diagram}
\vcenter{\xymatrix@-0.5pc{
\tilde C^\circ\ar[d]^\kappa\ar@/_1.5pc/_{\tilde \pi}[dd]
\ar[r]\ar@/^2.5pc/[rrr]^{\tilde f}
&C_\spl^\circ=\coprod_i C^\circ_i
\ar[dd]^(.73){\pi_\spl}|(.45)\hole
\ar[rr]^{f_\spl} &&X\\
C^\circ\ar[d]^\pi\ar@/_1pc/[rrru]_(.7)f\\
\scrM_\gl\ar[r]^(.32)\delta\ar[d]_{\varepsilon}&\scrM_\spl
=\prod_i\scrM_i\ar[d]^{\hat\varepsilon=\prod_i\varepsilon_i}\\
\fM_\gl^\ev\ar[r]^(.35){\delta^\ev}&\fM_\spl^\ev=\prod_i\fM^\ev_i.
}}
\end{equation}
The lower square is the cartesian square from Proposition~\ref{Prop: Gluing via
evaluation spaces} with strict vertical arrows. 

The strict map $\kappa:\tilde C^\circ\arr C^\circ$ is the map induced by
splitting the nodal sections of $C^\circ\arr \scrM_\gl$ given by $\bE\subseteq
\bS$ according to Proposition~\ref{prop:splitting-map}. The underlying morphism
$\ul\kappa$ of ordinary stacks is therefore the corresponding partial
normalization from Definition~\ref{Def: nodal sections}. 

The upper square thus identifies the pull-back of $C^\circ_\spl$ with the
pre-stabilization of $\tilde C^\circ$ (Definition~\ref{def:pre-stability}).
{This part of the diagram} is a pull-back of nodal curves, cartesian
only in the {category of stacks, because of the pre-stabilization}.

The morphism $\tilde f$ is as defined by the diagram. There is also the closed
substack $\tilde Z=\kappa^{-1}(Z)\to \tilde C^\circ$ of special points on
$\tilde C^\circ$ with projection $\tilde p:\tilde Z\to \scrM_\gl$, endowed with
the log structure making $\tilde Z\arr \tilde C^\circ$ strict.

\subsubsection{An obstruction theory for $\varepsilon$ and $\hat\varepsilon$}
\label{Ref: The gl and spl obstruction theories}
The discussion in \S\ref{ss: Obstruction theories with point conditions}
provides obstruction theories $\GG\arr \LL_{\scrM_\gl/\fM_\gl^\ev}$ for
$\varepsilon: \scrM_\gl\arr \fM_\gl^\ev$ and $\GG_\spl\to
\LL_{\scrM_\spl/\fM_\spl^{\ev}}$ for $\hat\varepsilon: \scrM_\spl\arr
\fM^\ev_\spl$ with
\begin{equation}
\label{Eqn: obstruction theory for widehat spaces}
\GG= R\tilde\pi_*\big(\tilde f^*\Omega_{X/B}\otimes
\omega_{\tilde\pi}(\tilde Z)\big),\quad
\GG_\spl= R{\pi_\spl}_*\big(f_\spl^*\Omega_{X/B}\otimes
\omega_{\pi_\spl}(Z_\spl)\big).
\end{equation}
Recall that this obstruction theory is obtained by taking the cone of a morphism
of perfect obstruction theories provided by Proposition~\ref{Prop: Compatibility
of obstruction theories}:
\[
\begin{CD}
L\hat\varepsilon^*\FF_\spl@>>> \EE_\spl\\
@V{L\hat\varepsilon^*\Psi}VV@VV{\Phi}V\\
L\hat\varepsilon^*\LL_{\fM_\spl^\ev/\fM_\spl}@>>> \LL_{\scrM_\spl/\fM_\spl}
\end{CD}
\]
\subsubsection{A justice of obstructions}\!\!\!\footnote{Our Babel of
coauthors proposes this collective noun for a system of compatible
obstructions.} 
We now have four deform\-ation/ob\-struction situations with corresponding
perfect obstruction theories. Given $T\to\scrM_\gl$ a morphism from an affine
scheme and $f_T:C^\circ_T \to X$, $h_T:Z_T\to X$, $\tilde f_T: \tilde
C^\circ_T\to X$, $\tilde h_T:\tilde Z_T\to X$ the respective base-changes to $T$
of the universal morphisms from the universal curve and universal sections,
pulled back to $\scrM_\gl$ in the last two instances, these are as follows. All
deformation situations are relative $\fM_\gl$, with the last two pulled
back from a deformation situation relative $\fM_\spl$.

\begin{enumerate}%
[leftmargin=0ex,labelwidth=14ex,align=left,itemindent=14ex,labelsep=0ex]
\item[($\scrM_\gl/\fM_\gl$)]
Deforming $f_T: C^\circ_T\to X$:
\[
\makebox[50ex][l]{$\EE= R\pi_*(f^*\Omega_{X/B}\otimes\omega_\pi)\,
\larr\,\LL_{\scrM_\gl/\fM_\gl}.$}
\]
\item[($\fM_\gl^\ev/\fM_\gl$)]
Deforming $h_T: Z_T\to X$:
\[
\makebox[50ex][l]{$L\varepsilon^*\FF= p_*(h^*\Omega_{X/B})\,\larr\,
L\varepsilon^*\LL_{\fM_\gl^\ev/\fM_\gl}.$}
\]
\item[($\scrM_\spl/\fM_\spl$)]
Deforming $\tilde f_T: \tilde C^\circ_T\to X$:
\[
\makebox[50ex][l]{$L\delta^*\EE_\spl= R\tilde\pi_*(\tilde f^*\Omega_{X/B}\otimes\omega_{\tilde\pi})\, \larr\,
L\delta^*\LL_{\scrM_\spl/\fM_\spl}.$}
\]
\item[($\fM_\spl^\ev/\fM_\spl$)]
Deforming $\tilde h_T: \tilde Z_T\to X$:
\[
\makebox[50ex][l]{$L\delta^* L\hat\varepsilon^* \FF_\spl=
\tilde p_*(\tilde h^*\Omega_{X/B})\,\larr\,
L\delta^*L\hat\varepsilon^*\LL_{\fM_\spl^\ev/\fM_\spl}.$}
\]
\end{enumerate}
\smallskip

\begin{lemma}
\label{Lem: GG on tilde C is GG on C}
There is a morphism of distinguished triangles
\[
\begin{CD}
L\delta^* L{\hat\varepsilon}^* \FF_\spl @>>> L\delta^*\EE_\spl@>>>\GG
@>>> L\delta^* L{\hat\varepsilon}^* \FF_\spl[1]\\
@VVV@VVV@|@VVV\\
L\varepsilon^*\FF @>>>\EE@>>>\GG @>>>L\varepsilon^*\FF[1] \\
\end{CD}
\]
with $\GG= L\delta^*\GG_\spl= R\tilde\pi_*(\tilde
f^*\Omega_{X/B}\otimes\omega_{\tilde\pi}(\tilde Z))$.
\end{lemma}

\begin{proof}
{
The lower row in the claimed diagram was produced in~\eqref{Eqn: EE-GG-FF
sequence} in the proof of Proposition~\ref{Prop: Point condition virtual bundle}
by applying $R\pi_*$ to \eqref{Eqn: Exact sequence dualizing complexes} tensored
with $f^*\Omega_{X/B}$. We claim that \eqref{Eqn: Exact sequence dualizing complexes} appears as the lower row in the following commutative diagram with exact rows:
\begin{equation}
\label{Eqn: Diagram to connect spl and gl obstruction bundles}
\vcenter{\xymatrix{
0\ar[r]&\kappa_*\omega_{\tilde\pi}\ar[r]\ar[d]&
\kappa_*(\omega_{\tilde\pi}(\tilde Z))\ar[r]\ar@{=}[d]&
\kappa_*\tilde\iota_*\cO_{\tilde Z}[1]\ar[r]\ar[d]&0\\
0\ar[r]&\omega_\pi\ar[r]&\kappa_*(\omega_{\tilde\pi}(\tilde Z))\ar[r]&
\iota_*\cO_Z[1]\ar[r]&0.
}}
\end{equation}
Away from the nodal locus $Z''\subset Z$ the upper and lower rows are identical,
and this identification defines the diagram there. \'Etale locally near a node,
the arrow $\kappa_*(\omega_{\tilde\pi}(\tilde Z))\arr \iota_*\cO_Z[1]$ takes the
difference of the residues of a differential with at most simple poles along the
two components of $\tilde Z''\simeq Z''\amalg Z''$ defined by the two branches at the
node. This map factors as $\kappa_*$ of the residue map
\[
\rho\colon \omega_{\tilde\pi}(\tilde Z'') \arr
\tilde\iota_*\cO_{\tilde Z''}[1]
\]
and the $[1]$-twist of the difference map
\[
\kappa_*\tilde\iota_*\cO_{\tilde Z''}=
\iota_*\cO_{Z''}\oplus\iota_*\cO_{Z''}\arr \iota_*\cO_{Z''},\quad
(a,b)\longmapsto a-b.
\]
The kernel of $\kappa_*\rho$ selects differentials without poles, that is,
$\kappa_*\omega_{\tilde\pi}$. This extends the construction of
Diagram~\eqref{Eqn: Diagram to connect spl and gl obstruction bundles} over the
nodal locus.

To produce the morphism of triangles in the statement it remains to show that
tensoring the upper row of \eqref{Eqn: Diagram to connect spl and gl obstruction
bundles} with $f^*\Omega_{X/B}$ and applying $R\pi_*$ leads to the upper row of
the claimed diagram. From Proposition~\ref{Prop: Point condition virtual bundle}
we already know that the middle term leads to $\GG$:
\[
\GG\stackrel{\eqref{Eqn: obstruction theory for widehat spaces}}{=}
R\tilde\pi_*\big(\tilde f^*\Omega_{X/B}\otimes
\omega_{\tilde \pi}(\tilde Z)\big)
=R\pi_*\big(f^*\Omega_{X/B}\otimes
\kappa_*(\omega_{\tilde \pi}(\tilde Z))\big).
\]
The other two terms are readily obtained by the projection formula for $\kappa$ using $\tilde\pi=\pi\circ\kappa$, $\tilde f=f\circ\kappa$, $\tilde h=\tilde f\circ\tilde \iota$, $\tilde p=\tilde\pi\circ\tilde\iota$:
\begin{align*}
L\delta^*\EE_\spl&=R\tilde\pi_*(\tilde f^*\Omega_{X/B}\otimes\omega_{\tilde\pi})
=R\pi_*(f^*\Omega_{X/B}\otimes\kappa_*\omega_{\tilde\pi})\\
L\delta^* L\hat\varepsilon^* \FF_\spl&= \tilde p_*(\tilde h^*\Omega_{X/B})
= R\tilde\pi_*(\tilde f^*\Omega_{X/B}\otimes \tilde\iota_*\cO_{\tilde Z})
= R\pi_*(f^*\Omega_{X/B}\otimes \kappa_*\tilde\iota_*\cO_{\tilde Z}).
\end{align*}}
\end{proof}

\begin{theorem}
\label{Thm: Virtual gluing via evaluation spaces2}
{Let $X\arr B$ be a log smooth morphism of fs logarithmic schemes
over $\kk$ fulfilling the assumptions stated at the beginning of
\S\ref{sec:stack}, {and $\tau,\tau_i,\bE,\bS$ as in
Proposition~\ref{Prop: Gluing via evaluation spaces}}. Then} with the notation
of \eqref{Eqn: short-hand notations gluing stacks}, we have
\begin{enumerate}
\item
The obstruction theory $\GG\arr \LL_{\scrM_\gl/\fM_\gl^\ev}$
for
\[
\scrM_\gl=\scrM_\gl(X/B,\tau)\arr\fM_\gl^\ev=\fM^\ev(\cX/B,\tau)
\]
coincides with the pull-back of {one of the obstruction theories
$\GG_\spl\rightarrow \LL_{\scrM_\spl/\fM_\spl^{\ev}}$ (Remark~\ref{Rem:
non-uniqueness of obstruction theories})} for
\[
\textstyle
\scrM_\spl=\prod_i \scrM(X/B,\tau_i)\arr \fM_\spl^\ev=
\prod_i \fM^\ev(\cX/B,\tau_i)
\]
described in {\S\ref{Ref: The gl and spl obstruction theories}}.
\item
If $\hat\varepsilon^!$ and $\varepsilon^!$ denote Manolache's virtual
pull-back defined using the two given obstruction theories for the vertical
arrows in Diagram~\eqref{diag:evaluation-gluing}, then for $\alpha\in
A_*\left(\fM^\ev(\cX/B,\tau)\right)$, we have the identity
\[
\hat\varepsilon^!\delta^\ev_*(\alpha)=
\delta_*\varepsilon^!(\alpha)\,.
\]
\end{enumerate}
\end{theorem}

\begin{proof}
{
(1) The morphism between the obstruction theories in question appear as the joint middle square in the following diagram of two adjacent cubes:
\[
\xymatrix@-1.2pc{
L\delta^*\EE_\spl\ar[rr]\ar[dd]\ar[rd]&& \GG\ar@{=}[dd]\ar@{-->}[rd]\ar[rr]
&&L\delta^* L\hat\varepsilon^* \FF_\spl[1]\ar[dd]\ar[rd]\\
&L\delta^*\LL_{\scrM_\spl/\fM_\spl}\ar[rr]\ar[dd]
&& L\delta^*\LL_{\scrM_\spl/\fM_\spl^\ev}\ar[dd]\ar[rr]
&&L\delta^*L\hat\varepsilon^* \LL_{\fM_\spl^\ev/\fM_\spl}[1]\ar[dd]\\
\EE\ar[rr]\ar[rd]&& \GG\ar@{-->}[rd]\ar[rr]&&L\varepsilon^*\FF[1]\ar[rd]&&\\
&\LL_{\scrM_\gl/\fM_\gl}\ar[rr]&& \LL_{\scrM_\gl/\fM^\ev_\gl}\ar[rr]
&&L\varepsilon^* \LL_{\fM^\ev_\gl/\fM_\gl}[1]
}
\]
{ The back face is the morphism of triangles from Lemma~\ref{Lem: GG
on tilde C is GG on C}.} The bottom face is commutative by the construction of
the obstruction theory with point conditions $\GG\arr
\LL_{\scrM_\gl/\fM_\gl^\ev}$ in \eqref{Eqn: Cone construction obstruction
theory} based on Proposition~\ref{Prop: Compatibility of obstruction theories}.
Similarly, the top face is commutative as the pull-back by $\delta$ of the
corresponding diagram for ${\GG_\spl}\arr
\LL_{\scrM_\spl/\fM_\spl^\ev}$. The front face of the diagram is the morphism of
distinguished triangles of cotangent complexes for the compositions
$\scrM_\gl\to \fM_\gl^\ev\to \fM_\gl$ and $\scrM_\spl\to
\fM_\spl^\ev\to\fM_\spl$, and hence is commutative as well.

{ For commutativity of the left face we argue in two steps.
First apply functoriality of obstruction theories, Lemma~\ref{Lem: Functoriality
of Phi}, to compare the pulled-back obstruction theory $(\scrM_\spl/\fM_\spl)$
for $f_\spl$ with the obstruction theory for $\tilde f$, both relative
$\fM_\spl$, to obtain the commutative square
\begin{equation}
\label{Eqn: EE_spl versus tilde EE}
\begin{CD}
L\delta^*\EE_\spl@>>>L\delta^*\LL_{\scrM_\spl/\fM_\spl}\\
@VVV@VVV\\
\tilde\EE@>>>\LL_{\scrM_\gl/\fM_\gl}.
\end{CD}
\end{equation}
Here $\tilde\EE=R\tilde\pi_*(\tilde f^*\Omega_{X/B}\otimes\omega_{\tilde\pi})$,
and we replaced the lower right-hand corner $\LL_{\scrM_\gl/\fM_\spl}$ by
$\LL_{\scrM_\gl/\fM_\gl}$ using functoriality of the cotangent complex. Note
also that the proof of Lemma~\ref{Lem: Functoriality of Phi} did not use the
general assumption in \S\ref{ss:Obstruction theories for pairs} that $M$ is an
open substack of the stack of diagrams described in \eqref{Eqn: objects of M},
so does apply to the non-universal family over $\scrM_\gl$ given by $\tilde f$.

We are then in the situation of \S\ref{sss: obstructions for markings} with $Y\arr S$ the universal curve over $\fM_\gl$, $Z$ the partial normalization of this curve, and $M=N=\scrM_\gl$. Thus Proposition~\ref{Prop: Compatibility of obstruction theories} provides the commutative square
\begin{equation}
\label{Eqn: tilde EE versus EE}
\begin{CD}
\tilde\EE@>>>\LL_{\scrM_\gl/\fM_\gl}\\
@VVV@VVV\\
\EE@>>>\LL_{\scrM_\gl/\fM_\gl}.
\end{CD}
\end{equation}
Again, this result did not use universality of the family of maps over
$\scrM_\gl$ given by $\tilde f$. Composing the two squares \eqref{Eqn: EE_spl
versus tilde EE} and \eqref{Eqn: tilde EE versus EE} proves commutativity of the
left face of our big diagram of adjacent cubes.

An analogous argument for the nodal locus $Z$ and its {pull-back $\tilde
Z\subset \tilde C^\circ$ instead of $C^\circ$ and $\tilde C^\circ$} also shows
commutativity of the right face.}

Thus the whole diagram is commutative
except possibly the middle, separating square that describes the morphism of
interest from the pull-back of the obstruction theory for
$(\scrM_\spl/\fM_\spl^\ev)$ to $(\scrM_\gl/\fM_\gl^\ev)$.

However, chasing the diagram, we see that the two morphisms from $\GG$ to the
front right corner $L\varepsilon^* \LL_{\fM_\gl^\ev/\fM_\gl}[1]$, one via the
top dashed arrow, the other via the bottom dashed arrow, agree. Their difference
factors over a homomorphism
\[
\GG\arr \LL_{\scrM_\gl/\fM_\gl}.
\]
The set of such homomorphisms acts transitively on the set of dashed arrows on
the bottom face defining the obstruction theory for $(\scrM_\gl/\fM_\gl^\ev)$,
as discussed in Remark~\ref{Rem: non-uniqueness of obstruction theories}. Thus
there is a choice of dashed bottom arrow making the separating middle square of
the diagram commutative, as claimed.
}

(2) This follows from the morphism $\delta^\ev$ being finite and representable,
hence projective, and the push-pull formula of
\cite[Thm.\,4.1,(iii)]{Mano}.
\end{proof}

\subsubsection{Gluing by the numbers}
We now achieve a numerical gluing formula for Gromov-Witten invariants for
classes in $\fM^\ev(\cX/B,\tau)$ whose push-forward to $\prod
\fM^\ev(\cX/B,\tau_i)$ decomposes as a sum of products of classes. This is for
example the case for point classes in $\fM^\ev(\cX/B,\tau)$, or if all gluing
strata are toric
\cite{Yixian}.

\begin{corollary}
\label{Cor: decomposition into product}
In the situation of Theorem~\ref{Thm: Virtual gluing via evaluation spaces2} let
$\alpha\in A_*\big(\fM(\cX/B,\tau)\big)$ and assume that there exists
$\alpha_{i,\mu}\in A_*\big(\fM(\cX/B,\tau_i)\big)$, $i=1,\ldots,r$,
$\mu=1,\ldots,m$, with
\[
\delta^\ev_*(\alpha)=\sum_{\mu=1}^m \alpha_{1,\mu}\times
\cdots\times \alpha_{r,\mu}\,.
\]
Then writing $\varepsilon_i: \scrM(X/B,\tau_i)\arr \fM^\ev(\cX/B,\tau_i)$ for
the canonical map, the following equality of associated virtual classes holds in
$A_*\big(\prod_i\scrM(X/B,\tau_i)\big)$:
\[
\delta_*\varepsilon^!(\alpha)=\sum_{\mu=1}^m \varepsilon_1^!(\alpha_{1,\mu})
\times\cdots\times\varepsilon_r^!(\alpha_{r,\mu}).
\]
\end{corollary}

\begin{proof}
The claimed formula follows readily from
Theorem~\ref{Thm: Virtual gluing via evaluation spaces2},(2) by observing that
\[
\hat\varepsilon^!(\alpha_{1,\mu}\times\cdots\times \alpha_{r,\mu})
=\varepsilon_1^!(\alpha_{1,\mu})\times\cdots\times
\varepsilon_r^!(\alpha_{r,\mu})\,.
\qedhere
\]
\end{proof}
\medskip

\subsubsection{Compatibility with contractions of types}
We end this section by noting that the relative obstruction theories are also
compatible with contraction morphisms relating different global types
(Definition~\ref{Def: global type},(1)).

\begin{proposition}
\label{Prop: obstruction theories are compatible with contraction morphisms}
Let {$X\arr B$ be as in Theorem~\ref{Thm: Virtual gluing via evaluation
spaces2} and assume $\tau'\arr\tau$ is} a contraction morphism of global types.
Then the commutative diagram
\[
\xymatrix
{
\scrM(X/B,\tau')\ar[r]\ar[d]_{\varepsilon'}&\scrM(X/B,\tau)
\ar[d]^{\varepsilon}\\
{\fM(\cX/B,\tau')}\ar[r]& \fM(\cX/B,\tau)
}
\]
is cartesian, and the relative obstruction theory for $\varepsilon$ pulls back
to the relative obstruction theory for $\varepsilon'$. Taking curve
classes into consideration, if $\btau=(\tau,{\bf A})$, the commutative diagram 
\begin{equation}
\label{eq:curve class diagram}
\vcenter{\xymatrix
{
\coprod_{\btau'=(\tau',{\bf A}')}\scrM(X/B,\btau')
\ar[r]\ar[d]_{\varepsilon'}&\scrM(X/B,\btau)\ar[d]^{\varepsilon}\\
\fM(\cX/B,\tau')\ar[r]& \fM(\cX/B,\tau)
}}
\end{equation}
is cartesian, and the same statement on relative obstruction theories holds.
Here, the disjoint union is over all decorations $\btau'$ of $\tau'$
such that the contraction morphism $\tau'\rightarrow\tau$ induces a
contraction morphism $\btau'\rightarrow\btau$.

Analogous results hold for weakly marked versions of the stacks (Definition
\ref{Def: stacks of decorated puncted maps}), and for evaluation stacks on the
bottom (Definition~\ref{Def: evaluation stack}).
\end{proposition}

\begin{proof}
That the diagrams are Cartesian follows from the definition of markings and
decorated markings of punctured maps (Definition~\ref{Def: stacks of decorated
puncted maps}). 

The statement about obstruction theories then follows from the functoriality
statement Lemma~\ref{Lem: Functoriality of Phi} and the construction in
\S\ref{ss: Obstruction theories with point conditions} of the relative
obstruction theory for $\scrM(X/B,\tau) \rightarrow \fM(\cX/B,\tau)$.
\end{proof}

\begin{remark}
\label{Rem: History of gluing}
The formalism for gluing presented here was found after many futile attempts
leading to practically useless gluing procedures. With hindsight compatibility
with the virtual formalism provides the strongest guiding principle that rules
out many alternative approaches. From this point of view one discovers the
imperative that one work with obstruction theories relative to a class of
unobstructed base stacks that induce the gluing. 

A first attempt would work with
moduli stacks $\fM(\cX/B)$ of punctured maps to the relative Artin fan $\cX\to
B$. This approach does indeed work, but it is often problematic for practical
applications because the gluing map $\fM(\cX/B,\btau)\to
\prod_i\fM(\cX/B,\btau_i)$ is neither representable nor proper, hence does not
allow pushing forward of cycles.

The key insight is to use evaluation stacks to add just enough information to
get rid of the stacky nature of the gluing in $\fM(\cX/B)$, thus leading to a
finite and representable splitting map $\delta^\ev$. In addition, $\delta^\ev$
fits into the expected gluing diagram stated in Corollary~\ref{Cor: Gluing
theorem for evaluation stacks} thus providing a practical path to explicit
computations.
\end{remark}

\subsection{Gluing in the degeneration setup}
\label{sec:degeneration gluing}

We now apply our gluing theorems to the degeneration situation previously
studied in \cite{decomposition}. In this case $B$ is a smooth affine curve over
$\Spec\kk$ with log structure trivial except at a marked point $b_0\in B$,
{and $\cA_X$ is assumed Zariski.} Base change to $b_0$ produces a log
smooth space $X_0$ over the standard log point $\Spec(\NN\to\kk)$. Let
$\beta=(g,\bar\bu,A)$ be a class of punctured maps to $X$. Note that
$\Sigma(X_0)=\Sigma(X)$, so we can view $\beta$ also as a class of punctured
maps to $X_0$. The fiber of the tropicalization $\Sigma(X_0)\arr
\Sigma(b_0)=\RR_{\ge0}$ of the projection $X_0\arr b_0$ over $1\in\RR_{\ge0}$
defines a polyhedral complex ${\Delta(X_0)=}\Delta(X)$. Restricting to
this fiber turns our cone complexes into the polyhedral complexes of traditional
tropical geometry.

The main result of \cite{decomposition} gives the following decomposition of the
virtual fundamental class of $\scrM(X_0/b_0,\beta)$ in terms of rigid tropical
maps to $\Delta(X_0)$. We emphasize that this result uses the marked
rather than weakly marked versions of the moduli stacks.

\begin{theorem}
Let $\beta$ be a class of stable logarithmic maps to $X_0/b_0$. Then
we have the following equality of Chow classes on
$\scrM(X_0/b_0,\beta)$:
\[
\big[\scrM(X_0/b_0,\beta)\big]^\virt
= \sum_{\btau=(\tau,\bA)} \frac{m_\tau}{|\Aut(\btau)|}
{j_{\btau}}_*\big[\scrM(X_0/b_0,\btau)\big]^\virt.
\]
The sum runs over representatives of isomorphism classes of realizable global
types $\btau$ of punctured maps to $X_0$ over $b_0$ of total class
$(g,\bar\bu,A)$ and with basic monoid $Q_\tau\simeq \NN$. The multiplicity
$m_\tau$ is the index of the image of the homomorphism $\NN\to Q_\tau$ given by
{ the map $\scrM(X_0/b_0,\beta)\arr b_0$}. The morphism $j_\btau: \scrM(X_0/b_0,\btau) \arr \scrM(X_0/b_0,\beta)$ is
induced by the contraction morphism $\tau\arr \beta$. {Finally,
$\Aut(\btau)$ denotes the group of automorphisms of the decorated type
$\btau$, i.e., automorphisms of the underlying graph $G$ preserving
$\bg,\bsigma,\bu$ and $\bA$}.
\end{theorem}

\subsubsection{{Degenerate types}}
Theorem~\ref{Thm: stratified version of decomposition paper} below is an
analogous result in the punctured case, which also provides a stratified version
in the case without punctures. Before stating this result we need some
preparations concerning types in degeneration situations. Since one works with
log spaces over $b_0$ and $\Sigma(b_0)=\RR_{\ge 0}$, all tropical objects come
with a map to {$\RR_{\ge 0}$}. We denote all these maps by
$p$ in the following. Assuming $X_0\arr b_0$ is the fiber over the unique marked
point $b_0\arr B$ in a log smooth curve $B$ over the trivial log point, the
tropicalization of a punctured map over the generic point $\eta\in \ul B$ maps
to $0\in\Sigma(B)=\Sigma(b_0)$ under $p$. Degenerations of families of punctured
maps over $\eta$ to $b_0$ then provide a contraction morphism of the associated
types (Definition~\ref{Def: global type},(1)). This motivates the following
definition.

\begin{definition}
\label{Def: types in degenerations}
Let $\tau=(G,\bg,\bsigma,\bu)$ be a realizable global type (Definition~\ref{Def:
global type},(2)) of punctured maps to $X_0/b_0$ (Definition~\ref{def:realizable
over B}) and $Q_\tau$ the associated basic monoid.
\begin{enumerate}
\item
We call $\tau$ \emph{generic} if {$(Q_\tau^\vee)_\RR$} and
$\bsigma(x)$ for each $x\in V(G)\cup E(G)\cup L(G)$ map to
$\{0\}\subset{\RR_{\ge0}}$ under $p$.
\item
A \emph{degeneration} of a realizable global type $\tau$ is a contraction
morphism $\tau'\arr\tau$ between realizable global types with $p:
Q_{\tau'}^\vee\arr \NN$ non-constant. The \emph{codimension} of $\tau'\arr\tau$
is defined as $\rk Q_{\tau'}^\gp-\rk Q_\tau^\gp$. In the case of codimension one
we define the \emph{multiplicity} $m_{\tau'}$ as the index of
$p^\gp(Q_{\tau'}^*)$ in $\ZZ$. Finally, $\Aut(\tau'/\tau)$ denotes the group of
automorphisms of $\tau'$ commuting with $\tau'\arr\tau$.
\end{enumerate}
Analogous notions are used in the decorated case (Definition~\ref{Def: stacks of
decorated puncted maps}).
\end{definition}

\subsubsection{{Degenerate types decompose}}
Let now $\btau=(G,\bg,\bsigma,\bu,\bA)$ be a generic realizable decorated global
type for $X/B$. By the assumption $p(\bsigma(x))=0$ we can view $\btau$ also as
a decorated global type for $X_b/b$ for $b\neq b_0$. The analogue to the main
results of \cite{decomposition} is:

\begin{theorem}
\label{Thm: stratified version of decomposition paper}
In the above situation{, additionally assuming $X$ is simple,} the
following holds.
\begin{enumerate}
\item For any point $j_b:\{b\}\hookrightarrow B$, one has
$j_b^![\scrM(X/B,\btau)]^{\virt}=[\scrM(X_b/b,\btau)]^{\virt}$.
\item
\begin{equation}
\label{Eqn: decomposition formula punctured}
\big[\scrM(X_0/b_0,\btau)\big]^\virt
= \sum_{\btau'=(\tau',\bA')} \frac{m_{\tau'}}{|\Aut(\btau'/\btau)|}
{j_{\btau'}}_*\big[\scrM(X_0/b_0,\btau')\big]^\virt
\end{equation}
The sum runs over representatives of isomorphism classes of degenerations
$\btau'=(\tau',\bA')\arr\btau$ of realizable global types
of punctured maps to $X/B$ of codimension one, with $m_{\tau'}$ its multiplicity.
\end{enumerate}
\end{theorem}

\begin{proof}
By Proposition~\ref{prop:realizable over B}, $\btau$ can be viewed both as a
type realizable over $B$ and as a type realizable over $b\in B$ for $b\not=b_0$.
Thus $\fM(\cX/B,\btau)$ is non-empty and
$\fM(\cX_b/b,\btau)=\fM(\cX/B,\btau)\times_B b$ is non-empty for $b \in
B\setminus\{b_0\}$. {By Proposition~\ref{Prop: pure-dimensional
fM(cX,tau)}, $\fM(\cX/B,\btau)$ is pure-dimensional. Further, by the same
proposition, every irreducible component of $\fM(\cX/B,\btau)$ contains a point
whose corresponding punctured map has tropical type $\btau$, as all other strata
are of lower dimension. By genericity of the type $\btau$, the stratum of
$\fM(\cX/B,\btau)$ of points with type $\btau$ maps to the open stratum of $B$.
Thus the restriction of $\fM(\cX/B,\btau)\rightarrow B$ to each irreducible
component is dominant. {There are no embedded components by the local
description in Remark~\ref{Rem: local structure of fM(cX/B,tau)}.}} We conclude
that the structure map $\fM(\cX/B,\btau)\rightarrow B$ is flat.

(1) then follows immediately from general properties of virtual pull-backs.

For (2), as in the proof of \cite[Thm.\,3.11]{decomposition},
we begin by showing the corresponding decomposition {as Chow classes}
\begin{equation}
\label{eq:decomposition rel}
[\fM(\cX_0/b_0,\tau)]=\sum_{\tau'\rightarrow\tau}
\frac{m_{\tau'}}{|\Aut(\tau'/\tau)|}\iota_{\tau'*}[\fM(\cX_0/b_0,\tau')].
\end{equation}
Here $\tau$ is the underlying global type of $\btau$, and $\tau'\rightarrow\tau$
runs over all contraction morphisms as in the statement of the theorem (without
the decoration). Finally, $\iota_{\tau'}:\fM(\cX_0/b_0,\tau')\rightarrow
\fM(\cX_0/b_0,\tau)$ is the natural morphism. However, using the smooth local
description of $\fM(\cX/B,\tau)$ given in Remark \ref{Rem: local structure of
fM(cX/B,tau)} and the fact that $|\Aut(\tau'/\tau)|$ is the degree of the finite
map $\iota_{\tau'}$ onto its image, we easily obtain the result using standard
toric geometry. We leave the details to the reader.

We now make use of the diagram \eqref{eq:curve class diagram} for a given 
choice of contraction $\tau'\rightarrow \tau$, and we see by
the push-pull result of \cite[Thm.\,4.1]{Mano} that
\begin{align}
\label{eq:M sum}
\begin{split}
\varepsilon^!\iota_{\tau'*}[\fM(\cX_0/b_0,\tau')]= {} & 
\sum_{\btau'=(\tau',{\bf A}')} j_{\tau'*}(\varepsilon')^![\fM(\cX_0/b_0,
\btau')]\\
= {} & 
\sum_{\btau'=(\tau',{\bf A}')} j_{\tau'*}[\scrM(X_0/b_0, \btau')]^{\virt},
\end{split}
\end{align}
where the sum is over all choices of decorations $\btau'$ of $\tau'$ giving a
contraction morphism $\btau'\rightarrow \btau$ compatible with
$\tau'\rightarrow\tau$. On the other hand, $\Aut(\tau'/\tau)$ acts on the set of
all such decorations, with the orbit of a decoration $\btau'$ having stabilizer
$\Aut(\btau'/\btau)$. Thus we may rewrite the {last summation
of \eqref{eq:M sum}} as
\[
\sum_{\btau'=(\tau',{\bf A}')}j_{\tau'*}[\scrM(X_0/b_0, \btau')]^{\virt}
\frac{|\Aut(\tau'/\tau)|}{|\Aut(\btau'/\btau)|},
\]
where now the sum is over a set of representatives of isomorphism classes of
type $\btau'$ with a contraction morphism $\btau'\rightarrow \btau$. Combining
this with the relation \eqref{eq:decomposition rel} then gives the desired
result.
\end{proof}

\sloppy
\subsubsection{Splitting and factoring decomposed degenerate types}
As a corollary of Theorem~\ref{Thm: Virtual gluing via evaluation spaces2} we
now obtain a formula for the computation of each summand
$\big[\scrM(X_0/b_0,\btau')\big]^\virt$ {in \eqref{Eqn: decomposition
formula punctured}} in terms of punctured Gromov-Witten theory of the strata.
For the statement note that if $\tau_v$ is a global type with only one vertex,
with associated stratum $\sigma\in\Sigma(X)$, then a $\tau_v$-marked punctured
map $(C^\circ/W,\bp,f)$ to $X$ has a factorization
\[
f:C^\circ\stackrel{f_\sigma}{\longrightarrow} X_\sigma\arr X,
\]
\fussy
where the stratum $X_\sigma$ is now endowed with the log structure making the
embedding $X_\sigma\arr X$ strict. The composition with this strict closed
embedding in fact induces an isomorphism
\[
\scrM(X_0/b_0,\tau_v)\stackrel{\simeq}{\longrightarrow}
\scrM(X_\sigma/b_0,\tau_v).
\]
Similarly, we obtain $\fM(\cX_0/b_0,\tau_v)\simeq \fM(\cX_\sigma/b_0,\tau_v)$
and $\fM^\ev(\cX_0/b_0,\tau_v)\simeq \fM^\ev(\cX_\sigma/b_0,\tau_v)$. Note also
that $X_\sigma\arr \cX_\sigma$ is strict and smooth despite $\cX_\sigma$ being
only idealized log smooth over $b_0$ {(see Proposition~\ref{Prop:
idealized log smooth})}. Thus the obstruction theory developed in \S\ref{ss:
Obstruction theories with point conditions} still applies with target
$X_\sigma\arr \cX_\sigma\arr b_0$ and yields the same result as with $X_0\arr
\cX_0\arr b_0$. Theorem~\ref{Thm: Virtual gluing via evaluation spaces2} applied
to our degeneration situation can therefore be stated as follows.

\begin{corollary}
\label{Cor: gluing formula for degenerations}
Let $(G,\bg,\bsigma,\bu,\bA)$ be a decorated type of punctured maps with basic
monoid $Q_\tau\simeq\NN$ and $\btau=(G,\bg,\bsigma,\bar\bu,\bA)$ the associated
decorated global type. Denote by $\btau_v$, $v\in V(G)$, the decorated global
types obtained by splitting $\btau$ at all edges, that is, for $\bE=E(G)$. Then
the diagram
\[
\xymatrix{
\scrM(X_0/b_0,\btau) \ar[r]^(.38){\delta}\ar[d]_{\varepsilon} &
\prod_{v\in V(G)} \scrM(X_{\bsigma(v)}/b_0,\btau_v)\ar[d]^{\hat\varepsilon
=\prod_{v\in V(G)}\varepsilon_v}\\
\fM^\ev(\cX_0/b_0,\btau) \ar[r]^(.38){\delta^\ev} &
\prod_{v\in V(G)} \fM^\ev(\cX_\bsigma(v)/b_0,\btau_v)
}
\]
with horizontal arrows the splitting maps from
Proposition~\ref{prop:splitting-map} finite and representable, is
cartesian, and it holds
\[
\delta_*\big[\scrM(X_0/b_0,\btau)\big]^\virt =
\hat\varepsilon^!\delta^\ev_*\big[\fM^\ev(\cX_0/b_0,\btau)\big].
\]
\end{corollary}

As in Corollary~\ref{Cor: decomposition into product}, a numerical formula in
terms of punctured Gromov-Witten invariants of the strata $X_\sigma$ of $X$ can
be derived assuming $\delta^\ev_*\big[\fM^\ev(\cX_0/b_0,\btau)\big]$ decomposes
into a sum of products. This is the case for example if all gluing strata
$X_{\bsigma(E)}$, $E\in E(G)$, are toric, as proved in \cite{Yixian} based on
Corollary~\ref{Cor: Gluing theorem for evaluation stacks}.


\begin{appendix}

\section{Contact orders}
\label{sec:contact}
Here we give a somewhat more sophisticated universal view on contact orders.
This was the point of view we originally planned to give, but
for most current applications, the simpler approach exposited
in \S\ref{subsec: global contact} suffices. Nevertheless, that approach
obscures some of the subtleties of contact orders, and at times
it may be worth having this more precise point of view.

For a target $X$ with fs log structure, consider the following \'etale
sheaves over $\ul{X}$:
\[
\ocM_X^{\vee} = \cHom(\ocM_{X}, \NN) \ \ \mbox{and} \ \ \ocM^*_X
= \cHom(\ocM_X,\ZZ) \cong  \cHom(\ocM^{\gp}_X,\ZZ).
\]

\begin{definition}
A \emph{family of contact orders of $X$} consists of a strict morphism $Z \to X$
and a section ${\bf u}\in \Gamma(Z,\ocM^*_Z)$ satisfying the following
condition. Let $u: \cM_Z \to \ocM_Z \stackrel{\bf u}\to \ZZ$ be the composite
homomorphism associated to ${\bf u}$. Then the map $\alpha: \cM_Z \to \cO_Z$
sends $u^{-1}(\ZZ \setminus \{0\})$ to $0$. 

We call the ideal $\cI_{\bf u} \subset \cM_Z$ generated by ${u}^{-1}(\ZZ
\setminus \{0\})$ the \emph{contact log-ideal} associated to ${\bf u}$, and
denote by $\ocI_{\bf u}$ the corresponding \emph{contact ideal} in $\ocM_Z$.
These are coherent sheaves of ideals.

The family of contact orders is said to be \emph{connected} if $Z$ is connected. 
\end{definition}

For simplicity, we will refer to ${\bf u}$ as the contact order when there is no
confusion about the strict morphism $Z \to X$. Given a family of contact orders
${\bf u}\in \Gamma(Z,\ocM^*_Z)$ of $X$, the \emph{pull-back} of ${\bf u}$ along a
strict morphism $Z' \to Z$ defines a family of contact orders ${\bf u}'\in
\Gamma(Z',\ocM^*_{Z'})$. 

\begin{example}
\label{example:contact}
To motivate this definition, consider a punctured map $f: C^{\circ} \to X$ over
$W$, and a punctured section $p \in \bp$. Take $\ul{Z}:=\ul{W}$, and give
$\ul{Z}$ the log structure given by pull-back of $\cM_X$ via $\ul{f}\circ p$, so
that $Z\rightarrow X$ is strict. Let ${\bf u}$ be the following composition 
\begin{equation}\label{equ:pun-map-family-contact}
\ocM_Z \stackrel{\bar{f}^{\flat}}{\longrightarrow} p^*\ocM_{C^{\circ}}
\longrightarrow\ocM_W\oplus\ZZ \longrightarrow \ZZ,
\end{equation}
where the middle arrow is the inclusion and the last arrow is the
projection to the second factor. 

We claim that ${\bf u}$ defines a family of contact orders of $X$. Indeed, let
$\delta \in \cM_Z$ and represent $f^\flat(\delta) =
(e_\delta,\sigma^{u_p(\delta)})$, where $\sigma$ is the element of $\cM_C$
corresponding to a local defining equation of the section $p$.

If $u_p(\delta)>0$ then
\[
\alpha_Z(\delta) = p^*\alpha_C(f^\flat(\delta))
= p^*\alpha_{C}(e_\delta)\cdot p^*\alpha_{C}(\sigma^{u_{p}(\delta)}) = 0
\]
since $p^*\alpha_C(\sigma)= 0$. 

If $u_p(\delta)<0$ then $f^\flat(\delta) \notin \cM_C$ and hence, by
Definition~\ref{def:puncturing} (2) we have $\alpha_Z(\delta) =0$.
\end{example}

The goal now is to define a universal family of contact orders for
the Artin fan $\cA_X$ of $X$.
 
\subsection{Family of contact orders of Artin cones.}
\label{subsubsec:families of contact orders}
Let {$\big(Z \to X, {\bf u}\in\Gamma( Z,\ocM^*_Z)\big)$} be a family of
contact orders of $X$. For any strict morphism $X \to Y$, ${\bf u}$ is naturally
a family of contact orders of $Y$ via the composition $Z \to X \to Y$.
Conversely, we can pull-back a contact order on $Y$ to $X$ by base-change:
{if we denote by $\cZ_X$ and $\cZ_Y$ the sheaves over $\Sch$ of families
of contact orders on $X$ and $Y$, respectively, and if we denote by $\cZ_X\to X$
the map forgetting the section ${\bf u}$, then $\cZ_X = \cZ_Y \times_YX$.} Thus
we may parameterize contact orders of the Artin fan $\cA_{X}$ instead of $X$:
{pulling back such parametrization gives a parametrization of contact
orders on $X$. This is the approach taken here, which is achieved in Proposition
\ref{prop:universal contact} and Definition~\ref{Def:connected-contact-order}.}
We first study the local case.

Consider a toric monoid $P$ with $\sigma=\Hom(P,\RR_{\ge 0})$,
$N_{\sigma}=\Hom(P,\ZZ){=P^*}$.
This gives the toric variety $\AA_\sigma
= \spec(P \stackrel{\alpha}{\rightarrow} \kk[P])$, torus
$T_\sigma:=\spec(\kk[P^{\gp}])$ and Artin cone 
\begin{equation}\label{equ:artin-cone}
\cA_{\sigma} = [\AA_\sigma/T_\sigma].
\end{equation}
Choose an integral vector $u \in N_\sigma$, which we view as $u\in \Hom(P,\ZZ)$.
Let $I_u$ be the ideal of $P$ generated by $u^{-1}(\ZZ \setminus \{0\})$. This
generates a $T_\sigma$-invariant ideal in $\kk[P]$, defining an invariant closed
subscheme $Z_{u,\sigma} \subseteq \AA_{\sigma}$ with quotient a closed
substack $\cZ_{u,\sigma} \subseteq \cA_{\sigma}$. We proceed to
construct a family of contact orders parametrized by $\cZ_{u,\sigma}$. 

For each face $\tau \prec \sigma$ (where $\prec$ denotes an inclusion of
faces) consider the prime ideal $\cK_{\tau \prec \sigma} = P \setminus
\tau^\perp$. It defines a toric stratum $Z_{\tau\prec\sigma}:=
V(\alpha(\cK_{\tau \prec \sigma})) \subseteq \AA_\sigma$ where the duals of the
stalks of $\ocM_{\cZ_{\tau\prec\sigma}}$ are identified with the faces
of $\sigma$ containing $\tau$. Note that the torus $T_\sigma$ acts on
$Z_{\tau\prec\sigma}$. Denote by $\cZ_{\tau\prec\sigma}:= [Z_{\tau\prec\sigma} /
\spec(\kk[P^{\gp}])] \subseteq \cA_{\sigma}$. 

\begin{lemma}
\label{lem:contact-strata}
We have $(\cZ_{u,\sigma})_{\red} = \bigcup_{\tau^\gp\ni u}\cZ_{\tau\prec\sigma}
\subseteq \cA_{\sigma}.$
\end{lemma}

\begin{proof}
The ideal $\sqrt{I_u}$ defines some union of strata and we identify those strata
$\cZ_{\tau\prec \sigma}$ on which it vanishes. If $u \notin \tau^\gp$ there is
an element $p \in \tau^\perp \cap P$ such that $u(p) \neq 0$. Therefore $p\in
I_u$ but the monomial $z^p$ does not vanish at the generic point of
$\cZ_{\tau\prec \sigma}$. Hence $(\cZ_{u,\sigma})_{\red}$ is contained in the
given union of strata. Conversely, if $u \in \tau^\gp$, and if $p\in
u^{-1}(\ZZ\setminus \{0\})$, then $p\notin \tau^\perp\cap P$, hence $z^p$
vanishes along $\cZ_{\tau\prec \sigma}$. Thus $\cZ_{\tau\prec\sigma}$ is
contained in $(\cZ_{u,\sigma})_{\red}$, proving the result.
\end{proof}

{Since $\ocM_{(\cZ_{u,\sigma})_\red}^*$ is the pull-back of $\ocM_{\cZ_{u,\sigma}}^*$ under the reduction $(\cZ_{u,\sigma})_\red\to \cZ_{u,\sigma}$, and reduction induces an isomorphism of \'etale sites, we have}
\[
\Gamma(\cZ_{u,\sigma}, \ocM_{\cZ_{u,\sigma}}^*)
= \Gamma((\cZ_{u,\sigma})_\red, \ocM_{(\cZ_{u,\sigma})_\red}^*).
\]
We define an element $\bfu_{u,\sigma}$ of this group by defining it on stalks in
a manner compatible with generization. For a point $z$ in the dense stratum of
$\cZ_{\tau\prec \sigma}$, with $F_{\tau}=P\cap \tau^{\perp}$, we have
$\ocM_{\cZ_{u,\sigma},z} = (P+F_{\tau}^{\gp}) / F_\tau^{\gp}$. Thus the
condition $u \in \tau^\gp$ guarantees that $u: P \to \ZZ$ descends to $u:
\ocM_{\cZ_{u,\sigma},z} \to \ZZ$. Being induced by the same element $u$, this is
compatible with generization. Note that the scheme $\cZ_{u,\sigma}$ was defined
in such a way so that $\alpha_{\cZ_{u,\sigma}}(\cI_{\bfu_{u,\sigma}})=0$, so
that $\cZ_{u,\sigma}$ acquires the structure of an idealized log stack. 

Thus $u$ defines a family of contact orders of $\cA_{\sigma}$
\begin{equation}\label{equ:cone-contact-component} {\bf u}_{u, \sigma}
\in\Gamma( \cZ_{u, \sigma} , \ocM^*_{\cZ_{u,\sigma}}). \end{equation} It is
connected since the most degenerate stratum $\cZ_{\sigma \prec \sigma}$ is
contained in the closure of $\cZ_{\tau \prec \sigma}$ for each face
$\tau$. 

\begin{lemma}\label{lem:local-contact-component}
For any connected family of contact orders ${\bf u}\in \Gamma( Z, \ocM^*_Z)$ of
$\cA_{\sigma}$, there exists a unique $u \in N_\sigma$ such that $\psi:Z \to
\cA_{\sigma}$ factors uniquely through $\cZ_{u, \sigma}$, and ${\bf
u}_{u,\sigma}$ pulls back to ${\bf u}$.
\end{lemma}

\begin{proof}
The global chart $P \to \ocM_{\cA_{\sigma}}$ over $\cA_{\sigma}$ pulls back to a
global chart $P \to \ocM_{Z}$ over $Z$. The composition $ P \longrightarrow
\ocM_{Z} \stackrel{{\bf u}}{\longrightarrow} \ZZ $ defines an integral vector $u
\in N_\sigma$. Consider the sheaf of monoid ideals $\cJ_u\subset
\cM_{\cA_{\sigma}}$ generated by $I_u$. By definition, the contact log-ideal
$\cI_{\bfu}$ is generated by $\psi^{-1}\cJ_u$. Since
$\alpha_Z(\cI_\bfu) = 0$ and since $\cJ_u$ defines
$\cZ_{u,\sigma}\subseteq \cA_\sigma$, we have the factorization $Z \to
\cZ_{u, \sigma}$ of $\psi$, with ${\bf u}$ the pull-back of ${\bf u}_{u ,
\sigma}$.
\end{proof}

We can now assemble all the $\cZ_{u,\sigma}$ by defining
\[
\cZ_{\sigma} = \coprod_{u\in N_{\sigma}} \cZ_{u,\sigma},
\]
and write $\psi_{\sigma}:\cZ_{\sigma}\rightarrow \cA_{\sigma}$ for the morphism
which restricts to the closed embedding $\cZ_{u,\sigma}\hookrightarrow
\cA_{\sigma}$ on each connected component $\cZ_{u,\sigma}$ of $\cZ_{\sigma}$.
Then the ${\bf u}_{u,\sigma}$ yield a section ${\bf u}_{\sigma} \in
\Gamma(\cZ_{\sigma}, \overline{\cM}^*_{\cZ_{\sigma}})$, giving the universal
family, over $\cZ_{\sigma},$ of contact orders of $\cA_{\sigma}$. This follows
immediately from Lemma~\ref{lem:local-contact-component} by restricting to
connected components.

\begin{proposition}
\label{prop:local-contact}
{Assume $Z$ is locally connected.}
For any family of contact orders ${\bf u}\in \Gamma( Z, \ocM^*_Z)$ of 
$\cA_{\sigma}$, $\psi:Z \to \cA_{\sigma}$ factors uniquely through 
$\cZ_{\sigma}$, and ${\bf u}_{\sigma}$ pulls back to ${\bf u}$.
\end{proposition}

\begin{corollary}
\label{cor:natural-inclusion}
If $\tau$ is a face of $\sigma$, viewing $\cA_{\tau}$ naturally as an open
substack of $\cA_{\sigma}$ we then have $\cZ_{\tau}\cong
\psi_{\sigma}^{-1}(\cA_{\tau})$, and the section ${\bf u}_{\sigma}\in
\Gamma(\cZ_{\sigma}, \overline{\cM}_{\cZ_{\sigma}}^*)$ pulls back to the section
${\bf u}_{\tau}\in \Gamma(\cZ_{\tau},\overline{\cM}_{\cZ_{\tau}}^*)$.
\end{corollary}

\begin{proof}
The statement is immediate from the universal property stated in
Proposition~\ref{prop:local-contact}.
\end{proof}

\subsection{Family of contact orders of Zariski Artin fans}
\label{sec:contact zariski}
{
We now consider the case of an Artin fan $\cA_X$. Recall that $\cA_{X}$ has an
\'etale cover by Artin cones. It was constructed in \cite[Prop.~3.1.1]{ACMW}
as a colimit of Artin cones $\cA_{\sigma}$, viewed as sheaves over $\Log$. 

\begin{definition}
\label{Def: Zariski Artin fan}
We say that the Artin fan $\cA_X$ is \emph{Zariski} if it admits a Zariski cover
by Artin cones.
\end{definition}

{A sufficient condition for $\cA_X$ to be Zariski is that $X$
is simple, because then $\cA_X$ is the Artin fan associated to the ordinary cone
complex $\Sigma(X)$ \cite[Thm.\,6.11]{CCUW}. Proposition~\ref{Prop: Log strata
in log smooth case} shows that $X$ is simple provided $X$ has Zariski log
structure and is log smooth over a simple $B$. The case $B$ a trivial log point
has previously been treated in \cite[Lem.\,2.6]{decomposition}.}

Fix a Zariski Artin fan $\cA_X$. Let $\cZ$ be the colimit of the $\cZ_{\sigma}$
viewed as sheaves over $\cA_X$. Note that $\cZ$ is obtained by gluing together
the local model $\cZ_{\sigma}$ for each Zariski open $\cA_{\sigma}
\subseteq \cA_X$ via the canonical identification given by
Corollary~\ref{cor:natural-inclusion}.\footnote{It should be possible to carry this
process out for more general Artin fans.}}

The following proposition classifies contact orders on $\cA_X$ by
globalizing Proposition~\ref{prop:local-contact}.

\begin{proposition}
\label{prop:universal contact}
There is a section ${\bf u}_{\cX}\in \Gamma(\cZ,\overline\cM_{\cZ}^*)$ making
$\cZ$ into a family of contact orders for $\cA_X$. This family of contact orders
is universal in the sense that for any family of contact orders ${\bf
u}\in\Gamma(Z,\overline\cM_Z^*)$ of $\cA_X$, $\psi:Z\rightarrow \cA_X$, there is
a unique factorization of $\psi$ through $\cZ\rightarrow\cA_X$ such that ${\bf
u}$ is the pull-back of ${\bf u}_{\cX}$.
\end{proposition}

\begin{proof} 
If $\cA_{\sigma}\rightarrow \cA_X$ is a Zariski open set, then by 
the construction of $\cZ$,
\[
\cZ \times_{\cA_X} \cA_{\sigma}=\cZ_{\sigma}.
\]
By Corollary~\ref{cor:natural-inclusion}, the sections ${\bf u}_{\sigma}$ glue
to give a section ${\bf u}_{\cX}\in \Gamma(\cZ,\overline{\cM}_{\cZ}^*)$,
yielding a family of contact orders in $\cA_X$.

Consider a family of contact orders $Z\rightarrow \cA_X$, ${\bf u}$. To show the
desired factorization, it suffices to prove the existence and uniqueness locally
on each Zariski open subset $\cA_{\sigma} \to \cA_{X}$, which follows from
Proposition~\ref{prop:local-contact}. 
\end{proof}

\begin{definition}\label{Def:connected-contact-order}
A \emph{connected contact order} for $X$ is a choice of connected component
of $\cZ$. 
\end{definition}

We end this discussion with a couple of properties of the space
$\cZ$ of contact orders.

\begin{proposition}\label{prop:irr-comp-contact}
Suppose that the Artin fan $\cA_X$ of $X$ is Zariski.
There is a one-to-one correspondence between irreducible components
of $\cZ$ and pairs $(u, \sigma)$ where $\sigma \in \Sigma(X)$ is a 
minimal cone such that $u \in \sigma^{\gp}$.
\end{proposition}

\begin{proof} 
Since $\cZ_\sigma \subseteq \cZ$ is Zariski open, an irreducible
component of $\cZ$ is the closure of an irreducible component of some
$\cZ_\sigma$, so we may assume $\cA_X = \cA_{\sigma}$. Then the statement
follows from the description of $\cZ_{u,\sigma}$ in
Lemma~\ref{lem:contact-strata}.
\end{proof}

\begin{remark}
\label{rem:mobius}
Note that if $u\in N_{\sigma}$ with $u\in \sigma$ or $-u\in \sigma$, then
$\cZ_{u,\sigma}$ is irreducible and reduced. In fact, topologically
$\cZ_{u,\sigma}$ is the closure of the stratum $\cZ_{\tau\prec\sigma}$ where
$\tau \subseteq \sigma$ is the minimal face containing $u$. Further,
the ideal generated by $u^{-1}(\ZZ\setminus\{0\})$ is precisely $P\setminus
F_{\tau}$, so that $\cZ_{u,\sigma}$ is reduced. In the case that $u\in\sigma$,
it is the contact orders associated to ordinary marked points, as developed in
\cite{Chen},\cite{AC},\cite{LogGW}.

For a simple non-reduced example let $P=\sigma_\ZZ^\vee$ be the
submonoid of $\NN^2$ generated by $(e,0),(0,e),(1,1)$ and $u:P\to\ZZ$ given by
$u(a,b)=a-b$. Then $I_u$ is generated by $(e,0),(0,e)$, and $\kk[P]/I_u\simeq
\kk[t]/(t^e)$ is non-reduced for $e>1$.

Thus the situation for more general contact orders associated to
punctures may be more complex than that for marked points.
\end{remark}

\begin{example}
Even in the Zariski case, there may be monodromy which creates a difference
between the point of view taken on contact orders in this appendix
and that taken in \S\ref{subsec: global contact}.
See Example~\ref{ex:mobius} for a simple
example with monodromy. There, taking $u=(0,1,0)$ as a tangent vector
to any of the top-dimensional cones of $\Sigma(X)$, the corresponding connected
contact order is a double cover of a one-dimensional closed subscheme of $X$.
Explicitly, $X$ contains $\ell$ strata isomorphic to $\PP^1$, forming
a cycle, i.e., a nodal elliptic curve. Then $u$ induces a family of
contact orders $Z\rightarrow X$ which is a double cover of this elliptic
curve. This curve has $2\ell$ irreducible components, in one-to-one 
correspondence with the set of pairs of the form
$(u,\sigma)$ and $(-u,\sigma)$ for $\sigma$ running over two-dimensional
cones of $\Sigma(X)$ tangent to $u$.

We may also obtain the irreducible components of $Z$ in the formalism
of \S\ref{subsec: global contact} by taking $\sigma$ to be a two-dimensional
cone of $\Sigma(X)$ for which $u$ is a tangent vector. This induces
a contact order $\bar u\in \fC_{\sigma}(X)$ distinct from the element of
$\fC_{\sigma}(X)$ induced by $-u$. In this case $Z_{\bar u}$ is just
the stratum of $X$ corresponding to $\sigma$.

{ Example~\ref{Expl: Infinite monodromy} provides an $X$ with $\cA_X$ Zariski where a similar monodromy
produces connected contact orders with an infinite number of
{irreducible} components. In this case one sees connected
components of moduli spaces of punctured maps with an infinity of irreducible
components.}
\end{example}

By the discussion in Remark~\ref{rem:mobius} above, additional hypotheses are
usually needed to obtain good control of moduli spaces of punctured maps. Here
is a simple criterion that often suffices in practice.

\begin{proposition}
\label{prop:finite-refinements}
Suppose $\ocM_X^{\gp}\otimes_{\ZZ}\QQ$ is generated by its global sections. {Assume further $X$ quasi-compact, with locally connected logarithmic strata}. Then every connected component of contact
orders of $\cA_X$ has finitely many irreducible components.
\end{proposition}

\begin{proof}
Let $V= \Gamma(\ul{X}, \ocM_X^{\gp}\otimes_{\ZZ}\RR)$, so that the induced map
$|\Sigma(X)|\rightarrow V^*$ is injective on each $\sigma \in\Sigma(X)$ as in
Proposition~\ref{Eqn: Sigma -> RR^r}. Suppose ${\bf u} \in \Gamma(\cZ,
\ocM^*_\cZ)$ is a connected component of contact orders of $\cA_X$. Denote the
composition $ V \stackrel{}{\longrightarrow}{ \ocM_{\cZ}^{\gp}\otimes_{\ZZ}\RR
\stackrel{{\bf u}}{\longrightarrow} \RR}$ by $v\in V^*$. For each irreducible
component of $\cZ$, its corresponding vector $u$ as in
Proposition~\ref{prop:irr-comp-contact} is then uniquely determined by $v$. By
Proposition~\ref{prop:irr-comp-contact} again $\cZ$ has finitely many
irreducible components, as $\Sigma(X)$ has finitely many cones by
quasi-compactness.
\end{proof}

\subsection{Connection with the global contact orders of \S\ref{subsec:
global contact}}
We continue to work with a Zariski $X$. For simplicity in this
discussion, let us also assume that $X$ is log smooth over $\Spec\kk$, so in
particular associated to any $\sigma\in \Sigma(X)$ is a closed stratum
$X_{\sigma} \subseteq X$ such that the dual cone of the stalk of
$\overline{\cM}_X$ at the generic point of $X_{\sigma}$ is $\sigma$.

In this case, $\cA_X$ is Zariski, and $\cA_{X_{\sigma}}$, the Artin fan of
$X_{\sigma}$, is also Zariski. Let $\cZ^{\sigma}\rightarrow \cA_{X_{\sigma}}$ be
the universal family of contact orders for $\cA_{X_{\sigma}}$. Further, write
$\cX_{\sigma}$ for the reduced closed stratum of $\cA_{X_{\sigma}}$
corresponding to $\sigma$. In this situation, we have:

\begin{proposition}
There is a one-to-one correspondence between $\fC_{\sigma}(X)$ and
the set of connected components of $\cZ^{\sigma}\times_{\cA_{X_{\sigma}}} 
\cX_\sigma$.
\end{proposition}

\begin{proof}
By the construction of the colimit of sets, $\fC_{\sigma}(X)$ is the
quotient of $\coprod_{\sigma\prec\sigma'\in\Sigma(X)} N_{\sigma'}$ by the
equivalence relation $\sim$ generated by the following set of relations.
Whenever given inclusions of faces $\sigma\prec\sigma'\prec\sigma''$ in
$\Sigma(X)$, one obtains an induced map $\iota_{\sigma'\sigma''}:
N_{\sigma'}\rightarrow N_{\sigma''}$. Then for $x\in N_{\sigma'}$, we have
$x\sim \iota_{\sigma'\sigma''}(x)$.

On the other hand, we may cover $\cA_{X_{\sigma}}$ with Zariski open sets
$\cA_{\sigma'}$ with $\sigma'$ running over $\sigma'\in\Sigma(X)$ with
$\sigma\prec\sigma'$. Note that by the construction of the universal contact
order of $\cA_{\sigma'}$, there is a one-to-one correspondence between
$N_{\sigma'}$ and the set of connected components of $\cZ_{\sigma'}=
\cZ^{\sigma}\times_{\cA_{X_{\sigma}}}\cA_{\sigma'}$, with $u\in N_{\sigma'}$
corresponding to $\cZ_{u,\sigma'}$. Note the same is then true of the set of
connected components of $\cZ_{\sigma'}\times_{\cA_{X_{\sigma}}} \cX_{\sigma}$,
with $u\in N_{\sigma'}$ corresponding to $\cZ_{u,\sigma'}
\times_{\cA_{X_{\sigma}}} \cX_{\sigma}$.

Define another equivalence relation $\approx$ on
$\coprod_{\sigma\prec \sigma'} N_{\sigma'}$ as follows. Suppose $u' \in
N_{\sigma'}, u''\in N_{\sigma''}$. Then $u'\approx u''$ if
$\cZ_{u',\sigma'}\times_{\cA_{X_{\sigma}}} \cX_{\sigma}$ and
$\cZ_{u'',\sigma''}\times_{\cA_{X_{\sigma}}} \cX_{\sigma}$ are open substacks of
the same connected component of $\cZ^{\sigma}\times_{\cA_{X_{\sigma}}}
\cX_{\sigma}$. The statement follows once we show that the two
equivalence relations $\sim$ and $\approx$ are equal.

Note that $\cA_{\sigma'}\cap \cA_{\sigma''}\cap\cX_{\sigma}$ may be covered with
sets $\cA_{\tau}\cap\cX_{\sigma}$ with $\tau\in\Sigma(X)$ running over those
$\tau$ with $\sigma\prec\tau\prec \sigma' \cap \sigma''$. This makes it
clear that $\approx$ is also generated by the following relations.
Suppose given $\sigma\prec\sigma'\prec\sigma''$ with $u'\in
N_{\sigma'}$, $u''\in N_{\sigma''}$. Then because of the inclusion
$\cZ_{\sigma'}\subseteq \cZ_{\sigma''}$ of Corollary~\ref{cor:natural-inclusion},
$\cZ_{\sigma',u'}\times_{\cA_{X_{\sigma}}}\cX_{\sigma}$ and $\cZ_{\sigma'',u''}
\times_{\cA_{X_{\sigma}}}\cX_{\sigma}$ may be both viewed as open substacks of
$\cZ^{\sigma}\times_{\cA_{X_{\sigma}}} (\cX_{\sigma}\times_{\cA_{X_{\sigma}}}
\cA_{\sigma''})$. If these two open substacks are not disjoint, then $u'\approx
u''$. However, it follows from Proposition~\ref{prop:local-contact} that this is
the case precisely when $\iota_{\sigma'\sigma''}(u')=u''$. Thus $\sim$ and
$\approx$ are the same equivalence relation, since they are generated by the
same set of relations.
\end{proof}


\section{Charts for morphisms of log stacks}
\label{section:charts}
We discuss here properties of charts of morphisms of algebraic log
stacks, due to a lack of a good reference. There are many standard
results involving existence and properties of charts for
morphisms between fs log schemes \'etale locally, e.g., \cite[II.2]{Ogus},
as well as local descriptions of log smooth or \'etale morphisms,
e.g., \cite[IV.3.3]{Ogus}. However, to apply these results to
morphisms of log stacks, one would need to pass to smooth neighborhoods,
which destroys any discussion of the more delicate condition of being
log \'etale. Thus it is far more convenient to think of charts as
being given by maps to toric stacks rather than toric varieties.
The results of this appendix are used in \S\ref{ss:idealized smoothness}
to describe local models for our moduli spaces of punctured maps to
Artin fans, but are also used extensively elsewhere, e.g., in \cite{GSAssoc}.

Here we fix a ground field $\kk$ of characteristic~$0$, as usual,
and all schemes and stacks are defined over $\Spec\kk$.
We define, given $P$ a fine monoid and $K\subseteq P$ a monoid ideal, 
\begin{equation}
\label{eq:artin charts}
\cA_P:=[\Spec\kk[P]/\Spec\kk[P^{\gp}]],\quad\quad\quad
\cA_{P,K}=[(\Spec \kk[P]/K)/\Spec\kk[P^{\gp}]].
\end{equation}
Here both stacks carry a canonical log structure coming from
$P$, and the second stack carries a canonical idealized log structure
induced by the monoid ideal $K$, see \cite[III.1.3]{Ogus}.

\begin{remark}
\label{rem:APghost}
If $P$ is a fine monoid, define $\overline{P} =P/P^{\times}$. Then
$\cA_{P}\cong \cA_{\overline P}$.
\end{remark}

\begin{proposition}
\label{prop:charts}
Let $f:X\rightarrow Y$ be a morphism of (idealized) fs log stacks over
$\Spec\kk$, with coherent sheaves of ideals $\cK_X$ and $\cK_Y$ in the idealized
case. Let $\ol x$ {be a geometric point of $\ul X$, $\ol y= f(\ol x)$,}
$P=\overline{\cM}_{X,\ol x}$, $Q=\overline{\cM}_{Y,\ol y}$,
($K=\overline{\cK}_{X,\ol x}$, $J=\overline{\cK}_{Y,\ol y}$ in the idealized
case). Then in the two cases, there are strict \'etale neighborhoods $X'$ and
$Y'$ of $\ol x$ and $\ol y$ respectively and commutative diagrams
\[
\xymatrix@C=15pt
{
X'\ar[d]\ar[r]&\cA_P\ar[d]\\
Y'\ar[r]&\cA_Q
}
\quad\quad\quad\quad
\xymatrix@C=15pt
{
X'\ar[d]\ar[r]&\cA_{P,K}\ar[d]\\
Y'\ar[r]&\cA_{Q,J}
}
\]
with horizontal arrows ({idealized}) strict. If further
$Y$ is already equipped with a strict morphism $Y\rightarrow \cA_Q$,
we may take $Y=Y'$.
\end{proposition}

\begin{proof}
If $Y$ is equipped with a strict morphism $Y\rightarrow\cA_Q$, we take $Y=Y'$.
Otherwise, there is a tautological morphism $Y\rightarrow \Log_{\kk}$. By
\cite[Cor.\,5.25]{LogStack}, $\Log_{\kk}$ has a strict \'etale cover by stacks
of the form $\cA_Q$ for various monoids $Q$. Thus we may choose a strict \'etale
morphism $\cA_Q\rightarrow\Log_{\kk}$ whose image contains the image of $\ol
y$, and take $Y'$ to be the \'etale neighborhood $Y\times_{\Log_\kk} \cA_Q$ of
$\ol y$. If the image of $\ol y$ in $\cA_Q$ is not the deepest stratum of
$\cA_Q$, we may replace $\cA_Q$ with a Zariski open subset of the form
$\cA_{Q'}$ where $Q'$ is a localization of $Q$ along some face and such that
$\ol y$ maps to the deepest stratum of $\cA_{Q'}$. By Remark \ref{rem:APghost},
$\cA_{Q'} \cong \cA_{Q'/(Q')^{\times}}$, so we may assume that $Q=\overline
{\cM}_{Y,\ol y}$.

Let $X''=X\times_Y Y'$ be the corresponding \'etale neighborhood of $\ol x$.
Similarly, we have a tautological strict morphism $X''\rightarrow \Log_{Y'}$.
Now $\Log_{Y'}$ can be covered by strict \'etale morphisms of the form
$\cA_P\times_{\cA_Q} Y'\rightarrow\Log_{Y'}$ for various $P$ such that the
projection to $\cA_P$ is strict, again by \cite[Cor.\,5.25]{LogStack}. Here we
range over fs monoids $P$ and morphisms $\theta:Q\rightarrow P$. Take
$X'=X''\times_{\Log_{Y'}} (\cA_P\times_{\cA_Q} Y')$ for a suitable choice of $P$
and $\theta:Q\rightarrow P$ so that $X'$ is an \'etale neighborhood of $\ol
x$. The projection of this stack to $\cA_P$ then yields the desired strict
morphism $X'\rightarrow \cA_P$ making the diagram commutative. As before, we can
pass to a Zariski open substack of $\cA_P$ to be able to assume that
$P=\overline{\cM}_{X,\ol x}$.

In the idealized case, the morphisms $X'\rightarrow\cA_P$, $Y'\rightarrow
\cA_Q$ factor through $\cA_{P,K}$ and $\cA_{Q,J}$ respectively, and
the factored morphisms are {idealized strict}.
\end{proof}

\begin{lemma}
\label{lem:idealized etale}
Let $\theta:Q\rightarrow P$ be a morphism of fs monoids, $J\subseteq Q$,
$K\subseteq P$ with $\theta(J)\subseteq K$. Then the induced morphism
$\theta:\cA_{P,K}\rightarrow\cA_{Q,J}$ is idealized log \'etale.
\end{lemma}

\begin{proof}
The morphism $\theta$ is clearly locally of finite presentation, so
we need only verify the formal lifting criterion. Suppose given a diagram
\[
\xymatrix@C=30pt
{
T_0\ar[r]^{g_0}\ar[d]_i & \cA_{P,K}\ar[d]\ar[d]^{\theta} 
\ar[r]&\cA_P\ar[d]^{\theta}\\
T\ar[r]_g\ar@{-->}[ru]_{g'}&\cA_{Q,J}\ar[r]&\cA_Q
}
\]
Here $i$ is a strict and {idealized strict} closed immersion with ideal
sheaf having square zero, and the right-hand horizontal arrows are strict, but
not {idealized strict}, closed immersions, with the right-hand square
commutative. We wish to show there is a unique $g'$ making the diagram commute.

By \cite[Cor.\,5.23]{LogStack}, $\cA_P\rightarrow\cA_Q$ is log \'etale. Thus
forgetting the idealized structure on $T$, we obtain a unique morphism
$h:T\rightarrow\cA_P$ in the above diagram making everything commute. Let
$\cK_{T_0}$, $\cK_T$ be the coherent sheaves of monoid ideals for $T_0$ and $T$
respectively giving the idealized structure. Note $\cK_{T_0}$ is the pull-back
of a sheaf of ideals $\overline{\cK}_{T_0} \subseteq \overline{\cM}_{T_0}$ under
the projection $\cM_{T_0} \rightarrow\overline{\cM}_{T_0}$, and similarly for
$T$. By strictness, $\overline{\cK}_{T_0}=\overline{\cK}_T$. Since $\bar
g_0^{\flat}(K)\subseteq \overline{\cK}_{T_0}$, as $g_0$ is an idealized
morphism, we have by commutativity that $\bar h^{\flat}(K) \subseteq
\overline{\cK}_T$. Hence $\alpha_T$ vanishes on any lift to $\cM_T$ of an
element $\bar h^{\flat}(k)$ for $k\in K$. It then follows that $h$ factors
through the closed immersion $\cA_{P,K}\rightarrow \cA_P$, and this
factorization yields the unique lifting $g'$.
\end{proof}

\begin{proposition}
\label{prop:etale charts}
With the hypotheses of Proposition~\ref{prop:charts}, suppose in addition that
$f$ is log smooth (resp. log \'etale, idealized log smooth, idealized log
\'etale). Then in the non-idealized case, the induced morphism $X'\rightarrow
Y'\times_{\cA_Q} \cA_P$ is smooth (resp.\ \'etale) and in the idealized case,
the induced morphism $X'\rightarrow Y' \times_{\cA_{Q,J}} \cA_{P,K}$ is smooth
(resp.\ \'etale).
\end{proposition}

\begin{proof}
In the non-idealized case, \cite[Thm.\,4.6]{LogStack}, shows that $X\rightarrow
Y$ is log smooth (\'etale) if and only if the tautological morphism
$\ul{X}\rightarrow \Log_Y$ is smooth (\'etale). It follows immediately by
base-change from the construction of the proof of Proposition~\ref{prop:charts}
that, if $f$ is log smooth (\'etale), the morphism $X''\rightarrow Y'$ is log
smooth (\'etale) and hence $\ul{X}''\rightarrow \Log_{Y'}$ is smooth (\'etale).
Thus by another base-change, we see that the projection $X'\rightarrow
Y'\times_{\cA_Q}\cA_P$ is smooth (\'etale).

The idealized case requires a little bit more work because the analogous
statement for idealized log smooth (\'etale) morphisms does not seem to appear
in the literature. First, by \cite[Lem.\,4.8]{LogStack}, $X''\rightarrow
\Log_{Y'}$ is locally of finite presentation, as $X\rightarrow Y$ is locally of
finite presentation, being idealized log \'etale, and thus $X'\rightarrow
Y'\times_{\cA_Q} \cA_P$ is locally of finite presentation. However, as in the
proof of Proposition~\ref{prop:charts}, $X'\rightarrow Y'\times_{\cA_Q} \cA_P$
factors through the closed immersion $Y'\times_{\cA_{Q,J}}
\cA_{P,K}\hookrightarrow Y'\times_{\cA_Q} \cA_P$, and thus $X'\rightarrow
Y'\times_{\cA_{Q,J}} \cA_{P,K}$ is also locally of finite presentation. So we
just need to show the formal lifting criterion, i.e., given a diagram
\begin{equation}
\label{diag:first lifting}
\vcenter{\xymatrix@C=30pt
{
\ul{T}_0\ar[r]^{\ul{g}_0} \ar[d]_{\ul{i}} & \ul{X}'\ar[d]_{\ul{f}'}\\
\ul{T}\ar[r]_>>>>{\ul{g}}\ar@{-->}[ru] & \ul{Y'\times_{\cA_{Q,J}} \cA_{P,K}}
}}
\end{equation}
where $i$ is a closed immersion with ideal of square zero, there is, \'etale
locally on $\ul{T}_0$, a dotted line as indicated, unique in the \'etale case.
Give $\ul{T_0}$ and $\ul{T}$ the idealized log structure making all arrows in
the above square strict and {idealized strict}. Via composition of $g$ with the
projection to $Y'$, we obtain a diagram
\[
\xymatrix@C=30pt
{
T_0\ar[r]^{g_0}\ar[d]_i & X'\ar[d]\\
T\ar[r]_{g''}\ar@{-->}_{g'}[ru] & Y'
}
\]
Formal idealized log smoothness (idealized log \'etaleness) then implies,
\'etale locally, a (unique) lift $g'$. It is then sufficient to show that
$f'\circ g'$ coincides with $g$ in \eqref{diag:first lifting}. However, by Lemma
\ref{lem:idealized etale}, the projection $Y'\times_{\cA_{Q,J}} \cA_{P,K}
\rightarrow Y'$ is idealized log \'etale, and hence by uniqueness in the formal
lifting criterion for idealized log \'etale morphisms, $f'\circ g'=g$.
\end{proof}


{
\section{Functorial tropicalization and the category of points}
\label{App: Functorial tropicalization}}

Various definitions of tropicalization in logarithmic geometry are available in
the literature \cite{Kato,LogGW,ACP,Ulirsch,decomposition,CCUW}. The
purpose of this appendix is to spell out the construction of tropicalization as
a functor from the category of fine log algebraic stacks to the category of
generalized cone complexes generalizing \cite[Prop.\,6.3]{Ulirsch} to cases with
monodromy, and closer in spirit to \cite[App.\,B]{LogGW}. This refines the
discussion in \cite[\S2.1]{decomposition}.

We adopt the definition from \cite[II.1]{KKMS}, \cite[\S2]{Payne},
\cite[\S2.2]{ACP} of a \emph{generalized cone complex} $\Sigma$ as a topological
space $|\Sigma|$ together with a \emph{presentation} given by a homeomorphism
with the colimit in the category of topological spaces of a diagram in $\Cones$
with all arrows face morphisms. Here we use the topology induced by embedding a
cone $\sigma$ into its vector space $N_\sigma\otimes_\ZZ\RR$. For any cone
$\sigma$ in a presentation we always include all face embeddings
$\tau\to\sigma$ in the diagram. The \emph{strata} of $|\Sigma|$ are the images
of the interiors of cones from the presentation. We consider generalized cone
complexes up to equivalence generated by adding more cones to a presentation. A
morphism of cone complexes $\Sigma\arr\Sigma'$ is given by a continuous map
$|\Sigma|\arr|\Sigma'|$ that locally lifts to a morphism of diagrams of
presentations. Unlike the cited references, we do not impose any finiteness
conditions since we want to admit situations with infinitely many strata.
\medskip

\subsection{Tropicalization of fine log schemes}

We begin by recalling the definition of the category of geometric points
$\Pt(X)$ of a scheme $X$ with arrows defined by specialization, following
\cite[VIII.7]{SGA4}, see also \cite[Sect.\,0GJ2]{stacks-project}. An object in
$\Pt(X)$ is a morphism $\ol x: \Spec \kappa\to X$ with $\kappa=\kappa(\ol x)$ an
algebraically closed field. Given $\ol x$ we have the associated local scheme
$X(\ol x)=\Spec\cO_{X,\ol x}$. A \emph{specialization} arrow $\ol x\to \ol y$ is
an $X$-morphism $\Spec \kappa(\ol x)\to X(\ol y)$ or, equivalently by
\cite[VIII.7, Prop.\,7.4]{SGA4}, an $X$-morphism $X(\ol x)\arr X(\ol y)$.

Composition with a morphism $f:X\to Y$ defines a functor
\[
f_*:\Pt(X)\arr \Pt(Y)
\]
compatible with composition, so $\Pt$ is a functor from the category of
schemes to the category of categories $\Cat$.

For each \'etale sheaf of sets $\cF$ on $X$, a specialization arrow $\ol x\to\ol
y$ in $\Pt(X)$ induces a generization map\footnote{We prefer ``generization
map'' over the common ``specialization map'' in this context since the map goes
from the stalk at the more special point to the stalk at the more generic
point.}
\begin{equation}
\label{Eqn: Generization hom}
\cF_{\ol y}\to \cF_{\ol x}.
\end{equation}
This assignment is compatible with morphisms of sheaves. Thus if $\Sh(X_\et)$
denotes the \'etale topos of $X$, we obtain a functor
\begin{equation}
\label{Eqn: Stalks functor}
\Stalks: \Sh(X_\et)\arr \Func(\Pt(X)\op,\Cat)
\end{equation}
associating to an \'etale sheaf its functor of stalks, a diagram in
$\Cat$ indexed by $\Pt(X)\op$. We emphasize that the generization homomorphism
\eqref{Eqn: Generization hom} does not only depend on $\ol x,\ol y$, but on the
choice of specialization arrow $\ol x\arr\ol y$.
\medskip

\begin{example}
Let $C$ be the nodal cubic. If $\ol\eta$ denotes a geometric generic point and
$\ol x$ a geometric point over the node, there are two different $C$-morphisms
\[
\Spec\kappa(\ol \eta)\arr X(\ol x)
\]
that reflect the specialization along the two branches of $C$ at $\ol x$. This
statement can most easily be seen by going over to the usual two-fold \'etale
cover $\pi:\tilde C\arr C$, and observing that each of the two lifts $\tilde{\ol
x}\in\Pt(\tilde C)$ of $\ol x$ has generization homomorphisms to both lifts of
$\ol\eta$.
\end{example}

Charts for the log structure define a locally finite stratification of $\ul X$
with a \emph{stratum} a maximal connected locally closed subset
$Z\subseteq\big|\ul X\big|$ with $\ocM_X|_Z$ locally constant. Denote by
$\operatorname{Strata}(X)$ the set of strata of $X$. For each $Z\in\Strata(X)$
choose a geometric point $\ol x=\ol x_Z$ of $Z$ and define
\begin{equation}
\label{Eqn: cone sigma_Z}
\sigma_Z=\Hom\big(\ocM_{X,\ol x},\RR_{\ge0}\big)\in\Cones.
\end{equation}
Different choices of $\ol x$ lead to isomorphic $\sigma_Z$, but the isomorphism
is only unique up to the monodromy action of the \'etale fundamental group
$\pi_1(Z,\ol x)$ of the stratum on $\ocM_{X,\ol x}$. More precisely, since the
automorphism group of a fine monoid is finite, 
arguing with
\cite[Lem.\,0DV5]{stacks-project} shows the following. There exists a finite
connected \'etale Galois cover
\[
f:\tilde Z\arr Z
\]
with $f^{-1}\ocM_X$ a constant sheaf. Lifting $\ol x$ to $\tilde Z$ yields an
isomorphism
\begin{equation}
\label{Eqn: trivialize ocM}
\Gamma(\tilde Z,f^{-1}\ocM_X)\stackrel{\simeq}{\longrightarrow} \ocM_{X,\ol x}
\end{equation}
by restriction. Now by definition, $\pi_1(Z,\ol x)$ acts on $f$, and the induced
action on $\Gamma(\tilde Z,f^{-1}\ocM_X)$ by pull-back corresponds to
the action of $\pi_1(Z,\ol x)$ on $\ocM_{X,\ol x}$ via \eqref{Eqn: trivialize
ocM}. The minimal choice of $f$ with $f^{-1}\ocM_X$ a constant sheaf
has connected $\tilde Z$ and is a Galois cover. Moreover, in the minimal case,
the action of $\pi_1(Z,\ol x)$ on $\ocM_{X,\ol x}$ factors over a faithful
action of the Galois group $\Aut(\tilde Z/Z)$.

For each stratum $Z$ with chosen geometric point $\ol x=\ol x_Z$ denote by
\begin{equation}
\label{Eqn: G_Z}
G_Z\subseteq \Aut(\ocM_{X,\ol x})
\end{equation}
the image of the monodromy action of $\pi_1(Z,\ol x)$ on $\ocM_{X,\ol
x}$. By the previous discussion, $G_Z\simeq\Aut(\tilde Z/Z)$ for any
minimal connected Galois cover $f:\tilde Z\to Z$ with
$f^{-1}\ocM_X$ a constant sheaf.

Now if $W\in\Strata(X)$ is another stratum, and $\ol w$ is a geometric point of
$W\cap\cl(Z)$, there exists a geometric point $\ol\eta$ of $Z$ and a
specialization arrow $\chi:\ol\eta\arr\ol w$
\cite[Sect.\,0BUP]{stacks-project}, hence a generization homomorphism
$\ocM_{X,\ol w}\arr \ocM_{X,\ol\eta}$. Since $\ocM_X$ is locally constant on the
strata there are also isomorphisms
\begin{equation}
\label{Eqn: choice of reference isos}
\ocM_{X,\ol w}\stackrel{\simeq}{\longrightarrow} \ocM_{\ol x_W},\quad
\ocM_{X,\ol\eta}\stackrel{\simeq}{\longrightarrow} \ocM_{\ol x_Z},
\end{equation}
for $\ol x_Z,\ol x_W$ the chosen reference points for the two strata. These
isomorphisms are unique up to composing with elements of $G_W$ and $G_Z$,
respectively. We call any morphism
\begin{equation}
\label{Eqn: sigma_Z -> sigma_W}
\iota:\sigma_Z\arr \sigma_W
\end{equation}
obtained by applying $\Hom(\,\cdot\,,\RR_{\ge0})$ to any of the compositions
\[
\ocM_{\ol x_W}\stackrel{\simeq}{\longrightarrow} \ocM_{X,\ol w}
\stackrel{\chi}{\longrightarrow}
\ocM_{X,\ol\eta}\stackrel{\simeq}{\longrightarrow} \ocM_{\ol x_Z}
\]
a \emph{specialization morphism} or \emph{specialization arrow}. Note that
$\iota$ also depends on the choice of $\ol w$, and hence the actions of $G_Z$
and $G_W$ on the set of specialization arrows may not be transitive. For $Z=W$
the set of specialization arrows equals $G_Z=G_W$.

If $f:X\arr Y$ is a morphism of fine log schemes, $Z\in\Strata(X)$ and $f(\ol
x_Z)$ a geometric point of $Z'\in\Strata(Y)$, then $\ol f^\flat:
f^{-1}\ocM_Y\arr \ocM_X$ together with a choice of isomorphism $\ocM_{Y,f(\ol
x_Z)}\simeq \ocM_{Y,\ol x_{Z'}}$ in \eqref{Eqn: choice of reference isos}
defines a morphism
\begin{equation}
\label{Eqn: Trop(f) on cones}
\varphi:\sigma_Z=
\Hom(\ocM_{X,\ol x_Z},\RR_{\ge0})\arr \Hom(\ocM_{Y,f(\ol x_Z)},\RR_{\ge0})
\stackrel{\simeq}{\longrightarrow} \Hom(\ocM_{Y,\ol x_{Z'}},\RR_{\ge0})
= \sigma_{Z'}
\end{equation}
between $\sigma_Z$ and $\sigma_{Z'}$ in $\Cones$. Note these are not in general
face morphisms. The set of all such arrows is compatible with specialization in
the sense that if $\iota:\sigma_Z\arr\sigma_W$ is a specialization morphism
\eqref{Eqn: sigma_Z -> sigma_W} in $X$ then there exists a specialization
morphism $\iota':\sigma_{Z'}\arr \sigma_{W'}$ in $Y$ and morphisms
$\varphi:\sigma_Z\arr\sigma_{Z'}$, $\psi:\sigma_W\arr\sigma_{W'}$ as in
\eqref{Eqn: Trop(f) on cones} making the following diagram commute:
\begin{equation}
\label{Eqn: Morphisms of specialization arrows}
\vcenter{\xymatrix@C=30pt{
\sigma_Z\ar[r]^{\iota}\ar[d]^{\varphi}&\sigma_W\ar[d]^{\psi}\\
\sigma_{Z'}\ar[r]^{\iota'}&\sigma_{W'}.
}}
\end{equation}

We are then in position to define the tropicalization of $X$ as a generalized
cone complex.

\begin{definition}
\label{Def: Tropicalization}
Let $X=(\ul X,\cM_X)$ be a fine log scheme. The \emph{tropicalization}
$\Sigma(X)$ of $X$ is the generalized cone complex defined by the diagram
in $\Cones$ with one object $\sigma_Z$ from \eqref{Eqn: cone
sigma_Z} for each stratum $Z\subset X$ and face morphisms $\sigma_Z\arr\sigma_W$
the set of specialization morphisms from \eqref{Eqn: sigma_Z -> sigma_W}.

A morphism $f:X\arr Y$ of fine log schemes induces the morphism
\[
\Sigma(f): \Sigma(X)\arr\Sigma(Y)
\]
defined by all arrows $\varphi:\sigma_Z\to\sigma_{Z'}$ as in \eqref{Eqn: Trop(f)
on cones}.
\end{definition}

Note that diagrams of specialization arrows as in \eqref{Eqn: Morphisms of
specialization arrows} show that the map of topological spaces
$|\Sigma(X)|\to|\Sigma(Y)|$ is well-defined and continuous, and that it lifts
locally to a morphism of presentations. Thus $\Sigma(f)$ indeed is a morphism of
generalized cone complexes.

We need to check that our definition of tropicalization does not depend on the
choices of a geometric point $\ol x_Z$ for each stratum $Z$ of $X$. 

\begin{lemma}
\label{Lem: Trop is independent of choices}
The definition of tropicalization in Definition~\ref{Def: Tropicalization} is
independent of choices.
\end{lemma}

\begin{proof}
Let $Z$ be a logarithmic stratum of $X$ and $\ol x'_Z$ another choice of
geometric point. Since $\ocM_X|_Z$ is locally constant there exists an
isomorphism
\[
\varphi:\sigma_Z= \Hom(\ocM_{X,\ol x _Z},\RR_{\ge0})
\stackrel{\simeq}{\longrightarrow}
\sigma'_Z=\Hom(\ocM_{X,\ol x' _Z},\RR_{\ge0})
\]
that is unique up to the action of $\pi_1(Z)$ on $\sigma_Z$. Replacing
$\sigma_Z$ by $\sigma'_Z$ and all arrows involving $\sigma_Z$ by composition
with $\varphi$ or $\varphi^{-1}$ as appropriate, gives an alternative
presentation of $|\Sigma(X)|$ as a colimit of a diagram in $\Cones$. By
construction, both diagrams are locally isomorphic, and hence they
lead to the same generalized cone complex. This argument is local to each
geometric point, thus also applies to any two different sets of choices of
geometric points.
\end{proof}

We finally check functoriality of this notion of tropicalization.

\begin{proposition}
\label{Prop: Functoriality of tropicalization of log schemes}
If $f:X\to Y$ and $g:Y\to Z$ are morphisms of fine log schemes then
$\Sigma(g\circ f)=\Sigma(g)\circ\Sigma(f)$.
\end{proposition}

\begin{proof}
Given a specialization morphism $\iota:Z\arr W$ of strata of $X$ there exist two
commutative diagrams of the form \eqref{Eqn: Morphisms of specialization arrows}
with horizontal arrows specialization morphisms $\iota':Z'\arr W'$ and
$\iota'':Z''\arr W''$ of strata in $Y$ and $Z$, respectively. The two small
commutative squares now define the local liftings of $\Sigma(f)$ and $\Sigma(g)$
to presentations, while their composition defines the lifting of $\Sigma(g\circ
f)$. The result is now obvious.
\end{proof}

\begin{remark}
A canonical and obviously functorial definition of $\Sigma(X)$ runs as follows.
The composition of the functor $\Stalks$ in \eqref{Eqn: Stalks functor} with
$\Hom(.,\RR_{\ge0})$ defines a diagram
\begin{equation}
\label{Eqn: Brutal tropicalization}
\Pt(X)\op\arr \Cones
\end{equation}
with all morphisms face inclusions. The reasoning in the proof of
Lemma~\ref{Lem: Trop is independent of choices} shows that the associated
generalized cone complex is canonically isomorphic to $\Sigma(X)$. We preferred
to base our definition on the more explicit treatment with one cone for each
stratum.
\end{remark}

\begin{remark}
One might think that a slightly refined definition could also give a functorial
notion of tropicalization as a diagram of cones associated to strata. This is,
however, not the case. The problem appears already with locally constant sheaves
in the \'etale topology, which can not be described by groupoids of sets
obtained from the associated representations of the \'etale
fundamental group. The \'etale fundamental group of a scheme $X$ depends on the
choice of a geometric point and is otherwise only defined up to non-unique
isomorphism. Thus a functorial definition would have to involve at least a
skeleton of $\Pt(X)$, and hence completely loses the combinatorial flavor of
tropicalization.
\end{remark}

\subsection{Tropicalization of fine log algebraic stacks}
Now let $X$ be a fine log algebraic stack, with $\cM_X$ and $\ocM_X$ sheaves in
the lisse-\'etale topology. To define the tropicalization $\Sigma(X)$ let
\[
h:U\arr X
\]
be a strict smooth surjection from a log scheme. Then $U\times_X U$ is a scheme
that comes with two projections to $U$. Tropicalizing defines a double arrow of
generalized cone complexes
\begin{equation}
\label{Eqn: double arrow of tropicalizations}
\Sigma(U\times_X U)\rightrightarrows \Sigma(U).
\end{equation}
For a geometric point $\ol x$ of $U\times_X U$, composition
with the two projections defines two geometric points $\ol x_1,\ol x_2$ of $U$.
Since both projections $U\times_X U\to U$ are strict, we have two isomorphisms
\begin{equation}
\label{Eqn: double arrow of stalks of ocM}
\ocM_{U,\ol x_i}\arr\ocM_{U\times_X U,\ol x},\quad i=1,2.
\end{equation}
These isomorphisms induce an equivalence relation on $\Strata(U)$, and provide
isomorphisms between stalks of $\ocM_U$ at pairs of geometric points in
equivalent strata. The quotient $\Strata(U)/\sim$ can easily be seen to be
independent of the choice of smooth cover $U\to X$, and in fact defines the set
$\Strata(X)$ of strata of the log algebraic stack $X$.

To define the tropicalization $\Sigma(X)$, we add $\Hom(.,\RR_{\ge0})$ of the
isomorphisms in~\eqref{Eqn: double arrow of stalks of ocM} to the set of arrows
in the diagram defining $\Sigma(U)$.

\begin{definition}
\label{Def: Tropicalization of log algebraic stack}
The tropicalization $\Sigma(X)$ of the fine log algebraic stack $X$ is the
generalized cone complex defined by the diagram of $\Sigma(U)$ with the
added isomorphisms induced by the tropicalization of~\eqref{Eqn:
double arrow of stalks of ocM}.
\end{definition}

Restricting the diagram defining $\Sigma(X)$ to one cone for each stratum of $X$
gives an alternative presentation with index category $\Strata(X)$.

We need to check independence of our definition of $\Sigma(X)$ from choices.

\begin{lemma}
The definition of $\Sigma(X)$ is independent of the choice of strict smooth
cover $U\arr X$.
\end{lemma}

\begin{proof}
It suffices to consider the composition of $U\arr X$ with a strict smooth
surjection $V\arr U$. We obtain the following commutative diagram of strict
smooth surjections of log schemes:
\begin{equation}
\label{Eqn: V times_X V versus U times_X U}
\vcenter{\xymatrix{
V\times_X V\ar[r]\ar@<-.5ex>[d]\ar@<.5ex>[d]&
U\times_X U\ar@<-.5ex>[d]\ar@<.5ex>[d]\\
V\ar[r]&U
}}
\end{equation}
Now all arrows are surjective on geometric points. Since smooth maps are open,
all arrows are also surjective on the set of generizations. Thus each cone and
arrow of $\Sigma(V)$ maps isomorphically to a cone or arrow of $\Sigma(U)$, and
each cone or arrow of $\Sigma(U)$ arises as an image. Moreover, if two cones
$\sigma_1,\sigma_2$ in $ \Sigma(U)$ belong to the same stratum in $X$, that is,
are isomorphic images of a cone $\sigma$ in $\Sigma(U\times_X U)$ appearing from
a geometric point in $U\times_X U$, then lifting this geometric point to
$V\times_X V$ provides a cone $\tilde\sigma$ in $\Sigma(V\times_X V)$ mapping to
cones $\tilde\sigma_1,\tilde\sigma_2$ in $\Sigma(V)$. The tropicalization
of~\eqref{Eqn: V times_X V versus U times_X U} now shows that the diagram of
cones
\[
\xymatrix{
\tilde\sigma_1\ar[d]&\tilde\sigma\ar[r]\ar[d]\ar[l]&\tilde\sigma_2\ar[d]\\
\sigma_1&\sigma\ar[r]\ar[l]&\sigma_2
}
\]
commutes up to composing the lower horizontal arrows with isomorphisms in
$\Sigma(U)$.

Taken together we see that the diagram defining $\Sigma(X)$ from $V\arr X$ just
adds a number of isomorphic cones to the diagram defining $\Sigma(X)$ from
$U\arr X$. Thus the corresponding generalized cone complexes are equivalent.
\end{proof}

The proof of functoriality of this notion of tropicalization now follows by
local lifting to a presentation as in Proposition~\ref{Prop: Functoriality of
tropicalization of log schemes}. We omit the details.

\begin{proposition}
If $f:X\to Y$ and $g:Y\to Z$ are morphisms of fine log algebraic stacks then
$\Sigma(g\circ f)=\Sigma(g)\circ\Sigma(f)$.
\end{proposition}

\subsection{Tropicalization in the log smooth case}
We end this section with some facts on logarithmic strata and tropicalization in
the Zariski log smooth case.

\begin{lemma}
\label{Lem: Log smooth implies irreducible log strata}
Let $f:X\arr B$ be a log smooth morphism of fine log schemes. Assume that $B$ is
locally Noetherian with geometrically unibranch logarithmic strata. Then the
logarithmic strata of $X$ are irreducible and geometrically unibranch.
\end{lemma}

\begin{proof}
First note that $X$ is locally Noetherian since $B$ is locally Noetherian and
$f$ is locally of finite presentation by the definition of log smoothness. Thus
a locally irreducible connected subset of $|X|$ is irreducible. It thus suffices
to show the stronger statement that each logarithmic stratum $Z$ of $X$ is
geometrically unibranch.

Let $z\in|Z|$ and $Z_B\subseteq B$ the logarithmic stratum containing $\ul
f(z)$. Being geometrically unibranch is a local property that is stable under
\'etale morphisms. By \cite[Thm.\,IV.3.3.1]{Ogus} we may thus replace $X$ and
$B$ by \'etale neighborhoods of $z$ and $\ul f(z)$ to obtain a commutative
diagram
\[
\xymatrix{
X\ar[r]^(.3)g\ar[rd]_f& B\times_{A_Q} A_P \ar[r]^(.7)k\ar[d]&
A_P\ar[d]^{A_\theta}\\
&B\ar[r]^h& A_Q
}
\]
with $A_P=\Spec\ZZ[P]$, $A_Q=\Spec\ZZ[Q]$, $A_\theta$ the morphism induced by a
homomorphism $\theta:Q\arr P$ of fine monoids, all horizontal arrrows strict,
the square cartesian, $g$ \'etale, and $h$ a neat chart at $z$. Thus
$Z_B=(h^{-1}(O))_\red$, where $O\subseteq A_Q$ is the closed torus orbit defined
by the monoid ideal $Q\setminus\{0\}$.

Since $k\circ g$ is strict, the composition $Z\arr B\times_{A_Q} A_P\arr A_P$
factors over the inclusion of a logarithmic stratum $Z_P\subseteq A_P$. Now
toric morphisms respect the decomposition into logarithmic strata. Thus
$\ul{A_\theta}(Z_P)$ is contained in a logarithmic stratum of $A_Q$. But $\ul
h(\ul f(z))\in \ul{A_\theta}(Z_P)$, so this latter stratum is the closed stratum
$O\subseteq A_Q$.

This shows that $g(Z)$ is contained in
\[
Z_B\times_{A_Q} A_P= Z_B\times_{A_Q} Z_P= Z_B\times_{O} Z_P=Z_B\times_\ZZ Z_P.
\]
Since $Z_B\times_\ZZ Z_P$ has constant ghost sheaf $\ocM$ it follows that $Z=
g^{-1}(Z_B\times_\ZZ Z_P)$, and hence $Z$ is \'etale over $Z_B\times_\ZZ Z_P$.
Here we are using that the preimage of a reduced subscheme under an \'etale
morphism remains reduced \cite[Prop.\,0250]{stacks-project}. Finally,
$Z_B\times_\ZZ Z_P$ is geometrically unibranch by the assumption on the strata
of $B$. This shows that $Z$ is geometrically unibranch at $z$.
\end{proof}

\begin{proposition}
\label{Prop: Log strata in log smooth case}
Let $f:X\arr B$ be a log smooth morphism of fine log schemes with $B$ locally
Noetherian and with geometrically unibranch logarithmic strata. Assume that $B$
is simple, that is, $\Sigma(B)$ is a cone complex rather than a generalized cone
complex, and that the log structure of $X$ is defined in the Zariski topology.
Then $X$ is simple as well, and the logarithmic strata of $X$ are irreducible
and geometrically unibranch.
\end{proposition}

\begin{proof}
Lemma~\ref{Lem: Log smooth implies irreducible log strata} shows the statement
on the log strata of $X$. Thus each logarithmic stratum $Z$ has a unique generic
point $\eta_Z$. It is then obvious that there is an arrow $\sigma_Z\arr
\sigma_W$ if and only if $\eta_W\in\cl(\eta_Z)$. Moreover, since $\cM_X$ is a
sheaf on the Zariski site, $\sigma_Z\arr \sigma_W$ must then be the dual of the
generization homomorphism $\cM_{X,\eta_W}\arr \cM_{X,\eta_Z}$. Thus there is at
most one such arrow, and hence $\Sigma(X)$ is a cone complex.
\end{proof}

\end{appendix}



\begin{thebibliography}{KKMSD}

\bibitem[AC]{AC}
D.~Abramovich and Q.~Chen:
\emph{Stable logarithmic maps to Deligne-Faltings pairs~II},
Asian J.\ Math.~\textbf{18} (2014), 465--488.

\bibitem[ACGS]{decomposition}
D.~Abramovich, Q.~Chen, M.~Gross and B.~Siebert:
\emph{Decomposition of degenerate Gromov-Witten invariants},
Compos.\ Math.~\textbf{156} (2020), 2020--2075.

\bibitem[ACMUW]{ACMUW}
D.~{Abramovich}, Q.~{Chen}, S.~{Marcus}, M.~Ulirsch and J.~{Wise}:
\emph{Skeletons and fans of logarithmic structures},
in: Nonarchimedean and tropical geometry, 287--336,
Simons Symp., Springer 2016.

\bibitem[ACMW]{ACMW}
D.~Abramovich, Q.~Chen, S.~Marcus and J.~Wise:
\emph{Boundedness of the space of stable logarithmic maps}, August 2014,
J.\ Eur.\ Math.\ Soc.~(JEMS)~\textbf{19} (2017), 2783--2809.

\bibitem[ACP]{ACP}
D.~Abramovich, L.~Caporaso and S.~Payne:
\emph{The tropicalization of the moduli space of curves},
Ann.\ Sci.\ \'Ec.\ Norm.\ Sup\'er.~(4)~\textbf{48} (2015), 765--809.

\bibitem[AW]{AW}
D.~Abramovich and J.~Wise:
\emph{Birational invariance in logarithmic Gromov-Witten theory},
Compos.\ Math.~\textbf{154} (2018), 595--620.

\bibitem[Ar]{Arguez}
H.~Arg\"uz,
\emph{A tropical and log geometric approach to algebraic structures in the ring
of theta functions},
\href{https://arxiv.org/pdf/1712.10260}{arXiv:1712.10260} [math.AG].

\bibitem[AG]{AG}
H.~Arg\"uz and M.~Gross:
\emph{The higher dimensional tropical vertex},
Geom.\ Topol.~\textbf{26} (2022), 2135--2235.

\bibitem[Ba]{Barrott}
L.~Barrott:
\emph{Logarithmic Chow theory},
\href{https://arxiv.org/pdf/1810.03746}{arXiv:1810.03746} [math.AG]

\bibitem[BNR22]{BNR1} L.~Battistella, N.~Nabijou and D.~Ranganathan:
\emph{Gromov-Witten theory via roots and logarithms},
\href{https://arxiv.org/abs/2203.17224}{arXiv:2203.17224} [math.AG].

\bibitem[BNR23]{BNR2} L.~Battistella, N.~Nabijou and D.~Ranganathan:
\emph{In preparation}.

\bibitem[BF]{Behrend-Fantechi}
K.~Behrend and B.~Fantechi:
\emph{The intrinsic normal cone},
Invent.\ Math.~\textbf{128} (1997), 45--88.

\bibitem[CCUW]{CCUW}
R.~Cavalieri, M.~Chan, M.~Ulirsch and J.~Wise:
\emph{A moduli stack of tropical curves},
Forum Math.\ Sigma~\textbf{8} (2020), Paper No.\,e23, 93pp.

\bibitem[Ch]{Chen}
Q.~Chen:
\emph{Stable logarithmic maps to Deligne-Faltings pairs~I},
Ann.\ of Math.~\textbf{180} (2014), 455--521

\bibitem[CJR1]{CJR19P}
Q.~Chen, F.~Janda and Y.~Ruan:
\emph{The logarithmic gauged linear sigma model},
Invent.\ Math.~\textbf{225}, 1077--1154, 2021.

\bibitem[CJR2]{CJR20P}
Q.~Chen, F.~Janda and Y.~Ruan:
\emph{Punctured logarithmic R-maps},
\href{https://arxiv.org/pdf/2208.04519}{arXiv:2208.04519} [math.AG]

\bibitem[CJR3]{CJR21P}
Q.~Chen, F.~Janda and Y.~Ruan:
\emph{The structural formulae of virtual cycles in logarithmic Gauged Linear Sigma Models},
in preparation.

\bibitem[CJRS+]{CJRS20P}
Q.~Chen, F.~Janda, R.~Pandharipande, A.~Pixton,  Y.~Ruan, A.~Sauvaget,
J.~Schmitt and D.~Zvonkine:
\emph{Witten's spin class and moduli spaces of canonical divisors},
in preparaton.

\bibitem[Co]{Conrad}
B.~Conrad:
\emph{Grothendieck duality and base change},
Lecture Notes in Math.~1750,
Springer~2000.

\bibitem[FWY1]{FWY1}
H.~Fan, L.~Wu and F.~You:
\emph{Structures in genus-zero relative Gromov--Witten theory},
J.\ Topol.~\textbf{13} (2020), 269--307.

\bibitem[FWY2]{FWY2}
H.-H.~Tseng and F.~You:
\emph{A Gromov-Witten theory for simple normal-crossing pairs without
log geometry},
\href{https://arxiv.org/pdf/2008.04844}{arXiv:2008.04844} [math.AG].

\bibitem[GP1]{GanatraPomerleano1}
S.~Ganatra and D.~Pomerleano:
\emph{A log PSS morphism with applications to Lagrangian embeddings},
J.\ Topol.~\textbf{14} (2021), 291--368.

\bibitem[GP2]{GanatraPomerleano2}
S.~Ganatra and D.~Pomerleano:
\emph{Symplectic cohomology rings of affine varieties in the topological limit},
Geom.\ Funct.\ Anal.~\textbf{30} (2020), 334--456.

\bibitem[G]{G22} M.~Gross:
\emph{Remarks on gluing punctured logarithmic maps},
\href{https://arxiv.org/pdf/2306.02661}{arXiv:2306.02661} [math.AG].

\bibitem[GHKS]{GHKS-cubic} M.~Gross, P.~Hacking, S.~Keel and B.~Siebert:
\emph{The mirror of the cubic surface}, 
London Math.\ Soc.\ Lecture Note Ser.~\textbf{478}, 150--182,
Cambridge University Press 2022.

\bibitem[GHS]{GStheta} M.~Gross, P.~Hacking and B.~Siebert:
\emph{Theta functions on varieties with effective anticanonical class},
to appear in Mem.\ Amer.\ Math.\ Soc.,
preprint~\texttt{arXiv:1601.07081 [math.AG]}, 123pp.

\bibitem[GS1]{GSAnnals} M.~Gross and B.~Siebert: \emph{From affine
geometry to complex geometry},
Ann.\ of Math.~{\bf 174}, (2011), 1301--1428.

\bibitem[GS2]{LogGW}
M.~Gross and B.~Siebert:
\emph{Logarithmic Gromov-Witten invariants},
J.\ Amer.\ Math.\ Soc.~\textbf{26} (2013), 451--510.

\bibitem[GS3]{Utah}
M.~Gross and B.~Siebert:
\emph{Intrinsic mirror symmetry and punctured Gromov-Witten invariants},
Algebraic geometry: Salt Lake City 2015, 199--230,
Proc.\ Sympos.\ Pure Math., 97.2, AMS~2018.

\bibitem[GS4]{GSAssoc}
M.~Gross and B.~Siebert:
\emph{Intrinsic mirror symmetry},
\href{https://arxiv.org/pdf/1909.07649}{arXiv:1909.07649} [math.AG].

\bibitem[GS5]{GSWallStructures}
M.~Gross and B.~Siebert:
\emph{Canonical wall structures and intrinsic mirror symmetry},
in preparation.

\bibitem[GJR1]{GJR17P}
S.~Guo, F.~Janda and Y.~Ruan:
\emph{A mirror theorem for genus two Gromov--Witten invariants of quintic
threefolds},
\href{https://arxiv.org/pdf/1709.07392}{arXiv:1709.07392} [math.AG].

\bibitem[GJR2]{GJR18P}
S.~Guo, F.~Janda and Y.~Ruan:
\emph{Structure of higher genus Gromov-Witten invariants of quintic 3-folds},
\href{https://arxiv.org/pdf/1812.11908}{arXiv:1812.11908} [math.AG].

\bibitem[GV]{GV05}
T.~Graber and R.~Vakil.
\newblock Relative virtual localization and vanishing of tautological classes
  on moduli spaces of curves.
\newblock {\em Duke Math. J.}, 130(1):1--37, 2005.

\bibitem[Ha]{Hartshorne}
R.~Hartshorne:
\emph{Residues and duality},
Lecture Notes in Math.~20,
Springer 1966

\bibitem[Hu]{HuybrechtsFM}
D.~Huybrechts:
\emph{Fourier-Mukai transforms in algebraic geometry},
Oxford University Press~2006.

\bibitem[Il]{Illusie}
L.~Illusie:
\emph{Complexe cotangent et d\'eformations I}.
Lecture Notes Math.~239, Springer~1971.

\bibitem[Jo]{Johnston}
S.~Johnston:
\emph{Birational Invariance in Punctured Log Gromov-Witten Theory},
\href{https://arxiv.org/pdf/2210.06079}{arXiv:2210.06079} [math.AG].

\bibitem[KV]{KV}
M.~Kapranov and K.~Vasserot:
\emph{The cohomological Hall algebra of a surface and factorization cohomology}, 
\href{https://arxiv.org/pdf/1901.07641}{arXiv:1901.07641} [math.AG].

\bibitem[Kf1]{FKato}
F.~Kato:
\emph{Log smooth deformation and moduli of log smooth curves},
Internat.\ J.\ Math.~\textbf{11} (2000), 215--232.

\bibitem[KK]{Kato}
K.~Kato:
\emph{Toric singularities},
Amer.\ Journal Math.\ {\bf 116} (1994) 1073--1099.

\bibitem[KKMS]{KKMS}G.\ Kempf, F.\ Knudsen, D.\ Mumford and B.\ Saint-Donat:
\emph{Toroidal Embeddings I}, Springer, LNM 339, 1973.


\bibitem[KLR]{KimLhoRuddat}
B.~Kim, H.~Lho and H.~Ruddat,
\emph{The degeneration formula for stable log maps},
Manuscripta Math.~\textbf{170} (2023), 63--107.

\bibitem[Kr]{Kresch}
A.~Kresch:
\emph{Cycle groups for {A}rtin stacks},
Invent.\ Math.~\textbf{138} (1999), 495--536.

\bibitem[La]{Ladkani}
S.~Ladkani:
\emph{On derived equivalences of categories of sheaves over finite posets},
J.\ Pure Appl.\ Algebra~\textbf{212} (2008), 435--451.

\bibitem[LR]{LR}
A.-M.~Li and Y.~Ruan:
\emph{Symplectic surgery and Gromov-Witten invariants of Calabi-Yau 3-folds},
Invent.\ Math., {\bf 145} (2001), 151--218.

\bibitem[Li1]{JunLi1}
J.~Li,
\emph{Stable morphisms to singular schemes and relative stable morphisms},
J.\ Differential Geom~\textbf{57} (2000) 509--578.

\bibitem[Li2]{JunLi2}
J.\ Li:
\emph{A degeneration formula of {GW}-invariants},
J.\ Differential Geom.~\textbf{60} (2002), 199--293.



\bibitem[Ma]{Mano}
C.~Manolache:
\emph{Virtual pull-backs},
J.\ Algebraic Geom.~\textbf{21} (2012), no.~2, 201--245.

\bibitem[Mo]{Mochizuki}
S.~Mochizuki:
\emph{The geometry of the compactification of the {H}urwitz scheme},
Publ.\ Res.\ Inst.\ Math.\ Sci.~\textbf{31} (1995), 355--441.

\bibitem[Og]{Ogus}
A.~Ogus:
\emph{Lectures on logarithmic algebraic geometry},
Cambridge University Press, 2018.

\bibitem[Ol1]{LogStack}
M.~Olsson:
\emph{Logarithmic geometry and algebraic stacks}, Ann.\ Sci.
\'Ecole Norm.\ Sup. (4)~\textbf{36} (2003), no.~5, 747--791.

\bibitem[Ol2]{LogCot} M.\ Olsson:
\emph{The logarithmic cotangent complex},
Math.\ Ann.~\textbf{333}  (2005), 859--931.

\bibitem[Pa]{Parker}
B.~Parker:
\emph{Gromov-Witten invariants of exploded manifolds},
\href{https://arxiv.org/pdf/1102.0158}{arXiv:1102.0158} [math.SG],

\bibitem[Pa1]{Parker: log}
B.~Parker:
\emph{Log geometry and exploded manifolds},
Abh.\ Math.\ Semin.\ Univ.\ Hambg.~\textbf{82} (2012), 43--81. 

\bibitem[Pa2]{Parker: reg}
B.~Parker:
\emph{Holomorphic curves in exploded manifolds: regularity},
Geom.\ Topol.~\textbf{23} (2019), 1621--1690.

\bibitem[Pa3]{Parker: cmp}
B.~Parker:
\emph{Holomorphic curves in exploded manifolds: compactness},
\href{https://arxiv.org/pdf/0911.2241}{arXiv:0911.2241} [math.SG].

\bibitem[Pa4]{Parker: Kuranishi}
B.~Parker:
\emph{Holomorphic curves in exploded manifolds: Kuranishi structure},
\href{https://arxiv.org/pdf/1301.4748}{arXiv:1301.4748} [math.SG].

\bibitem[Pa5]{Parker: vfc}
B.~Parker:
\emph{Holomorphic curves in exploded manifolds: virtual fundamental class},
Geom.\ Topol.~\textbf{23} (2019), 1877--1960.

\bibitem[Pa6]{Parker: gluing}
B.~Parker:
\emph{Tropical gluing formulae for Gromov-Witten invariants},
\href{https://arxiv.org/pdf/1703.05433}{arXiv:1703.05433} [math.SG].

\bibitem[PPZ]{PPZ16P}
R.~Pandharipande, A.~Pixton and D.~Zvonkine:
\emph{Tautological relations via $r$-spin structures},
J.\ Algebraic Geom.~\textbf{28} (2019), 439--496.

\bibitem[Pas]{Pascaleff}
J.~Pascaleff:
\emph{On the symplectic cohomology of log Calabi-Yau surfaces},
Geom.\ Topol.~\textbf{23} (2019), 2701--2792.

\bibitem[Pay]{Payne}
\emph{Toric vector bundles, branched covers of fans, and the resolution property},
J.\ Algebraic Geom.~\textbf{18} (2009), 1--36.

\bibitem[Ra]{Ranganathan}
D.~Ranganathan:
\emph{Logarithmic Gromov-Witten theory with expansions},
Algebr.\ Geom.~\textbf{9} (2022), 714--761.

\bibitem[Se]{Seidel}
P.~Seidel:
\emph{A biased view of symplectic cohomology},
Current developments in mathematics, 2006, 211--253, Int. Press~2008. 

\bibitem[SGA4]{SGA4} \emph{Théorie des topos et cohomologie étale des schémas},
Séminaire de Géométrie Algébrique du Bois-Marie 1963–1964 (SGA 4, Vol.2).
Directed by M.~Artin, A.~Grothendieck and J.L.~Verdier. Lecture Notes in
Mathematics, Vol. 270. Springer~1972.


\bibitem[SP]{stacks-project}
The {Stacks Project Authors}:
\emph{\itshape Stacks project},
\url{http://stacks.math.columbia.edu}, 2017.

\bibitem[Te]{Tehrani}
M.F.~Tehrani:
\emph{Pseudoholomorphic curves relative to a normal crossings symplectic
divisor: compactification},
Geom.\ Topol.~\textbf{26} (2022), 989--1075.

\bibitem[Ul]{Ulirsch}
M.~Ulirsch,
\emph{Non-Archimedean geometry of Artin fans.}
Adv. Math.  345  (2019), 346--381.


\bibitem[Wi]{Wise-minimality}
J.~Wise:
\emph{Moduli of morphisms of logarithmic schemes},
Algebra Number Theory~\textbf{10} (2016), 695--735.

\bibitem[Wu]{Yixian}
Y.~Wu:
\emph{Splitting of Gromov-Witten invariants with toric gluing strata},
preprint~2020.

\bibitem[Yu]{Yu}
T.Y.~Yu:
\emph{Enumeration of holomorphic cylinders in log Calabi-Yau surfaces. II.
Positivity, integrality and the gluing formula},
Geom.\ Topol.~\textbf{25} (2021), 1--46.

\end{thebibliography}
\end{document}